\documentclass[
  10pt,
  a4paper,
  bibliography=totoc,  %
  DIV=12,
  BCOR=20mm,             %
  enabledeprecatedfontcommands
]{scrbook}
\usepackage{etex}

\setkomafont{title}{\normalfont\bfseries}
\setkomafont{section}{\normalfont\bfseries}
\setkomafont{subsection}{\normalfont\itshape}
\setkomafont{subsubsection}{\normalfont\itshape}

\usepackage[table]{xcolor}

\usepackage[ngerman,english]{babel}
\usepackage[utf8]{inputenc}
\usepackage[T1]{fontenc}

\usepackage[top=35mm, bottom=35mm, width=150mm, bindingoffset=15mm]{geometry}

\usepackage{morewrites}

\usepackage{amsmath}
\usepackage{amssymb}
\usepackage{amsfonts}
\usepackage{amsthm}
\usepackage{mathtools}
\usepackage{thmtools}

\usepackage{wasysym}

\usepackage{chngcntr}

\makeatletter
\renewcommand\thmt@listnumwidth{4.3em}
\makeatother

\usepackage[export]{adjustbox}

\usepackage[inline]{enumitem}
\usepackage{textcomp}

\usepackage{cite}

\usepackage{pdfpages}

\usepackage{booktabs}
\usepackage{pgf}
\usepackage{tikz}
\usepackage{pgfplots}
\pgfplotsset{compat=1.14}
\usetikzlibrary{plotmarks}
\usetikzlibrary{shapes.misc}
\usetikzlibrary{matrix,positioning}
\usetikzlibrary{arrows.meta}
\usetikzlibrary{external}
\usepackage{subcaption}
\usepackage{multirow}
\usepackage{tabu}

\usepackage{placeins}

\usepackage{bibentry}
\nobibliography*

\usepackage[notlot,notlof,numbib]{tocbibind}

\usepackage{titleps}

\usepackage[
  algo2e,
  lined,
  boxed,
  commentsnumbered,
  linesnumbered,
  ruled,
  algochapter
]{algorithm2e}
\SetKwProg{Fn}{Function}{\string:}{end}
\SetKwProg{Find}{find}{\string:}{}
\SetKw{st}{such that}

\usepackage[frozencache]{minted}

\counterwithin{listing}{chapter}
\usepackage[Sonny]{fncychap} %

\newminted{cpp}{
  frame=lines,
  obeytabs=true,
  tabsize=2,
  baselinestretch=1.0,
  linenos=true,
  mathescape,
  fontsize=\footnotesize
}
\newmintedfile[cppfile]{cpp}{
  frame=lines,
  obeytabs=true,
  tabsize=2,
  baselinestretch=1.0,
  linenos=true,
  mathescape,
  fontsize=\footnotesize
}
\newminted{python}{
  frame=lines,
  obeytabs=true,
  tabsize=4,
  baselinestretch=1.0,
  linenos=true,
  mathescape,
  fontsize=\footnotesize
}
\newmintedfile{python}{
  frame=lines,
  obeytabs=true,
  tabsize=4,
  baselinestretch=1.0,
  linenos=true,
  mathescape,
  fontsize=\footnotesize
}

\usepackage[
  bookmarks,
  raiselinks,
  pageanchor,
  hyperindex,
  colorlinks,
  citecolor=black,
  linkcolor=black,
  urlcolor=black,
  filecolor=black,
  menucolor=black
]{hyperref}

\usepackage[capitalise,noabbrev]{cleveref}

\usepackage{calc}

\usepackage[
  acronym,            %
  xindy
]{glossaries}

\definecolor{Xred}{RGB}{228,26,28}
\definecolor{Xblue}{RGB}{14,102,174}
\definecolor{Xgreen}{RGB}{255,127,0}
\definecolor{Xorange}{RGB}{77,175,74}
\definecolor{Xviolet}{RGB}{152,78,163}

\author{Andreas Buhr\\\footnotesize{geboren in München}}
\title{Towards Automatic and Reliable\\
Localized Model Order Reduction}
\subtitle{%
Inauguraldissertation zur Erlangung des Doktorgrades der Naturwissenschaften
im Fachbereich Mathematik und Informatik der Mathematisch-Naturwissenschaftlichen
Fakult\"at der Westf\"alischen Wilhelms-Universit\"at M\"unster.
}
\date{M\"unster\\$-2019-$}

\hypersetup{
  pdftitle={Towards Automatic and Reliable Localized Model Order Reduction, Andreas Buhr},
  pdfauthor={Andreas Buhr (andreas@andreasbuhr.de)}
}

\newcommand{\draftmarker}{}

\newpagestyle{main}[\small]{
\headrule
\sethead[\chaptername~\thechapter\quad\chaptertitle][][]
{}{}{\thesection\quad\sectiontitle}
\setfoot[\thepage][\draftmarker][]{}{\draftmarker}{\thepage}}

\renewpagestyle{plain}[\small]{
\sethead[][][]{}{}{}
\setfoot[\thepage][\draftmarker][]{}{\draftmarker}{\thepage}
}

\renewpagestyle{empty}[\small]{
\sethead[][][]{}{}{}
\setfoot[][\draftmarker][]{}{\draftmarker}{}
}

\newpagestyle{bib}[\small]{
\headrule
\sethead[Bibliography][][]{}{}{Bibliography}
\setfoot[\thepage][\draftmarker][]{}{\draftmarker}{\thepage}}

\newpagestyle{intro}[\small]{
\headrule
\sethead[Introduction][][]{}{}{Introduction}
\setfoot[\thepage][\draftmarker][]{}{\draftmarker}{\thepage}}

\newpagestyle{gloss}[\small]{
\headrule
\sethead[Notations and Abbreviations][][]{}{}{Notations and Abbreviations}
\setfoot[\thepage][\draftmarker][]{}{\draftmarker}{\thepage}}

\newpagestyle{mytoc}[\small]{
\headrule
\sethead[Contents][][] {}{}{Contents}
\setfoot[\thepage][\draftmarker][]{}{\draftmarker}{\thepage}}

\newpagestyle{justline}[\small]{
\headrule
\sethead[][][] {}{}{}
\setfoot[\thepage][\draftmarker][]{}{\draftmarker}{\thepage}}

\newcommand{\norm}[1]{{\left\lVert{#1}\right\rVert}}
\newcommand{\abs}[1]{{\left|{#1}\right|}}
\DeclareMathOperator{\spn}{span}

\newcommand\Cpp{C\nolinebreak[4]\hspace{-.05em}\raisebox{.4ex}{\relsize{-3}{\textbf{++}}}\xspace}

\newcommand\myunderline{}
\DeclareRobustCommand\myunderline[1]{\underline{#1}}

\providecommand{\eprint}[2][]{\url{#2}}

\newtheorem{theorem}{Theorem}[section]
\newtheorem{proposition}[theorem]{Proposition}
\newtheorem{lemma}[theorem]{Lemma}
\newtheorem{definition}[theorem]{Definition}
\newtheorem{corollary}[theorem]{Corollary}
\newtheorem{remark}[theorem]{Remark}

\AfterEndEnvironment{theorem}{\noindent\ignorespaces}
\AfterEndEnvironment{proposition}{\noindent\ignorespaces}
\AfterEndEnvironment{lemma}{\noindent\ignorespaces}
\AfterEndEnvironment{definition}{\noindent\ignorespaces}
\AfterEndEnvironment{corollary}{\noindent\ignorespaces}
\AfterEndEnvironment{remark}{\noindent\ignorespaces}
\AfterEndEnvironment{proof}{\noindent\ignorespaces}

\makeglossaries
\newacronym{esmsfem}{ESMsFEM}{Edge Spectral Multiscale Finite Element Method}
\newacronym{rfb}{RFB}{Residual Free Bubbles Method}
\newacronym{srbe}{SRBE}{Seamless Reduced Basis Element Method}
\newacronym{rbhdd}{RBHDD}{Reduced Basis for Heterogeneous Domain Decomposition}
\newacronym{rdf}{RDF}{Reduced Basis - Domain Decomposition - Finite Element Method}
\newacronym{fetidp}{FETI-DP}{finite element tearing and interconnecting - dual primal}
\newacronym{bddc}{BDDC}{balancing domain decomposition by constraints}
\newacronym{dtn}{DtN}{Dirichlet-to-Neumann}
\newacronym{rbhm}{RBHM}{Reduced Basis Hybrid Method}
\newacronym{dfg}{DFG}{Deutsche Forschungsgemeinschaft}
\newacronym{scm}{SCM}{Successive Constraint Method}
\newacronym{dof}{DoF}{Degree of Freedom}
\newacronym{pod}{POD}{Proper Orthogonal Decomposition}
\newacronym{svd}{SVD}{Singular Value Decomposition}
\newacronym{pcb}{PCB}{Printed Circuit Board}
\newacronym{pde}{PDE}{Partial Differential Equation}
\newacronym{gmsfem}{GMsFEM}{Generalized Multiscale Finite Element Method}
\newacronym{cemgmsfem}{CEM-GMsFEM}{Constraint Energy Minimization - Generalized Multiscale Finite Element Method}
\newacronym{iid}{i.i.d.}{independent and identically distributed}
\newacronym{cdf}{CDF}{cumulative distribution function}
\newacronym{fem}{FEM}{Finite Element Method}
\newacronym{lrbms}{LRBMS}{Localized Reduced Basis Multiscale Method}
\newacronym{gfem}{GFEM}{Generalized Finite Element Method}
\newacronym{scrbe}{SCRBE}{Static Condensation Reduced Basis Element Method}
\newacronym{cms}{CMS}{Component Mode Synthesis}
\newacronym{acms}{ACMS}{Approximate Component Mode Synthesis}
\newacronym{prscrbe}{PR-SCRBE}{Port Reduced Static Condensation Reduced Basis Element Method}
\newacronym{msfem}{MsFEM}{Multiscale Finite Element Method}
\newacronym{lod}{LOD}{Local Orthogonal Decomposition Method}
\newacronym{dgrbe}{DGRBE}{Discontinuous Galerkin Reduced Basis Element}
\newacronym{rbe}{RBE}{Reduced Basis Element}
\newacronym{rb}{RB}{Reduced Basis}
\newacronym{dd}{DD}{Domain Decomposition}
\newacronym{hmm}{HMM}{Heterogeneous Multiscale Method}
\newacronym{vmm}{VMM}{Variational Multiscale Method}
\newacronym{deim}{DEIM}{Discrete Empirical Interpolation Method}
\newacronym{pou}{POU}{Partition of Unity}
\newacronym{faas}{FaaS}{Function as a Service}
\newacronym{arbilomod}{ArbiLoMod}{method designed for \underline{Arbi}trary \underline{Lo}cal \underline{Mod}ifications}
\newacronym{dg}{DG}{Discontinuous Galerkin}
\newacronym{dune}{DUNE}{Distributed and Unified Numerics Environment}
\newacronym{fe}{FE}{Finite Element}
\newacronym{fea}{FEA}{Finite Element Analysis}
\newacronym{pymor}{pyMOR}{pyMOR -- Model Order Reduction with Python}
\newacronym{mathematica}{Mathematica\textsuperscript{\textregistered}}{Mathematica\textsuperscript{\textregistered} by Wolfram}
\newacronym{pec}{PEC}{Perfect Electric Conductor}
\newacronym{cnr}{CNR}{Conditional Numerical Reproducibility}
\newacronym{cbwr}{CBWR}{Conditional BitWise Reproducibility}
\newacronym{mkl}{Intel MKL}{Intel Math Kernel Library}

\newglossaryentry{friedrichs}
{
  name={\ensuremath{c_F}},
  sort={friedrich},
  description={constant in Friedrich's inequality}
}
\newglossaryentry{cest}
{
  name={\ensuremath{c_\mathrm{est}}},
  sort={cest},
  description={constant in probabilistic a posteriori norm estimator}
}
\newglossaryentry{ceff}
{
  name={\ensuremath{c_\mathrm{eff}}},
  sort={ceff},
  description={effectivity constant of probabilistic a posteriori norm estimator}
}
\newglossaryentry{epsalgofail}
{
  name={\ensuremath{\varepsilon_\mathrm{algofail}}},
  sort={epsalgofail},
  description={maximum probability of randomized algorithm to fail}
}
\newglossaryentry{epstestfail}
{
  name={\ensuremath{\varepsilon_\mathrm{testfail}}},
  sort={epstestfail},
  description={maximum probability of one randomized test to fail}
}
\newglossaryentry{algotol}
{
  name={\ensuremath{\normalfont{\texttt{tol}}}},
  sort={algotol},
  description={target error of \cref{algo:adaptive_range_approximation}}
}
\newglossaryentry{boundary}
{
  name={\ensuremath{\Gamma}},
  sort={boundary},
  description={boundary of $\Omega$}
}
\newglossaryentry{emptyset}
{
  name={\ensuremath{\emptyset}},
  sort={emptyset},
  description={the empty set}
}
\newglossaryentry{transfer_source}
{
  name={\ensuremath{S}},
  sort={sourcespace},
  description={source space of the transfer operator \gls{transferoperator}}
}
\newglossaryentry{tracespace}
{
  name={\ensuremath{S_i}},
  sort={intermediate space},
  description={intermediate space of localized training configuration}
}
\newglossaryentry{transfer_range}
{
  name={\ensuremath{R}},
  sort={rangespace},
  description={range space of the transfer operator \gls{transferoperator}}
}
\newglossaryentry{transfer_rrange}
{
  name={\ensuremath{\widetilde{R}^n}},
  sort={rangespace reduced},
  description={reduced range space of the transfer operator \gls{transferoperator} of size $n$}
}
\newglossaryentry{dimension_transfer_source}
{
  name={\ensuremath{N_S}},
  sort={dimension sourcespace},
  description={dimension of \gls{transfer_source}}
}
\newglossaryentry{dimension_transfer_range}
{
  name={\ensuremath{N_R}},
  sort={dimension rangespace},
  description={dimension of  \gls{transfer_range}}
}
\newglossaryentry{source_inner}
{
  name={\ensuremath{\myunderline{M}_S}},
  sort={inner product in S},
  description={inner product matrix in \gls{transfer_source}}
}
\newglossaryentry{source_inner_eigenvector}
{
  name={\ensuremath{\myunderline{m}_{S,i}}},
  sort={inner product in S eigenvector},
  description={eigenvector of inner product matrix in \gls{transfer_source}}
}
\newglossaryentry{range_inner}
{
  name={\ensuremath{{\myunderline{M}}_{R}}},
  sort={inner product in R},
  description={inner product matrix in \gls{transfer_range}}
}
\newglossaryentry{transferop}
{
  name={\ensuremath{T_i}},
  sort={transferoperatori},
  description={affine linear compact operator, mapping from \gls{tracespace}
  to \gls{lfspace}}
}
\newglossaryentry{transferopl}
{
  name={\ensuremath{T_i^l}},
  sort={transferoperatoril},
  description={linear part of \gls{transferop}}
}
\newglossaryentry{transferopa}
{
  name={\ensuremath{T_i^a}},
  sort={transferoperatoria},
  description={affine part of \gls{transferop}}
}
\newglossaryentry{transferopmu}
{
  name={\ensuremath{T_{\mu,i}}},
  sort={transferoperatormui},
  description={parameterized affine linear compact operator, mapping from \gls{tracespace}
  to \gls{lfspace}}
}
\newglossaryentry{transferoperator}
{
  name={\ensuremath{T}},
  sort={transferoperator},
  description={affine linear compact transfer operator, mapping from some outer space \gls{transfer_source}
  to some inner space \gls{transfer_range}}
}
\newglossaryentry{transferoperatorl}
{
  name={\ensuremath{T^l}},
  sort={transferoperatorl},
  description={linear part of \gls{transferoperator}}
}
\newglossaryentry{transferoperatora}
{
  name={\ensuremath{T^a}},
  sort={transferoperatora},
  description={affine part of \gls{transferoperator}}
}
\newglossaryentry{number_testvectors}
{
  name={\ensuremath{n_t}},
  sort={number of test vectors},
  description={number of test vectors in randomized range finder}
}
\newglossaryentry{probof}
{
  name={\ensuremath{\mathrm{Pr}}},
  sort={probability},
  description={operator giving the probability of a statement. E.g.: $\mathrm{Pr}( 1 > 0 ) = 1$}
}
\newglossaryentry{normaldistribution}
{
  name={\ensuremath{\mathcal{N}}},
  sort={normaldistribution},
  description={normal distribution}
}
\newglossaryentry{neigh}
{
  name={\ensuremath{\mathcal{N}}},
  sort={neighborhood},
  description={neighborhood}
}
\newglossaryentry{distribution}
{
  name={\ensuremath{\mathcal{D}}},
  sort={distribution},
  description={'distribution of' operator}
}
\newglossaryentry{identity}
{
  name={\ensuremath{\mathbf{1}}},
  sort={identity},
  description={identity operator}
}
\newglossaryentry{cdfhn}
{
  name={\ensuremath{\mathrm{CDF_{hn}}}},
  sort={cdfhn},
  description={cumulative distribution function of the half-normal distribution}
}
\newglossaryentry{erf}
{
  name={\ensuremath{\mathrm{erf}}},
  sort={erf},
  description={error function: $\gls{erf}(x) := \frac{2}{\sqrt{\pi}} \int_0^x e^{-t^2} \mathrm d t$}
}
\newglossaryentry{R}
{
  name={\ensuremath{\mathbb{R}}},
  sort={reals},
  description={set of real numbers}
}
\newglossaryentry{C}
{
  name={\ensuremath{\mathbb{C}}},
  sort={complex},
  description={set of complex numbers}
}
\newglossaryentry{K}
{
  name={\ensuremath{\mathbb{K}}},
  sort={field},
  description={field, either $\gls{R}$ or $\gls{C}$}
}
\newglossaryentry{domain}
{
  name={\ensuremath{\Omega}},
  sort={omega},
  description={calculation domain, assumed to be Lipschitz}
}
\newglossaryentry{subdomain}
{
  name={\ensuremath{\omega_i}},
  sort={omegai},
  description={subdomain of localizing space decomposition}
}
\newglossaryentry{traindomain}
{
  name={\ensuremath{\omega_i^*}},
  sort={omegaistar},
  description={subdomain of localized training configuration}
}
\newglossaryentry{coarsemeshentity}
{
  name={\ensuremath{\mathcal{E}}},
  sort={coarse mesh entity},
  description={coarse mesh entity}
}
\newglossaryentry{subdomainnol}
{
  name={\ensuremath{\omega_i^\mathrm{nol}}},
  sort={omega nol},
  description={subdomain of non overlapping space decomposition}
}
\newglossaryentry{subdomainol}
{
  name={\ensuremath{\omega_i^\mathrm{ol}}},
  sort={omega ol},
  description={subdomain of overlapping space decomposition}
}
\newglossaryentry{subdomainext}
{
  name={\ensuremath{\omega_i^\mathrm{ext}}},
  sort={omega ext},
  description={subdomain of wirebasket space decomposition, extension domain}
}
\newglossaryentry{domains}
{
  name={\ensuremath{I}},
  sort={domains},
  description={domains associated: $I_X$ is the set of domain numbers containing all domains which intersect with $X$}
}
\newglossaryentry{numdomainsnol}
{
  name={\ensuremath{N_\mathrm{nol}}},
  sort={number of subdomains non overlapping},
  description={number of subdomain of non overlapping domain decomposition}
}
\newglossaryentry{numdomainsol}
{
  name={\ensuremath{N_\mathrm{ol}}},
  sort={number of subdomains overlapping},
  description={number of subdomain of overlapping domain decomposition}
}
\newglossaryentry{numcoarseentities}
{
  name={\ensuremath{N_{\mathcal{E}}}},
  sort={coarse mesh entities},
  description={number of inner coarse mesh entities}
}
\newglossaryentry{trainingspace}
{
  name={\ensuremath{\mathrm{Tr}}},
  sort={training space},
  description={training space}
}
\newglossaryentry{extensionspace}
{
  name={\ensuremath{E}},
  sort={extension space},
  description={extension space}
}
\newglossaryentry{couplingspace}
{
  name={\ensuremath{C}},
  sort={coupling space},
  description={coupling space}
}
\newglossaryentry{dimension}
{
  name={\ensuremath{d}},
  sort={dimension},
  description={dimension of the problem, 2 or 3}
}
\newglossaryentry{enrichmentfraction}
{
  name={\ensuremath{r}},
  sort={enrichment fraction},
  description={enrichment fraction}
}
\newglossaryentry{fspace}
{
  name={\ensuremath{V}},
  sort={functionspace},
  description={function space}
}
\newglossaryentry{fspaceh}
{
  name={\ensuremath{V_h}},
  sort={functionspace discrete},
  description={discrete function space}
}
\newglossaryentry{fspacedg}
{
  name={\ensuremath{V^\mathrm{DG}}},
  sort={functionspace dg},
  description={DG function space}
}
\newglossaryentry{lfspace}
{
  name={\ensuremath{V_i}},
  sort={functionspace local},
  description={local function space}
}
\newglossaryentry{lfspacetrain}
{
  name={\ensuremath{V\big|_{\omega_i^*}}},
  sort={functionspace restricted training},
  description={function space, restricted to training domain}
}
\newglossaryentry{lfspacepou}
{
  name={\ensuremath{V_i^\mathrm{pou}}},
  sort={functionspace local partition of unity},
  description={local function space of partition of unity decomposition}
}
\newglossaryentry{lfspaceres}
{
  name={\ensuremath{V_i^\mathrm{res}}},
  sort={functionspace local restriction},
  description={local function space of restriction decomposition}
}
\newglossaryentry{lfspacebasic}
{
  name={\ensuremath{V_i^\mathrm{basic}}},
  sort={functionspace local basic},
  description={local function space of basic decomposition}
}
\newglossaryentry{lfspacewb}
{
  name={\ensuremath{V_i^\mathrm{wb}}},
  sort={functionspace local wb},
  description={local function space of wirebasket decomposition}
}
\newglossaryentry{rlfspace}
{
  name={\ensuremath{\widetilde{V}_i}},
  sort={functionspace local},
  description={local function space}
}
\newglossaryentry{rlfspacewb}
{
  name={\ensuremath{\widetilde{V}_i^\mathrm{wb}}},
  sort={functionspace local},
  description={reduced local function space of wirebasket decomposition}
}
\newglossaryentry{febasis}
{
  name={\ensuremath{\mathcal{B}}},
  sort={finite element basis},
  description={finite element basis}
}
\newglossaryentry{rfspace}
{
  name={\ensuremath{\widetilde{V}}},
  sort={reduced functionspace},
  description={reduced function space}
}
\newglossaryentry{sol}
{
  name={\ensuremath{u}},
  sort={solution},
  description={solution}
}
\newglossaryentry{rsol}
{
  name={\ensuremath{\widetilde u}},
  sort={reduced solution},
  description={reduced solution}
}
\newglossaryentry{h1}
{
  name={\ensuremath{H^1(\Omega)}},
  sort={h1},
  description={space of all square integrable functions with square integrable derivative}
}
\newglossaryentry{h}
{
  name={\ensuremath{h}},
  sort={mesh size},
  description={mesh size of fine mesh}
}
\newglossaryentry{H}
{
  name={\ensuremath{H}},
  sort={mesh size dd},
  description={mesh size of coarse mesh / domain decomposition}
}
\newglossaryentry{hcurl}
{
  name={\ensuremath{H(\mathrm{curl}, \Omega)}},
  sort={hcurl},
  description={space of all square integrable functions with square integrable curl}
}
\newglossaryentry{a}
{
  name={\ensuremath{a}},
  sort={bilinear form},
  description={bilinear form}
}
\newglossaryentry{f}
{
  name={\ensuremath{f}},
  sort={linear form},
  description={linear form}
}
\newglossaryentry{coercconst}
{
  name={\ensuremath{\alpha}},
  sort={coercivity constant},
  description={coercivity constant of $a$}
}
\newglossaryentry{rcoercconst}
{
  name={\ensuremath{\widetilde{\alpha}}},
  sort={reduced coercivity constant},
  description={reduced coercivity constant of $a$}
}
\newglossaryentry{infsupconst}
{
  name={\ensuremath{\beta}},
  sort={inf-sup constant},
  description={inf-sup constant of $a$}
}
\newglossaryentry{rinfsupconst}
{
  name={\ensuremath{\widetilde{\beta}}},
  sort={reduced inf-sup constant},
  description={reduced inf-sup constant of $a$}
}
\newglossaryentry{lsv}
{
  name={\ensuremath{q}},
  sort={left singular vector},
  description={left singular vector}
}
\newglossaryentry{rsv}
{
  name={\ensuremath{v}},
  sort={right singular vector},
  description={right singular vector}
}
\newglossaryentry{singular_value}
{
  name={\ensuremath{\sigma}},
  sort={singular value},
  description={singular value}
}
\newglossaryentry{kronecker}
{
  name={\ensuremath{\delta}},
  sort={kronecker},
  description={kronecker delta}
}
\newglossaryentry{contconst}
{
  name={\ensuremath{\gamma}},
  sort={continuity},
  description={continuity constant of $a$}
}
\newglossaryentry{rcontconst}
{
  name={\ensuremath{\widetilde{\gamma}}},
  sort={reduced continuity},
  description={reduced continuity constant of $a$}
}
\newglossaryentry{upsilon}
{
  name={\ensuremath{\Upsilon}},
  sort={spaceset},
  description={set of spaces in wirebasket space decomposition}
}
\newglossaryentry{ovl}
{
  name={\ensuremath{O}},
  sort={overlapping spaces},
  description={spaces of overlapping space decomposition}
}
\newglossaryentry{randomized_a_posteriori}
{
  name={\ensuremath{\Delta}},
  sort={randomized a posteriori norm estimator},
  description={randomized a posteriori norm estimator}
}
\newglossaryentry{numspaces}
{
  name={\ensuremath{N_{V_i}}},
  sort={number of  spaces},
  description={number of local spaces}
}
\newglossaryentry{numospaces}
{
  name={\ensuremath{N_{OS}}},
  sort={number of ovl spaces},
  description={number of overlapping spaces}
}
\newglossaryentry{orthogonalp}
{
  name={\ensuremath{P}},
  sort={orthogonal projection},
  description={orthogonal projection, $P_X$ maps into $X$}
}
\newglossaryentry{feinterpolation}
{
  name={\ensuremath{\mathcal{I}}},
  sort={interpolation operator},
  description={interpolation operator $\gls{fspace} \rightarrow \gls{fspaceh}$}
}
\newglossaryentry{extend}
{
  name={\ensuremath{\mathrm{Extend}}},
  sort={extend operator},
  description={$\mathrm{Extend}$ operator}
}
\newglossaryentry{mapboundvr}
{
  name={\ensuremath{c_{i, \widetilde V}}},
  sort={norm of local mapping},
  description={norm of local mapping \gls{lfspacemap}}
}
\newglossaryentry{interpolationbound}
{
  name={\ensuremath{c_I}},
  sort={stability constant of FE interpolation},
  description={stability constant of FE interpolation \gls{feinterpolation}}
}
\newglossaryentry{pufuncconstant}
{
  name={\ensuremath{{c_{\mathrm{pu}}^\prime}}},
  sort={constant to delimit derivative of partition of unity},
  description={constant to delimit derivative of partition of unity in $\norm{\nabla \gls{pouf}}_\infty \leq \gls{pufuncconstant}  H_i^{-1}$}
}
\newglossaryentry{scaledpoincareconstant}
{
  name={\ensuremath{c_{\mathrm pc}}},
  sort={constant in rescaled poincare inequality},
  description={constant in rescaled Poincaré inequality $\norm{\varphi - \bar \varphi_i}_{L^2(\gls{subdomainol})} \leq c_{\mathrm pc} H_i \norm{\nabla \varphi}_{L^2(\gls{subdomainol})}$}
}
\newglossaryentry{decompositionboundvr}
{
  name={\ensuremath{c_{\mathrm{pu}, \widetilde V}}},
  sort={stability of space decomposition improved},
  description={improved stability constant of localizing space decomposition}
}
\newglossaryentry{image}
{
  name={\ensuremath{\mathrm{Im}}},
  sort={image of},
  description={image of}
}
\newglossaryentry{kernel}
{
  name={\ensuremath{\mathrm{Ker}}},
  sort={kernel of},
  description={kernel of}
}
\newglossaryentry{L}
{
  name={\ensuremath{L}},
  sort={length of \exrangea{}},
  description={length of \exrangea{}}
}
\newglossaryentry{W}
{
  name={\ensuremath{W}},
  sort={width of \exrangea{}},
  description={width of \exrangea{}}
}

\newglossaryentry{electric_field}
{
  name={\ensuremath{\mathbf{E}}},
  sort={electric field},
  description={electric field}
}
\newglossaryentry{angular_frequency}
{
  name={\ensuremath{\omega}},
  sort={angular frequency},
  description={angular frequency}
}

\newglossaryentry{some_field}
{
  name={\ensuremath{\mathbf{X}}},
  sort={some field},
  description={some field}
}
\newglossaryentry{expectedvalue}
{
  name={\ensuremath{\mathbb{E}}},
  sort={expected value},
  description={expectation value}
}
\newglossaryentry{some_operator}
{
  name={\ensuremath{O}},
  sort={some operator},
  description={some bounded linear operator}
}
\newglossaryentry{speedoflight}
{
  name={\ensuremath{C}},
  sort={speed of light},
  description={speed of light}
}

\newglossaryentry{electric_permittivity}
{
  name={\ensuremath{\varepsilon}},
  sort={electric permittivity},
  description={electric permittivity}
}
\newglossaryentry{magnetic_permeability}
{
  name={\ensuremath{\mu_0}},
  sort={magnetic permeability},
  description={magnetic permeability}
}
\newglossaryentry{electric_flux}
{
  name={\ensuremath{\mathbf{D}}},
  sort={electric flux},
  description={electric flux}
}
\newglossaryentry{magnetic_field}
{
  name={\ensuremath{\mathbf{H}}},
  sort={magnetic field},
  description={magnetic field}
}
\newglossaryentry{magnetic_flux}
{
  name={\ensuremath{\mathbf{B}}},
  sort={magnetic flux},
  description={magnetic flux}
}
\newglossaryentry{charge_density}
{
  name={\ensuremath{\varrho}},
  sort={charge density},
  description={charge density}
}
\newglossaryentry{current_density}
{
  name={\ensuremath{\mathbf{J}}},
  sort={current density},
  description={current density}
}
\newglossaryentry{trainingset}
{
  name={\ensuremath{\Xi}},
  sort={training set},
  description={training set, $\Xi \subset \mathcal{P}$}
}
\newglossaryentry{parameter}
{
  name={\ensuremath{\mu}},
  sort={parameter},
  description={parameter}
}
\newglossaryentry{pouf}
{
  name={\ensuremath{\varrho_i}},
  sort={partition of unity},
  description={partition of unity function, $\sum_i \varrho_i \equiv 1$}
}
\newglossaryentry{parameterspace}
{
  name={\ensuremath{\mathcal{D}}},
  sort={parameter space},
  description={parameter space}
}
\newglossaryentry{numparameter}
{
  name={\ensuremath{P}},
  sort={number of parameters},
  description={number of parameters}
}
\newglossaryentry{realpart}
{
  name={\ensuremath{\mathfrak{Re}}},
  sort={real part},
  description={real part}
}
\newglossaryentry{imagpart}
{
  name={\ensuremath{\mathfrak{Im}}},
  sort={imag part},
  description={imaginary part}
}
\newglossaryentry{heat_conductivity}
{
  name={\ensuremath{\sigma}},
  sort={heat conductivity},
  description={heat conductivity}
}
\newglossaryentry{electric_conductivity}
{
  name={\ensuremath{\kappa}},
  sort={electric conductivity},
  description={specific electric conductivity}
}
\newglossaryentry{helmholtz_wavenumber}
{
  name={\ensuremath{k}},
  sort={helmholtz wavenumber},
  description={Helmholtz wavenumber}
}

\newglossaryentry{heatsource}
{
  name={\ensuremath{q}},
  sort={heat source},
  description={heat source}
}

\newglossaryentry{residual}
{
  name={\ensuremath{\mathcal{R}}},
  sort={residual},
  description={residual}
}
\newglossaryentry{riesz}
{
  name={\ensuremath{\mathsf{R}}},
  sort={riesz},
  description={Riesz isomorphism}
}
\newglossaryentry{overlappingspaces}
{
  name={\ensuremath{J}},
  sort={overlappingspaces},
  description={maximum number of domains \gls{subdomain} at any point in \gls{domain}}
}
\newglossaryentry{overlappingspacestrain}
{
  name={\ensuremath{J^*}},
  sort={overlappingspacestrain},
  description={maximum number of domains \gls{traindomain} at any point in \gls{domain}}
}
\newglossaryentry{orthogonal_classes}
{
  name={\ensuremath{\bar{J}}},
  sort={orthogonalclasses},
  description={number of subspace sets required so all spaces in each set are orthogonal}
}
\newglossaryentry{lfspacemap}
{
  name={\ensuremath{\mathcal{P}_{V_i}}},
  sort={mapping fspace},
  description={mapping from full space \gls{fspace} to local space $\gls{lfspace}$}
}
\newglossaryentry{tracespacemap}
{
  name={\ensuremath{\mathcal{P}_{S_i}}},
  sort={mapping tracespace},
  description={mapping from full space \gls{fspace} to local space $\gls{tracespace}$}
}
\newglossaryentry{lfspacemappou}
{
  name={\ensuremath{\mathcal{P}_{V_i^\mathrm{pou}}}},
  sort={mapping fspace partition of unity},
  description={mapping from full space \gls{fspaceh} to local space $\gls{lfspacepou}$}
}
\newglossaryentry{lfspacemapres}
{
  name={\ensuremath{\mathcal{P}_{V_i^\mathrm{res}}}},
  sort={mapping fspace res},
  description={mapping from full space \gls{fspaceh} to local space $\gls{lfspaceres}$}
}
\newglossaryentry{lfspacemapbasic}
{
  name={\ensuremath{\mathcal{P}_{V_i^\mathrm{basic}}}},
  sort={mapping fspace basic},
  description={mapping from full space \gls{fspaceh} to local space $\gls{lfspacebasic}$}
}
\newglossaryentry{lfspacemapwb}
{
  name={\ensuremath{\mathcal{P}_{V_i^\mathrm{wb}}}},
  sort={mapping fspace wirebasket},
  description={mapping from full space \gls{fspaceh} to local space $\gls{lfspacewb}$}
}
\newglossaryentry{errest}
{
  name={\ensuremath{\Delta(\gls{rsol}_{\gls{parameter}})}},
  sort={errest},
  description={error estimator, standard \gls{rb} variant}
}
\newglossaryentry{rerrest}
{
  name={\ensuremath{\Delta_\mathrm{rel}(\gls{rsol}_{\gls{parameter}})}},
  sort={errest relative},
  description={relative error estimator, standard \gls{rb} variant}
}
\newglossaryentry{locerrest_l}
{
  name={\ensuremath{\Delta^{l}_{loc}(\gls{rsol}_{\gls{parameter}})}},
  sort={locerrest_l},
  description={error estimator, localized, improved variant with local constants}
}
\newglossaryentry{locerrest_g}
{
  name={\ensuremath{\Delta^{g}_{loc}(\gls{rsol}_{\gls{parameter}})}},
  sort={locerrest_g},
  description={error estimator, localized, improved variant with global constant}
}
\newglossaryentry{rerrest4}
{
  name={\ensuremath{\Delta^{g}_{loc,rel}(\gls{rsol}_{\gls{parameter}})}},
  sort={errest4 relative},
  description={relative error estimator, localized, improved variant with global constant}
}
\newglossaryentry{ritzds}
{
  name={\ensuremath{D_S}},
  sort={ritz isomorphism},
  description={Ritz isomorphism $\gls{ritzds}: \gls{transfer_source} \rightarrow \gls{R}^{\gls{dimension_transfer_source}}$}
}
\newglossaryentry{ritzdsinv}
{
  name={\ensuremath{D_S^{-1}}},
  sort={ritz isomorphism inverse},
  description={Ritz isomorphism $\gls{ritzdsinv}: \gls{R}^{\gls{dimension_transfer_source}} \rightarrow \gls{transfer_source}$}
}

\newcommand{\nnull}{\setminus \{0\}}
\newcommand{\exa}{\textit{Thermal Block Example}}
\newcommand{\exb}{\textit{Thermal Channels Example}}
\newcommand{\exbv}{\textit{Thermal Channels Variant}}
\newcommand{\exc}{\textit{2D Maxwell Example}}
\newcommand{\exd}{\textit{Olimex A64 Example}}
\newcommand{\exrangea}{\textit{Rangefinder Example 1}}
\newcommand{\exrangeb}{\textit{Rangefinder Example 2}}
\newcommand{\oldsplit}{\textit{traditional offline/online splitting}}
\newcommand{\newsplit}{\textit{improved offline/online splitting}}
\newcommand{\etaspace}{{\gls{fspace}_\eta}}
\newcommand{\hdim}{{N}}

\newcommand{\rbdim}{{\widetilde{\hdim}}}
\newcommand{\etadim}{{\hdim_\eta}}
\newcommand{\etabase}{\psi^\eta}
\newcommand{\coef}[1]{{\mathsf{#1}}}

\newcommand{\wt}{\widetilde}
\newcommand{\dx}{\ \mathrm{d}x}
\newcommand{\df}{\ \mathrm{d}}

\newcommand{\diam}{\operatorname{diam}}

\newcommand{\extension}[1]{\gls{extensionspace}\left(#1\right)}
\newcommand{\training}[1]{\gls{trainingspace}\left(#1\right)}
\newcommand{\coupling}[1]{\gls{couplingspace}\left(#1\right)}
\newcommand{\rcoupling}[1]{\widetilde C(#1)}

\newcommand{\dualpair}[2]{{#1 \left( #2 \right) }}

\newcommand{\defaultparameter}{\overline{{\gls{parameter}}}}

\newcommand{\rbasis}{{\widetilde{\mathcal{B}}}}

\newcommand{\jhat}{\hat{\jmath}}

\newcommand{\argmax}{\operatornamewithlimits{arg\,max}}

\newcommand{\supp}{\operatorname{supp}}
\newcommand{\spanset}{\operatorname{span}}

\newcommand{\encloseimage}[1]{
\setlength{\fboxrule}{0pt}%
\fbox{\raisebox{-.45\height}{
\setlength{\fboxrule}{1pt}%
#1
}}}

\newcommand{\argmin}{\operatorname{argmin}}

\newcommand{\cest}{{\gls{cest}(\gls{number_testvectors}, \gls{epstestfail})}}

\newcommand{\ceff}{{\gls{ceff}(\gls{number_testvectors}, \gls{epstestfail})}}

\newcommand{\cestshort}{{\gls{cest} } }
\newcommand{\ceffshort}{{\gls{ceff} } }
\newcommand{\algotol}{\gls{algotol}}

\newcommand{\localenrichment}{\widehat u_n}
\newcommand{\localenrichedsolution}[1]{\widehat u_{n,#1}}

\newcommand{\futurebox}[1]{
\fboxsep1em
\fboxrule.2pt
\colorbox{futurecolor}{
\begin{minipage}{\textwidth - 2.5em}
#1
\end{minipage}
}
}

\newcommand{\myline}{\noindent\makebox[\linewidth]{\rule{\textwidth}{0.4pt}}}
\newcommand{\textbfit}[1]{\textbf{\textit{#1}}}
\definecolor{gelb}{RGB}{255,165,0}
\definecolor{oran}{RGB}{255,80,9}
\definecolor{lila}{RGB}{108,0,228}
\definecolor{gruen}{RGB}{0,184,0}
\definecolor{ei}{RGB}{251,255,238}
\definecolor{dunkelgruen}{RGB}{0,120,0}
\definecolor{warm_red}{RGB}{213,16,16}
\definecolor{grau}{RGB}{155,155,140}

\definecolor{futurecolor}{RGB}{230,255,230}

\pgfplotscreateplotcyclelist{btPlotCycle}{%
solid, color=Xred, mark size=1, every mark/.append style={solid, fill=Xred!50!white}, mark=square*\\
solid, color=Xblue, mark=x\\%
solid, color=Xgreen, mark size=1.2, mark=diamond*\\
solid, color=Xorange, mark size=1, mark=*\\
}
\newlength{\mylength}

\begin{document}
\frontmatter
\includepdf{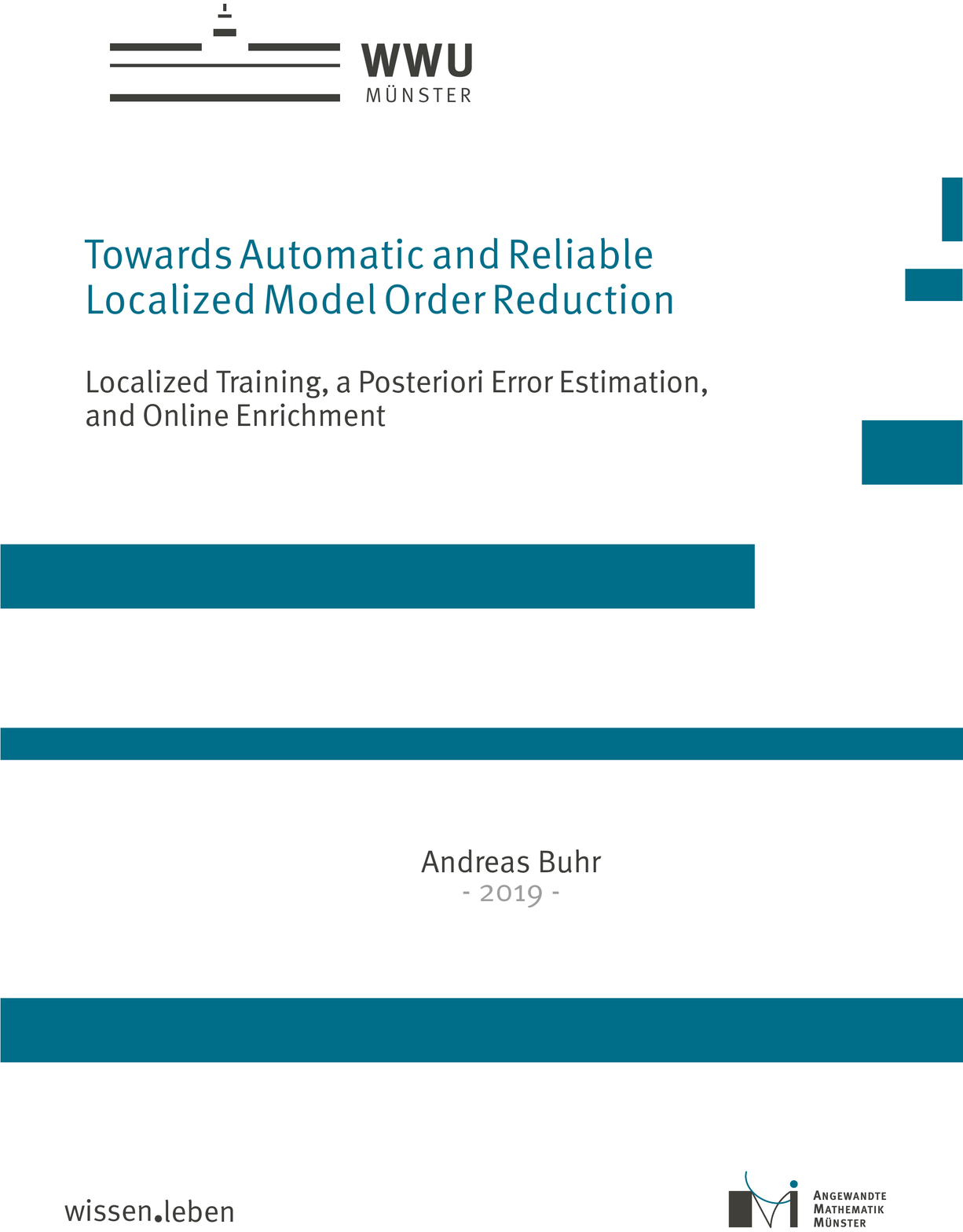}
\begin{titlepage}
  \maketitle
\end{titlepage}

\par\vspace*{\fill}
\hrule
~\\[\baselineskip]
\textbf{Dekan:} Prof.\ Dr.\ Xiaoyi Jiang \\
\textbf{Erstgutachter:} Prof.\ Dr.\ Mario Ohlberger \\
\textbf{Zweitgutachter:} Prof.\ Dr.\ Christian Engwer \\
\textbf{Externer Gutachter:} Prof.\ Dr.\ Robert Scheichl, Heidelberg \\
~\\
\textbf{Tag der m\"undlichen Pr\"ufung: 11. Juli 2019}\\
\textbf{Tag der Promotion: 11. Juli 2019}

\clearpage
\pagestyle{justline}
\chapter*{Abstract}
{\Large English Abstract}\vspace{10pt}
\\
Finite element based simulation of phenomena governed
by partial differential equations is a standard
tool in many engineering workflows today.
However, the simulation of complex geometries
is computationally expensive.
Many engineering workflows require
multiple simulations with small, non parametric
changes in between.
The use of localized model order reduction
for subsequent simulations of geometries with localized
changes is very promising.
It produces lots of computational tasks with little
dependencies and thus parallelizes well.
Furthermore, the possibility to reuse
intermediary results in the subsequent simulations
can lead to large computational savings.
In this thesis, we investigate different
aspects of localized model order reduction
and propose various improvements.
A simulation methodology named ArbiLoMod, 
comprising a localized training, 
a localized
a posteriori error estimator
and an enrichment procedure is proposed.
A new localized a posteriori error estimator
with computable constants is presented and analyzed.
A new training algorithm which is based on a transfer
operator is derived. It can be shown
to converge nearly as fast as
the singular value decay
of this operator.
The transfer operator's spectrum is observed
to decay fast in electromagnetic simulations
in printed circuit boards.
New online enrichment algorithms 
are proposed.
All results are supported by numerical experiments, for
which the source code for reproduction is provided.
\clearpage

\noindent
{\Large German Abstract}\vspace{10pt}
\\
Finite Elemente basierte Simulationen von 
Phänomenen, die durch partielle Differentialgleichungen
beschrieben werden können, sind heutzutage
fester Bestandteil von vielen
Arbeitsabläufen im Ingenieurswesen.
Allerdings sind Simulationen von komplexen Geometrien
sehr rechenintensiv.
In vielen Arbeitsabläufen sind wiederholte
Simulationen mit lokalisierten, nichtparametrischen
Änderungen dazwischen notwendig.
Lokalisierte Modellordnungsreduktion 
verspricht ein großes Verbesserungspotential
für Sequenzen von Simulationen
mit geringfügigen Änderungen zwischen den Simulationen.
Bei der Anwendung von lokalisierte Modellordnungsreduktion
entsteht eine große Anzahl von Berechnungsabschnitten
mit nur wenigen Datenabhängigkeiten, was
Parallelisierung einfach macht.
Weiterhin existiert die Möglichkeit,
Zwischenergebnisse in den folgenden 
Simulationen wiederzuverwenden, 
was zu Einsparungen von Rechenzeit führen kann.
In dieser Dissertationsschrift untersuchen
wir verschiedene Aspekte von lokalisierter Modellordnungsreduktion
und schlagen einige Verbesserungen vor.
Eine Simulationsmethodik, genannt ArbiLoMod,
die ein lokalisiertes Training, 
eine a posteriori Fehlerabschätzung
sowie ein online enrichment enthält,
wird vorgeschlagen.
Ein neuer lokalisierter a posteriori Fehlerschätzer
mit berechenbaren Konstanten wird gezeigt und analysiert.
Ein neuer Trainingsalgorithmus 
wird präsentiert, der auf einem
Transferoperator basiert
und beweisbar nahezu
so gut konvergiert wie
der Singulärwertabfall dieses Operators.
In elektromagnetischen Simulationen von
Leiterplatten wird ein schneller Abfall
des Spektrums des Transferoperators beobachtet.
Neue Algorithmen zum online enrichment
werden vorgeschlagen.
Alle Ergebnisse werden mit numerischen Experimenten
bestätigt, deren Quellcode zwecks Reproduktion
zur Verfügung gestellt wird.

\chapter*{Acknowledgments}
I would like to thank my supervisors, Mario Ohlberger and Christian Engwer, and Stefan Reitzinger 
from CST AG
for the opportunity to pursue this research project. 
The idea to have interactive simulations of physical phenomena in complex geometries has fascinated
me for a long time. 
Having had the opportunity to concentrate on research heading in this direction
for an extended period of time
is a privilege which I appreciate greatly.
I would like to thank 
Mario Ohlberger, Christian Engwer,
Stephan Rave, and Kathrin Smetana
for productive collaborations on the central research questions,
which lead to significant progress toward this goal.
Finally I would like to thank all colleagues from my work groups
for a stimulating and enjoyable working environment
and 
Julia Brunken,
Annika Bürger,
Dennis Eickhorn,
René Fritze,
Tim Keil, and
Tobias Leibner
for proofreading parts of this thesis.
\vfill

\noindent
This thesis was supported by CST AG: \\
\\
CST Computer Simulation Technology AG (as of August 2017: GmbH)\\
Bad Nauheimer Stra\ss e 19 \\
64289 Darmstadt \\
Germany

\clearpage
\pagestyle{mytoc}
\tableofcontents
\clearpage
\pagestyle{justline}
\listoftheorems
\listoffigures
\listoftables
\listofalgorithmes
\cleardoublepage

\mainmatter
\clearpage
\pagestyle{main}
\chapter{Introduction}
\section{Target Application}
Nowadays, finite element based simulation of phenomena which are
governed by partial differential equations, such as
heat conduction, elasticity, fluid
flow, or electromagnetics is an essential part of 
many engineering workflows. Commercial tools are available from
different vendors.
However, many tools have a simple simulation pipeline which goes from
geometry to meshing, to matrix assembly, to the solution of
a matrix equation, to postprocessing and finally visualization
of results.
Often, all intermediary results are discarded and the complete pipeline is
reprocessed on any non parametric change of the model under consideration,
even if the change is localized.

The research presented in this thesis was motivated by the desire
to improve on this. 
Often, engineers improve a structure manually in an iterative approach,
where in each iteration they change the structure slightly and resimulate.
A particular example is the design of printed circuit boards (PCBs).
The design of PCBs has all of the above mentioned properties:
PCBs are nowadays very complex, there is no
scale separation, improvements are often obtained by
local changes of the electronic components and conductive tracks,
and when solved in frequency domain, it is a parameterized problem
with the frequency as a parameter.
The idea to use localized model reduction for PCBs
was already published by the authors in \cite{Buhr2009}.
A possible change is depicted in 
\cref{fig:boards}.
The same applies to integrated circuit (IC) packages, which are in structure similar to PCBs.
\begin{figure}
\centering
\includegraphics[width=.9\textwidth,viewport=0 220 1476 470, clip=true]{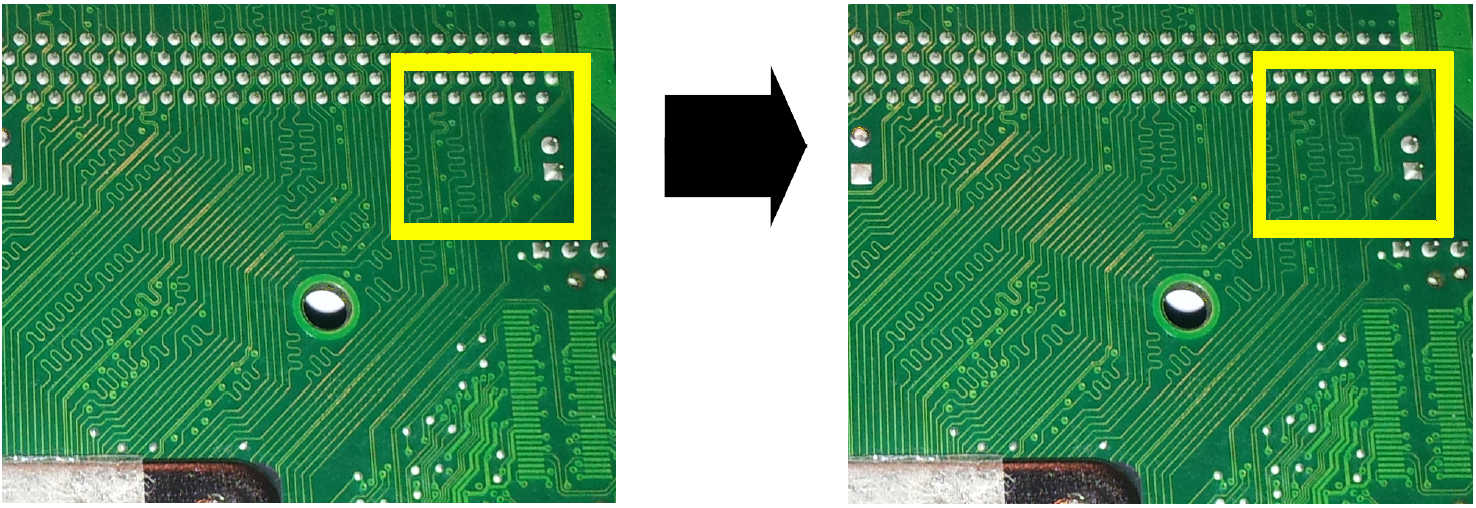}
\caption[DDR memory channel on a printed circuit board subject to local modification.]
{DDR memory channel on a printed circuit board subject to local modification of conductive tracks (own photo, manipulated with GIMP).}
\label{fig:boards}
\end{figure}
When the changes are localized, one can take advantage of this.
To this end, we designed a processing pipeline which
keeps everything localized as long as possible.
With a localized meshing procedure, the mesh has to be regenerated
only in an environment of the change.
With a localized matrix assembly, matrices have to be reassembled only in an environment
of the change.
When only a small part of the geometry is changed, this bears the potential
for large speedups.

The next step of the pipeline is the solution of
a matrix equation, which cannot easily be localized, since everything is coupled.
Localizing this part is the central goal of this thesis.
To this end, we designed a simulation methodology named \gls{arbilomod}.
It localizes
this step by first creating localized, reduced models and then 
coupling these local models to solve the global problem. This way,
only the reduced models in an environment of the change have to be
regenerated. The following solution of the global problem is fast, 
as it is based on the reduced models.
While this localized pipeline was designed to improve simulation times after
localized changes, it has in addition the advantage
that the computational workload it creates parallelizes well. It fits
to the current development of computers to become more and more parallel.
While only a few localized meshes, matrices and reduced models
are created upon a localized change to the model, 
thousands of tasks have to be processed in the initial run.
These tasks have little data dependencies and are thus well suited 
for cloud environments.

\subsection*{Vision}
Today, the use of finite element based simulation tools needs
skilled engineers and is a time consuming task where simulations
often take hours, days or even weeks.
We believe that in the future, by advances in computing technology
and the development of more efficient algorithms, simulation
tools will provide results almost instantaneously.
The goal is that results are 'just there' and the engineer
using the simulation software can forget about the fact that there is a 
sophisticated simulation happening.

\section{Changing Computing Landscape}
The finite element method was developed in the second
half of the previous century.
As the clock frequencies of PCs as well as the amount 
of available RAM have risen, the analysis of 
increasingly bigger models became feasible.
For a long time, sequential  performance of CPUs was
rising according to Moore's law.
The speed doubled approximately every 18 month
until around the year 2005, when high end CPUs
reached over 3 GHz. Since then, at the time
of writing for 14 years, CPUs became more parallel by
broader SIMD units and an increasing number of cores
per socket, but the clock rates stayed around
3 GHz.
Applications need to be adapted to these parallel architectures
to exploit the available computing power.
There are efforts to use this kind
of parallelism  in today's FEM codes,
see for example the EXA-DUNE project \cite{exadune}
within the
``German Priority Programme 1648 -- Software for Exascale Computing''
by the \gls{dfg}.

But now we witness the rise of cloud computing,
where everyone has potential access to thousands of
computers, albeit accessed through a rather slow network connection.
To successfully exploit the potential of cloud computing
for finite element analysis in an interactive setting, 
new algorithms with greatly reduced communication demands are required.
For example, if a user wants to start 1000 compute tasks within a second
and has a 20 MB/s internet connection, each task can be given only 20kB
of input data.
If the task runs on a recent multicore machine, it can have $\mathcal{O}(10^{11})$
floating point operations available in one second.
This discrepancy of $\mathcal{O}(10^{11})$ floating point operations
to $\mathcal{O}(10^5)$ bytes of input data
is the challenge in using highly parallel cloud computing
for interactive \gls{fea} applications.
The algorithms we develop in this thesis are targeted at this
setting.

\section{Research Questions}
\label{sec:research questions}
As already mentioned,
localized model order reduction has the 
potential to vastly decrease simulation times in a finite element based
simulation software for sequences of geometries
with localized changes, as it is common
in many engineering workflows.
It has attracted a lot of attention from different research communities
recently.
The main questions in research on localized model
order reduction are the following.
\begin{enumerate}
\item
\textbf{\textit{How to decompose the global problem?}}\\
Each localized model order reduction method
has to decompose the global problem
in local parts and to couple the local parts.
The question how to do this decomposition is central, as it
impacts all other questions.
While in general, a decomposition into local problems
of very different kind is possible, 
we concentrate on the decomposition of a global solution space
into localized subspaces in this thesis,
as this gives a consistent mathematical framework.
\item
\textbf{\textit{How to generate local reduced models?}}\\
Having decomposed the global problem into
local coupled parts, the second questions
is how to generate reduced
approximations of these. This is usually the question
getting the most attention.
The models are often constructed to fulfill
some form of local convergence criterion or error bound.
In order to obtain the projected benefits, especially
the reuse of reduced models after a localized geometry change,
it is necessary to use only local information in the construction of local
reduced models.
This step is often called ``training''.
\item
\textbf{\textit{How does the local error propagate to the global problem?}}\\
Having constructed reduced approximations of all parts of
a localizing decomposition, one can compute an
approximate solution of the global problem using these.
Quantifying a priori
how the errors of the local reduced models
influence
the error of the global approximate solution
is another topic in research on localized model order reduction.
\item
\textbf{\textit{How to estimate the global error a posteriori?}}\\
Having computed a global approximate solution,
one wants to quantify its error. This is
the topic of a posteriori error estimation. It is
especially important in situations where an a priori estimate
is not available.
But even when an a priori estimate is available,
an a posteriori estimate might give a much sharper bound.
Most existing a posteriori error estimators for reduced
models involve global computations and are
thus unsuitable for the localized setting.
\item
\textbf{\textit{How can the local reduced models be improved, once we solved the global reduced problem?}}\\
Having solved the global reduced problem,
the a posteriori error estimator might indicate an
insufficient quality of the obtained
approximation.
In order to improve, one wants to improve
the local, reduced problems. 
This is the question of ``online enrichment''.
\end{enumerate}
Besides these five research questions,
there are additional implementation related
questions which come into focus when
the knowledge is transferred to the application
domain and the method is implemented in an
engineering context, among those:
\begin{enumerate}
\setcounter{enumi}{5}
\item
\textbf{\textit{How to mesh the geometry localized?}}\\
When using \gls{fem} based methods, as we do throughout
this thesis, a mesh is required. On order to
realize the advantages of localized model order reduction, also
the meshing has to be localized.
Meshing is a very complex task in practice. The additional
constraints of localization may lead to additional
difficulties.
\item
\textbf{\textit{How to implement the method parallelized?}}\\
The intended method should react quickly
to geometry changes and should be heavily parallelized.
The resulting workload requires many computing resources,
but only in short bursts. To be cost
effective, the computational resources must be shared with other users.
Implementing a suitable runtime system is a challenging task.
\end{enumerate}

\section{Existing Research}
Our motivating application is the handling
of localized, unforeseen changes 
in the model under consideration.
Despite its importance, this question
gained surprisingly little attention
from the research community.
Similar to our motivation 
is localized model order reduction for
crack propagation, where the location of
the crack is not known in advance
\cite{Kerfriden2013}.
While only few researchers share our
motivation of arbitrary localized changes,
the field of localized model order
reduction has gained a large amount
of attention recently.
Localized model order reduction is a research
topic in between of three traditional
research fields. 
It can be seen as an extension to the
\gls{rb} method, using basis functions
with localized support only instead of
basis functions with global support, as
it is common in \gls{rb} methods.
It can also be seen as a generalization
of numerical multiscale methods, relaxing
the requirement of scale separation
and using an adaptive number of basis functions
for each coarse grid element.
As a third option, it can be seen as
a part of domain decomposition research,
as the construction of a coarse space
good enough so that the coarse
space solution is of sufficient quality.
In the following, we review related
developments in these three fields.
\subsection{Localization of Reduced Basis Methods}
\label{sec:intro rb localization}
The \gls{rb} method \cite{noor1980reduced,Fink1983,Patera2007} is by now an established model order
reduction method for parameterized problems in real-time or many-query contexts.
Introductory texts were published recently \cite{Benner2017b,Hesthaven2016,Quarteroni2016}.
Multiple approaches combining the \gls{rb}
method with some form of localization were published.

In 2002, Maday and R\o nquist
published the \gls{rbe}
method,
introduced in \cite{Maday2002}
and applied to the thermal fin problem and the Stokes problem \cite{Maday2004,Lovgren2006},
combining the reduced basis approach with a domain decomposition,
coupling local basis vectors by polynomial Lagrange multipliers on
domain boundaries in a mortar type.
Later, the \gls{rbe} was applied to the two dimensional Maxwell's
problem in \cite{Chen2011}, using
\gls{dg} coupling instead of mortar type coupling.
This was named \gls{srbe}.
This approach is unsuitable for our setting.
While using polynomials at the boundaries might work
in most cases, it could fail for complex interfaces
or if singularities in the solution are close to the interfaces.
In addition, to construct the reduced spaces,
in \cite{Chen2011} it is proposed to build up a global
\gls{rb} space first and then decompose these function into local parts.
This is incompatible with our goal of only local computations
and reuse after changes.

The \gls{rbhm} which is built upon the \gls{rbe}
was introduced in \cite{Iapichino2012}.
It combines the local \gls{rb} approximations
with a global coarse \gls{fe} space,
but still using polynomial Lagrange
multipliers at domain boundaries.

The
\gls{rdf},
whose preliminary version was introduced in
\cite{Iapichino2012b} %
and which was introduced in
\cite{Iapichino2016} %
uses parameterized interface spaces
of Langrange- or Fourier functions 
for basis generation
and
it includes \gls{fe} basis functions
along interfaces in the global reduced space.
It thereby has more accurate solutions,
but
this again is unsuitable for our
setting, as we aim at relatively
large interfaces, where model
reduction at the interfaces is crucial
and including all \gls{fe} basis
functions at the interface would be
prohibitively expensive.

The \gls{dgrbe}
was introduced in
\cite{Antonietti2016}
and applied to the Stokes problem in
\cite{Pacciarini2016a}.
It extends on the \gls{rdf},
but it removes the \gls{fe}
functions at the interfaces and instead,
like the \gls{srbe},
couples local bases on subdomains using
a \gls{dg} coupling.
Local basis functions are computed
by prescribing parameterized Neumann data
on the subdomain boundaries.
For the Neumann data, either using all
\gls{fe} functions at the interface
or using Legendre functions is proposed.
Both is unsuitable for our setting. Using
all \gls{fe} functions leads to too many
possible boundary conditions. 
Legendre functions might not approximate
the true interface data in complex
geometries or if singularities are
close to the domain boundaries.

Also combining \gls{rb} with domain decomposition
but different in motivation is the
\gls{rbhdd} \cite{Maier2014,Martini2014,martini2017certified},
aiming at the coupling of different
physical formulations on the domains.

Similar to ours in motivation is the
\gls{scrbe}
method \cite{PhuongHuynh2012}, which also aims at systems
composed of components, where the geometry of the components can
be mapped to reference geometries.
While the connection between the components is simply achieved by
polynomial Lagrange multipliers in the RBE, significant research
has been conducted on choosing the right coupling spaces in
the context of the \gls{scrbe}. Choosing the right space at the interfaces
is called ``Port Reduction'' and was included in the name, leading to the
\gls{prscrbe}
method
\cite{Eftang2013,Eftang2014} which employs a so-called ``pairwise training''.
A variation of this idea is used in our method.
Recently, also an algorithm to obtain optimal interface spaces
for the \gls{prscrbe} was proposed \cite{Smetana2016} and a framework
for a posteriori error estimation was introduced \cite{Smetana2015}.
While the \gls{prscrbe} performs excellent in using the potentials of cloud
environments, its goals are different from ours.
While being able
to handle various changes to the simulated system,
it does not aim at arbitrary modifications. And it does not try
to hide the localization from the user, but rather exposes it to let
the user decide how the domain decomposition should be done.
One further drawback is that the \gls{scrbe} is only
specified for non intersecting domain interfaces and
thus not applicable
to a domain decomposition of 3D space but only
to systems which naturally decompose into components like
bridges or cranes.

Very close to our motivation is the \gls{lrbms}.
It
is a localized model order reduction method building
upon a non overlapping domain decomposition and
a \gls{dg} coupling at the domain boundaries.
Reduced spaces are constructed for each domain.
The \gls{lrbms} was developed in the groups of Mario Ohlberger 
at the University of Münster and Bernard Haasdonk at the
University of Stuttgart during the last decade.
The first publication mentioning localized spaces and the DG coupling
appeared in 2011 \cite{KAULMANN2011},
the first publication mentioning the \gls{lrbms}
appeared 2012 \cite{Albrecht2012}.
This method interpolates between a standard \gls{dg} method
when the domain decomposition is chosen to be equal to the fine mesh
and a classic \gls{rb} method when the domain decomposition
consists of only one element containing the full domain.
It has been proposed to generate local
basis functions by solving the global,
unreduced problem and restricting
the global solution to the local domains.
This strategy is uncommon in the realm
of localized model order reduction
and thus a specialty of the \gls{lrbms}.
Improvement of the local reduced spaces by
online enrichment was introduced in \cite{Albrecht2013},
solving local problems on localized oversampling domains
in an 
\textit{solve $\rightarrow$ estimate $\rightarrow$  mark $\rightarrow$ refine}
loop.
Rigorous a posteriori error estimators have been
developed in \cite{Ohlberger2014} and \cite{Ohlberger2015},
estimating not only the reduction error with
respect to a reference solution, but also
the full error with respect to the analytic solution.
Recently, the \gls{lrbms} was adapted to parabolic problems in 
\cite{Ohlberger2017}.
In \cite{Ohlberger2017a} and \cite{Ohlberger2018}, 
usage of the \gls{lrbms} in the context of PDE constrained 
optimization was investigated,
performing online enrichment of the localized reduced approximation spaces inside of
an optimizer loop.
Unlike other localized model order reduction methods, no local training procedures have been
investigated so far in the context of the
\gls{lrbms}.
The implementation of the \gls{lrbms} led to a
comprehensive software library extending the \gls{dune} \cite{Bastian2008a,Bastian2008}
framework, namely the \texttt{dune-xt} modules
\cite{Milk2017}
and \texttt{dune-gdt}\footnote{\url{https://github.com/dune-community/dune-gdt}}.
The main difference between the \gls{lrbms} and the \gls{arbilomod}
is that \gls{arbilomod} is based on a conforming discretization,
allowing for reuse of existing conforming
\gls{fe} software libraries to build an \gls{arbilomod}
implementation on top.
Furthermore, during the development of
\gls{arbilomod}, there was a focus on localized training
procedures (presented in \cref{sec:codim_n_training,chap:training}).
While these training procedures could also be
applied in the context of \gls{lrbms},
no research has been conducted in this area to date.

\subsection{Generalization of Numerical Multiscale Methods}
Numerical multiscale methods are designed for applications in which
the problem has a fine structure which is too fine
to be discretized directly, as for example
in ground water flow simulations.
In most cases, these methods require
a scale separation.
For problems with scale separation, there are established methods
to reduce the model such as
\gls{msfem}, \gls{vmm}, \gls{hmm}, or \gls{lod}.

\gls{vmm} was introduced in \cite{Hughes1995,Hughes1998a}.
An adaptive variant was introduced in \cite{Larson2005,Larson2007}
which was later extended in \cite{Larson2009,LARSON2009b}.
\gls{hmm} was introduced in \cite{weinan2003multiscale,Weinan2003,weinan2005}.
An a posteriori error estimator was derived in \cite{Ohlberger2005}.
Fully discrete a priori estimates are presented in \cite{Abdulle2005}.
\gls{lod} was introduced in \cite{Malqvist2014,Henning2014a}.
Efficient implementation was recently analyzed in 
\cite{Engwer2019}.

All the established methods which employ a
coarse grid and use a fixed number of 
coarse basis functions per coarse grid cell
are unsuitable for our setting,
as the existence of many channels through
a coarse grid cell might require
a larger number of ansatz functions in some cells.

Closest to our
approach are the developments in the
\gls{gmsfem}, which originated
in developments in the \gls{msfem}.
\gls{msfem} was introduced in
\cite{Hou1997}.
An introductory book to the \gls{msfem}
was published in \cite{Efendiev2009}.
Very similar in spirit is the \gls{rfb}
\cite{Sangalli2003} where it is proposed to enrich
a standard \gls{fe} space by 
intra-element bubble functions removing
the inside residual, essentially constructing
as space very similar to the \gls{msfem} space.
An adaptive \gls{msfem}
was published in \cite{Henning2014}.
In \cite{efendiev2011},
the \gls{msfem} space is enriched with
eigenfunctions 
multiplied with a partition of unity,
thus an adaptive number of basis functions
per coarse cell is used.
A postprocessing step is used to differentiate
between local basis functions for inclusions, which
can be eliminated, and local basis functions for channels.
This development lead to the \gls{gmsfem},
introduced in \cite{Efendiev2013}.
Plenty of interesting ideas were
published in the context of the \gls{gmsfem}.
\gls{gmsfem} uses the idea of
the \gls{msfem} and
constructs reduced spaces which are spanned
by ansatz functions on local patches, using only local
information. It allows for non-fixed
number of ansatz functions on each local patch.
\gls{gmsfem} uses local eigenproblems and a partition of unity
for basis generation.
Adaptive enrichment for the \gls{gmsfem} is presented by Chung et al.\ in
\cite{Chung2018,Chung2014b} and
online-adaptive enrichment in \cite{Chung2015a}.
Very recently, the \gls{esmsfem} has been proposed in \cite{Fu2018},
using dominant eigenvectors of the local \gls{dtn}
operator multiplied with a partition of unity as basis functions.

\subsection{Domain Decomposition Methods}
The generation of localized approximation spaces 
is closely related to the problem of constructing
coarse spaces in domain decomposition methods.
Especially the adaptive construction of multiscale coarse spaces
for problems with heterogeneous coefficients with
high contrast is interesting in this regard.
Similar local problems and similar constants are used.
However, most research on coarse spaces in domain decomposition
methods targets at estimates for the condition of the
preconditioned operator matrix while
generation of local bases in localized model order reduction
targets at the approximation properties of
local spaces.

\gls{msfem} coarse spaces are used in \cite{Aarnes2002}
for a non overlapping domain decomposition method.
In \cite{Graham2007,Graham2007a}, 
coarse spaces are constructed where robustness of the
\gls{dd} iteration in case of large coefficient
variations inside of domains can be shown.
These approaches still have a fixed number of
coarse space functions per domain.
In \cite{Galvis2010}, a large gap in the spectrum
of local eigenvalue problems is used
to steer the adaptive number of coarse grid basis functions
per domain.
In \cite{Dolean2010,Dolean2012}, coarse spaces are constructed adaptively
using an eigenvalue problem based on the \gls{dtn} operator in each subdomain.
Robustness with regard to large coefficient variations
which are not confined in a domain (i.e.~channels)
is shown in \cite{Spillane2013},
where coarse spaces are constructed from
an eigenvalue problem including
the partition of unity in the overlap region.
Also in the context of \gls{fetidp}, there
were similar developments.
\cite{Mandel2007,Pechstein2008,Klawonn2016}.
Recently, ``optimal'' coarse spaces have been proposed
leading to convergence in one iteration,
see for instance \cite{Gander2017}.
An interesting connection between 
multiscale methods and domain decomposition methods has been established in
\cite{Kornhuber2018,Kornhuber2017,Kornhuber2016a}.
Coarse spaces from \gls{acms} are proposed for two-level overlapping
Schwarz in \cite{Heinlein2018}.

\subsection{Other Related Developments}
CMS \gls{cms}
component mode synthesis (CMS) introduced in {\color{white}a} \cite{Hurty1965,Bampton1968}: uses an approximation based on the eigenmodes of local constrained eigenvalue problems (approximation in the interior of the domain)

To realize a fast simulation response also for large component-based structures it is vital to reduce  the number of \gls{dof}s on the ports, too. Within the CMS approach this is realized by utilizing an eigenmode expansion \cite{Bourquin1992,Hetmaniuk2010,Jakobsson2011,Hetmaniuk2014}.

\newpage
\section{Collaborations and Previous Publications}
The research presented in this thesis was conducted 
in collaboration with different partners.
Most of the results presented in this dissertation have
been published by the authors and collaborators in peer reviewed journals or
conference proceedings.
In the following, all collaborations and previous publications
are given.
\begin{enumerate}
\item
The numerically stable formulation of the standard a posteriori
error estimator for the reduced basis method which is presented in
\cref{sec:improved splitting}
was developed together with Christian Engwer, Mario Ohlberger, and Stephan Rave.
The results have been published in the article%
\vspace{5pt}
\\
\bibentry{Buhr2014}
\cite{Buhr2014}.
\item
The simulation method \gls{arbilomod}
in \cref{chap:arbilomod},
was devised together with
Christian Engwer, Mario Ohlberger and Stephan Rave.
The first publication presenting \gls{arbilomod}
was the article 
\vspace{5pt}
\\
\bibentry{Buhr2017}
\cite{Buhr2017}.
\item
\gls{arbilomod} was tested on the time harmonic Maxwell's equations
in collaboration with Christian Engwer, Mario Ohlberger and Stephan Rave,
the results
are given in \cref{sec:arbilomod experiments for maxwell}
and were published in the article 
\vspace{5pt}
\\
\bibentry{Buhr2017c}
\cite{Buhr2017c}.
\item
The randomized range finder algorithm
presented in \cref{sec:randomized_range_finder} was developed together with Kathrin Smetana.
Results were published in the article
\vspace{5pt}
\\
\bibentry{Buhr2018}
\cite{Buhr2018}.
\item
The convergence analysis for the online enrichment
given in \cref{sec:enrichment convergence} was done by the author
alone. It was published in the article 
\vspace{5pt}
\\
\bibentry{Buhr2017e}
\cite{Buhr2017e}.
\item
The localized a posteriori error estimator
presented in \cref{sec:a_posteriori} and
the randomized training procedure shown in \cref{sec:randomized_range_finder}
were published along with a common framework
for localized model order reduction
in the book chapter
\vspace{5pt}
\\
\bibentry{morhandbook}
\cite{morhandbook}.
\end{enumerate}
\clearpage

\section{Overview over this Document}
In the following \cref{chap:preliminaries}
we will introduce the abstract, variational
problem, detail how stationary
heat conduction and the time harmonic 
Maxwell's equation lead to a 
problem in variational form and
introduce the numerical examples used.
In \cref{chap:arbilomod}
we present the method \gls{arbilomod} and
show how the parts integrate.
Thereafter, 
three chapters diving deeper into
selected topics follow.
In \cref{chap:a posteriori}
we detail variants
of a localized a posteriori error estimator.
In \cref{chap:training}
we apply the theory of randomized
numerical linear algebra to the problem
of localized training, yielding
a priori and a posteriori error bounds
for randomized training algorithms.
In \cref{chap:enrichment}
some first a priori results on convergence 
of online enrichment are shown
and a new enrichment algorithm is introduced.
\cref{chap:conclusion}
concludes this thesis.

\chapter{Preliminaries}
\label{chap:preliminaries}
In this chapter, we present the abstract,
variational
problem definition 
consisting of a bilinear or sesquilinear form and a linear 
or antilinear form
on an abstract Hilbert space.
We will build upon this formulation throughout this thesis.
We introduce the Galerkin projected, reduced problem.
Based on the coercivity or the inf-sup stability
of the bilinear (sesquilinear) form we will review
the most important results, namely
uniqueness, stability and best approximation.
We introduce the parameterized problem
and its affine decomposition.
Furthermore, we introduce an abstract notation
for localized model order reduction.
We detail how the problems of
stationary heat conduction
and time harmonic electromagnetism
in two or three dimensions
lead to a problem in this
variational form.
Finally, we introduce the four examples
which are used for numerical experiments
in this thesis.

\section{Variational Problem}
Let \gls{fspace} be a Hilbert space over the
field \gls{K} where \gls{K} is either \gls{R} or \gls{C}.
Let further \gls{a} be a bilinear or sesquilinear form, 
which is coercive or inf-sup stable and bounded.
Let \gls{f} be a linear or antilinear, bounded form on \gls{fspace}.
We then define the variational problem as follows.
\begin{definition}[Variational problem]
\label{def:variational problem}
Find $\gls{sol}$ in $\gls{fspace}$, such that
\begin{equation}
\gls{a}(\gls{sol}, v) = \gls{f}(v) \qquad \forall v \in \gls{fspace} .
\label{eq:variational problem}
\end{equation}
\end{definition}
We denote the stability constant by \gls{contconst}.
It is defined as
\begin{equation}
\gls{contconst} := \sup_{\varphi_1 \in \gls{fspace} \nnull, \varphi_2 \in \gls{fspace} \nnull}
\frac{
\abs{\gls{a}(\varphi_1, \varphi_2) }
}{
\norm{\varphi_1} \norm{\varphi_2}
} .
\end{equation}
The coercivity constant \gls{coercconst} and the inf-sup constant \gls{infsupconst} are defined as
\begin{eqnarray}
\gls{coercconst} &:=& \inf_{\varphi \in \gls{fspace} \nnull} 
\frac{\abs{\gls{a}(\varphi,\varphi)}}{\norm{\varphi}^2}\\
\gls{infsupconst} &:=& \inf_{\varphi_1 \in \gls{fspace} \nnull}{\sup_{\varphi_2 \in \gls{fspace} \nnull} 
\frac{\abs{\gls{a}(\varphi_1, \varphi_2)}}{\norm{\varphi_1} \norm{\varphi_2}}} .
\end{eqnarray}
In the inf-sup stable case, we further assume that
\begin{equation}
\inf_{\varphi_2 \in \gls{fspace} \nnull} \sup_{\varphi_1 \in \gls{fspace} \nnull} \frac{\abs{\gls{a}(\varphi_1, \varphi_2)}}{\norm{\varphi_1} \norm{\varphi_2}}
> 0
.
\end{equation}
In the coercive case ($\gls{coercconst} > 0$), the variational problem in \cref{def:variational problem}
has a unique solution due to the Lax-Milgram lemma \cite[Theorem 2.1]{laxmilgram}
and it holds
\begin{equation}
\norm{\gls{sol}} \leq \frac{1}{\gls{coercconst}} \norm{\gls{f}}_{\gls{fspace}'}.
\end{equation}
In the inf-sup stable case ($\gls{infsupconst} > 0$), the variational problem in \cref{def:variational problem}
has a unique solution \cite[Théorème 3.1]{Necas1962}
and it holds
\begin{equation}
\norm{\gls{sol}} \leq \frac{1}{\gls{infsupconst}} \norm{\gls{f}}_{\gls{fspace}'}.
\end{equation}
\begin{remark}[Antilinear inner product]
We denote the inner product between two vectors $a$ and $b$ in \gls{fspace} by $(a,b)$.
In complex vector spaces, we define the inner product $(a,b)$ to be linear
in the second argument and antilinear in the first
argument, i.e.~for two vectors $a$ and $b$ and a scalar $\lambda \in \gls{C}$
we have
\begin{equation}
(\lambda a, b) = \bar \lambda (a,b) \quad \text{ and } \quad (a, \lambda b) = \lambda (a,b) .
\end{equation}
We use this definition because it is prevalent in
numerical codes, cf.~\cref{chap:complex handling in pymor}.
\end{remark}
\section{Reduced Problem}
\label{sec:reduced problem}
At the core of projection based model order reduction
is the (Petrov-)Galerkin projection of the variational problem
in \cref{def:variational problem}.
We consider only Galerkin projection in this thesis.
To define the projected variational
problem, let \gls{rfspace} be a subspace
of \gls{fspace}. We define
\begin{definition}[Reduced variational problem]
\label{def:reduced variational problem}
Find $\gls{rsol}$ in $\gls{rfspace}$, such that
\begin{equation}
\gls{a}(\gls{rsol}, v) = \gls{f}(v) \qquad \forall v \in \gls{rfspace} .
\label{eq:reduced variational problem}
\end{equation}
\end{definition}
Accordingly, we define the reduced stability constant \gls{rcontconst},
the reduced coercivity constant \gls{rcoercconst}, and the reduced
inf-sup constant \gls{rinfsupconst} by
\begin{eqnarray}
\gls{rcontconst} &:=& \sup_{\varphi_1 \in \gls{rfspace} \nnull, \varphi_2 \in \gls{rfspace} \nnull}
\frac{
\abs{\gls{a}(\varphi_1, \varphi_2) }
}{
\norm{\varphi_1} \norm{\varphi_2}
}\\
\gls{rcoercconst} &:=& \inf_{\varphi \in \gls{rfspace} \nnull} 
\frac{\abs{\gls{a}(\varphi,\varphi)}}{\norm{\varphi}^2}\\
\gls{rinfsupconst} &:=& \inf_{\varphi_1 \in \gls{rfspace} \nnull} \sup_{\varphi_2 \in \gls{rfspace} \nnull} \frac{\abs{\gls{a}(\varphi_1, \varphi_2)}}{\norm{\varphi_1} \norm{\varphi_2}}.
\end{eqnarray}
\subsection{Properties of Reduced Variational Problem in Coercive Case}
From these definitions follows that $\gls{rcontconst} \leq \gls{contconst}$ and 
$\gls{rcoercconst} \geq \gls{coercconst}$ and thus 
in the coercive case it holds
\begin{equation}
\norm{\gls{rsol}}
\leq \frac{1}{\gls{rcoercconst}} \norm{\gls{f}}_{\gls{fspace}'}
\leq \frac{1}{\gls{coercconst}} \norm{\gls{f}}_{\gls{fspace}'}
.
\end{equation}
In the coercive case, the reduced variational problem inherits its stability from the full variational problem.
Furthermore, it holds the following best approximation theorem.
\begin{theorem}[Best approximation for coercive problems (Céa's Lemma)]
\label{thm:coercive bestapproximation}
Let \gls{a} be coercive and symmetric.
The error of the solution $\gls{rsol}$ of the 
reduced variational problem in \cref{def:reduced variational problem}
with respect to the full variational problem in \cref{def:variational problem}
is bounded by the constant $\sqrt{\tfrac{\gls{contconst}}{\gls{coercconst}}}$ times the smallest error possible in \gls{rfspace}.
It holds 
\begin{equation}
\norm{\gls{sol} - \gls{rsol}} \leq \sqrt{\frac{\gls{contconst}}{\gls{coercconst}}} \inf_{\widetilde v \in \gls{rfspace}} \norm{\gls{sol} - \widetilde v}.
\end{equation}
\end{theorem}
\begin{proof}
\cite[Proposition 3.1]{cea1964approximation}
\end{proof}

\subsection{Properties of Reduced Variational Problem in inf-sup Stable Case}
\label{sec:reduction_inf_sup}
While the reduced problem inherits the coercivity constant from the full
problem in the coercive case, the inf-sup stable case is more involved.
The stability still holds, i.e.
\begin{equation}
\norm{\gls{rsol}} \leq \frac{1}{\gls{rinfsupconst}} \norm{\gls{f}}_{\gls{fspace}'}
\end{equation}
and there is still a best approximation estimate:
\begin{theorem}[Best approximation for inf-sup stable problems]
\label{thm:bestapproxinfsup}
Let \gls{a} be inf-sup stable.
The error of the solution $\gls{rsol}$ of the 
reduced variational problem in \cref{def:reduced variational problem}
with respect to the full variational problem in \cref{def:variational problem}
is bounded by the constant $\tfrac{\gls{contconst}}{\gls{rinfsupconst}}$ times the smallest error possible in \gls{rfspace}.
It holds
\begin{equation}
\norm{\gls{sol} - \gls{rsol}}
\leq
\frac{\gls{contconst}}{\gls{rinfsupconst}}
\inf_{\widetilde v \in \gls{rfspace}}
\norm{\gls{sol} - \widetilde v}
.
\end{equation}
\end{theorem}
\begin{proof}
\cite{Xu2003}
\end{proof}

However,
the reduced inf-sup constant \gls{rinfsupconst}
might be larger or smaller than the inf-sup constant
in the full space \gls{infsupconst}.
This means that the reduced variational problem might not 
be inf-sup stable at all or have a very small inf-sup constant,
even if the full problem has good constants.
\section{Parameterized Problem}
\label{sec:parameterized problem}
While the methods developed in this thesis work well for
non parametric problems, they work especially well
for parametric problems, inheriting
the favorable properties of the reduced basis method,
as already described in \cref{sec:intro rb localization} above.
For the parametric case,
let $\gls{parameterspace} \subset \gls{R}^{\gls{numparameter}}$
be the parameter space and $\gls{parameter}$ in \gls{parameterspace}
be the parameter.
The entities
\gls{a},
\gls{f},
\gls{sol},
\gls{rsol},
\gls{coercconst},
\gls{rcoercconst},
\gls{infsupconst},
\gls{rinfsupconst},
\gls{contconst}, and
\gls{rcontconst}
are replaced by parameter dependent counterparts
$\gls{a}_{\gls{parameter}}$,
$\gls{f}_{\gls{parameter}}$,
$\gls{sol}_{\gls{parameter}}$,
$\gls{rsol}_{\gls{parameter}}$,
$\gls{coercconst}_{\gls{parameter}}$,
$\gls{rcoercconst}_{\gls{parameter}}$,
$\gls{infsupconst}_{\gls{parameter}}$,
$\gls{rinfsupconst}_{\gls{parameter}}$,
$\gls{contconst}_{\gls{parameter}}$, and
$\gls{rcontconst}_{\gls{parameter}}$
.
The spaces \gls{fspace} and \gls{rfspace}
are not parameter specific.
Moreover, we assume that $\gls{a}_{\gls{parameter}}$ and $\gls{f}_{\gls{parameter}}$ exhibit
an affine parameter dependence, i.e.~there exist parameter independent
bilinear forms $\gls{a}^q : \gls{fspace} \times \gls{fspace} \rightarrow \gls{K}$ $(1 \leq q \leq
Q_a)$, continuous (anti-)linear functionals $\gls{f}^q \in \gls{fspace}'$ $(1 \leq q \leq Q_f)$
 and
coefficient functionals $\theta_a^q : \gls{parameterspace} \rightarrow \gls{K}$ and $\theta_f^q : \gls{parameterspace} \rightarrow \gls{K}$ 
such that
\begin{equation}
\label{eq:affineaf}
a_{\gls{parameter}}(\varphi_1,\varphi_2) = \sum_{q=1}^{Q_a} \theta_a^q({\gls{parameter}}) a^q(\varphi_1, \varphi_2)
\qquad\mathrm{and} \qquad
f_{\gls{parameter}}(\varphi) = \sum_{q=1}^{Q_f} \theta_f^q({\gls{parameter}}) f^q(\varphi).
\end{equation}
We further assume that there is a lower bound for the coercivity constant
$\gls{coercconst}_\mathrm{LB} \in \gls{R}$ for which holds
\begin{equation}
0 < \gls{coercconst}_\mathrm{LB} \leq \gls{coercconst}_{\gls{parameter}} \qquad \forall \gls{parameter} \in \gls{parameterspace}
\end{equation}
and an upper bound for the continuity constant $\gls{contconst}_\mathrm{UB} \in \gls{R}$ for which holds
\begin{equation}
\gls{contconst}_\mathrm{UB} \geq \gls{contconst}_{\gls{parameter}} \qquad \forall \gls{parameter} \in \gls{parameterspace}
.
\end{equation}

\section{Localized Model Order Reduction}
\label{sec:localized model order reduction}
In this section, we introduce an abstract
notation for localized model order reduction.
First, we define an abstract localizing space decomposition
and subsequently detail how we define
the model order reduction based on this decomposition.
\subsection{Localizing Space Decomposition}
\label{sec:localizing space decomposition}
Localization of the reduced basis method is based on a space decomposition
of a global space \gls{fspace}.
We define an abstract localizing space decomposition.
\begin{definition}[Localizing Space Decomposition]
\label{def:localizing space decomposition}
A localizing space decomposition consists of
\begin{enumerate}
\item
an open, bounded and connected domain $\gls{domain} \subset \gls{R}^{\gls{dimension}}$
with smooth boundary,
\item
a Hilbert space \gls{fspace} of functions on \gls{domain} which is
a continuous or discrete function space,
\item
\gls{numspaces} subdomains $\gls{subdomain}$ which form an 
overlapping or non overlapping domain decomposition of the global 
domain $\gls{domain}$,
\item
\gls{numspaces} local spaces $\gls{lfspace}$ which are
subspaces of the global ansatz space $\gls{fspace}$ and satisfy
\begin{equation}
\sum_i \gls{lfspace} = \gls{fspace} 
\end{equation}
and
\begin{equation}
\supp(\varphi) \subseteq \gls{subdomain} \qquad \forall \varphi \in \gls{lfspace},
\end{equation}
\item
\gls{numspaces} linear mappings
$\gls{lfspacemap} : \gls{fspace} \rightarrow \gls{lfspace}$
for which it holds
\begin{equation}
\sum_i \gls{lfspacemap}(\varphi) = \varphi \qquad \forall \varphi \in \gls{fspace} . 
\end{equation}
\end{enumerate}
\end{definition}

\subsection{Localized Model Order Reduction}
\label{sec:lmor}
Based on the abstract space decomposition introduced above, we can
define an abstract scheme for localized model order
reduction.
The core task in localized model order reduction is
the construction of reduced subspaces $\gls{rlfspace} \subset \gls{lfspace}$
for all localized spaces. Each local reduced subspace $\gls{rlfspace}$
should approximate the local part of the solution $\gls{lfspacemap}(\gls{sol}_{\gls{parameter}})$.
The global reduced space is then defined as
\begin{equation}
\gls{rfspace} := \sum_i \gls{rlfspace}
\end{equation}
which is used as the reduced space
in the Problem in \cref{def:reduced variational problem}.
\subsection{Examples for Localizing Space Decomposition}
\label{sec:introduction space decompositions}
We present two examples for localizing space decompositions.
The space decomposition used mostly throughout this thesis,
the ``wirebasket space decomposition'',
is introduced in \cref{sec:wirebasket space decomposition}.
\subsubsection{Restriction Decomposition}
\label{sec:restriction decomposition}
Let \gls{fspacedg} be a discrete \gls{dg} space on a fine mesh
on the domain $\gls{domain}$.
Let $\gls{subdomainnol}$, $i \in \{1, \dots, \gls{numdomainsnol}\}$
be a non overlapping domain decomposition which is resolved by the fine mesh:
\begin{equation}
\overline{\gls{domain}} = \bigcup_{i=1}^{\gls{numdomainsnol}}\overline{\gls{subdomainnol}}  \qquad \gls{subdomainnol} \cap \omega_j^\mathrm{nol} = \emptyset\ \mathrm{for}\ i\ne j
.
\label{eq:dd}
\end{equation}
In this setting, we can
define the local spaces as
\begin{equation}
\gls{lfspaceres} := \left\{ \varphi|_{\gls{subdomainnol}} \, \Big| \, \varphi \in \gls{fspacedg} \right\}
\end{equation}
and define the linear mappings $\gls{lfspacemapres}$ as the restriction of $\varphi \in \gls{fspacedg}$
to $\gls{subdomainnol}$:
\begin{equation}
\gls{lfspacemapres}(\varphi) := \varphi|_{\gls{subdomainnol}}.
\end{equation}
We call 
$\left\{
\gls{domain}, \gls{fspacedg},
\{\gls{subdomainnol}\}_{i=1}^{\gls{numdomainsnol}},
\{\gls{lfspaceres}\}_{i=1}^{\gls{numdomainsnol}},
\{\gls{lfspacemapres}\}_{i=1}^{\gls{numdomainsnol}}
\right\}$
the "restriction decomposition".
This is the space decomposition used in the \gls{lrbms} \cite{Ohlberger2015}.
It has the advantages that it is easy to understand,
the mappings to the local subspaces are projections,
and
functions which are in two different local subspaces
can not be linear dependent.
Its main disadvantage is that  the function space has to be a \gls{dg}
space.
A restriction decomposition can not be implemented on top of
existing \gls{fem} codes using conforming function spaces.

\subsubsection{Partition of Unity Decomposition}
\label{sec:poudecomposition}
Another option is to base the space decomposition on a
continuous partition of unity on an overlapping domain decomposition
$\gls{subdomainol}$.
Let $\gls{subdomainol}$, $i \in \{1, \dots, \gls{numdomainsol}\}$
be an overlapping domain decomposition which is resolved by the fine mesh.
Let further $\gls{pouf}$ be a partition of unity functions for which it holds
\begin{equation}
\gls{pouf} \geq 0,
\qquad
\sum_i \gls{pouf} \equiv 1,
\qquad
\mathrm{supp}(\gls{pouf}) \subseteq \gls{subdomainol},
\qquad
\norm{\nabla \gls{pouf}}_{L^\infty} < C
\end{equation}
and let \gls{fspaceh} be a global finite element space.
With a linear interpolation operator $\gls{feinterpolation} : \gls{fspace} \rightarrow \gls{fspaceh}$,
the local spaces are defined as
\begin{equation}
\gls{lfspacepou} := \left\{\gls{feinterpolation}(\gls{pouf} \varphi)\ \ \Big| \ \ \varphi \in \gls{fspaceh} \right\}
\end{equation}
and the linear mappings are defined as
\begin{equation}
\gls{lfspacemappou}(\varphi) := \gls{feinterpolation}(\gls{pouf} \varphi) \qquad \forall \varphi \in \gls{fspaceh} .
\end{equation}
We call 
$\left\{
\gls{domain}, \gls{fspaceh},
\{\gls{subdomainol}\}_{i=1}^{\gls{numdomainsol}},
\{\gls{lfspacepou}\}_{i=1}^{\gls{numdomainsol}},
\{\gls{lfspacemappou}\}_{i=1}^{\gls{numdomainsol}}
\right\}$
the "partition of unity decomposition".
This space decomposition is quite popular. 
It or its continuous counterpart is used for example
in the partition of unity finite element method \cite{Melenk1996}
and in the \gls{gfem} \cite{Strouboulis2001}.
It is easy to understand and easy to implement.
Its main disadvantage is the fact that functions
in different local subspaces sharing support
can be linear dependent.
This can lead to singular system matrices.
Note that the restriction decomposition can be seen as a
partition of unity decomposition with a discontinuous
partition of unity.

\section{Heat Conduction}
In this section, we define the problem 
of stationary heat conduction
and derive the corresponding bilinear and 
linear form of the underlying physical
equations.
While we aim at the time harmonic Maxwell's equation
in this thesis, it is often beneficial
to test numerical methods on the simpler case
of the stationary heat equation first.
The stationary heat equation leads to a coercive
bilinear form
and thus avoids stability
problems in the model order reduction,
which are discussed above in 
\cref{sec:reduction_inf_sup}.
\subsection{Full Heat Conduction Equation}
The full heat equation is a parabolic \gls{pde}
describing the diffusion of heat over time.
The unknown is the temperature $u(x,t)$,
a scalar field dependent on space and time.
With the space dependent heat conductivity $\gls{heat_conductivity}(x)$
and a space and time dependent heating
term $\gls{heatsource}(x,t)$, the full heat equation reads
\begin{equation}
\frac{\partial}{\partial t} \gls{sol}(x,t) 
=
\mathrm{div} \Big( \gls{heat_conductivity}(x) \nabla \gls{sol}(x,t) \Big) + \gls{heatsource}(x,t)
.
\label{eq:full heat equation}
\end{equation}
The gradient of the temperature $\nabla \gls{sol}(x,t)$ leads
to a flux of thermal energy of $\gls{heat_conductivity}(x) \nabla \gls{sol}(x,t)$.
The divergence of this flux of thermal energy results
in a temperature change over time.
This equation is assumed to hold pointwise. It does
not define a solution \gls{sol} without
specifying a domain, boundary, and initial conditions.
\subsection{Stationary Heat Conduction}
Systems governed by the full heat equation (\cref{eq:full heat equation})
evolve towards a steady state if the heat term $\gls{heatsource}(x,t)$ does not depend
on time (and thus is stationary).
This steady state is characterized by a vanishing time
derivative. We can compute the steady state solution using
the stationary heat equation, which we derive
from the full heat equation (\cref{eq:full heat equation})
by setting the time derivative to zero and
replacing the space and time dependent fields by only space 
dependent fields $\gls{sol}(x)$ and $\gls{heatsource}(x)$.
We obtain
\begin{equation}
- \mathrm{div} \Big( \gls{heat_conductivity}(x) \nabla \gls{sol}(x) \Big) = \gls{heatsource}(x)
.
\label{eq:stationary heat equation}
\end{equation}
\subsection{Weak Formulation of Heat Conduction}
For the \gls{fe} approximation, one considers 
the stationary heat equations in weak form. 
We consider the problem on a domain \gls{domain}
which we assume to be Lipschitz.
We multiply by a sufficiently smooth test function $v$
and integrate over the calculation domain $\gls{domain}$.
We obtain
\begin{equation}
- \int_{\gls{domain}} \mathrm{div} \Big( \gls{heat_conductivity}(x) \nabla \gls{sol}(x) \Big) v(x) \dx = \int_{\gls{domain}} \gls{heatsource}(x) v(x) \dx
.
\label{eq:first weak heat equation}
\end{equation}
Using the identity
\begin{equation}
\mathrm{div} \Big( \gls{heat_conductivity}(x) \nabla \gls{sol}(x) v(x) \Big) = \mathrm{div} \Big( \gls{heat_conductivity}(x) \nabla \gls{sol}(x) \Big) v(x) + \gls{heat_conductivity}(x) \nabla \gls{sol}(x) \cdot \nabla v(x)
\end{equation}
and Stokes' law
\begin{equation}
\int_{\gls{domain}} \mathrm{div} \Big( \gls{heat_conductivity}(x) \nabla \gls{sol}(x) v(x) \Big) 
\dx = 
\int_{\partial \gls{domain}}
\gls{heat_conductivity}(x) \Big( \nabla \gls{sol}(x) \cdot n \Big) 
v(x) \dx
\end{equation}
where $n$ is the outer unit normal, we can rewrite \cref{eq:first weak heat equation}
as
\begin{equation}
\int_{\gls{domain}} \gls{heat_conductivity}(x) \nabla \gls{sol}(x) \cdot \nabla v(x) \dx
-\int_{\partial \gls{domain}}
\gls{heat_conductivity}(x) \Big( \nabla \gls{sol}(x) \cdot n \Big) 
v(x) \dx
= \int_{\gls{domain}} \gls{heatsource}(x) v(x) \dx
.
\end{equation}
In this form, we only require $H^1$ regularity
and can formulate the problem in the
variational form
with the bilinear form and linear form
\begin{align}
\label{eq:weak heat equation}
a(u,v) &:= 
\int_{\gls{domain}} \gls{heat_conductivity}(x) \nabla \gls{sol}(x) \cdot \nabla v(x) \dx
-\int_{\partial \gls{domain}}
\gls{heat_conductivity}(x) \Big( \nabla \gls{sol}(x) \cdot n \Big) 
v(x) \dx \\
\gls{f}(v) &:= \int_{\gls{domain}} \gls{heatsource}(x) v(x) \dx
. \nonumber
\end{align}
For a parameterized heat conductivity $\gls{heat_conductivity}_{\gls{parameter}}(x)$, the bilinear form inherits the affine parameter
dependence if the heat conductivity itself
has an affine parameter dependence of the form
\begin{equation}
\label{eq:affineafheat}
\gls{heat_conductivity}_{\gls{parameter}}(x)
= \sum_{q=1}^{Q_{\gls{heat_conductivity}}} \theta_{\gls{heat_conductivity}}^q({\gls{parameter}}) {\gls{heat_conductivity}}^q(x)
.
\end{equation}
On the boundary \gls{boundary}
we impose either Dirichlet boundary conditions
\begin{equation}
u(x) = 0\qquad \text{on } \ \gls{boundary}_D
\end{equation}
or Neumann boundary conditions
\begin{equation}
\nabla \gls{sol}(x) \cdot n = 0\qquad \text{on } \gls{boundary}_N
\end{equation}
where $\gls{boundary}_D \cup \gls{boundary}_N = \gls{boundary} := \partial \gls{domain}$
and we assume $\gls{boundary}_D$ has a measure greater zero,
so Friedrich's inequality holds.

We consider the space
\begin{equation}
\label{eq:fspace heat}
V := \Big\{ \varphi \in H^1(\gls{domain}) \quad \Big| \quad \varphi|_{\gls{boundary}_D} = 0 \Big\}
\end{equation}
where the restriction to the boundary is to be understood in the sense of traces.

The problem defined in \cref{def:variational problem}
with the bilinear and linear form
defined in \cref{eq:weak heat equation}
poses a well defined problem in the space $V$ defined above
if $\gls{heat_conductivity}$ is bound from below and from above, i.e.
\begin{equation}
0 < \gls{heat_conductivity}_\mathrm{min} \leq \gls{heat_conductivity}(x) \leq \gls{heat_conductivity}_\mathrm{max} \qquad \forall x \in \gls{domain}
.
\end{equation}
In the usual $H^1(\gls{domain})$ norm it holds $\gls{contconst} \leq \gls{heat_conductivity}_\mathrm{max}$
and $\gls{coercconst} \geq \frac{\gls{heat_conductivity}_\mathrm{min}}{1 + \gls{friedrichs}^2}$
where $\gls{friedrichs}$ is the constant in the Friedrich's inequality $\norm{u}_{L^2(\gls{domain})} \leq \gls{friedrichs} \norm{\nabla u}_{L^2(\gls{domain})}$.
Because of the coercivity and continuity of \gls{a}, we could also use
the energy norm induced by $\norm{\varphi}_E^2 := (\varphi,\varphi)_E := \gls{a}(\varphi,\varphi)$ where the coercivity and continuity constant are equal to one,
but we stick to the $H^1$ norm throughout this thesis.

For the \gls{fe} discretization, we use linear finite elements.

\section{Maxwell's Equations}
 \subsection{Full Maxwell's Equations in Original Form}
The
Maxwell's equations were
first published by James Clerk Maxwell in 1861
\cite{Maxwell1861} and 1865
\cite{maxwell1864dynamical}.
However, he did not have the
vector notation which is common today.
The form of the Maxwell's equations which
is commonly in use today and stated below
was published by Oliver Heaviside in 1892
\cite{heavisideelectricalpapers}.
In Maxwell's equations, the unknowns are
the electric field \gls{electric_field},
the electric flux \gls{electric_flux},
the magnetic field \gls{magnetic_field},
and the magnetic flux \gls{magnetic_flux}.
In addition, there is the charge density
\gls{charge_density} and the
current density \gls{current_density},
which are sources or unknowns, depending
on the problem setting.
\gls{electric_field}, \gls{electric_flux}, \gls{magnetic_field}, \gls{magnetic_flux}, and \gls{current_density}
are space- and time-dependent vector fields while \gls{charge_density}
is a space- and time-dependent scalar field.
The classic Maxwell's equations read
\begin{subequations}
\begin{eqnarray}
\nabla \cdot \gls{electric_flux} &=& \gls{charge_density} \label{eq:maxwell gauss law}\\
\nabla \cdot \gls{magnetic_flux} &=& 0 \label{eq:maxwell magnetic gauss law}\\
\nabla \times \gls{electric_field} &=& - \frac{\partial}{\partial t} \gls{magnetic_flux} \label{eq:maxwell faradays law}\\
\nabla \times \gls{magnetic_field} &=& \frac{\partial}{\partial t} \gls{electric_flux} + \gls{current_density} .\label{eq:maxwell apere law}
\end{eqnarray}
\end{subequations}
Together with the electric permittivity \gls{electric_permittivity}
and the magnetic permeability \gls{magnetic_permeability}
in the material laws
\begin{equation}
\gls{electric_flux} = \gls{electric_permittivity} \gls{electric_field}
\qquad
\gls{magnetic_flux} = \gls{magnetic_permeability} \gls{magnetic_field}
\end{equation}
and suitable boundary conditions they define the problem.
Equation \eqref{eq:maxwell gauss law} is also known as Gauss's law,
Equation \eqref{eq:maxwell magnetic gauss law} is also known as Gauss's law for magnetism,
Equation \eqref{eq:maxwell faradays law} is also known as Faraday's law,
and
Equation \eqref{eq:maxwell apere law} is also known as Ampère's law with Maxwell's addition.

The material properties \gls{electric_permittivity} and \gls{magnetic_permeability}
are in general time, space, and field dependent tensors,
but we assume them to be only space dependent scalars throughout this thesis.
We only consider examples where the magnetic permeability is equal
to the magnetic permeability in vacuum \gls{magnetic_permeability}.
\subsection{Time Harmonic Maxwell's Equation}
Depending on the application, it might be advantageous
to consider the time harmonic Maxwell's equations.
The time harmonic Maxwell's equations can be derived in
two ways (see also \cite{monk2003finite,Zaglmayr2006}).
One either assumes that all quantities are time harmonic
and uses the ansatz
\begin{equation}
\gls{some_field}(x,t) = \gls{realpart} \left( \widehat{\gls{some_field}}(x) e^{i \gls{angular_frequency} t} \right)
\end{equation}
for $\gls{some_field}$ in $\Big\{ \gls{electric_field}, \gls{electric_flux}, \gls{magnetic_field}, \gls{magnetic_flux}, \gls{charge_density},
\gls{current_density} \Big\}$
where 
$
\widehat{\gls{electric_field}}(x),
\widehat{\gls{electric_flux}}(x),
\widehat{\gls{magnetic_field}}(x),
\widehat{\gls{magnetic_flux}}(x), \text{ and }
\widehat{\gls{current_density}}(x)
$
are complex valued vector fields only
dependent on space and
$
\widehat{\gls{charge_density}}(x)
$
is a complex valued scalar field
only dependent on space.

The second option to obtain the time harmonic Maxwell's equations
is to use the Fourier transform of the fields involved and use
\begin{align}
\widehat{\gls{some_field}}(x, \gls{angular_frequency}) &= \int_{-\infty}^\infty \gls{some_field}(x,t) \ e^{- i \gls{angular_frequency} t} \df t &
\gls{some_field}(x, t) &= \frac{1}{2 \pi} \int_{-\infty}^\infty \widehat{\gls{some_field}}(x,\gls{angular_frequency}) \ e^{i \gls{angular_frequency} t} \df\gls{angular_frequency} 
\end{align}
for $\gls{some_field}$ in $\Big\{ \gls{electric_field}, \gls{electric_flux}, \gls{magnetic_field}, \gls{magnetic_flux}, \gls{charge_density},
\gls{current_density} \Big\}$.
Under the assumption that
the materials are linear,
one can obtain the time harmonic Maxwell's equations
\begin{subequations}
\begin{eqnarray}
\nabla \cdot \widehat{\gls{electric_flux}} &=& \widehat{\gls{charge_density}}\\
\nabla \cdot \widehat{\gls{magnetic_flux}} &=& 0 \\
\nabla \times \widehat{\gls{electric_field}} &=& - i \gls{angular_frequency} \widehat{\gls{magnetic_flux}}\\
\nabla \times \widehat{\gls{magnetic_field}} &=& i \gls{angular_frequency} \widehat{\gls{electric_flux}} + \widehat{\gls{current_density}}
\end{eqnarray}
\end{subequations}
where \gls{angular_frequency} is the angular frequency.
In the following, we do not differentiate
between the time-dependent and time-harmonic quantities.
\subsection{Reduction to one Unknown Field}
\label{sec:reduction one field}
Maxwell's equations can be reduced to one equation
for the electric field by first 
dividing Faraday's law (Equation \eqref{eq:maxwell faradays law})
by the magnetic permeability \gls{magnetic_permeability}
and taking its curl, obtaining
\begin{equation}
\nabla \times \frac{1}{\gls{magnetic_permeability}} \nabla \times \gls{electric_field} = - \nabla \times \frac{\partial}{\partial t} \gls{magnetic_field}.
\end{equation}
Exchanging the curl with the time derivative on the right hand side
and plugging in Ampère's law (Equation \eqref{eq:maxwell apere law})
one obtains
\begin{equation}
\nabla \times \frac{1}{\gls{magnetic_permeability}} \nabla \times \gls{electric_field} = 
- \frac{\partial ^2}{\partial t^2} \gls{electric_permittivity} \gls{electric_field}
- \frac{\partial}{\partial t} \gls{current_density}
\end{equation}
or in the time harmonic setting
\begin{equation}
\label{eq:thcurlcurl}
\nabla \times \frac{1}{\gls{magnetic_permeability}} \nabla \times \gls{electric_field} = 
\gls{angular_frequency}^2 \gls{electric_permittivity} \gls{electric_field}
- i \gls{angular_frequency} \gls{current_density}
.
\end{equation}
We will use this equation as starting point for our simulations.

\subsection{Weak Formulation of Time Harmonic Maxwell's Equation}
For the \gls{fe} approximation, one considers Maxwell's
equations in weak form. 
We consider the problem on a domain \gls{domain}
which we assume to be Lipschitz.
For the time harmonic setting, one starts with
\cref{eq:thcurlcurl}, multiplies it with a sufficiently smooth,
vector valued test function
and integrates over the domain \gls{domain}.
In the weak setting, we denote the solution for the
electric field by \gls{sol} and the test function
by $\overline v$
and omit the space dependency for notational brevity.
One obtains
\begin{equation}
\int_{\gls{domain}}
\left(
\nabla \times \frac{1}{\gls{magnetic_permeability}} \nabla \times \gls{sol}
-\gls{angular_frequency}^2 \gls{electric_permittivity} \gls{sol} 
\right)
\cdot \overline v
 \dx
= 
- i \gls{angular_frequency} 
\int_{\gls{domain}}
\gls{current_density} \cdot \overline v
 \dx
\qquad \forall v \in V .
\end{equation}
We use Stokes' theorem (see e.g.~\cite[Eq. 3.27]{monk2003finite})
\begin{equation}
\int_{\gls{domain}}
\left( \nabla \times \gls{some_field} \right) \cdot \overline v \df V
=
\int_{\gls{domain}}
\gls{some_field} \cdot \left( \nabla \times \overline v \right) \df V
+ \int_{\partial \gls{domain}}
\left( n \times \gls{some_field} \right) \cdot \overline v \df A
\end{equation}
where $n$ is the outer normal on the boundary.
Applying this
with $\gls{some_field} = \frac{1}{\gls{magnetic_permeability}} \nabla \times u$
yields
\begin{equation}
\int_{\gls{domain}}
\frac{1}{\gls{magnetic_permeability}} \left( \nabla \times \gls{sol} \right)
\cdot 
\left( \nabla \times \overline v \right)
-\gls{angular_frequency}^2 \gls{electric_permittivity} \gls{sol} 
\cdot \overline v
 \dx
+ \int_{\partial \gls{domain}}
\left( n \times 
\frac{1}{\gls{magnetic_permeability}} \nabla \times \gls{sol}
\right) \cdot \overline v
 \dx
= 
- i \gls{angular_frequency} 
\int_{\gls{domain}}
\gls{current_density} \cdot \overline v
 \dx
\qquad \forall v \in V .
\end{equation}
The term
\begin{equation}
\int_{\partial \gls{domain}}
\left( n \times 
\frac{1}{\gls{magnetic_permeability}} \nabla \times \gls{sol}
\right) \cdot \overline v
 \dx
\end{equation}
is a triple product, which is invariant under cyclic permutation.
So it is equal to 
\begin{equation}
\label{eq:maxwell boundary term}
\int_{\partial \gls{domain}}
\left(\frac{1}{\gls{magnetic_permeability}} \nabla \times \gls{sol} \right)
\cdot 
\left( 
\overline v
\times 
n 
\right)
 \dx
.
\end{equation}
On the boundary $\gls{boundary} := \partial \gls{domain}$
we impose either Dirichlet, Neumann, or
conducting wall boundary conditions. We denote
the parts of the boundary by $\gls{boundary}_D$,
$\gls{boundary}_N$, and $\gls{boundary}_R$
respectively and assume
$\gls{boundary}_D \cup \gls{boundary}_N \cup \gls{boundary}_R = \gls{boundary}$.
On the Dirichlet boundary we set the tangential part
of the electric field to zero
\begin{equation}
\gls{sol} \times n = 0\qquad \text{on } \gls{boundary}_D.
\end{equation}
On the Neumann boundary we set the normal derivative to the electric
field to zero
\begin{equation}
n \times \nabla \times \gls{sol} = 0\qquad \text{on } \gls{boundary}_N.
\end{equation}
On the conducting wall boundary, we
impose the radiation condition
(cf.~\cite[Eq.~(1.25c)]{monk2003finite}\footnote{Note that we use a different sign convention than Peter Monk.})
\begin{equation}
\left( n \times 
\frac{1}{\gls{magnetic_permeability}} \nabla \times \gls{sol}
\right) \cdot \overline v
=
i \gls{angular_frequency} \gls{electric_conductivity} \gls{sol}_T
\end{equation}
where 
$\gls{sol}_T = n \times (\gls{sol} \times n)$ is the tangential part
of the solution at the boundary
and \gls{electric_conductivity} is the specific electric
conductivity at the boundary.

The problem can thus be formulated using a sesquilinear form
and an antilinear form defined by
\begin{align}
\label{eq:maxwell bilinear linear form}
a(\gls{sol},v;\gls{angular_frequency}) :=& \int_{\gls{domain}} \frac{1}{\gls{magnetic_permeability}} (\nabla \times \gls{sol}) \cdot (\nabla \times \overline v)
- \gls{electric_permittivity} \gls{angular_frequency} ^2 (\gls{sol} \cdot \overline v) \dx \nonumber
\\&
+ i \gls{angular_frequency} \gls{electric_conductivity} \int_{\gls{boundary}_R} (\gls{sol} \times n) \cdot (\overline v \times n) \dx,
\nonumber\\
f(v;\gls{angular_frequency}) :=& - i \gls{angular_frequency} \int_{\gls{domain}} (\gls{current_density} \cdot \overline v ) \dx
\end{align}
where $\gls{angular_frequency}$ is again the angular frequency. 
Note that we have chosen the sesquilinear form to be antilinear in the test
function, as this is the common choice in the literature (cf.~\cref{chap:complex handling in pymor}).
We see $\gls{angular_frequency}$ as a parameterization to this problem.

We consider the space
\begin{equation}
\label{eq:maxwellspace}
V := \Big\{ \varphi \in \gls{hcurl} \, \Big| \, \left(\varphi \times n \right)\big|_{\gls{boundary}_D} = 0 \Big\}
\end{equation}
where the restriction to the boundary is to be understood in the sense of traces.
The space \gls{hcurl} is defined as usual as
\begin{equation}
\gls{hcurl} := \Big\{ \varphi \in \left( L^2(\gls{domain}) \right) ^{\gls{dimension}}
\ \ \Big|\ \ 
\nabla \times \varphi \in \left( L^2(\gls{domain}) \right) ^{\gls{dimension}}
\Big\}.
\end{equation}
We consider the problem defined in \cref{def:variational problem}
with the bilinear and linear form
defined in \cref{eq:maxwell bilinear linear form}
in the space $\gls{fspace}$ defined above.
We use the problem specific inner product
\begin{eqnarray}
(\varphi_1, \varphi_2) &:=&
\int_{\gls{domain}} \frac{1}{\gls{magnetic_permeability}} (\nabla \times \varphi_2) \cdot (\nabla \times \overline \varphi_1)
+ \gls{electric_permittivity} \gls{angular_frequency}_\mathrm{max} ^2 (\varphi_2 \cdot \overline \varphi_1) \dx \nonumber
\\&&
+ \ \gls{angular_frequency}_\mathrm{max} \gls{electric_conductivity} \int_{\gls{boundary}_R} (\varphi_2 \times n) \cdot (\overline \varphi_1 \times n) \dx
\label{eq:maxwell energy norm}
\end{eqnarray}
and the thereby induced norm where $\gls{angular_frequency}_\mathrm{max}$
is the problem specific maximum angular frequency considered.
We can thus estimate the continuity constant by one, but we cannot
in general estimate the inf-sup constant from below.
Obtaining a lower bound for the inf-sup constant
is a challenging task which we omit in this thesis.
Lower bounds can be obtained numerically for example
by solving an eigenvalue problem or by the
\gls{scm} \cite{Huynh2007,Chen2009},
but these approaches are unsuitable for our targeted
application.

The bilinear form has an affine parameter dependence
\begin{equation}
\gls{a}(\gls{sol},v;\gls{angular_frequency}) =
\gls{a}_{cc}(u,v) - \gls{angular_frequency}^2 \gls{a}_{ii}(u,v)
+ i \gls{angular_frequency} \gls{a}_{b}(u,v)
\end{equation}
with
\begin{subequations}
\begin{equation}
\gls{a}_{cc} :=  \int_{\gls{domain}} \frac{1}{\gls{magnetic_permeability}} (\nabla \times \gls{sol}) \cdot (\nabla \times \overline v)\\
\end{equation}
and
\begin{equation}
\gls{a}_{ii} := \int_{\gls{domain}} \gls{electric_permittivity} (\gls{sol} \cdot \overline v) \dx
\end{equation}
and
\begin{equation}
\gls{a}_{b} := \int_{\gls{boundary}_R} \gls{electric_conductivity} (\gls{sol} \times n) \cdot (\overline v\times n) \dx
.
\end{equation}
\end{subequations}

\subsection{Finite Element Approximation in \gls{hcurl}}
\label{sec:nedelec ansatz}
\begin{figure}
\centering
\begin{tabular}{ccc}
\includegraphics[width=0.25\textwidth]{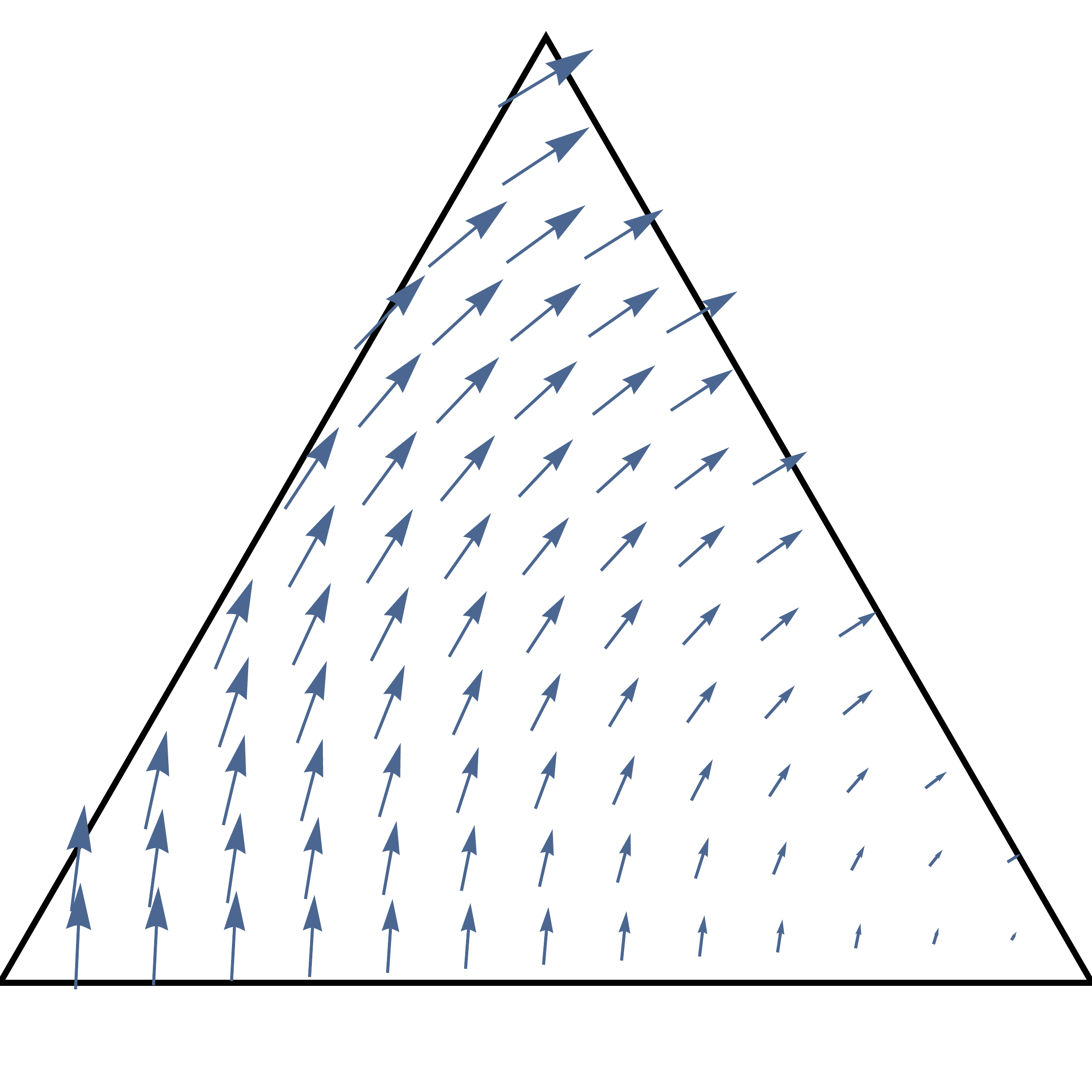}&
\includegraphics[width=0.25\textwidth]{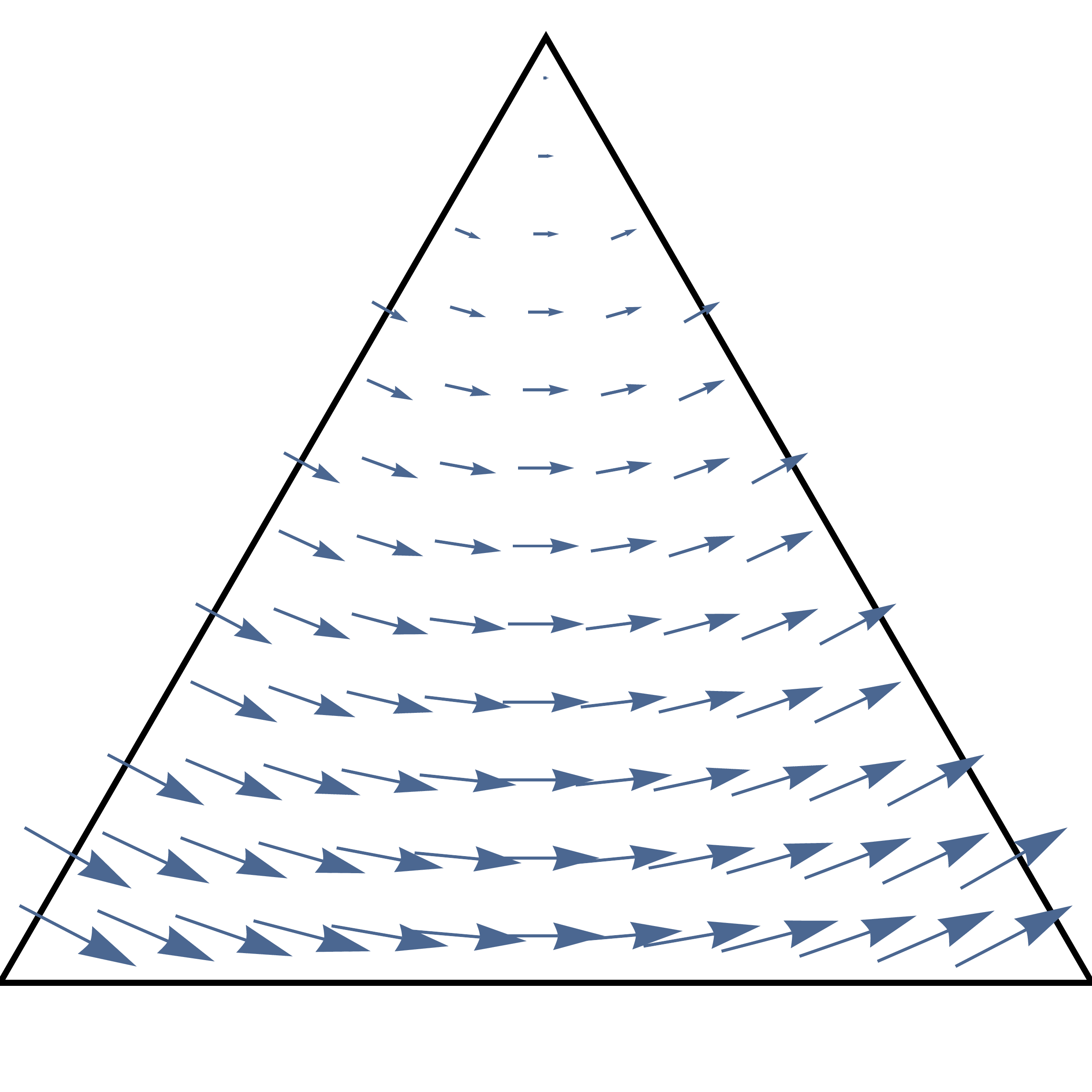}&
\includegraphics[width=0.25\textwidth]{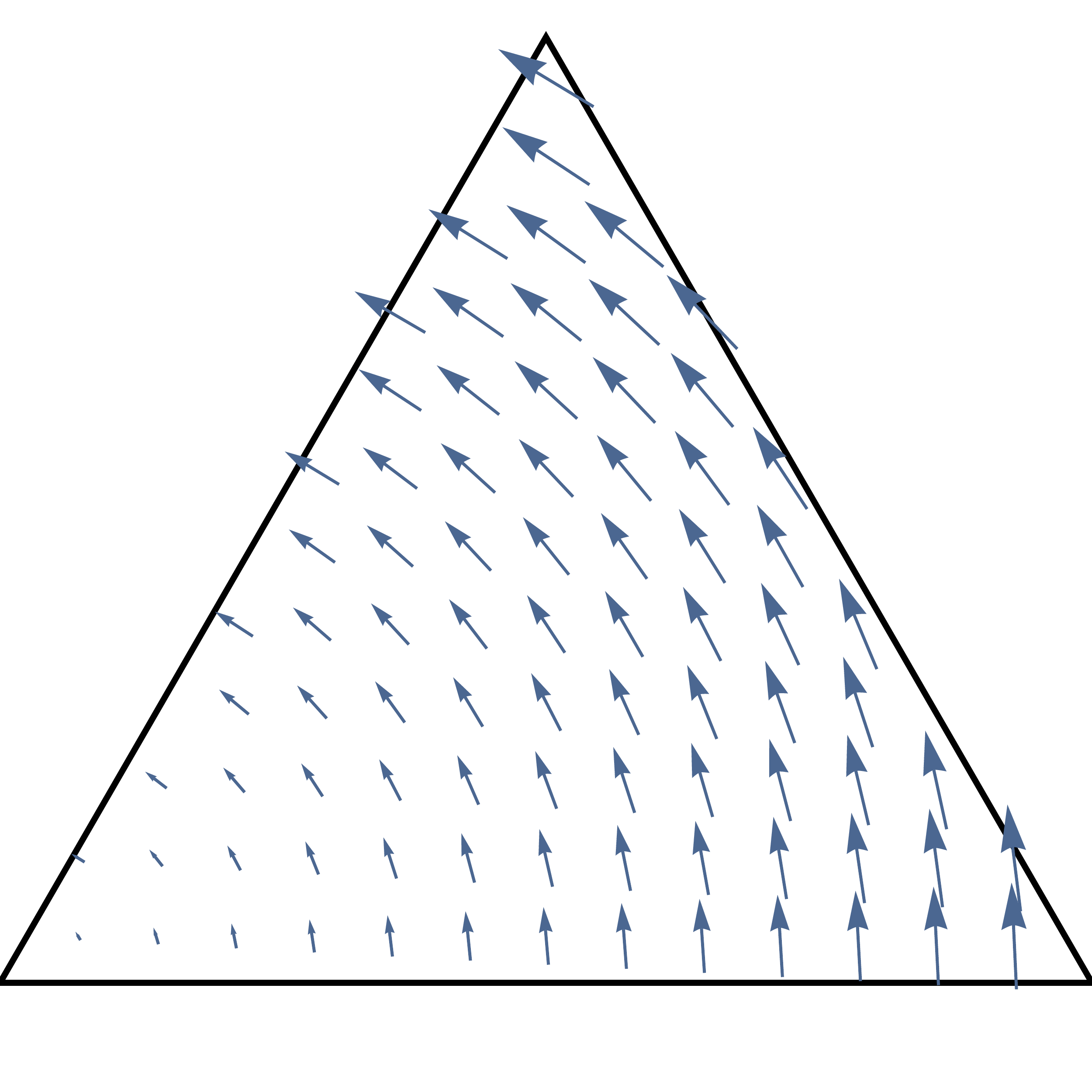}\\
first & second & third
\end{tabular}
\caption[Visualization of Nédélec ansatz functions (edge elements) for one triangle.]
        {Visualization of Nédélec ansatz functions (edge elements) for one triangle (reproduction: \cref{repro:fig:nedelec_ansatz}).}
\label{fig:nedelec_ansatz}
\end{figure}

For the finite element approximation of \gls{hcurl}
the most common approach is to use Nédélec's ansatz functions.
Nédélec published two articles introducing
his ansatz functions, one in 1980 \cite{Nedelec1980}
 and the second in
1986 \cite{Nedelec1986}.
For the construction of higher order
\gls{hcurl} conforming elements, we refer to
\cite{Zaglmayr2006}.
In this thesis, Nédélec elements of first kind
of lowest order are used, as formulated
in \cite[Lemma 4.12]{Zaglmayr2006}.
These elements are also called Whitney elements,
as they appeared in \cite{whitney1957}.
For a visualization of these ansatz functions, see
\cref{fig:nedelec_ansatz}.
\subsection{2D Version of Maxwell's Equations}
There are two common variants of deriving two dimensional
equations from the three dimensional Maxwell's equations,
which are shown in the following.
We denote the three space dimensions by $x, y,$ and $z$
and we will eliminate the $z$-direction.
In both variants, one assumes that the derivatives
with respect to the $z$ dimension vanish for all
quantities, i.e.
\begin{equation}
0=
\frac{\partial}{\partial z} 
\gls{electric_field}_x
=
\frac{\partial}{\partial z} 
\gls{electric_field}_y
=
\frac{\partial}{\partial z} 
\gls{electric_field}_z
=
\frac{\partial}{\partial z} 
\gls{electric_flux}_x
=
\frac{\partial}{\partial z} 
\gls{electric_flux}_y
=
\frac{\partial}{\partial z} 
\gls{electric_flux}_z
=
\frac{\partial}{\partial z} 
\gls{magnetic_field}_x
= \dots
\qquad
.
\end{equation}
We present the two approaches in the following.
\subsubsection{Normal Electric Approach}
The first common approach is to assume that
the electric field has only a component perpendicular
to the $xy$-plane and the magnetic field has only
components in the plane. One assumes
\begin{equation}
0=
\gls{electric_field}_x =
\gls{electric_field}_y =
\gls{magnetic_field}_z =
\gls{charge_density} =
\gls{current_density}_x =
\gls{current_density}_y
.
\end{equation}
This assumptions lead to a good approximation
of electric fields between two \gls{pec}
plates separated by a very small gap.
With these assumptions, Maxwell's equations reduce to
\begin{subequations}
\begin{align}
\nabla \cdot \gls{electric_flux} &= 
\frac{\partial}{\partial x} \gls{electric_flux}_x + 
\frac{\partial}{\partial y} \gls{electric_flux}_y + 
\frac{\partial}{\partial z} \gls{electric_flux}_z
&=&~
0
&&
\\
\nabla \cdot \gls{magnetic_flux} &= 
\frac{\partial}{\partial x} \gls{magnetic_flux}_x + 
\frac{\partial}{\partial y} \gls{magnetic_flux}_y + 
\frac{\partial}{\partial z} \gls{magnetic_flux}_z
&=&~
\frac{\partial}{\partial x} \gls{magnetic_flux}_x + 
\frac{\partial}{\partial y} \gls{magnetic_flux}_y 
&=&~
0 \\
\nabla \times \gls{electric_field} &= 
\begin{pmatrix}
\partial_y \gls{electric_field}_z - \partial_z \gls{electric_field}_y \\
\partial_z \gls{electric_field}_x - \partial_x \gls{electric_field}_z \\
\partial_x \gls{electric_field}_y - \partial_y \gls{electric_field}_x 
\end{pmatrix}
&=&
\begin{pmatrix}
\partial_y \gls{electric_field}_z \\
- \partial_x \gls{electric_field}_z \\
0
\end{pmatrix}
&=&
- \frac{\partial}{\partial t} \gls{magnetic_flux}\\
\nabla \times \gls{magnetic_field} &= 
\begin{pmatrix}
\partial_y \gls{magnetic_field}_z - \partial_z \gls{magnetic_field}_y \\
\partial_z \gls{magnetic_field}_x - \partial_x \gls{magnetic_field}_z \\
\partial_x \gls{magnetic_field}_y - \partial_y \gls{magnetic_field}_x 
\end{pmatrix}
&=&
\begin{pmatrix}
0\\
0\\
\partial_x \gls{magnetic_field}_y - \partial_y \gls{magnetic_field}_x 
\end{pmatrix}
&=&~
\frac{\partial}{\partial t} \gls{electric_flux} + \gls{current_density}.
\end{align}
\end{subequations}
or in short
\begin{subequations}
\begin{eqnarray}
\frac{\partial}{\partial x} \gls{magnetic_flux}_x + 
\frac{\partial}{\partial y} \gls{magnetic_flux}_y 
&=& 0 \\
\partial_x \gls{magnetic_field}_y - \partial_y \gls{magnetic_field}_x 
&=& \frac{\partial}{\partial t} \gls{electric_flux}_z + \gls{current_density}_z\\
\begin{pmatrix}
\partial_y \gls{electric_field}_z \\
- \partial_x \gls{electric_field}_z
\end{pmatrix}
&=&
- \frac{\partial}{\partial t}
\begin{pmatrix}
\gls{magnetic_flux}_x \\
\gls{magnetic_flux}_y
\end{pmatrix} .
\end{eqnarray}
\end{subequations}
Applying the same procedure as in \cref{sec:reduction one field}, 
we end up with the Helmholtz-like equation
\begin{equation}
\mathrm{div} \left( \frac{1}{\gls{magnetic_permeability}} \mathrm{grad}\left(\gls{electric_field}_z\right) \right)
+ \gls{angular_frequency}^2 \gls{electric_permittivity} \gls{electric_field}_z = i \gls{angular_frequency} \gls{current_density}_z
\end{equation}
for the time harmonic case.
This equation is scalar.
In the development of numerical methods for Maxwell's equations,
it is often a good first step to analyze this equation.
It requires only little modifications of codes which
were written for the heat equation but already
exposes the challenges of inf-sup stable problems.
In this thesis, this equation is
used in \cref{subsect:helmholtz}.
\subsubsection{Normal Magnetic Approach}
\label{sec:maxwell normal magnetic approach}
The second common approach is to assume that the electric field has only
components in the plane and that the magnetic field is
perpendicular to the plane. One assumes
\begin{equation}
0 =
\gls{electric_field}_z
= \gls{magnetic_field}_x
= \gls{magnetic_field}_y
= \gls{current_density}_z .
\end{equation}
With these assumptions, Maxwell's equations reduce to
\begin{subequations}
\begin{align}
\nabla \cdot \gls{electric_flux} &= 
\frac{\partial}{\partial x} \gls{electric_flux}_x + 
\frac{\partial}{\partial y} \gls{electric_flux}_y + 
\frac{\partial}{\partial z} \gls{electric_flux}_z
&&=
\frac{\partial}{\partial x} \gls{electric_flux}_x + 
\frac{\partial}{\partial y} \gls{electric_flux}_y
&&=
\gls{charge_density}\\
\nabla \cdot \gls{magnetic_flux} &= 
\frac{\partial}{\partial x} \gls{magnetic_flux}_x + 
\frac{\partial}{\partial y} \gls{magnetic_flux}_y + 
\frac{\partial}{\partial z} \gls{magnetic_flux}_z
&&=
0
&&=
0 \\
\nabla \times \gls{electric_field} &= 
\begin{pmatrix}
\partial_y \gls{electric_field}_z - \partial_z \gls{electric_field}_y \\
\partial_z \gls{electric_field}_x - \partial_x \gls{electric_field}_z \\
\partial_x \gls{electric_field}_y - \partial_y \gls{electric_field}_x 
\end{pmatrix}
&&=
\begin{pmatrix}
0 \\
0 \\
\partial_x \gls{electric_field}_y - \partial_y \gls{electric_field}_x 
\end{pmatrix}
&&=
- \frac{\partial}{\partial t} \gls{magnetic_flux}\\
\nabla \times \gls{magnetic_field} &= 
\begin{pmatrix}
\partial_y \gls{magnetic_field}_z - \partial_z \gls{magnetic_field}_y \\
\partial_z \gls{magnetic_field}_x - \partial_x \gls{magnetic_field}_z \\
\partial_x \gls{magnetic_field}_y - \partial_y \gls{magnetic_field}_x 
\end{pmatrix}
&&=
\begin{pmatrix}
\partial_y \gls{magnetic_field}_z \\
- \partial_x \gls{magnetic_field}_z \\
0
\end{pmatrix}
&&=
\frac{\partial}{\partial t} \gls{electric_flux} + \gls{current_density}.
\end{align}
\end{subequations}
or in short
\begin{subequations}
\begin{eqnarray}
\frac{\partial}{\partial x} \gls{electric_flux}_x + 
\frac{\partial}{\partial y} \gls{electric_flux}_y 
&=& \gls{charge_density} \\
\partial_x \gls{electric_field}_y - \partial_y \gls{electric_field}_x 
&=&
- \frac{\partial}{\partial t}
\gls{magnetic_flux}_z\\
\begin{pmatrix}
\partial_y \gls{magnetic_field}_z \\
- \partial_x \gls{magnetic_field}_z
\end{pmatrix}
&=& \frac{\partial}{\partial t} 
\begin{pmatrix}
\gls{electric_flux}_x \\
\gls{electric_flux}_y
\end{pmatrix} + 
\begin{pmatrix}
\gls{current_density}_x\\
\gls{current_density}_y
\end{pmatrix}
.
\end{eqnarray}
\end{subequations}
Applying the same procedure as in \cref{sec:reduction one field}, 
one obtains
\begin{equation}
\begin{pmatrix}
\partial_y \frac{1}{\gls{magnetic_permeability}} \partial_y \gls{electric_field}_x 
- \partial_y \frac{1}{\gls{magnetic_permeability}} \partial_x \gls{electric_field}_y
\\
\partial_x \frac{1}{\gls{magnetic_permeability}} \partial_x \gls{electric_field}_y 
- \partial_x \frac{1}{\gls{magnetic_permeability}} \partial_y \gls{electric_field}_x
\end{pmatrix}
+ \gls{angular_frequency}^2 \gls{electric_permittivity} 
\begin{pmatrix}
\gls{electric_field}_x \\
\gls{electric_field}_y
\end{pmatrix}
=
i \gls{angular_frequency}
\begin{pmatrix}
\gls{current_density}_x \\
\gls{current_density}_y
\end{pmatrix}
.
\end{equation}
In the development of numerical methods for Maxwell's equations,
it is often a good second step to analyze this equation.
It is vector-valued, thus requiring vector-valued
ansatz functions like Nédélec ansatz functions and
the usual discretization has a low frequency
instability.
But the problems arising from discretizing this equation
are still two dimensional and thus the number of
unknowns stays rather small in comparison with a 
three dimensional problem.
In this thesis, this equation is used in the example presented in \cref{sec:example 2d maxwell}
and analyzed in \cref{sec:arbilomod experiments for maxwell}.

\section{Examples Used for Numerical Experiments}
In this section, we define the problems which
are used as examples in the numerical experiments.
We use four different examples:
The \exa{}, the \exb{}, the \exc{},
and the \exd{}.
Their numerical treatment is increasingly challenging.
\exb{} adds high contrast and thereby strong nonlocal
effects, \exc{} is additionally vector-valued
and inf-sup stable and \exd{}
is the most challenging,
having channels, inf-sup stability, vector-valued unknowns
and a high \gls{dof} count.
In \cref{sec:rangefinder artificial test}
we will introduce two additional examples:
The \exrangea{} and \exrangeb{}.
They are used only in that particular chapter and thus stated there.
\subsection{Thermal Block Example}
The first numerical example we denote by
\exa{}.
It is a simple parametric example which 
is often used in introductory texts on parameterized
model order reduction (for example in \cite{Patera2007} or \cite{Haasdonk2017a}).

We solve the problem in \cref{def:variational problem}
with the bilinear form and the linear form defined as in
\cref{eq:weak heat equation}
and the space \gls{fspace} defined in \cref{eq:fspace heat}.
\begin{figure}
\begin{center}
\ifx\JPicScale\undefined\def\JPicScale{1}\fi
\unitlength \JPicScale mm
\begin{tikzpicture}[x=\unitlength,y=\unitlength,inner sep=0pt]
\definecolor{userFillColour}{rgb}{1,0.8,1}
\draw [fill=userFillColour](40,40) rectangle (60,60);
\definecolor{userFillColour}{rgb}{1,0.8,0.8}
\draw [fill=userFillColour](20,40) rectangle (40,60);
\definecolor{userFillColour}{rgb}{0.8,1,0.8}
\draw [fill=userFillColour](40,20) rectangle (60,40);
\definecolor{userFillColour}{rgb}{1,1,0.8}
\draw [fill=userFillColour](20,20) rectangle (40,40);
\draw (20,15) node {0};
\draw (60,15) node {1};
\draw (15,20) node {0};
\draw (15,60) node {1};
\draw (30,30) node {$\gls{parameter}_{00}$};
\draw (50,30) node {$\gls{parameter}_{10}$};
\draw (30,50) node {$\gls{parameter}_{01}$};
\draw (50,50) node {$\gls{parameter}_{11}$};
\end{tikzpicture}
\end{center}
\caption[Geometry of \exa{}.]
{Geometry of \exa{}: Four blocks of constant, but parameterized, heat conductivity.}
\label{fig:thermal block geometry}
\end{figure}
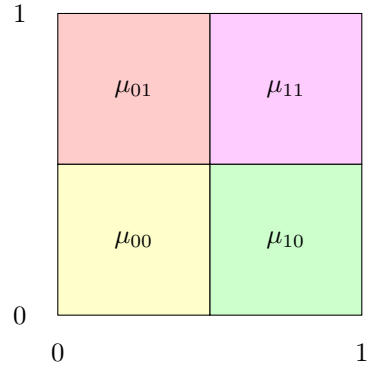
On the domain $\gls{domain} := (0, 1)^2$ we consider the stationary
heat conduction (cf. \cref{eq:stationary heat equation})
\begin{equation}
- \mathrm{div} \Big( \gls{heat_conductivity}(x) \nabla \gls{sol}(x) \Big) = 1
\label{eq:thermalblock}
\end{equation}
with Dirichlet zero boundary conditions on $\gls{boundary}_D := \partial \gls{domain}$. The heat conductivity is defined as
\begin{equation}
\gls{heat_conductivity}_{\gls{parameter}} := \sum_{i,j=0}^1 {\gls{parameter}}_{ij}\cdot\chi_{[i/2, (i+1)/2]\times [j/2, (j+1)/2]}
\end{equation}
(see \cref{fig:thermal block geometry}), 
denoting by $\chi_A$ the characteristic function of the set $A$.
This heat conductivity has an affine parameter decomposition
as specified in \cref{eq:affineafheat} and thus
the corresponding bilinear form has an affine parameter decomposition.
Denoting the local blocks by 
\begin{equation}
\begin{split}
\gls{domain}_{00} := (0,0.5) \times (0, 0.5)
\qquad
\gls{domain}_{10} := (0.5,1) \times (0, 0.5)
\\
\qquad
\gls{domain}_{01} := (0,0.5) \times (0.5, 1)
\qquad
\gls{domain}_{11} := (0.5,1) \times (0.5, 1)
\end{split}
\end{equation}
the affine decomposition of the bilinear form is
\begin{equation}
\begin{split}
\gls{a}_{\gls{parameter}}(u,v) =
\gls{parameter}_{00} \int_{\gls{domain}_{00}} \nabla u \cdot \nabla v \dx
+
\gls{parameter}_{10} \int_{\gls{domain}_{10}} \nabla u \cdot \nabla v \dx\\
\hfill
+
\gls{parameter}_{01} \int_{\gls{domain}_{01}} \nabla u \cdot \nabla v \dx
+
\gls{parameter}_{11} \int_{\gls{domain}_{11}} \nabla u \cdot \nabla v \dx
.
\end{split}
\end{equation}
The parameters
${\gls{parameter}} :=\{{\gls{parameter}}_{ij}\}_{i,j=0}^1$ were allowed to vary in the space $\gls{parameterspace} = [0.1, 1.0]^4$.

The problem is discretized using linear finite elements on a regular mesh with
$500\times 500 \times 4$ triangular entities,
an example solution is shown in \cref{fig:thermal block solution}. 
Discretization is done using \gls{pymor}.
The \exa{} is used
in \cref{sec:thermal block numericalexample}.
\begin{figure}
\begin{center}
\includegraphics[scale=0.3, trim=00 230 0 200, clip]{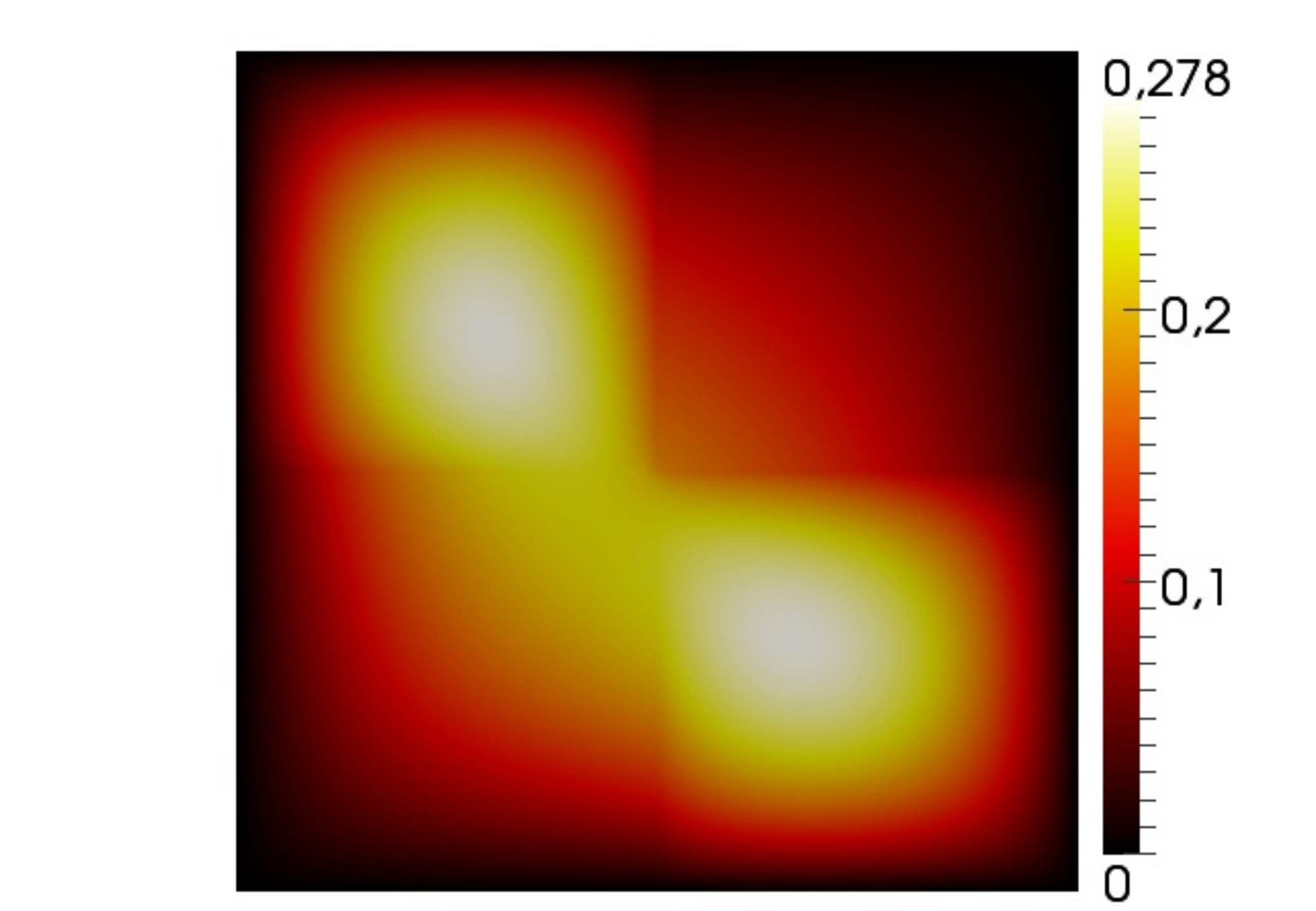}
\end{center}
\caption[High-dimensional solution of \exa{}.]
        {High-dimensional solution of \cref{eq:thermalblock} in \exa{} for ${\gls{parameter}} = (0.1, 1.0, 0.4, 0.1)$.
(reproduction: \cref{repro:fig:thermal block solution})
}
\label{fig:thermal block solution}
\end{figure}

\subsection{Example for Stationary Heat Equation With Channels}
\label{sec:exb}
The second numerical example is an example of heat conduction,
but with the additional features of having a very high
contrast (up to $10^5$) and conducting channels. We denote the
second numerical example by \exb.
We solve the problem in \cref{def:variational problem}
with the bilinear form and the linear form defined as in
\cref{eq:weak heat equation}
and the space \gls{fspace} defined in \cref{eq:fspace heat}.
\begin{figure}
\centering
\includegraphics[width=0.3\textwidth]{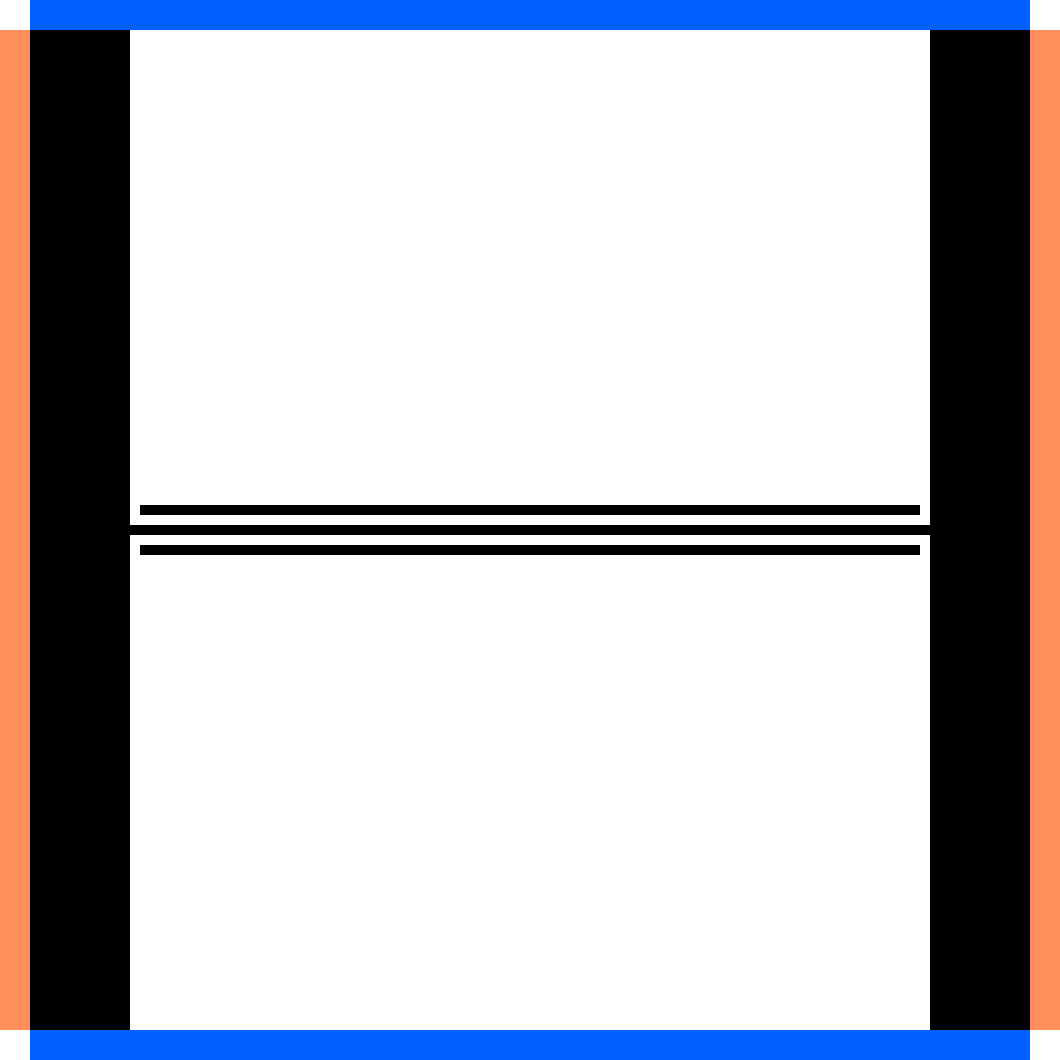}
\caption[First structure in sequence of simulated structures in \exb.]
{First structure in sequence of simulated structures in \exb. Unit square with
high and low conductivity regions. White: constant conductivity region
($\gls{heat_conductivity} = 1$), black: parameterized high conductivity region
($\gls{heat_conductivity}_{\gls{parameter}} = 1 + {\gls{parameter}}$).
Homogeneous Neumann boundaries $\gls{boundary}_N$ at top and bottom marked blue,
inhomogeneous Dirichlet boundaries $\gls{boundary}_D$ at left and
right marked light red.}
\label{fig:geo1}
\end{figure}

\begin{figure}
\footnotesize
\centering
\begin{tabular}{c|c|c|c|cc}
\# & left (zoom) & full structure & right (zoom) &
\multicolumn{2}{l}{solution for $\gls{heat_conductivity}_{\gls{parameter}} = 10^5$}
\\
\hline
1 &
\encloseimage{\includegraphics[width=0.17\textwidth,frame,trim=50 450 850 450,clip]{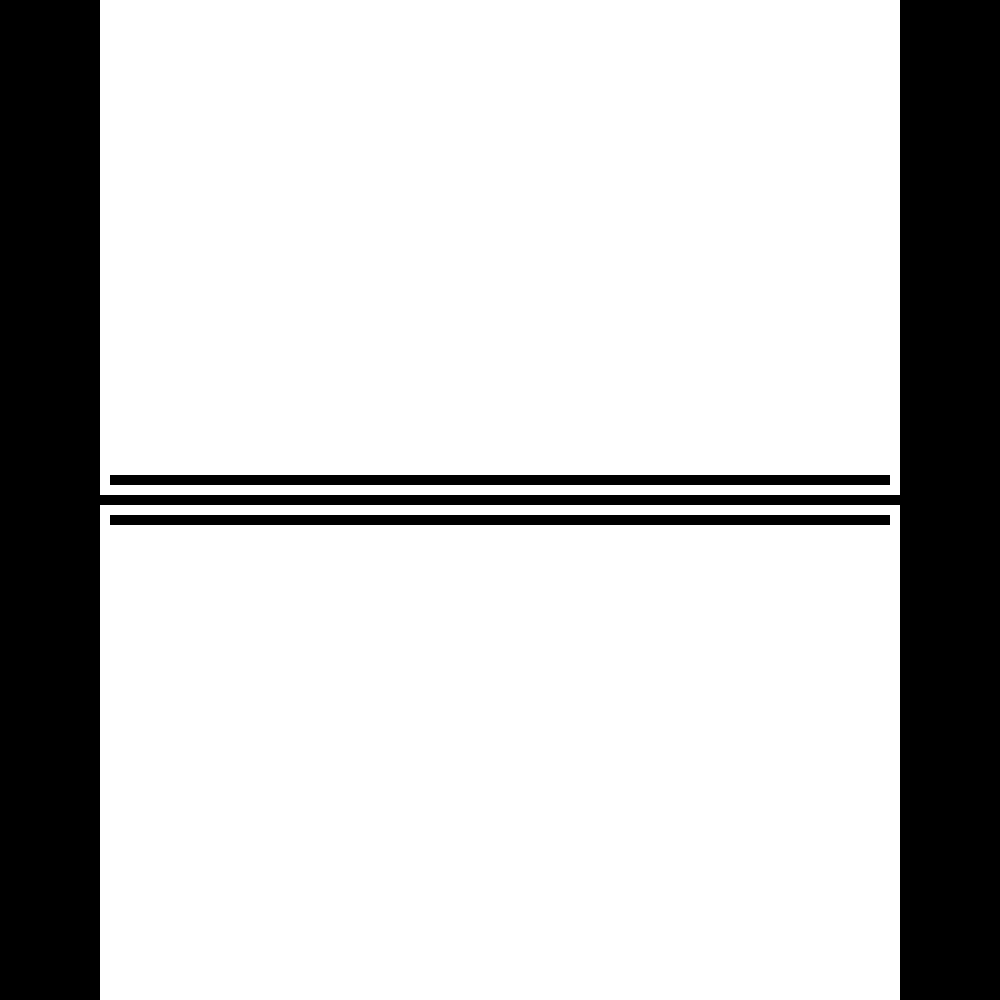}} &
\encloseimage{\includegraphics[width=0.17\textwidth,frame]{h_seq1}} &
\encloseimage{\includegraphics[width=0.17\textwidth,frame,trim=850 450 50 450,clip]{h_seq1}} &
\encloseimage{\includegraphics[width=0.17\textwidth]{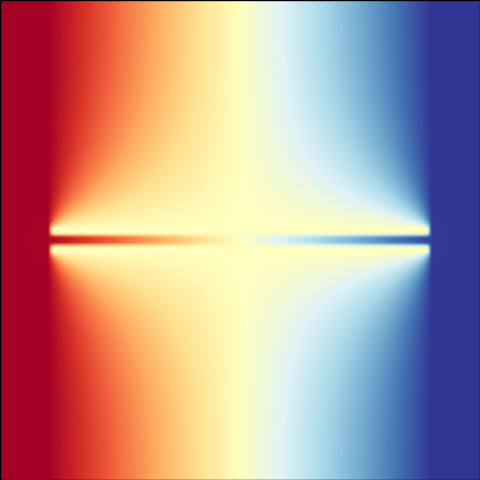}} &
\hspace{-25pt}
\encloseimage{
\includegraphics[width=0.05\textwidth]{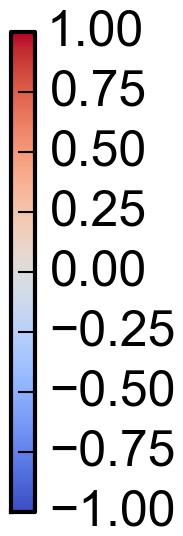}
}
\\
\hline
2 &
\encloseimage{\includegraphics[width=0.17\textwidth,frame,trim=50 450 850 450,clip]{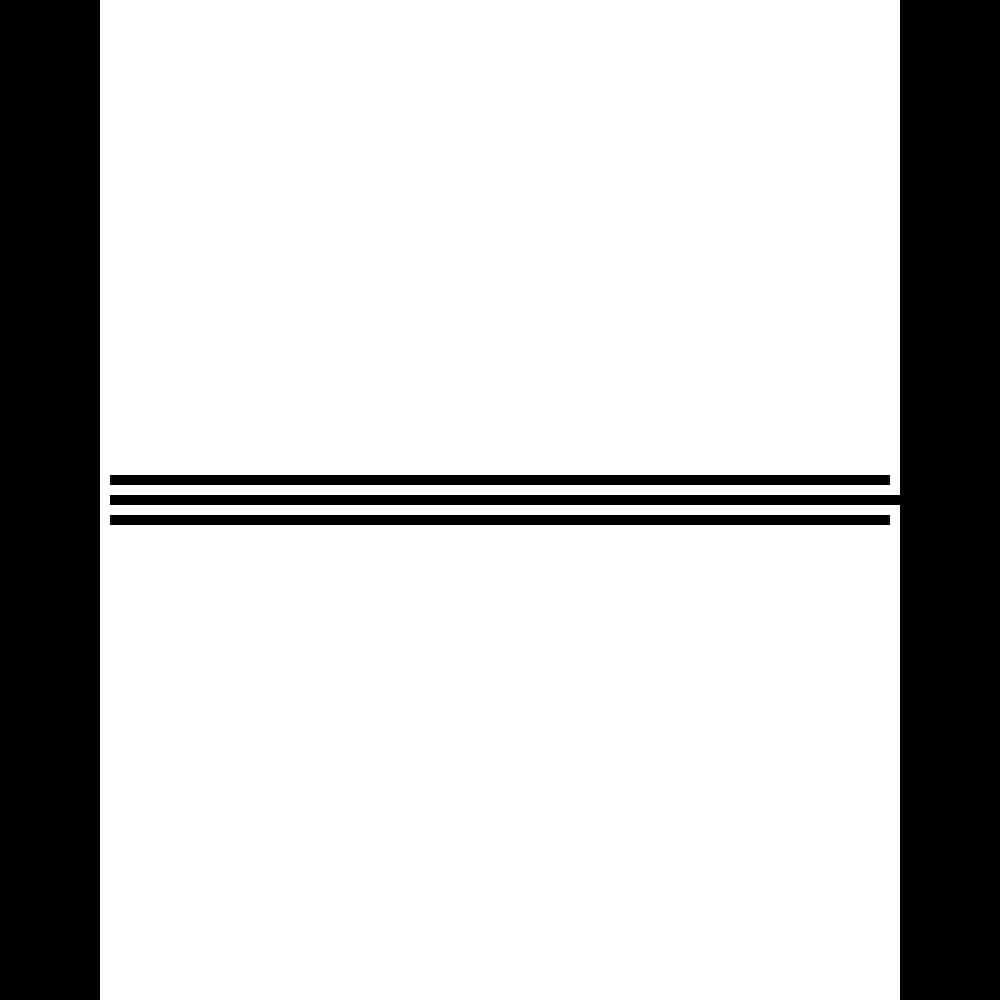}} &
\encloseimage{\includegraphics[width=0.17\textwidth,frame]{h_seq2}} &
\encloseimage{\includegraphics[width=0.17\textwidth,frame,trim=850 450 50 450,clip]{h_seq2}} &
\encloseimage{\includegraphics[width=0.17\textwidth]{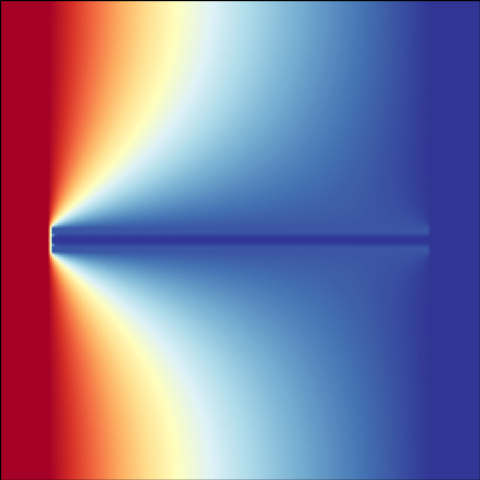}}
\\
\hline
3 &
\encloseimage{\includegraphics[width=0.17\textwidth,frame,trim=50 450 850 450,clip]{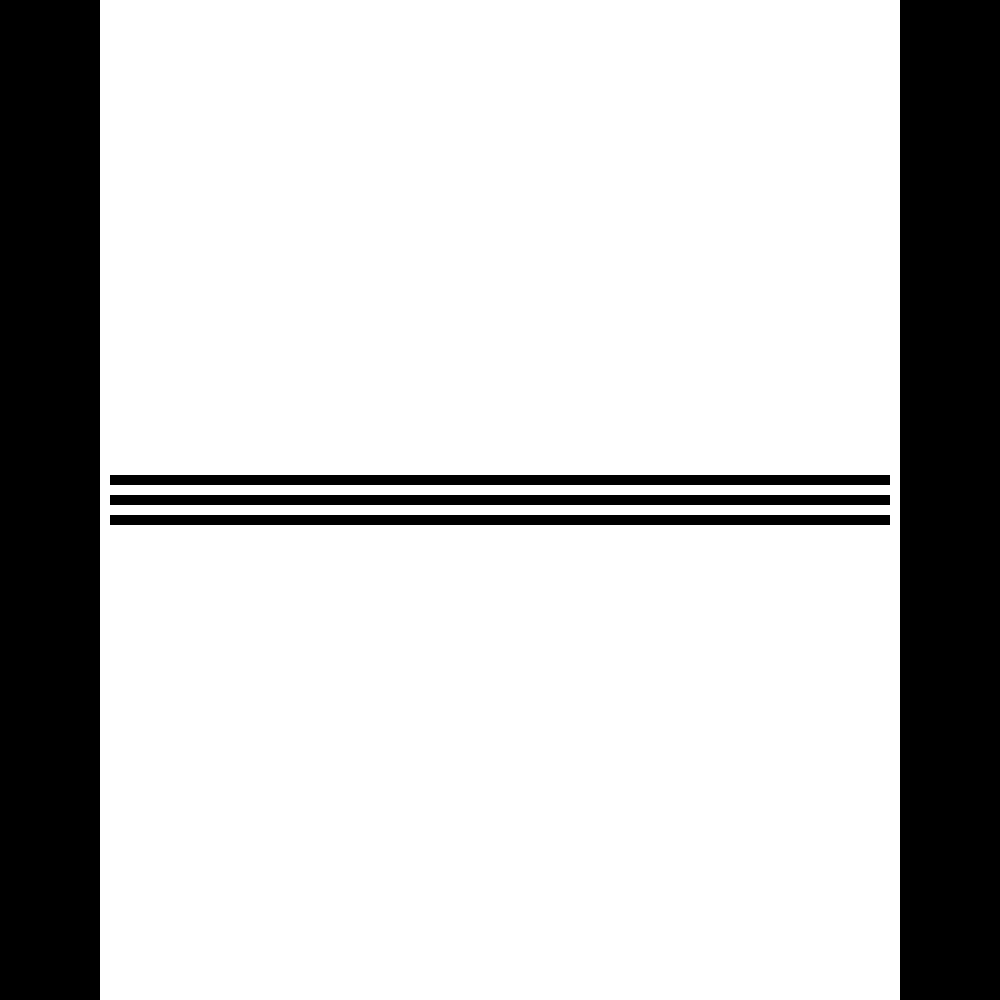}} &
\encloseimage{\includegraphics[width=0.17\textwidth,frame]{h_seq3}} &
\encloseimage{\includegraphics[width=0.17\textwidth,frame,trim=850 450 50 450,clip]{h_seq3}} &
\encloseimage{\includegraphics[width=0.17\textwidth]{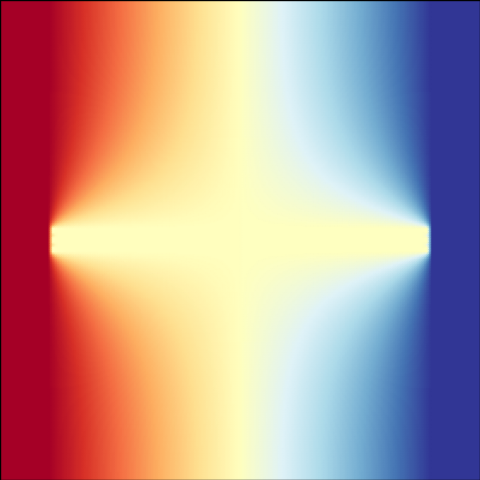}}
\\
\hline
4 &
\encloseimage{\includegraphics[width=0.17\textwidth,frame,trim=50 450 850 450,clip]{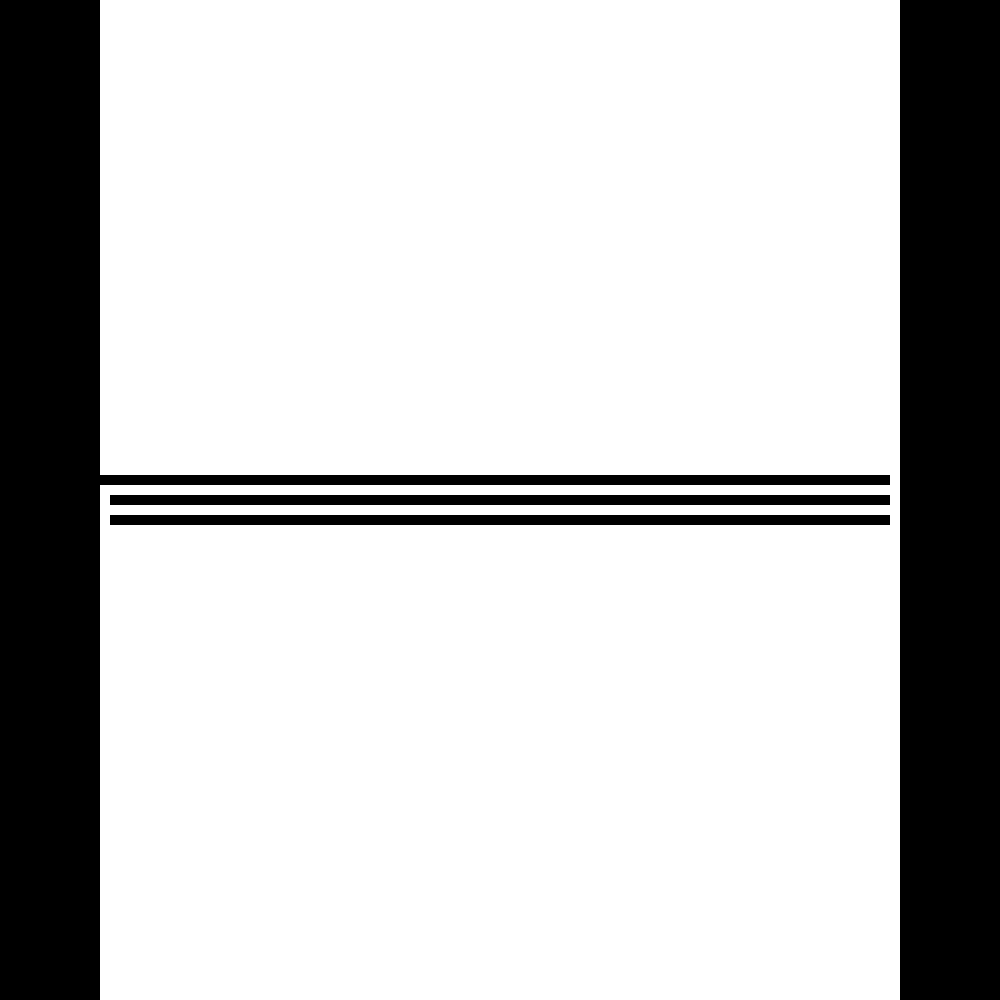}} &
\encloseimage{\includegraphics[width=0.17\textwidth,frame]{h_seq4}} &
\encloseimage{\includegraphics[width=0.17\textwidth,frame,trim=850 450 50 450,clip]{h_seq4}} &
\encloseimage{\includegraphics[width=0.17\textwidth]{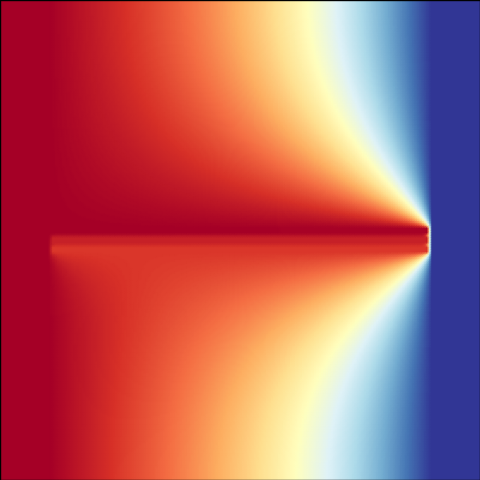}}
\\
\hline
5 &
\encloseimage{\includegraphics[width=0.17\textwidth,frame,trim=50 450 850 450,clip]{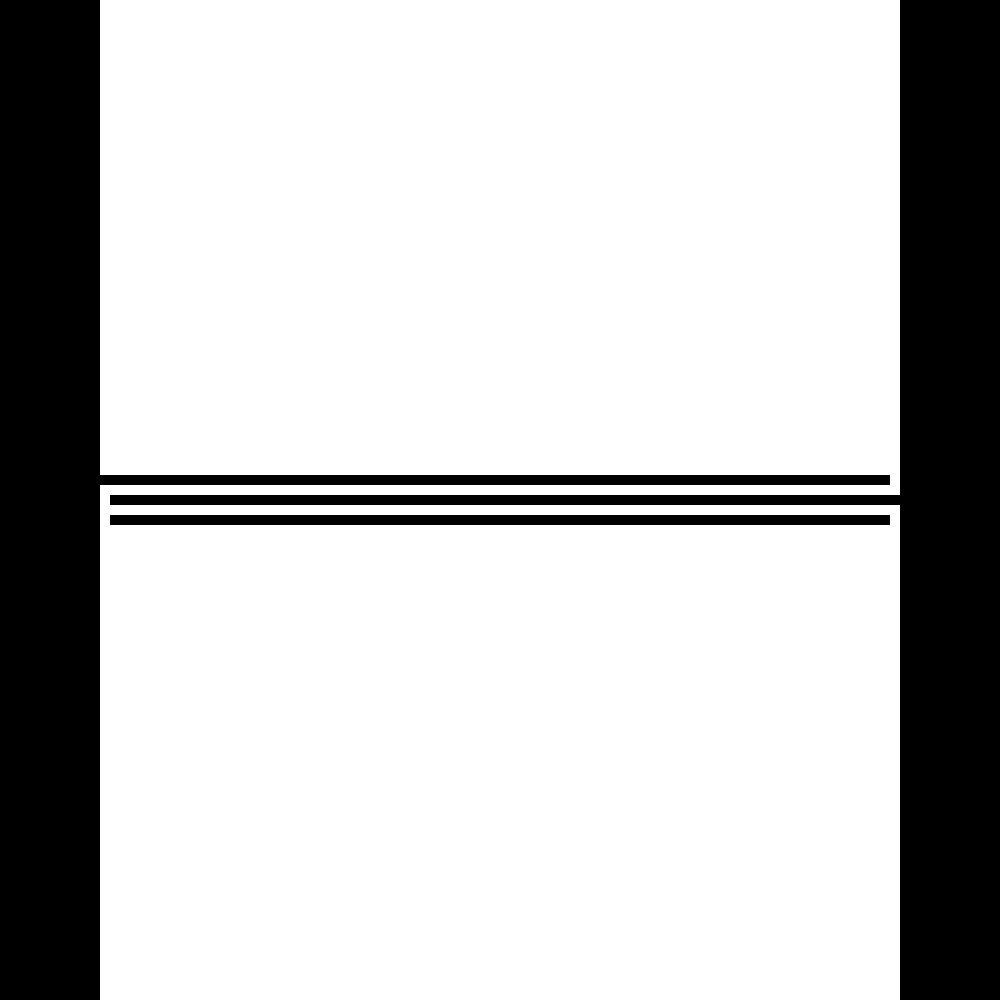}} &
\encloseimage{\includegraphics[width=0.17\textwidth,frame]{h_seq5}} &
\encloseimage{\includegraphics[width=0.17\textwidth,frame,trim=850 450 50 450,clip]{h_seq5}} &
\encloseimage{\includegraphics[width=0.17\textwidth]{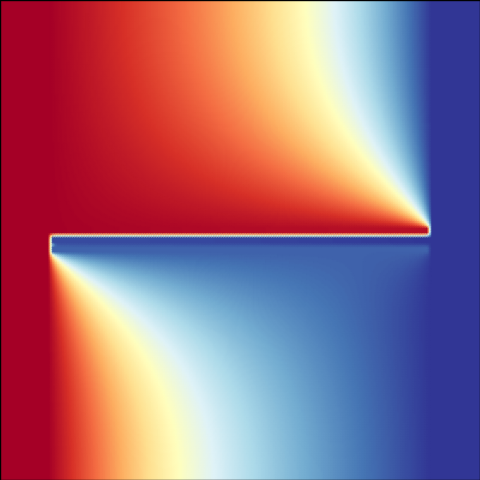}}
\end{tabular}
\caption[Sequence of structures simulated in \exb{} along with solutions.]
{Sequence of structures simulated in \exb{} along with solutions for
one parameter value. Localized changes in the geometry cause global changes
in the solution.
(reproduction: \cref{repro:fig:structures})
}
\label{fig:structures}
\end{figure}

As we target on changing geometries with local changes,
we define a sequence of five geometries.
We consider
heat conduction without heat sources in the domain
on the unit square $\gls{domain} := (0,1)^2$.
We approximate $\gls{sol}_{\gls{parameter}}$ solving
\begin{equation}
-\nabla \cdot ( \sigma_{\gls{parameter}} \nabla \gls{sol}_{\gls{parameter}}) = 0
.
\end{equation}
We apply homogeneous Neumann boundaries at the top and the bottom
\begin{equation}
\gls{boundary}_N
	:= ( (0,1) \ \times \ 0 ) \cup ( (0,1) \ \times \ 1 )
\end{equation}
and inhomogeneous Dirichlet boundaries at the left and right:
\begin{equation}
\gls{sol}_{\gls{parameter}} = 1 \ \mathrm{on} \ \gls{boundary}_{D,1} := 0 \ \times \ (0,1)
\qquad
\gls{sol}_{\gls{parameter}} = -1 \ \mathrm{on} \ \gls{boundary}_{D,-1} := 1 \ \times \ (0,1)
\end{equation}
(see also \cref{fig:geo1}).

The unit square is partitioned into two regions:
one region with constant heat conductivity $\gls{heat_conductivity}_{\gls{parameter}} = 1$ and
one region with constant, but parameterized
conductivity
\begin{equation}\gls{heat_conductivity}_{\gls{parameter}} = 1 + {\gls{parameter}}
\text{ where }
\gls{parameter} \in \gls{parameterspace} := [10^0, 10^5].
\end{equation}
We call this second region the ``high conductivity region'' $\Omega_{h,i}$.
For reproduction or benchmarking, we give
precise definitions in the following.
In the first geometry, the high conductivity region is
\begin{eqnarray}
\Omega_{h,1} &:=&
\big[(0.0 ,0.1) \times (0, 1) \big]
\ \cup \
\big[(0.9, 1.0) \times (0, 1) \big]
\ \cup \ \\ &&
\nonumber
\big[(0.11, 0.89) \times (0.475, 0.485) \big]
\ \cup \
\big[(0.1, 0.9) \times (0.495, 0.505) \big]
\ \cup \ \\ &&
\nonumber
\big[(0.11, 0.89) \times (0.515, 0.525) \big],
\end{eqnarray}
see also \cref{fig:geo1}.
The inhomogeneous Dirichlet boundary conditions are handled by a shift function
$u_s$ and we solve for $u_0$ having homogeneous Dirichlet values where
$u = u_0 + u_s$.
The parameterization of the conductivity in the high conductivity region,
$\gls{heat_conductivity}_{\gls{parameter}}$, leads to a term in the affine decomposition of the bilinear form.
The affine decomposition of the bilinear form and linear form are
\begin{eqnarray}
a_{\gls{parameter}}(u_0,\varphi) &=& {\gls{parameter}} \int_{\Omega_{h,i}}  \nabla u_0 \cdot \nabla \varphi \dx + \int_{\Omega}  \nabla u_0 \cdot \nabla \varphi \dx \\
\nonumber
\gls{f}_{\gls{parameter}}(\varphi) &=& - {\gls{parameter}} \int_{\Omega_{h,i}} \nabla u_s \cdot \nabla \varphi \dx - \int_{\Omega} \nabla u_s \cdot \nabla \varphi \dx .
\end{eqnarray}
The coercivity constant $\gls{coercconst}_{\gls{parameter}}$ of the corresponding bilinear form with
respect to the $H^1$ norm is bounded from
below by $\gls{coercconst}_{LB} := \frac{\gls{heat_conductivity}_\mathrm{min}}{\gls{friedrichs}^2+1}$ where $\gls{heat_conductivity}_\mathrm{min}$
is the minimal conductivity and $\gls{friedrichs} = \frac{1}{\sqrt{2}\pi}$.
The problem is discretized using $P^1$ ansatz functions on a structured triangle mesh with maximum triangle size $\gls{h}$. The mesh is carefully constructed to resolve the
high conductivity regions, i.e.~$\gls{h}$ is chosen to be $1/n$ where $n$ is a multiple of 200.
To mimic ``arbitrary local modifications'', the high conductivity region is changed slightly four times, which leads to a sequence of five structures to be simulated in total. The high conductivity regions are defined as:
\begin{eqnarray}
\Omega_{h,2} &:=& \Omega_{h,1} \setminus \big[(0.1, 0.11) \times (0.495, 0.505) \big]\\
\Omega_{h,3} &:=& \Omega_{h,2} \setminus \big[(0.89, 0.9) \times (0.495, 0.505) \big] \nonumber\\
\Omega_{h,4} &:=& \Omega_{h,3} \cup \big[(0.1, 0.11) \times (0.515, 0.525) \big] \nonumber\\
\Omega_{h,5} &:=& \Omega_{h,4} \cup \big[(0.89, 0.9) \times (0.495, 0.505) \big]
\nonumber
\end{eqnarray}
These modifications only affect a very small portion of the domain
(actually, only 0.01\%), but for high contrast configurations, they lead to global
changes in the solution, see \cref{fig:structures}.
The \exb{} is used in \cref{sec:numerical_example}.
A slightly modified version of this example,
which is driven by a right hand side instead of inhomogeneous 
Dirichlet boundary conditions is defined and used in \cref{sec:enrichment numerical example}.

\subsection{Example for 2D Maxwell Equation}
\label{sec:example 2d maxwell}
\begin{figure}
\begin{center}
\includegraphics[width=0.3\textwidth]{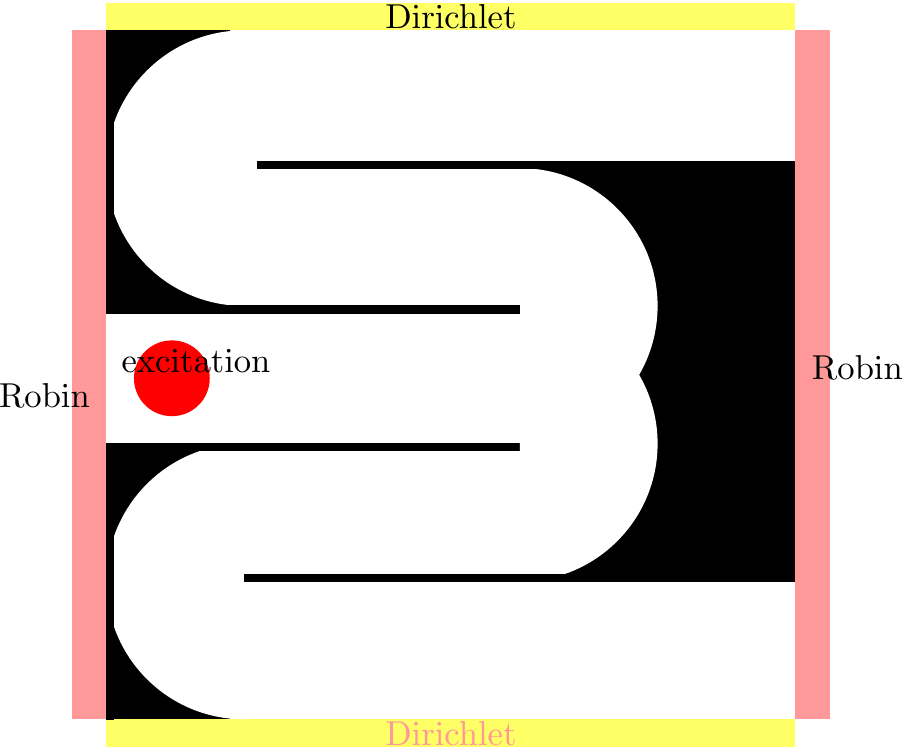}
\hspace{20pt}
\includegraphics[width=0.3\textwidth]{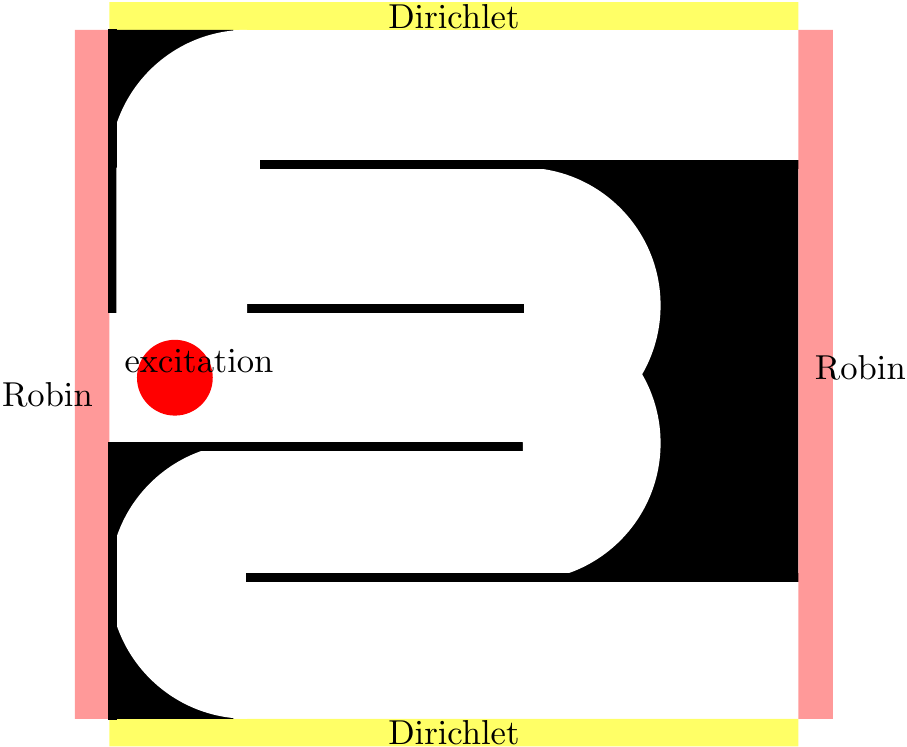}
\end{center}
\caption[Geometries simulated in \exc.]
{Geometries simulated in \exc.
Black area is not part of the domain and treated as Dirichlet zero boundary. Note that the change is topology changing.}
\label{fig:example 2d maxwell}
\end{figure}
The third example is a two dimensional Maxwell example.
We solve the problem in \cref{def:variational problem}
with the bilinear form and the linear form defined as in
\cref{eq:maxwell bilinear linear form} 
and the space \gls{fspace} as defined in \cref{eq:maxwellspace}.
With the two dimensional Maxwell's equations in the normal magnetic approach 
as specified in \cref{sec:maxwell normal magnetic approach}.
We simulate on the unit square $\gls{domain} := (0,1)\times (0,1)$ with robin boundary conditions
at $\gls{boundary}_R := 0 \times (0,1) \cup 1 \times (0,1)$ and Dirichlet zero boundary conditions
at $\gls{boundary}_D :=(0,1) \times 0 \cup (0,1) \times 1$. The surface impedance
parameter $\gls{electric_conductivity}$ is chosen as the impedance of free space,
$\gls{electric_conductivity} = 1 / 376.73$ Ohm. We introduce some structure by inserting \gls{pec} into the domain
and thus shrinking \gls{domain} and enlarging $\gls{boundary}_D$,
see \cref{fig:example 2d maxwell}. The \gls{pec} is modeled as Dirichlet zero boundary condition. 
Note that the geometry is slightly asymmetric intentionally, to produce more interesting behavior.
The mesh does not resolve the geometry. Rather, we use a structured mesh and remove all degrees of freedom which
are associated with an edge whose center is inside of the \gls{pec} structure. The structured mesh consists of
100 $\times$ 100 squares, each of which is divided into four triangles. With each edge, one degree of freedom is
associated, which results in 60200 degrees of freedom, some of which are ``disabled'' because they are
in \gls{pec} or on a Dirichlet boundary. The parameter domain is the range from 10 MHz to 1 GHz.
For the discretization, the software package \gls{pymor} is used.
To simulate an ``arbitrary local modification'', the part of the \gls{pec} within $(0.01, 0.2) \times (0.58,0.80)$ is removed and the simulation domain is enlarged.
The excitation is a current of
\begin{equation}
\gls{current_density}(x,y) := \mathrm{exp}\left( - \frac{(x - 0.1)^2 + (y - 0.5)^2}{1.25 \cdot 10^{-3}}\right) \cdot \mathrm{e}_y
.
\end{equation}
\begin{figure}
\begin{center}
\begin{tabular}{ccc}
\includegraphics[width=0.25\textwidth]{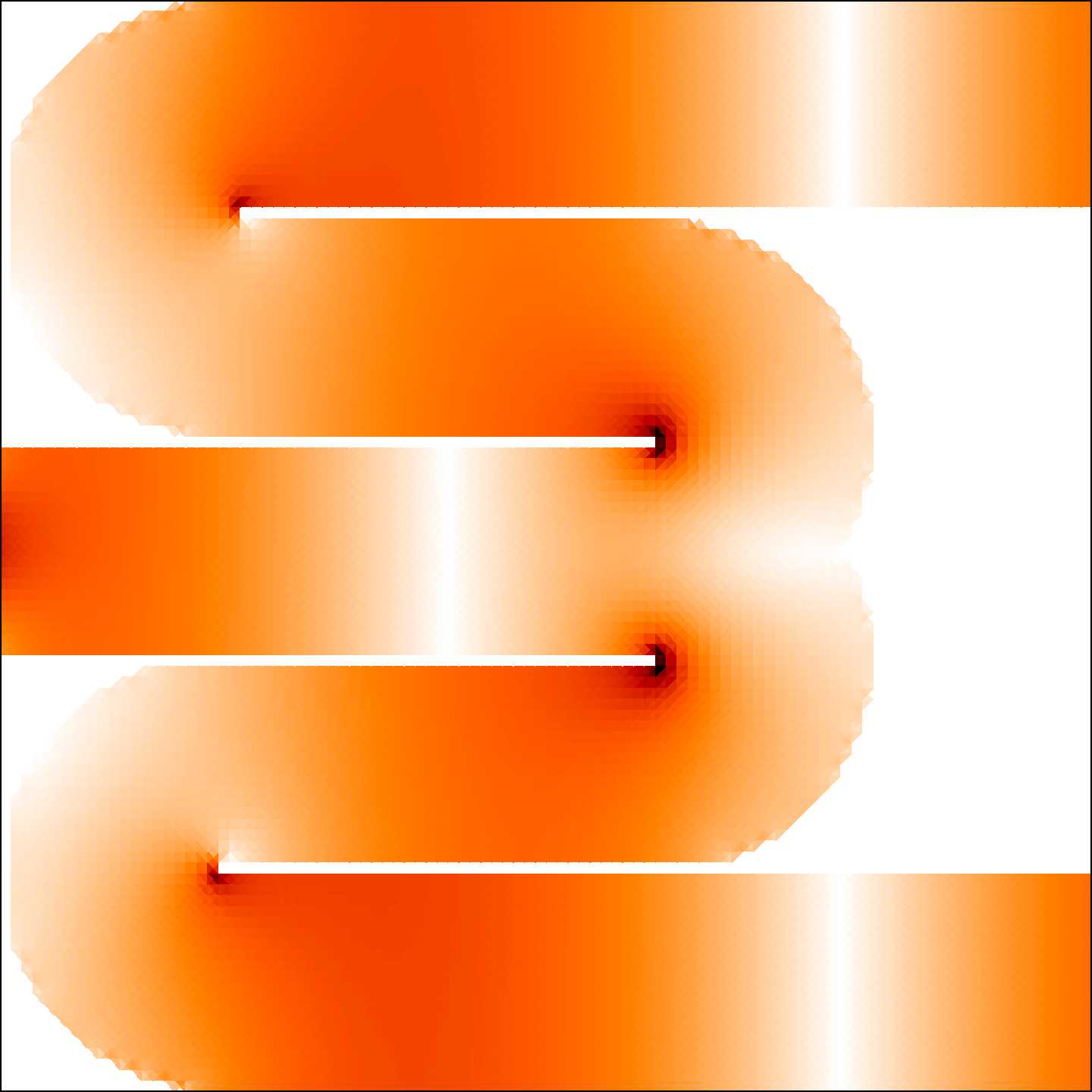}&
\includegraphics[width=0.25\textwidth]{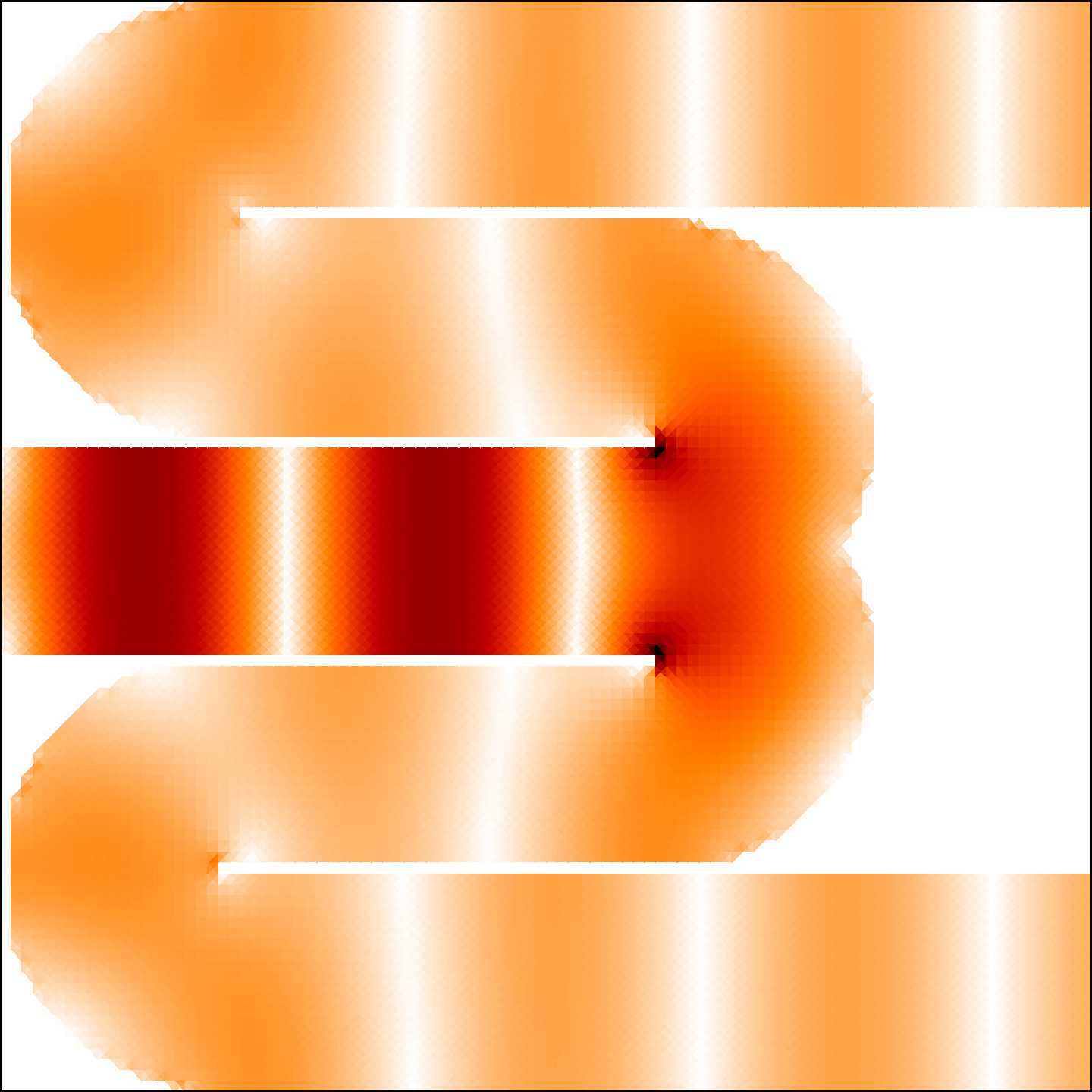}&
\includegraphics[width=0.25\textwidth]{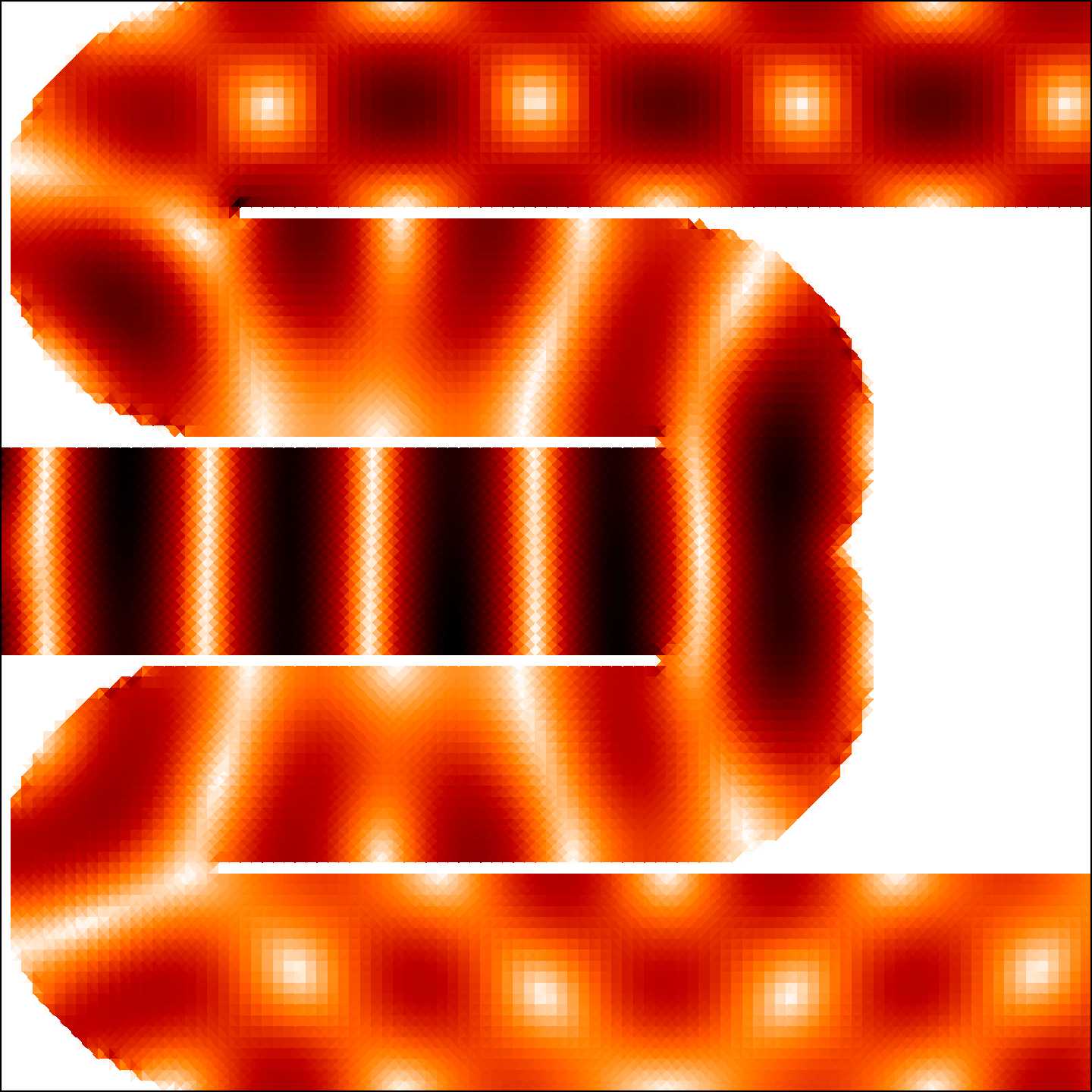}\\
\includegraphics[width=0.25\textwidth]{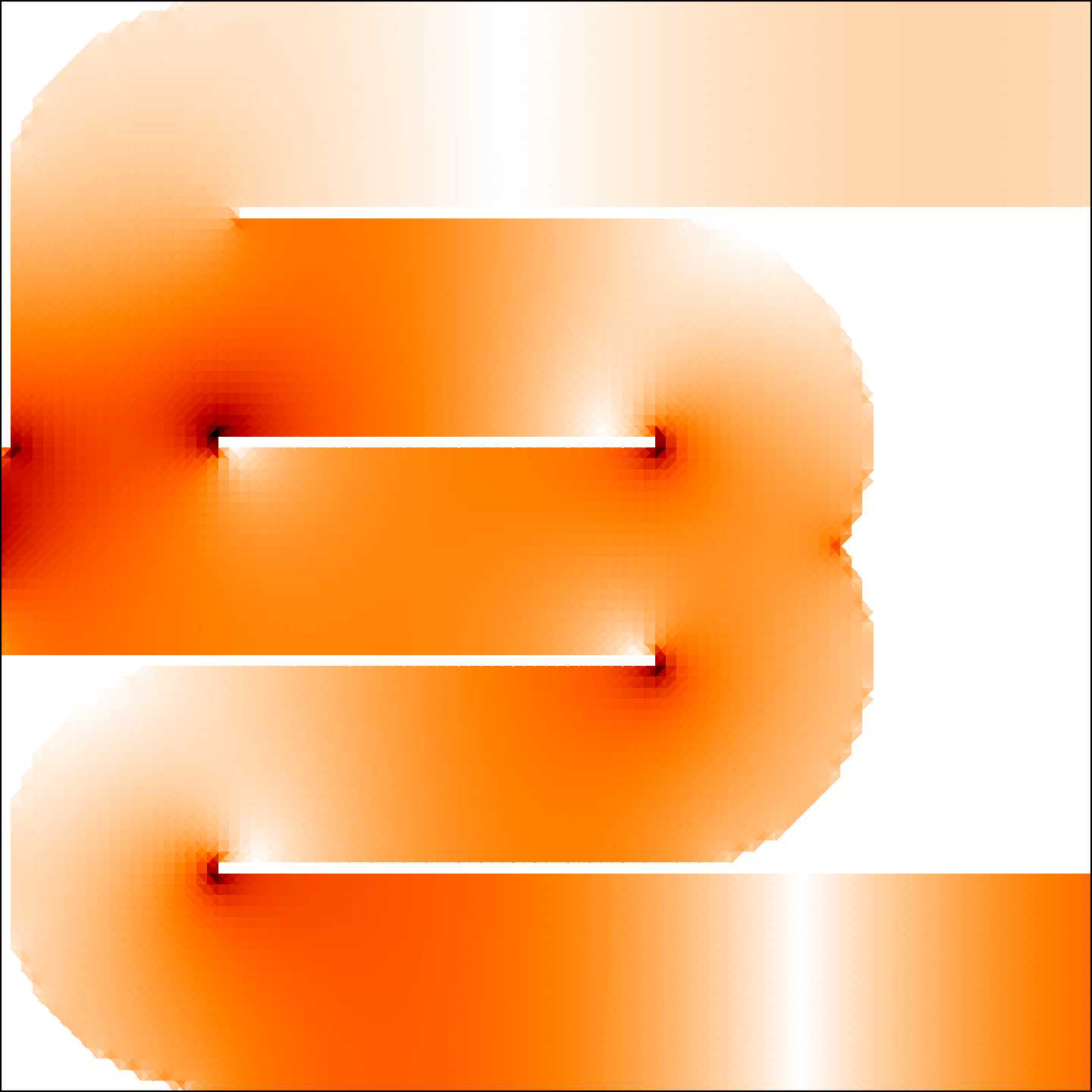}&
\includegraphics[width=0.25\textwidth]{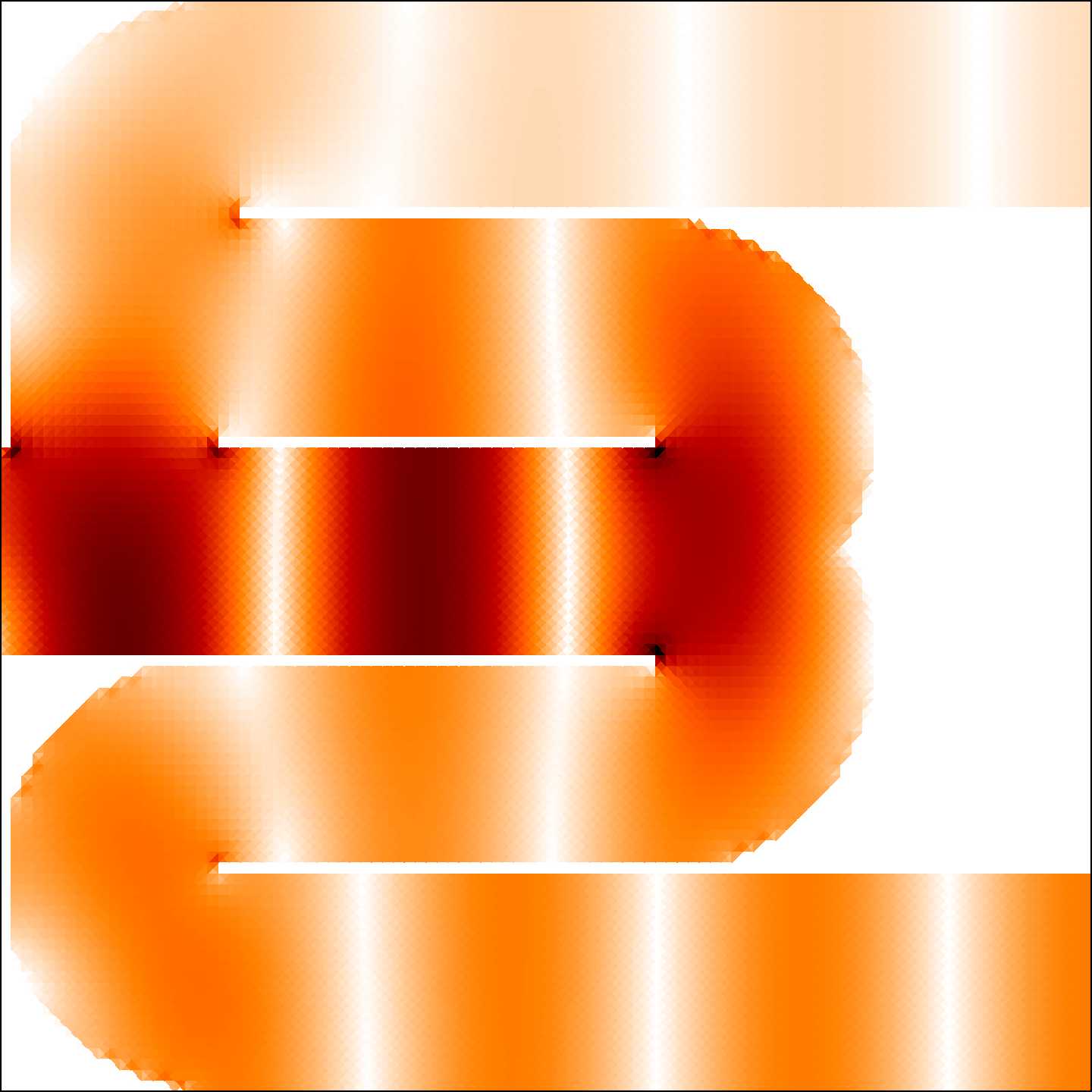}&
\includegraphics[width=0.25\textwidth]{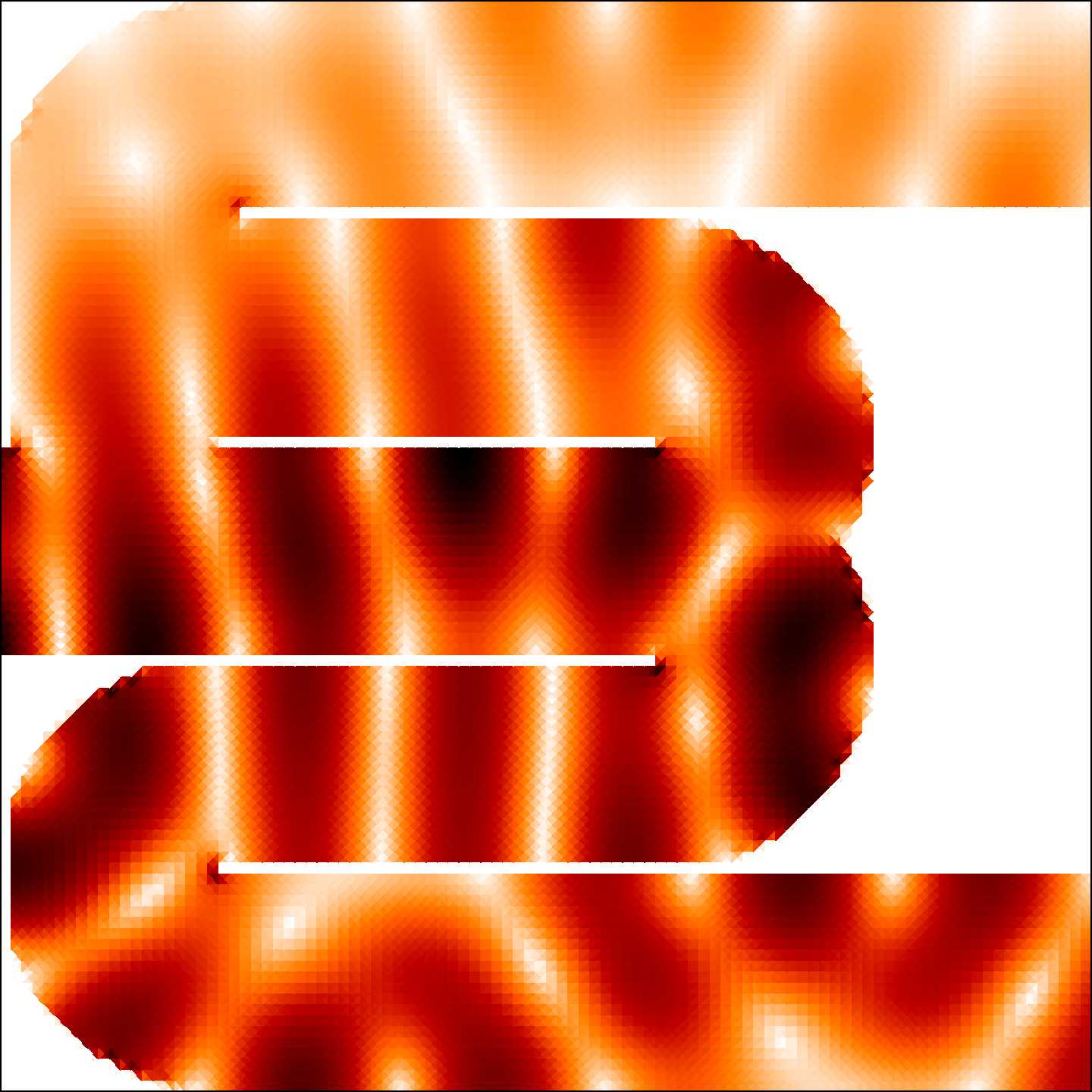}
\end{tabular}
\end{center}
\caption[Example solutions for \exc{}.]
{Example solutions for \exc{} for $\gls{angular_frequency}= 2 \pi \cdot 186$\,MHz, $\gls{angular_frequency} = 2 \pi \cdot 561$\,MHz and $\gls{angular_frequency}= 2 \pi \cdot 1$\,GHz for the first and second geometry. Plotted is $|\gls{realpart}(\gls{electric_field})|$.
(reproduction: \cref{repro:fig:maxwell solutions})
}
\label{fig:maxwell solutions}
\end{figure}
To get an impression of the solutions, some example solutions are plotted in \cref{fig:maxwell solutions}.
The \exc{} is used in \cref{sec:arbilomod experiments for maxwell}.

\subsection{Olimex A64 Example}
\label{sec:olimex description}
The fourth example is the \exd{}.
Here we solve the problem in \cref{def:variational problem}
with the bilinear form and the linear form defined as in
\cref{eq:maxwell bilinear linear form}
and the space $V$ defined in \cref{eq:maxwellspace}.
The \exd{} is an example from the targeted application domain:
Signal integrity simulation in a high frequency \gls{pcb}.
It features high geometrical complexity:
The geometry is described by six layers of metal, 
separated by insulating layers (cf.~\cref{fig:stackup}).
In addition, we simulate a layer of air above and below
the board.
In each metal layer, the shape of the metal
is described by polygons.
The polygonal description of all metal layers
have combined 280.378 vertices.
Photos of the board are shown in
\cref{fig:olimex picture 1,fig:olimex picture 2}.
In the experiments, we will be especially
concerned with the metal traces connecting CPU
and RAM, see right picture in \cref{fig:olimex picture 2}.
Meandering traces are used to get correct signal runtimes.
The geometry is provided by the creator and manufacturer
Olimex\footnote{\url{https://www.olimex.com}}. It is published
on Github\footnote{\url{https://github.com/OLIMEX/OLINUXINO}}.
We provide a snapshot of the project files used
in this thesis
at Zenodo \cite{Olimex_zenodo}
and a snapshot of the software project
KiCad, which can read the project file, at \cite{diss_zenodo}.
The calculation domain is
\begin{equation}
\gls{domain} :=
\left[ (-0.35, 90.35) \times (-0.35, 62.85) \times (-0.35, 1.95) \right] mm^3,
\end{equation}
which is the size of the \gls{pcb}
($
\left[ (0, 90) \times (0, 62.5) \times (0, 1.6) \right] mm^3
$)
and a surrounding layer of air of $0.35mm$ thickness.
On the outer boundaries, we use Neumann boundary conditions
on $\gls{boundary}_N := \partial \gls{domain}$.
The metal layers are modeled as thin \gls{pec}
sheets which pose internal Dirichlet boundary conditions.
In the air layers, we assume an electric permittivity
of $\gls{electric_permittivity} = \gls{electric_permittivity}_0 = \frac{1}{\gls{magnetic_permeability} \gls{speedoflight}^2} \approx 8.854 \cdot 10^{-12} \frac{F}{m}$ where $\gls{speedoflight} = 299792458 \frac{m}{s}$ is the speed of light.
In the insulating layers, we assume an electric permittivity of $\gls{electric_permittivity} = \gls{electric_permittivity}_r \cdot \gls{electric_permittivity}_0$ with $\gls{electric_permittivity}_r = 4.5$ which is a typical relative permittivity in FR4
material which is a typical material for \gls{pcb}s.
For the discretization we use a structured mesh of
$960 \times 673 \times 33 = 21,320,640$
cubes with lowest order
Nédélec ansatz functions, which leads to approximately
65 million \gls{dof}s,
see \cref{fig:olimex picture mesh}
for a part of the mesh.
The \exd{} is used in \cref{sec:olimex decay}.
\begin{figure}
\begin{center}
\includegraphics[width=0.4\textwidth]{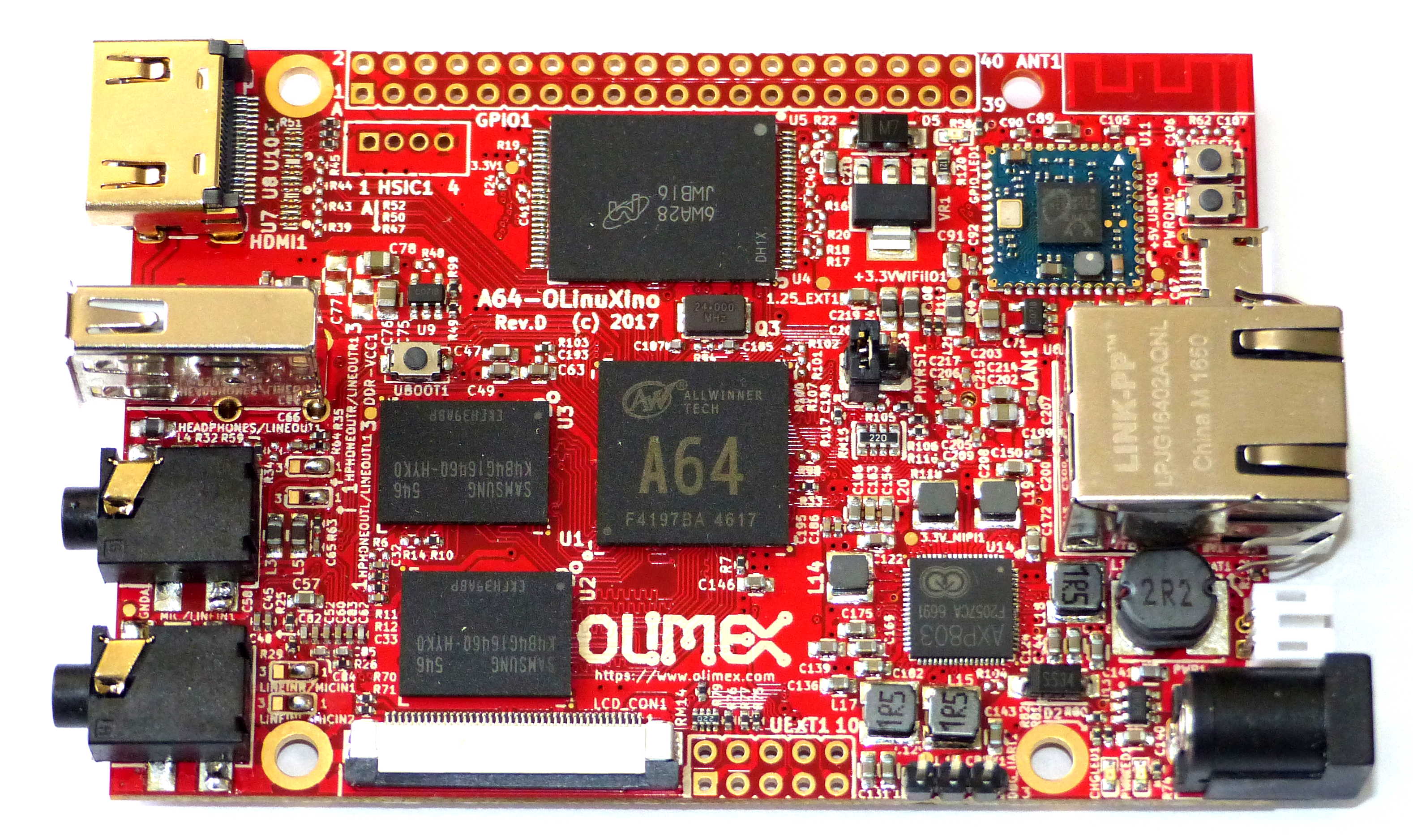}
\hspace{10pt}
\includegraphics[width=0.4\textwidth]{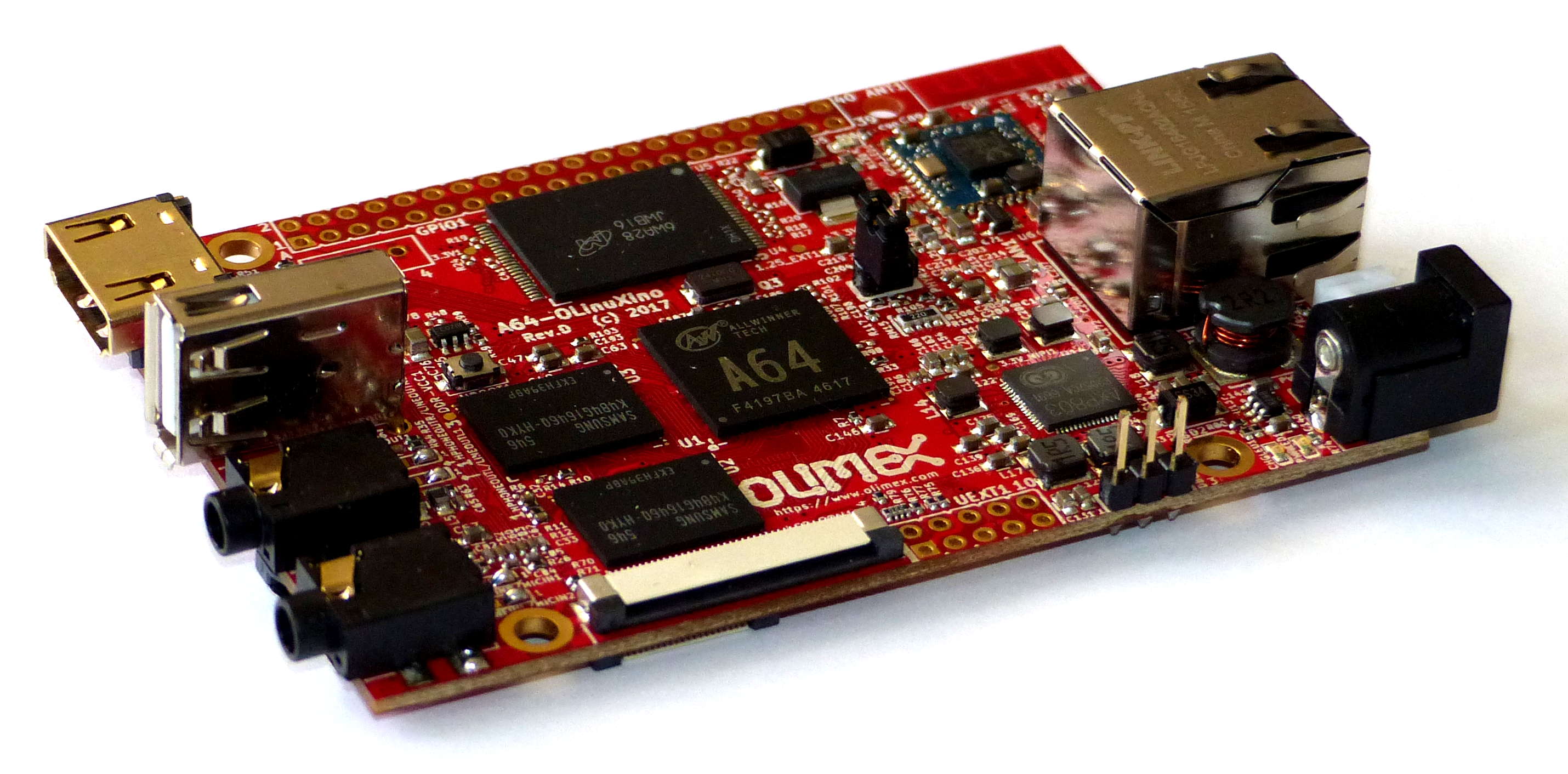}
\end{center}
\caption{Front and perspective view of \exd{}.}
\label{fig:olimex picture 1}
\end{figure}

\begin{figure}
\begin{center}
\includegraphics[width=0.4\textwidth]{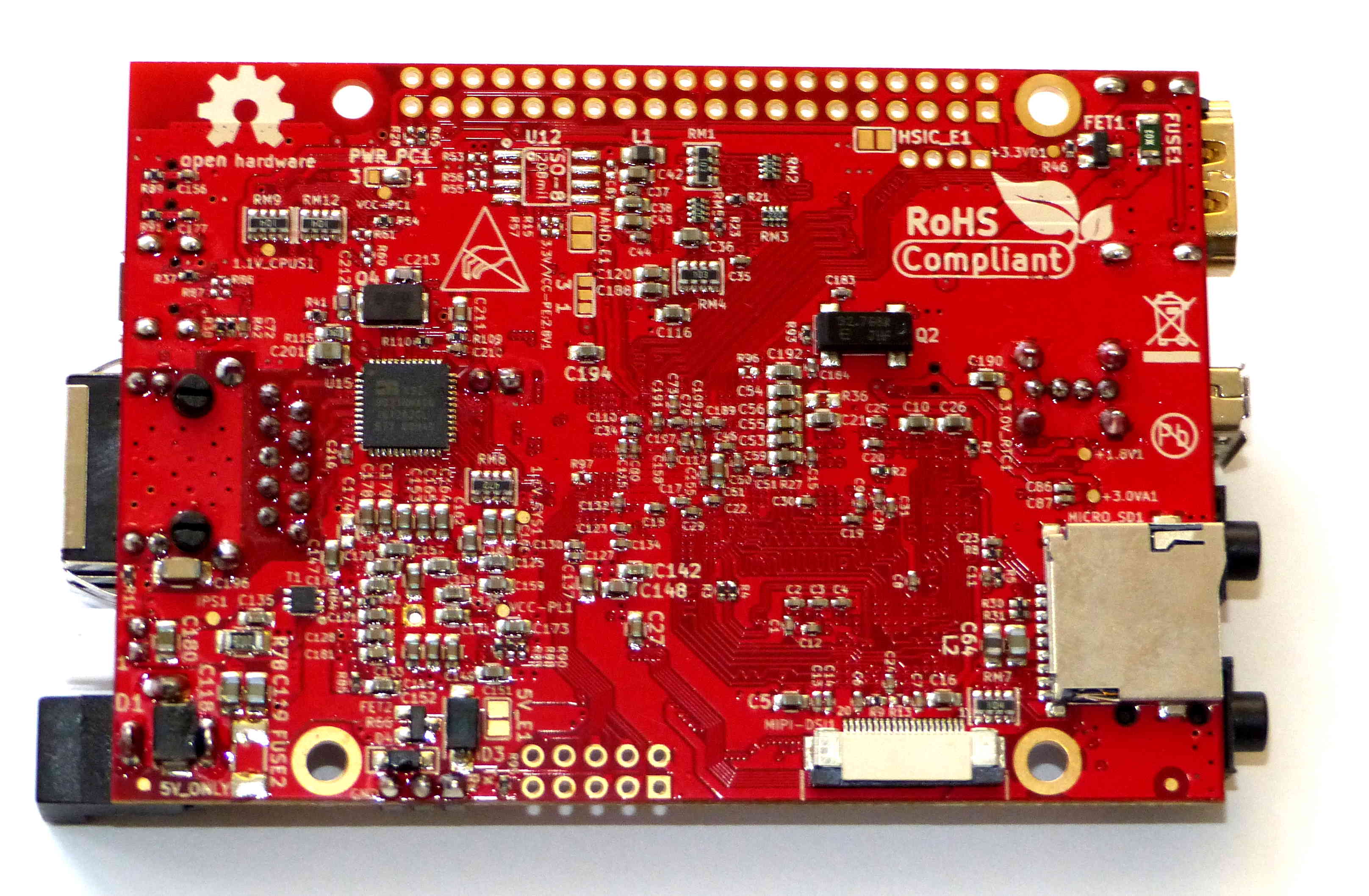}
\hspace{10pt}
\includegraphics[width=0.5\textwidth]{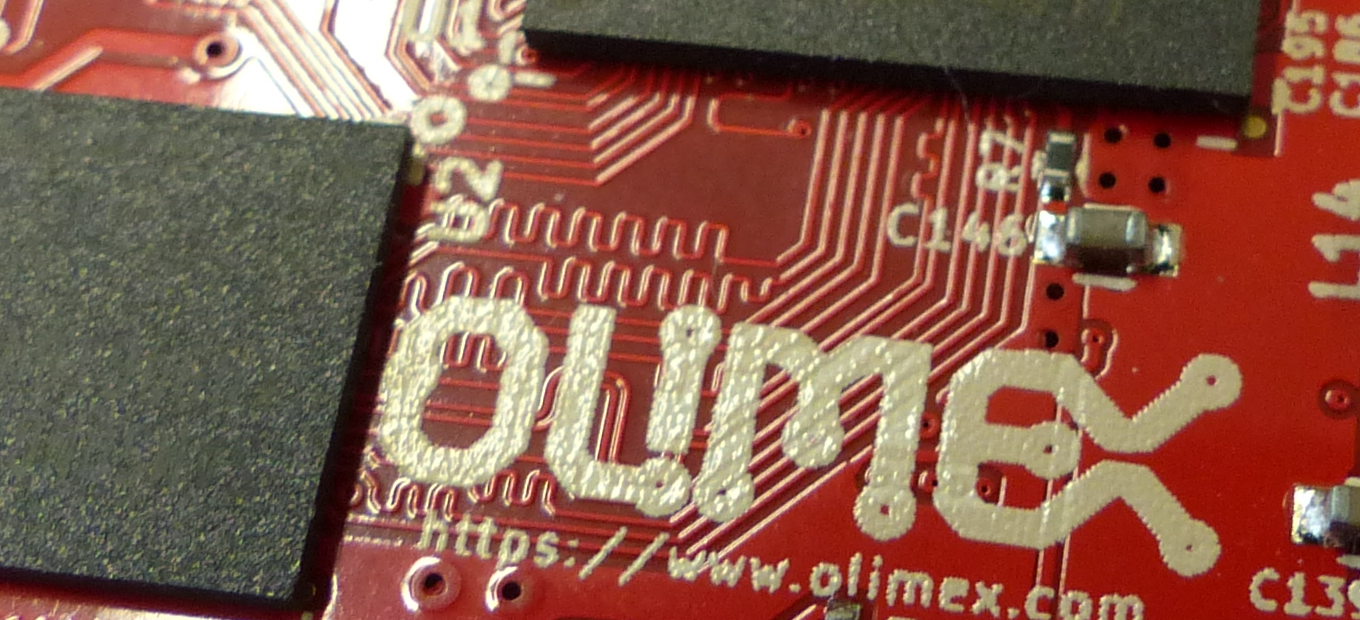}
\end{center}
\caption[Back and detail view of \exd{}.]
{
Back and detail view of \exd.
In the right picture,  the meandering traces between the CPU (top)
and the RAM (left) are visible.}
\label{fig:olimex picture 2}
\end{figure}

\begin{figure}
\begin{center}
\includegraphics[width=0.495\textwidth]{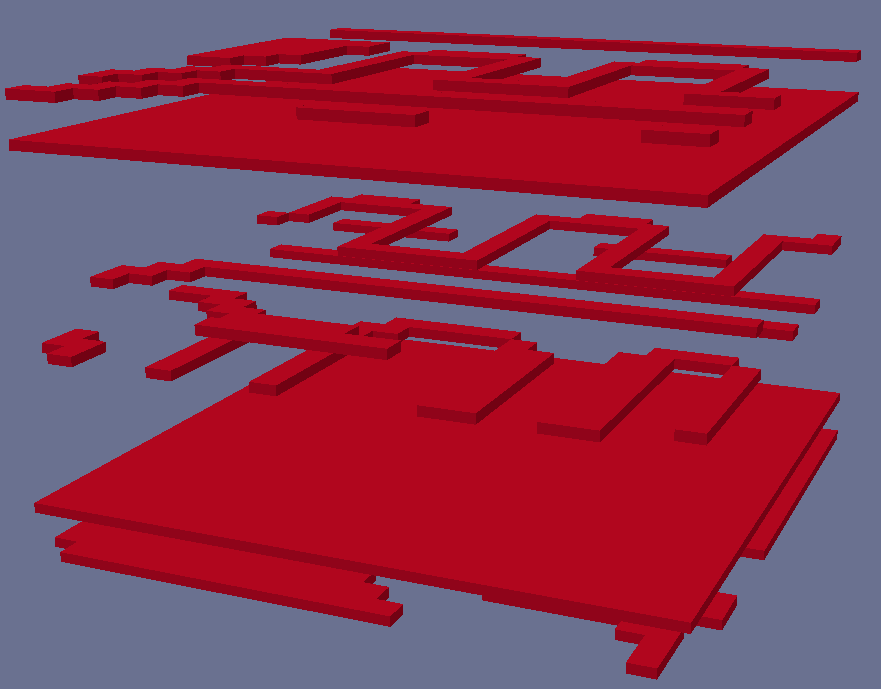}
\includegraphics[width=0.4\textwidth]{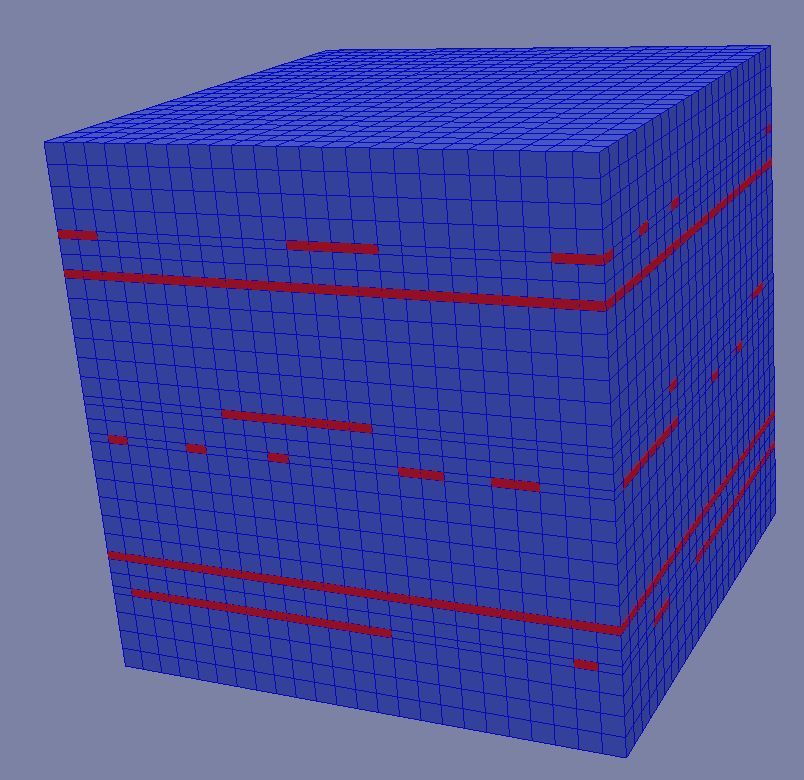}
\end{center}
\caption[Mesh detail of \exd{}.]
{
Meshed detail of \exd{}.
Metal parts, modeled as Dirichlet boundaries, are shown in red.
A staircase approximation of the geometry is used.
}
\label{fig:olimex picture mesh}
\end{figure}

\begin{figure}
\begin{center}
\def\JPicScale{0.56}
{\footnotesize
\ifx\JPicScale\undefined\def\JPicScale{1}\fi
\unitlength \JPicScale mm
\begin{tikzpicture}[x=\unitlength,y=\unitlength,inner sep=0pt]
\draw (0,0) rectangle (20,17.5);
\draw [fill=yellow](0,17.5) rectangle (20,19.25);
\definecolor{userFillColour}{rgb}{0.4,1,0.4}
\draw [fill=userFillColour](0,19.25) rectangle (20,25.6);
\draw [fill=yellow](0,25.6) rectangle (20,27.35);
\definecolor{userFillColour}{rgb}{0.4,1,0.4}
\draw [fill=userFillColour](0,27.35) rectangle (20,52.35);
\draw [fill=yellow](0,52.35) rectangle (20,54.1);
\definecolor{userFillColour}{rgb}{0.4,1,0.4}
\draw [fill=userFillColour](0,54.1) rectangle (20,60.45);
\draw [fill=yellow](0,60.45) rectangle (20,62.2);
\definecolor{userFillColour}{rgb}{0.4,1,0.4}
\draw [fill=userFillColour](0,62.2) rectangle (20,87.2);
\draw [fill=yellow](0,87.2) rectangle (20,88.95);
\definecolor{userFillColour}{rgb}{0.4,1,0.4}
\draw [fill=userFillColour](0,88.95) rectangle (20,95.3);
\draw [fill=yellow](0,95.3) rectangle (20,97.05);
\draw (0,97.05) rectangle (20,114.55);
\draw (40,8.75) node {air};
\draw (40,18.38) node {copper};
\draw (40,22.43) node {prepreg};
\draw (40,26.48) node {copper};
\draw (40,39.85) node {FR4};
\draw (40,53.23) node {copper};
\draw (40,57.27) node {prepreg};
\draw (40,61.33) node {copper};
\draw (40,74.7) node {FR4};
\draw (40,88.08) node {copper};
\draw (40,92.12) node {prepreg};
\draw (40,96.17) node {copper};
\draw (40,105.8) node {air};
\draw (70,8.75) node {350 $\mu$m};
\draw (70,18.38) node {35 $\mu$m};
\draw (70,22.43) node {127 $\mu$m};
\draw (70,26.48) node {35 $\mu$m};
\draw (70,39.85) node {500 $\mu$m};
\draw (70,53.23) node {35 $\mu$m};
\draw (70,57.27) node {127 $\mu$m};
\draw (70,61.33) node {35 $\mu$m};
\draw (70,74.7) node {500 $\mu$m};
\draw (70,88.08) node {35 $\mu$m};
\draw (70,92.12) node {127 $\mu$m};
\draw (70,96.17) node {35 $\mu$m};
\draw (70,105.8) node {350 $\mu$m};
\draw (100,8.75) node {$\gls{electric_permittivity}_r = 1$};
\draw (100,22.43) node {$\gls{electric_permittivity}_r = 4.5$};
\draw (100,39.85) node {$\gls{electric_permittivity}_r = 4.5$};
\draw (100,57.27) node {$\gls{electric_permittivity}_r = 4.5$};
\draw (100,74.7) node {$\gls{electric_permittivity}_r = 4.5$};
\draw (100,92.12) node {$\gls{electric_permittivity}_r = 4.5$};
\draw (100,105.8) node {$\gls{electric_permittivity}_r = 1$};
\end{tikzpicture}
}
\end{center}
\caption[Stackup of \exd{}.]
{Stackup of \exd{}: Six layers of metal with insulating layers in between.
Data provided by Olimex in private communication.
}
\label{fig:stackup}
\end{figure}

\chapter{ArbiLoMod }
\label{chap:arbilomod}
In order to overcome the limitations
of existing methods, we devised 
the simulation methodology
\gls{arbilomod},
named after its main design goal of
handling arbitrary local modifications.
In this chapter, \gls{arbilomod}
is presented. We discuss the design goals,
review the most important design decisions
and define the method.
\gls{arbilomod} is a localized model
order reduction method as introduced in 
\cref{sec:lmor}, based on a localizing
space decomposition and a Galerkin
projection.
To define it, first the wirebasket
space decomposition is presented.
Second, localized basis generation
algorithms for interfaces and volume
spaces are introduced.
Third, a residual based online enrichment
is proposed.
The fourth major ingredient, 
a localized a posteriori error estimator, is postponed
to \cref{chap:a posteriori} because of its
importance, but also because it is coupled
only loosely to \gls{arbilomod}
and might be useful in other settings.
We conclude this chapter by numerical
experiments for the stationary
heat equation and time harmonic Maxwell's
equation.

\section{Method Design}
\gls{arbilomod}
was devised to be a method engineers can rely
upon, while at the same time incorporating
advanced mathematical algorithms.
Besides the ability to handle arbitrary local modifications,
the design goals were:
\begin{enumerate}
\item
Capability of handling interfaces with very complex
geometry and complex field patterns on the interfaces.
\item
Little communication, possibly at the cost 
of CPU cycles, to be performant in current cloud environments.
\item
Implementable on top of existing, conforming \gls{fe} codes.
\item
Little data dependencies, in order
to reuse as many intermediate results as possible after
a localized geometry change.
\item
Reliable results: The method gives a result of sufficient
quality or raise a warning.
\end{enumerate}
\begin{figure}
\begin{center}
\def\JPicScale{0.7}
\ifx\JPicScale\undefined\def\JPicScale{1}\fi
\unitlength \JPicScale mm
\begin{tikzpicture}[x=\unitlength,y=\unitlength,inner sep=0pt]
\draw (0,0) rectangle (100,10);
\draw (8,12) rectangle (28,52);
\draw (6,52) rectangle (30,54);
\draw (6,10) rectangle (30,12);
\draw (40,12) rectangle (60,52);
\draw (38,52) rectangle (62,54);
\draw (38,10) rectangle (62,12);
\draw (72,12) rectangle (92,52);
\draw (70,52) rectangle (94,54);
\draw (70,10) rectangle (94,12);
\draw (0,54) -- (100,54) -- (50,76) .. controls (50,76) and (50,76) .. (50,76) -- cycle;
\draw (50,57) node {Localized Model Order Reduction};
\draw (50,4) node {Space Decomposition};
\draw (18,32) node {\rotatebox[origin=c]{70}{Training}};
\draw (50,32) node {\rotatebox[origin=c]{70}{Error Estimator}};
\draw (82,32) node {\rotatebox[origin=c]{70}{Enrichment}};
\draw [color=white](-5,-5) rectangle (105,80);
\end{tikzpicture}
\def\JPicScale{1}
\end{center}
\caption{Main components of localized model order reduction}
\label{fig:saeulen}
\end{figure}
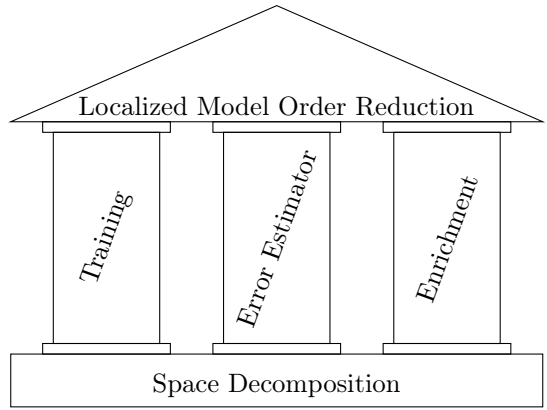

To meet these requirements, we designed a new
space decomposition which will be described in 
the following \cref{sec:space_decomposition}
and structured three core parts (see \cref{fig:saeulen})
of the
algorithm around it:
\begin{itemize}
\item
Localized training
\item
Localized a posteriori error estimation
\item
Localized online enrichment
\end{itemize}
Those parts form a global loop,
as visualized in \cref{fig:workflow}.
Even though there is a global loop,
\gls{arbilomod} is not designed
to be an iterative solution method.
Rather, it should deliver a solution
of sufficient quality in the first iteration.
Only if this fails, online enrichment is employed.

\subsection{Design Decision: Handling Channels}
Numerical multiscale methods like the
\gls{msfem} or the \gls{lod} have a fixed
number of basis functions per coarse grid element and
the basis functions are computed on an oversampling domain.
When applying such a method to a problem with channels,
some modification has to be made.
When a number of channels passes an element and
its oversampling region, the basis construction
can not know which of these channels is active.
So at least one basis function per channel is necessary.
If the number of channels is higher than the fixed number
of basis functions per coarse grid element, the method
probably fails to deliver good results.
There are three possible solutions in this situation.
The first one is to adapt the coarse mesh, steered
by an error indicator. That way, the coarse
mesh cells shrink until they contain less channels
than they have basis functions.
The second option is to increase the oversampling region,
to yield more information about the channels in the
basis construction.
The third option is to allow a variable, adaptive number
of basis functions per coarse grid element.

To keep the implementation simple, we would
like to keep the coarse mesh fixed and independent
of the geometry, therefore we do not want to
adapt the coarse mesh.
The second option, increasing the oversampling
region, would lead to very large oversampling
regions in the simulation of \gls{pcb}s,
as the channels are very long.
Additionally, it would break the desired
localization.
The option we haven chosen for \gls{arbilomod}
is the third one: Having an adaptive number
of basis functions per coarse grid element.

\subsection{Design Decision: Interface Spaces}
One major decision in the design of
localized model order reduction methods is
whether or not to use interface spaces.
Many methods do not use interface
spaces:
For example
\gls{gfem},
\gls{lrbms},
\gls{srbe}, and
\gls{gmsfem}
do not employ interface spaces
while
\gls{cms},
\gls{rbe}, and
\gls{prscrbe}
use them.
The methods without interface spaces tend to be
easier to implement, as there is usually
only one type of local space, while
methods with interface spaces have at least
spaces associated with volumes and
spaces associated with interfaces.
The same is reflected in the mathematical
description of
the methods. Interface spaces lead to more
complexity in notation.
However, we decided
to employ interface spaces,
as they lead to a reduction of computational
requirements in our setting, due to two reasons:
First, interface spaces can usually be approximated
by smaller reduced spaces than volume spaces.
This is beneficial when coupling
is done through the interface spaces,
it can reduced communication in a parallel
implementation.
Second, the generation of a reduced approximation
space for an interface space is
usually computationally less expensive than the 
generation of a volume space.
This is beneficial in a possible
parallel implementation, where the
generation of the same space on different
compute nodes is desirable to avoid
communication of unreduced quantities.
\begin{figure}
\centering
\ifx\JPicScale\undefined\def\JPicScale{1}\fi
\unitlength \JPicScale mm
\begin{tikzpicture}[x=\unitlength,y=\unitlength,inner sep=0pt]
\draw (20,86) ellipse [x radius=15, y radius=4];
\draw (20,86) node {Start};
\draw (20,70) node {Reduced Interface Spaces};
\draw (70,48) -- (50,44);
\draw (70,48) -- (90,44);
\draw (50,44) -- (70,40);
\draw (70,40) -- (90,44);
\draw (70,44) node {converged?};
\draw (20,74) node {Generate Initial};
\draw (90,44) -- (92,44);
\draw (66,50) node {no};
\draw (98,50) node {yes};
\draw [-{Latex[width=3mm]}](20,82) -- (20,76);
\draw (0,76) rectangle (40,68);
\draw [-{Latex[width=3mm]}](20,68) -- (20,62);
\draw (20,56) node {Reduced Volume Spaces};
\draw (20,60) node {Generate};
\draw (0,62) rectangle (40,54);
\draw (20,42) node {Reduced Solution};
\draw (20,46) node {Calculate};
\draw [-{Latex[width=3mm]}](20,54) -- (20,48);
\draw (0,48) rectangle (40,40);
\draw [-{Latex[width=3mm]}](40,44) -- (50,44);
\draw [-{Latex[width=3mm]}](92,44) -- (92,82);
\draw (92,86) ellipse [x radius=16, y radius=4];
\draw (92,86) node {End};
\draw (70,56) node {Reduced Interface Spaces};
\draw (70,60) node {Enrich};
\draw (50,62) rectangle (90,54);
\draw [-{Latex[width=3mm]}](70,48) -- (70,54);
\draw [-{Latex[width=3mm]}](50,58) -- (40,58);
\end{tikzpicture}
\caption[Overview of ArbiLoMod.]
{Overview of \gls{arbilomod}.
Generation of initial reduced interface spaces by training
and of reduced volume spaces by greedy basis generation is subject of \cref{sec:training_and_greedy}
and \cref{chap:training}.
Convergence is assessed with the localized a posteriori error estimator presented in \cref{chap:a posteriori}.
Enrichment is discussed in \cref{sec:enrichment,chap:enrichment}.
On each geometry change, the procedure starts over,
where new initial interface spaces are only generated in the region
affected by the change.
}
\label{fig:workflow}
\end{figure}

\subsection{Design Decision: Using a-harmonic Extensions}
The decision to employ interface
spaces is contrary to the design goal
of being implementable on
top of existing, conforming \gls{fe}
codes. Existing, conforming \gls{fe}
codes usually do not feature interface
spaces
for mortar-type coupling.
Interface spaces
can be formed by all ansatz
functions having support on that
interface.
We define these spaces below
as ``basic spaces'' in \cref{sec:basic spaces}.
However, using these
spaces directly would have some disadvantages:
First, the mappings into the local subspaces
would have a large operator norm
in the $H^1$ norm or energy norm,
because very steep gradients would be introduced.
Second, it would lead to
a bad condition of the reduced system matrices,
because of these gradients.
Third, the properties of the interface spaces
would be very dependent on the fine mesh.
These drawbacks can be avoided
when using \gls{a}-harmonic extensions which spread the support
of the interface functions to the adjacent
domains by solving a local problem.
This will be described in detail below in
\cref{sec:decomposition}.
Additionally, it has the advantage
of smaller reduced volume spaces when the
extension is done by solving the 
underlying equation.

\section{Space Decomposition}
\label{sec:wirebasket space decomposition}
\label{sec:space_decomposition}
In the following, the wirebasket space decomposition,
which is used in \gls{arbilomod},
is defined. Two other common localizing space decompositions
were introduced in \cref{sec:introduction space decompositions}:
the partition of unity decomposition
on overlapping domains
and the restriction decomposition
on non overlapping domains.
In contrast to the partition of unity decomposition and the restriction decomposition,
the wirebasket decomposition which we  introduce here is problem
dependent, as the bilinear form is used in its construction.
The choice of the space decomposition heavily 
affects the performance of every part of the method.

To define the wirebasket space decomposition, we assume that a variational
problem as defined in \cref{def:variational problem}
is given, which depends on a number of parameters $\gls{parameter} \in \gls{parameterspace}$
as described in \cref{sec:parameterized problem}.
We further assume that the global function
space \gls{fspaceh} is a
\gls{fe} approximation space
on a fine mesh which resolves
a non overlapping domain decomposition $\{\gls{subdomainnol}\}_{i=1}^{\gls{numdomainsnol}}$
as defined in \cref{sec:restriction decomposition}.
The wirebasket space decomposition is defined on non overlapping
domains, but has continuous local spaces with
overlapping support.
The non overlapping domain decomposition forms a coarse mesh.
The local spaces $\gls{lfspacewb}$ are associated with the vertices,
edges, faces and volumes of this coarse mesh.
The functions associated with a mesh entity are defined 
by the function values on that entity.
However, they are extended to the neighboring domains
to form continuous functions.
In the following, we introduce definitions for the different kinds of
spaces needed in the wirebasket space decomposition.
We first define ``basic subspaces'' and then, based
on these, we define the wirebasket decomposition.

We formulate the method on a discrete level.
The definitions here mirror the implementation,
wherein the \gls{fe} ansatz functions are grouped
and associated with entities of the coarse mesh, and
afterwards basis functions spanning multiple of
these groups are constructed.
\subsubsection{Basic Subspaces}
\label{sec:basic spaces}
$\gls{fspaceh}$ denotes a discrete ansatz space, spanned by ansatz functions
$\{\psi_i\}_{i=1}^N =: \gls{febasis}$.
We assume that
the ansatz functions have a localized support, which is true
for many classes of ansatz functions like Lagrange- or
N\'ed\'elec-type functions.
To obtain the subspaces we classify the
ansatz functions by their support and define each subspace as the
span of all ansatz functions of one class.
Let $\gls{subdomainnol}$ be a non overlapping domain
decomposition as introduced in \cref{sec:introduction space decompositions}.
For each $\psi$ in $\gls{febasis}$, we call $\gls{domains}_\psi$
the set
of indices of subdomains that have non-empty intersection with the support of $\psi$, i.e.
\begin{equation}
\gls{domains}_\psi :=
\Big\{i \in \{1, \dots, \gls{numdomainsnol}\} \ \Big| \ \supp(\psi) \cap \gls{subdomainnol} \ne \emptyset \Big\}.
\label{eq:febasisdomains}
\end{equation}
Let further $\{\gls{coarsemeshentity}_i\}_{i=1}^{\gls{numcoarseentities}}$ be a collection of all inner mesh entities of
the coarse mesh formed by the subdomains $\gls{subdomainnol}$,
i.e.~all mesh elements, all inner mesh faces, all inner mesh edges and all inner
mesh vertices, where $\gls{numcoarseentities}$ is the number of all inner mesh entities of the coarse mesh.
We associate one basic subspace with each of them.
In order to do so, we call $\gls{domains}_{\gls{coarsemeshentity}_i}$
the set of indices of subdomains that contain the coarse mesh element $\gls{coarsemeshentity}_i$, i.e.
\begin{equation}
\label{eq:domains in mesh}
\gls{domains}_{\gls{coarsemeshentity}_j} :=
\Big\{i \in \{1, \dots, \gls{numdomainsnol}\} \ \Big| \ \gls{coarsemeshentity}_j \subseteq \overline{\gls{subdomainnol}} \Big\}.
\end{equation}

We define index sets for each
codimension:
\begin{align}\label{def:codimsets}
\gls{upsilon}_0 := \Big\{ i \in \{1, \dots, \gls{numcoarseentities}\} \ \Big| \ \mathrm{codim}(\gls{coarsemeshentity}_i) = 0\Big\}, \nonumber \\
\gls{upsilon}_1 := \Big\{ i \in \{1, \dots, \gls{numcoarseentities}\} \ \Big| \ \mathrm{codim}(\gls{coarsemeshentity}_i) = 1\Big\}, \nonumber \\
\gls{upsilon}_2 := \Big\{ i \in \{1, \dots, \gls{numcoarseentities}\} \ \Big| \ \mathrm{codim}(\gls{coarsemeshentity}_i) = 2\Big\}.
\end{align}

The classification is very similar
to the classification of mesh nodes in domain decomposition methods, see
for example
\cite[Definition 3.1]{Klawonn2006}
or \cite[Definition 4.2]{Toselli2005}. 
\begin{definition}[Basic subspaces]
\label{def:basic_space_def}
For each element $i \in \{1, \dots, \gls{numcoarseentities}\}$  we define a basic subspace $\gls{lfspacebasic}$ of $\gls{fspaceh}$ as:
\begin{equation}
\gls{lfspacebasic} := \spanset\Big\{\psi \in \gls{febasis} \ \Big| \ \gls{domains}_\psi = \gls{domains}_{\gls{coarsemeshentity}_i} \Big\}.
\nonumber
\end{equation}
\end{definition}
\begin{remark}[Basic space decomposition]
The definition of $\gls{lfspacebasic}$ induces a direct decomposition of $\gls{fspaceh}$. It holds
\begin{equation}
\gls{fspaceh} = \bigoplus_{i \in \{1, \dots, \gls{numcoarseentities}\}} \gls{lfspacebasic}.
\nonumber
\end{equation}
\end{remark}
\subsubsection{Wirebasket Space Decomposition}
\label{sec:decomposition}
Based on the basic subspaces defined in the previous section,
we define the spaces of the wirebasket space decomposition here.

For each of the spaces $\gls{lfspacebasic}$ defined in \cref{def:basic_space_def},
we calculate extensions. The extensions are computed on the
``extension space'' $\extension{\gls{lfspacebasic}}$ which is defined as
\begin{equation}
\label{eq:definition_extension_space}
\extension{\gls{lfspacebasic}} := \bigoplus \Big\{V_j^\mathrm{basic} \ \Big| \ \gls{domains}_{\gls{coarsemeshentity}_j} \subseteq \gls{domains}_{\gls{coarsemeshentity}_i} \Big\}.
\end{equation}
\begin{figure}
\footnotesize
\centering
\def\myscale{1.2}
\def\circsize{0.06}
\def\circsizetwo{0.09}
\begin{tikzpicture}[scale=\myscale]
\node at (0.5+0.125,0.5) [color=Xred] {\Huge 1};
\draw [step=2.5mm,color=gray,very thin] (-0.15,-0.15) grid (1.15,1.15);
\draw [step=1cm,black,line width=0.7mm] (-0.15,-0.15) grid (1.15,1.15);
\foreach \x in {1, ..., 3} {
  \foreach \y in {1, ..., 3} {
    \draw (\x * 0.25, \y * 0.25) circle(\circsizetwo);
  }
}
\foreach \x in {1, ..., 3} {
  \foreach \y in {1, ..., 3} {
    \draw [fill=black] (\x * 0.25, \y * 0.25) circle(\circsize);
  }
}
\node at (0.5,-0.5) {$i \in \gls{upsilon}_0$};
\end{tikzpicture}
\begin{tikzpicture}[scale=\myscale]
\node at (0.5+0.125,0.5) [color=Xred] {\Huge 1};
\node at (1.5+0.125,0.5) [color=Xred] {\Huge 2};
\draw [step=2.5mm,color=gray,very thin] (-0.15,-0.15) grid (2.15,1.15);
\draw [step=1cm,black,line width=0.7mm] (-0.15,-0.15) grid (2.15,1.15);
\foreach \x in {1, ..., 7} {
  \foreach \y in {1, ..., 3} {
    \draw (\x * 0.25, \y * 0.25) circle(\circsizetwo);
  }
}
\foreach \x in {4} {
  \foreach \y in {1, ..., 3} {
    \draw [fill=black] (\x * 0.25, \y * 0.25) circle(\circsize);
  }
}
\node at (1,-0.5) {$i \in \gls{upsilon}_1$};
\end{tikzpicture}
\begin{tikzpicture}[scale=\myscale]
\node at (0.5+0.125,1.5) [color=Xred] {\Huge 1};
\node at (1.5+0.125,1.5) [color=Xred] {\Huge 2};
\node at (0.5+0.125,0.5) [color=Xred] {\Huge 3};
\node at (1.5+0.125,0.5) [color=Xred] {\Huge 4};
\draw [step=2.5mm,color=gray,very thin] (-0.15,-0.15) grid (2.15,2.15);
\draw [step=1cm,black,line width=0.7mm] (-0.15,-0.15) grid (2.15,2.15);
\foreach \x in {1, ..., 7} {
  \foreach \y in {1, ..., 7} {
    \draw (\x * 0.25, \y * 0.25) circle(\circsizetwo);
  }
}
\foreach \x in {4} {
  \foreach \y in {4} {
    \draw [fill=black] (\x * 0.25, \y * 0.25) circle(\circsize);
  }
}
\node at (1,-0.5) {$i \in \gls{upsilon}_2$};
\end{tikzpicture}
\begin{tikzpicture}[scale=\myscale]
\draw [fill=black] (0.15,1.2) circle(\circsize);
\node at (0.4,1.2) [right] {\gls{dof} of $\gls{lfspacebasic}$};
\draw (0.15,0.9) circle(\circsizetwo);
\node at (0.4,0.9) [right] {\gls{dof} of $\extension{\gls{lfspacebasic}}$};
\draw [color=gray, very thin] (0,0.6) -- (0.3,0.6);
\node at (0.4,0.6) [right] {mesh line};
\draw [line width=0.7mm] (0,0.3) -- (0.3,0.3);
\node at (0.4,0.3) [right] {domain boundary};
\node at (0.15,0) [color=Xred] {\huge 1};
\node at (0.4, 0) [right] {domain number};
\node at (0,-1) {};
\end{tikzpicture}
\caption[Visualization of basic spaces $\gls{lfspacebasic}$ and their extension spaces for $Q^1$ ansatz functions.]
        {Visualization of basic spaces $\gls{lfspacebasic}$ and their extension spaces for $Q^1$ ansatz functions
(one \gls{dof} per mesh node).}
\label{fig:extension_space}
\end{figure}
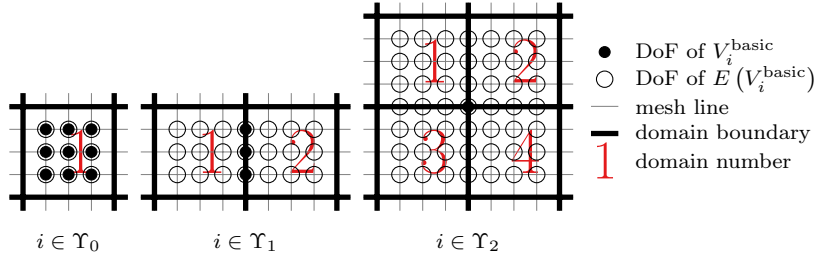
The functions in one extension space have support in a corresponding extension
domain
\begin{equation}
\gls{subdomainext} := 
\overset{\circ}{\overline{
\sum_{j \in \gls{domains}_{\gls{coarsemeshentity}_i}} \omega_j^\mathrm{nol}
}}
.
\end{equation}
Examples for extension spaces are given in \cref{fig:extension_space}.
For each space $\gls{lfspacebasic}$, a linear extension operator $\gls{extend}$ is defined:
\begin{equation}
\gls{extend} : \gls{lfspacebasic} \rightarrow \extension{\gls{lfspacebasic}}.
\end{equation}
For all $i$ in $\gls{upsilon}_0$, $\gls{extend}$ is just the identity.
For all $i$ in $\gls{upsilon}_1$, we extend by solving the
homogeneous version of the equation with Dirichlet zero
boundary values for one (arbitrary) chosen
$\defaultparameter \in \gls{parameterspace}$.
For all $i$ in $\gls{upsilon}_2$, $\gls{extend}$ is defined by first
extending linearly to zero on all edges in the extension domain, i.e.~in all basic spaces in $ \extension{\gls{lfspacebasic}}$
which belong to $\gls{upsilon}_1$. Then, in a second step, the homogeneous version
of the equation with Dirichlet boundary values is solved on the basic spaces in
$ \extension{\gls{lfspacebasic}}$ which belong to $\gls{upsilon}_0$.
The procedure is visualized in \cref{fig:vertexextension}.
The functions constructed by this two step procedure are
the same as the \gls{msfem} basis functions used by Hou and Wu \cite{Hou1997}.
Note that the base functions for $\gls{lfspacewb}, i \in \gls{upsilon}_2$ form a partition 
of unity in the interior of the coarse partition of the domain. They can be completed to form 
a partition of unity on the whole coarse partition of the domain if suitable base functions for 
the vertices at the boundary of the domain are added. 
This will be used
below for 
the robust and efficient localization of an a posteriori error estimator. 
Examples of extended functions for all codimensions are given in \cref{fig:extended_base}.
In the case of the Laplace equation these basis functions coincide with the hat functions on the 
coarse partition (see \cref{fig:vertexextension}).  
The choice of hat function can be an alternative 
choice that allows for a better a priori bound of the constants in the localized a posteriori error estimator, 
as their gradient is controlled by $1/\gls{H}$ -- where $\gls{H}$ denotes the mesh size of the macro partition -- independent of the contrast of the data.
\newlength{\foobar}
\setlength{\foobar}{2cm}
\def\myscale{1.2}
\def\circsize{0.06}
\def\circsizetwo{0.08}
\begin{figure}
\footnotesize
\centering
\begin{subfigure}[b]{0.19\textwidth}
\centering
\begin{tikzpicture}[scale=\myscale]
\draw [step=0.4mm,color=gray,very thin] (-0.15,-0.15) grid (1.15,1.15);
\draw [step=1cm,black,line width=0.7mm] (-0.15,-0.15) grid (1.15,1.15);
\node at (0.5,0.5) {\includegraphics[width=\myscale cm]{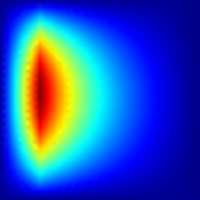}};
\end{tikzpicture}
\caption{$i \in \gls{upsilon}_0$}
\label{fig:extended_base_0}
\end{subfigure}
\begin{subfigure}[b]{0.24\textwidth}
\centering
\begin{tikzpicture}[scale=\myscale]
\draw [step=0.4mm,color=gray,very thin] (-0.15,-0.15) grid (2.15,1.15);
\draw [step=1cm,black,line width=0.7mm] (-0.15,-0.15) grid (2.15,1.15);
\node at (1,0.5) {\includegraphics[width=\myscale\foobar]{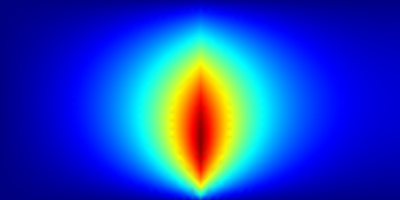}};
\end{tikzpicture}
\caption{$i \in \gls{upsilon}_1$}
\label{fig:extended_base_1}
\end{subfigure}
\begin{subfigure}[b]{0.26\textwidth}
\centering
\begin{tikzpicture}[scale=\myscale]
\draw [step=0.4mm,color=gray,very thin] (-0.15,-0.15) grid (2.15,2.15);
\draw [step=1cm,black,line width=0.7mm] (-0.15,-0.15) grid (2.15,2.15);
\node at (1,1) {\includegraphics[width=\myscale\foobar]{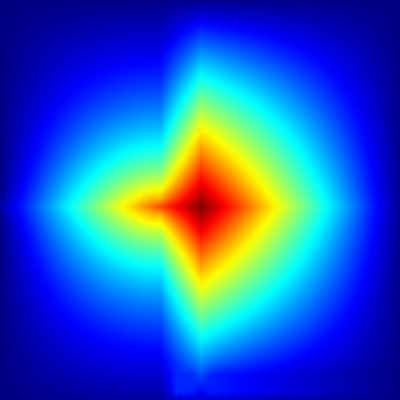}};
\end{tikzpicture}
\caption{$i \in \gls{upsilon}_2$}
\label{fig:extended_base_2}
\end{subfigure}
\begin{subfigure}[b]{0.26\textwidth}
\centering
\begin{tikzpicture}[scale=\myscale]
\draw [line width=0.6mm] (0,0) -- (0.3,0) node[right]{domain \!boundary};
\draw [color=gray, very thin] (0,0.3) -- (0.3,0.3) node[right,color=black]{mesh line};
\node at (0,-1) {};
\end{tikzpicture}
\end{subfigure}
\caption[Visualization of some example elements of the local subspaces $\gls{lfspacewb}$.]
{
Visualization of some example elements of the local subspaces $\gls{lfspacewb}$
for inhomogeneous coefficients.
The structure in the solution results from variations in the
heat conduction coefficient.
}
\label{fig:extended_base}
\end{figure}
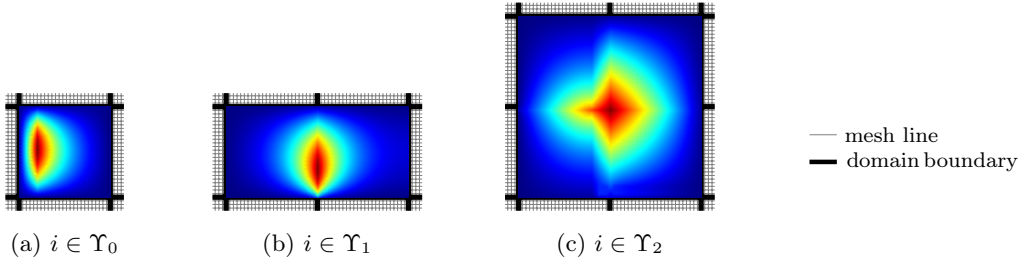

\begin{figure}
\centering
\begin{subfigure}[b]{0.3\textwidth}
\centering
\includegraphics[width=\textwidth]{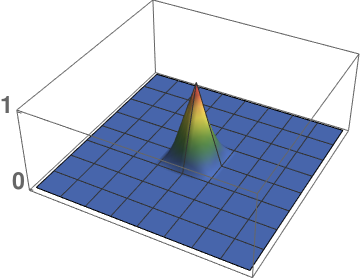}
\caption{Value 1 in $\gls{lfspacebasic}$}
\label{fig:vertexextension1}
\end{subfigure}
\begin{subfigure}[b]{0.3\textwidth}
\centering
\includegraphics[width=\textwidth]{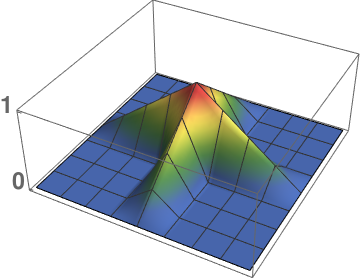}
\caption{Linear in $V_j^\mathrm{basic}, \ j \in \gls{upsilon}_1$}
\label{fig:vertexextension2}
\end{subfigure}
\begin{subfigure}[b]{0.3\textwidth}
\centering
\includegraphics[width=\textwidth]{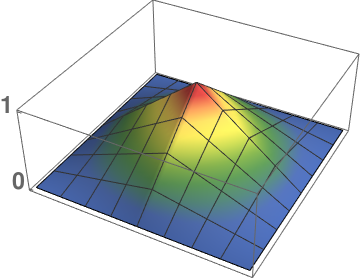}
\caption{Solving in $V_k^\mathrm{basic}, \ k \in \gls{upsilon}_0$}
\label{fig:vertexextension3}
\end{subfigure}
\caption[$\gls{extend}$ operator is executed in two steps for spaces $\gls{lfspacewb}, i \in \gls{upsilon}_2$.]
{$\gls{extend}$ operator is executed in two steps for spaces $\gls{lfspacewb}, i \in \gls{upsilon}_2$. 
Mesh and spaces as depicted in \cref{fig:extension_space}.
(reproduction: \cref{repro:fig:vertexextension})
}
\label{fig:vertexextension}
\end{figure}

For the communication avoiding properties of the \gls{arbilomod}, it is
important to note that extensions can be calculated independently on
each domain, i.e. $\gls{extend}(\varphi)|_{\gls{subdomainnol}}$ can be calculated
having only information about $\varphi$ and $\gls{subdomainnol}$, without
knowledge about other domains.
Using this operator, we define the local subspaces $\gls{lfspacewb}$:

\begin{definition}[Extended subspaces]
\label{def:extended_space_def}
For each element $i \in \{1, \dots, \gls{numcoarseentities}\}$  we define an extended subspace $\gls{lfspacewb}$ of $\gls{fspaceh}$ as:
\begin{equation}
\label{eq:space_decomposition}
\gls{lfspacewb} := \Big\{ \gls{extend} (\varphi) \ \Big| \ \varphi \in \gls{lfspacebasic} \Big\}.
\end{equation}
According to the definition in \eqref{def:codimsets} we call $\gls{lfspacewb} $ a cell, face, or vertex space,
if $i \in \gls{upsilon}_0, i \in \gls{upsilon}_1$, or $i \in \gls{upsilon}_2$, respectively (cf.\ \cref{fig:extended_base}).
\end{definition}
\begin{remark}[Extended subspaces]
The definition of $\gls{lfspacewb}$ induces a direct decomposition of $\gls{fspaceh}$:
\begin{equation}
\gls{fspaceh} = \bigoplus_{i \in \{1, \dots, \gls{numcoarseentities}\}} \gls{lfspacewb}.
\end{equation}
\end{remark}
\noindent Space decompositions of the same spirit are used
in the context of Component Mode Synthesis
(CMS), see \cite{Hetmaniuk2010,Hetmaniuk2014}
and recently, it was also applied in the \gls{prscrbe}
context \cite[Section 5.3.3]{marioalee2018}.
\subsubsection{Projections}
\label{sec:projection_operators}

\begin{definition}[Local projection operators]
\label{def:local_projection_operators}
The projection operators
$\gls{lfspacemapbasic} : \gls{fspaceh} \rightarrow \gls{lfspacebasic}$
and
$\gls{lfspacemapwb} : \gls{fspaceh} \rightarrow \gls{lfspacewb}$
are defined by the relation
\begin{equation}
\varphi =
\sum_{i \in \{1, \dots, \gls{numcoarseentities}\}} \gls{lfspacemapbasic}(\varphi)
= \sum_{i \in \{1, \dots, \gls{numcoarseentities}\}} \gls{lfspacemapwb}(\varphi)
\qquad \forall \varphi \in \gls{fspaceh}.
\end{equation}
As both the subspaces $\gls{lfspacebasic}$ and the subspaces $\gls{lfspacewb}$ form
a direct decomposition of the space $\gls{fspaceh}$, the projection operators are
uniquely defined by this relation.
\end{definition}

The implementation of the projection operators
$\gls{lfspacemapbasic}$ is very easy: It is just extracting
the coefficients of the basis functions forming $\gls{lfspacebasic}$ out
of the global coefficient vector.
The implementation of the projection operators $\gls{lfspacemapwb}$
is more complicated and involves the solution of local problems,
see \cref{algo:projections}.
\IncMargin{1em}
\begin{algorithm2e}
\DontPrintSemicolon%
\SetAlgoVlined%
\SetKwFunction{SpaceDecomposition}{SpaceDecomposition}%
\SetKwInOut{Input}{Input}
\SetKwInOut{Output}{Output}
\Fn{\SpaceDecomposition{$\varphi$}}{
  \Input{function $\varphi \in \gls{fspaceh}$}
  \Output{decomposition of $\varphi$}
  \tcc{iterate over all codimensions in decreasing order}
  \For{
    $\mathrm{codim} \in \{\gls{dimension}, \dots, 0\}$
  }%
  {
    \For{
        $i \in \gls{upsilon}_\mathrm{codim}$
    }
    {
        $\varphi_i \leftarrow \gls{extend}(\gls{lfspacemapbasic}(\varphi))$\;
        $\varphi \leftarrow \varphi - \varphi_i$
    }
  }
  \Return $\{\varphi_i\}$\;
}
\caption{Projections of wirebasket decomposition \gls{lfspacemapwb}.}
\label{algo:projections}
\end{algorithm2e}
\DecMargin{1em}

\noindent
With the projections defined, the definition 
of the wirebasket space decomposition
\begin{equation}
\left\{
\gls{domain}, \gls{fspaceh},
\{\gls{subdomainext}\}_{i=1}^{\gls{numcoarseentities}},
\{\gls{lfspacewb}\}_{i=1}^{\gls{numcoarseentities}},
\{\gls{lfspacemapwb}\}_{i=1}^{\gls{numcoarseentities}}
\right\}
\end{equation}
is complete.

\section{Local Basis Generation}
\label{sec:training_and_greedy}
For each local subspace $\gls{lfspacewb}$, an initial reduced local
subspace $\gls{rlfspacewb} \subseteq \gls{lfspacewb}$ is generated, using
only local information from an environment around the support of the
elements in $\gls{lfspacebasic}$. The strategy used to construct
these
reduced local subspaces depends on the type of the
space: whether $i$ belongs to $\gls{upsilon}_0$, $\gls{upsilon}_1$ or
$\gls{upsilon}_2$. The three strategies are given in the
following.
The local basis generation algorithms on different domains for the same codimension
can be run in parallel, completely independent of each other.
The basis generation for the spaces $\gls{rlfspacewb}, i \in \gls{upsilon}_0$
requires the spaces $\gls{rlfspacewb}, i \in \gls{upsilon}_1 \cup \gls{upsilon}_2$ (for the 2D case).
See \cref{sec:runtime_and_communication} for further
discussion of the potential parallelization.
As the algorithms only use local information, their results do not change
when the problem definition is changed outside of the area they took into
account. So there is no need to rerun the algorithms in this case.
Our numerical results indicate that the spaces obtained by
local trainings and greedys have good approximation properties
(see \cref{sec:numerical_experiments}).
The quality of the obtained solution will be guaranteed by
the a posteriori
error estimator presented in \cref{chap:a posteriori}.
\subsection{Basis Construction for Reduced Vertex Spaces}
\label{sec:codim_2_spaces}
The spaces $\gls{lfspacewb}$ for $i \in \gls{upsilon}_2$ are spanned by
only one function (see \cref{fig:extended_base_2} for an example)
and are thus one dimensional. The reduced spaces
are therefore chosen to coincide with the original space, i.e.
$
\gls{rlfspacewb} := \gls{lfspacewb}, \forall i \in \gls{upsilon}_2.
$
\subsection{Local Training for Basis Construction of Reduced Face Spaces}
\label{sec:codim_n_training}
To generate an initial reduced local subspace
$\gls{rlfspacewb}, \ i \in \gls{upsilon}_1$
we use a local training procedure.
We present a simple training procedure here.
An improved variant is the topic of \cref{chap:training}.
To sample the parameter space \gls{parameterspace},
a discrete set $\gls{trainingset} \subset \gls{parameterspace}$
is used.
The main four steps of the training are to
\begin{enumerate}
\item solve the equation on a small domain around the space in question
with zero boundary values for all parameters in the training set $\gls{trainingset}$,
\item
solve the homogeneous equation
repeatedly on a small domain around the space in question with
random boundary values for all parameters in $\gls{trainingset}$,
\item
apply the space decomposition to all obtained local solutions
to obtain the part belonging to the space in question
and
\item
use a greedy procedure to create a space approximating this set.
\end{enumerate}
The complete algorithm is given in \cref{algo:training}
and explained below.

The training is inspired by the ``Empirical Port Reduction'' introduced in
Eftang et al.\ \cite{Eftang2013} but differs in some key points.
The main differences are:
(1)~%
Within~\cite{Eftang2013}, the trace of solutions at the interface to be trained is used.
This leads to the requirement that interfaces between domains do not intersect.
In \gls{arbilomod}, a space decomposition is used instead. This allows ports to intersect,
which in turn allows the decomposition of space into domains.
(2)~%
The ``Empirical Port Reduction'' trains with a pair of domains. We use an environment
of the interface in question, which contains six domains in the 2D case.
In 3D, it contains 18 domains.
(3)~%
\gls{prscrbe} aims at providing a library of domains
which can be connected at their interfaces. The reduced interface spaces
are used in different domain configurations and have to be valid in all of them.
Within the context of \gls{arbilomod}, no database of domains is created and 
the interface space is constructed only for the configuration at hand,
which simplifies the procedure.
(4)~%
The random boundary values used in \cite{Eftang2013} are generalized
Legendre polynomials with random coefficients.
In \gls{arbilomod}, the finite element basis functions with random coefficients
are used, which simplifies the construction, especially when there
is complex structure within the interface.

\subsubsection*{The Training Space}
\begin{figure}
\centering
\def\myscale{1.3}
\def\circsize{0.06}
\def\circsizetwo{0.09}
\begin{tikzpicture}[scale=\myscale]
\node at (0.5+0.125,1.5) [color=Xred] {\Huge 1};
\node at (1.5+0.125,1.5) [color=Xred] {\Huge 2};
\node at (2.5+0.125,1.5) [color=Xred] {\Huge 3};
\node at (0.5+0.125,0.5) [color=Xred] {\Huge 4};
\node at (1.5+0.125,0.5) [color=Xred] {\Huge 5};
\node at (2.5+0.125,0.5) [color=Xred] {\Huge 6};
\draw [step=2.5mm,color=gray,very thin] (-0.2,-0.2) grid (3.2,2.2);
\draw [step=1cm,black,line width=0.7mm] (-0.2,-0.2) grid (3.2,2.2);
\foreach \x in {1, ..., 11} {
  \foreach \y in {1, ..., 7} {
    \draw (\x * 0.25, \y * 0.25) circle(\circsizetwo);
  }
}
\foreach \x in {5, ..., 7} {
  \foreach \y in {4} {
    \draw [fill=black] (\x * 0.25, \y * 0.25) circle(\circsize);
  }
}
\foreach \x in {1, ..., 11} {
  \foreach \y in {0,8} {
    \draw [Xred,fill=Xred!60] (\x * 0.25, \y * 0.25) circle(\circsize);
  }
}
\foreach \x in {0,12} {
  \foreach \y in {0, ..., 8} {
    \draw [Xred,fill=Xred!60] (\x * 0.25, \y * 0.25) circle(\circsize);
  }
}
\node at (1.5,-0.5) {$i \in \gls{upsilon}_1$};
\end{tikzpicture}
\begin{tikzpicture}[scale=\myscale]
\draw [fill=black] (0.15,1.5) circle(\circsize);
\node at (0.4,1.5) [right] {\gls{dof} of $\gls{lfspacebasic}$};
\draw (0.15,1.2) circle(\circsizetwo);
\node at (0.4,1.2) [right] {\gls{dof} of $\training{\gls{lfspacewb}}$};
\draw (0.15,0.9) [Xred!,fill=Xred!60] circle(\circsize);
\node at (0.4,0.9) [right]{\gls{dof} of $\coupling{\training{\gls{lfspacewb}}}$};
\draw [color=gray, very thin] (0,0.6) -- (0.3,0.6);
\node at (0.4,0.6) [right] {mesh line};
\draw [line width=0.7mm] (0,0.3) -- (0.3,0.3);
\node at (0.4,0.3) [right] {domain boundary};
\node at (0.15,0) [color=Xred] {\huge 1};
\node at (0.4, 0) [right] {domain number};
\node at (0,-.9) {};
\end{tikzpicture}
\caption[Visualization of basic spaces \gls{lfspacebasic}, training space, and coupling space.]
{Visualization of spaces for a codim 1 coarse mesh entity.
Basic space \gls{lfspacebasic}, training space, and the coupling space of training space for $Q^1$ ansatz functions
(one \gls{dof} per mesh node).}
\label{fig:training_space}
\end{figure}
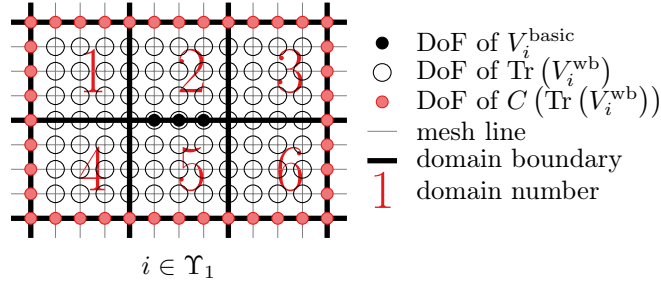
To train a basis for $\gls{rlfspacewb}$, $i \in \gls{upsilon}_1$, 
we define a training
space $\training{\gls{lfspacewb}}$ on an environment
associated with the face $\gls{coarsemeshentity}_i$.
The following definition is geometric
and tailored to a domain decomposition in rectangular domains. 
More complex domain decompositions would need a more complex definition here.
We define the neighborhood
\begin{equation}
\gls{neigh}_j := \Big\{i \in \{1, \dots, \gls{numdomainsnol}\} \ \Big| \ \overline{\gls{subdomainnol}} \cap \overline{\gls{coarsemeshentity}_j} \ne \emptyset \Big\}
\end{equation}
and the training spaces
\begin{equation}
\label{eq:arbilomod training}
\training{\gls{lfspacewb}} := \bigoplus \Big\{V_j^\mathrm{basic} \ \Big| \ \gls{domains}_{\gls{coarsemeshentity}_j} \subseteq \gls{neigh}_i \Big\}.
\end{equation}
The training space is coupled to the rest of the system via its
coupling space%
\begin{equation}
\label{eq:arbilomod training coupling}
\coupling{\training{\gls{lfspacewb}}}
:=
\bigoplus \Big\{V_j^\mathrm{basic} \ \Big| \ \gls{domains}_{\gls{coarsemeshentity}_j} \cap \gls{neigh}_i \ne \emptyset , \gls{domains}_{\gls{coarsemeshentity}_j} \nsubseteq \gls{neigh}_i \Big\}.
\end{equation}
A sketch of the degrees of freedom associated with the respective spaces is given in \cref{fig:training_space}.
These definitions are also suitable in the 3D case.
We have fixed the size of the neighborhood to one domain from the interface in question in each direction.
This facilitates the setup of local problems and the 
handling of local changes: After a local change, the affected domains are determined.
Afterwards, all trainings have to be redone 
for those spaces which contain an affected domain in their training domain.
While a larger or smaller training domain might be desirable in some cases 
(see \cite{henning_oversampling}), it is not
necessary:
As missing global information is added in the enrichment step,
\gls{arbilomod} always converges to the desired accuracy,
even if the training domain is not of optimal size. So the advantages of having
the size of the training domain fixed to one domain outweighs its drawbacks.

The reduced basis must be rich enough to
handle two types of right hand sides up to a given accuracy $\varepsilon_\mathrm{train}$:
(a) source terms and boundary conditions, and
(b) arbitrary values on the coupling interface,
both in the whole parameter space $\gls{parameterspace}$.
We define an extended parameter
space $\gls{parameterspace} \times \coupling{\training{\gls{lfspacewb}}}$. For this
parameter space we construct a training space
$\gls{trainingset} \times G \subset \gls{parameterspace} \times \coupling{\training{\gls{lfspacewb}}}$, where $G$
denotes an appropriate sampling of
$\coupling{\training{\gls{lfspacewb}}}$.
We use the finite element
basis $\gls{febasis}_{\coupling{\training{\gls{lfspacewb}}}}$ on the coupling space and
generates $M$ random coefficient vectors $r_i$ of size $N_{\mathcal{B}_C} :=
\dim({\coupling{\training{\gls{lfspacewb}}}})$. With this an
individual coupling space function $\varphi \in
\coupling{\training{\gls{lfspacewb}}}$ is constructed
as
\begin{equation}
\varphi_i = \sum_{j=1}^{N_{\mathcal{B}_C}} r_{ij} \phi_j; \qquad\phi_j \in \gls{febasis}_{\coupling{\training{\gls{lfspacewb}}}}.\label{eq:3}
\end{equation}
For our numerical experiments in \cref{sec:numerical_experiments},
we use uniformly distributed random coefficients over the
interval $[-1,1]$ and Lagrange basis functions.
For each ${\gls{parameter}} \in \gls{trainingset}$ and each pair $({\gls{parameter}},g_c) \in \gls{trainingset} \times G$
we construct snapshots $u_f$ and $u_c$ as solutions for
right hand sides $f_{\gls{parameter}}(.)$, $a_{{\gls{parameter}}}(g_c,.)$ respectively, i.e.
find $u_f, u_c \in \training{\gls{lfspacewb}}$ such that
\begin{align}
\label{eq:trainingproblem}
  a_{{\gls{parameter}}}(u_f,\phi) &= \dualpair{f_{\gls{parameter}}}{\phi} \qquad \forall \phi \in \training{\gls{lfspacewb}},\\
  a_{{\gls{parameter}}}(u_c,\phi) & = - a_{{\gls{parameter}}}(g_c,\phi) \quad \forall \phi \in \training{\gls{lfspacewb}}\,.
  \nonumber
\end{align}
Existence and uniqueness of $u_f$ and $u_c$ are guaranteed in the coercive case,
as the problem \cref{eq:trainingproblem} is a Galerkin
projected variant of the full problem and the coercivity is
maintained during projection (cf.~\cref{sec:reduced problem}).
In the inf-sup stable case, existence and uniqueness of $u_f$ and $u_c$
are not guaranteed.
Based on the set of snapshots which we call $Z$, a reduced basis $\rbasis_{\gls{rlfspacewb}}$
is constructed using the greedy procedure \cref{algo:snapshot_greedy}.
In the numerical experiments, the $V$-norm and $V$-inner product are used.
The complete generation of the reduced face spaces $\gls{rlfspacewb}, \ i \in \gls{upsilon}_1$ with basis $\rbasis_{\gls{rlfspacewb}}$
is summarized in \cref{algo:training}.
\begin{algorithm2e}
\DontPrintSemicolon%
\SetAlgoVlined%
\SetKwFunction{SnapshotGreedy}{SnapshotGreedy}%
\SetKwInOut{Input}{Input}
\SetKwInOut{Output}{Output}
\Fn{\SnapshotGreedy{$Z,\varepsilon_\mathrm{train}$}}{
  \Input{set of elements to approximate $Z$,\\ training tolerance $\varepsilon_\mathrm{train}$}
  \Output{basis of approximation space $\mathfrak{B}$}
  $\mathfrak{B} \leftarrow \emptyset$ \;
  \While{$\max_{z \in Z} \norm{z} > \varepsilon_\mathrm{train}$}{
    $\hat z \leftarrow \argmax_{z \in Z} \norm{z}$\;
    $\hat z \leftarrow \frac{\hat z}{\norm{\hat z}}$\;
    $Z \leftarrow \left\{z - (z,\hat z)\hat z \ | \ z \in Z \right\}$\;
    $\mathfrak{B} \leftarrow \mathfrak{B} \cup \{\hat z\}$ \;
  }
  \Return $\mathfrak{B}$\;
}
\caption{SnapshotGreedy}
\label{algo:snapshot_greedy}
\end{algorithm2e}
\begin{algorithm2e}
\DontPrintSemicolon%
\SetAlgoVlined%
\SetKwFunction{Training}{Training}%
\SetKwFunction{RandomSampling}{RandomSampling}%
\SetKwFunction{SnapshotGreedy}{SnapshotGreedy}%
\SetKwInOut{Input}{Input}
\SetKwInOut{Output}{Output}
\Fn{\Training{$i, M, \varepsilon_\mathrm{train}$}}{
  \Input{space identifier $i$,\\ number of random samples $M$,\\ training tolerance $\varepsilon_\mathrm{train}$}
  \Output{reduced local subspace $\gls{rlfspacewb}$}
  $G \leftarrow$ \RandomSampling{$i$, M}\;
  $Z \leftarrow \emptyset$\;
  \ForEach{${\gls{parameter}} \in \gls{trainingset}$}{
    \Find{$u_f \in \training{\gls{lfspacewb}}$ \st}{
      $
      \qquad a_{{\gls{parameter}}}(u_f,\phi) = f_{\gls{parameter}}(\phi) \qquad \forall \phi \in \training{\gls{lfspacewb}}
      $}
    $Z \leftarrow Z \cup \gls{lfspacemapwb}(u_f)$\;
    \ForEach{$g_c \in G$}{
      \Find{$u_c \in \training{\gls{lfspacewb}}$ \st}{
        $
        \qquad a_{{\gls{parameter}}}(u_c + g_c,\phi) = 0 \qquad \forall \phi \in \training{\gls{lfspacewb}}
        $}
      $Z \leftarrow Z \cup \gls{lfspacemapwb}(u_c)$\;
    }
  }
  $\rbasis_{\gls{rlfspacewb}} \leftarrow$ \SnapshotGreedy{$Z, \varepsilon_\mathrm{train}$}\;
  \Return $\spanset(\rbasis_{\gls{rlfspacewb}})$\;
}
\caption[Training to construct reduced face spaces.]
        {Training to construct reduced face spaces $\gls{rlfspacewb}, \ i \in \gls{upsilon}_1$.}
\label{algo:training}
\end{algorithm2e}
\subsection{Basis Construction for Reduced Cell Spaces
Using Local Greedy}
\label{sec:codim_0_greedy}
For each cell space $\gls{lfspacewb}, \ i \in \gls{upsilon}_0$ we create a reduced space
$\gls{rlfspacewb}$.
The complete algorithm is given in \cref{algo:local_greedy} and explained here.
These spaces should be able to
approximate the solution in the associated part of the space decomposition for any
variation of functions from reduced vertex or face spaces that are coupled with it.
We define the reduced coupling space $\rcoupling{\gls{lfspacewb}}$
and its basis $\rbasis_{\rcoupling{\gls{lfspacewb}}}$.
\begin{equation}
\gls{upsilon}^C_i := \Big\{j \in 1, \dots,  \gls{numcoarseentities}\ | \ i \in \gls{domains}_{\mathcal{E}_j},
\gls{domains}_{\mathcal{E}_j} \neq \{i\} \Big\}
\end{equation}
\begin{equation}
\rcoupling{\gls{rlfspacewb}} := \bigoplus_{j \in \gls{upsilon}^C_i} \widetilde{V}_j^\mathrm{wb}
\qquad \qquad
\rbasis_{\rcoupling{\gls{lfspacewb}}} := \bigcup_{j \in \gls{upsilon}^C_i}
\rbasis_{\widetilde{V}_j^\mathrm{wb}}
\end{equation}
We introduce an extended training set:
$\gls{trainingset} \times \{1, \dots,
N_\rbasis+1\}$, $N_{\rbasis} := \dim(\rcoupling{\gls{rlfspacewb}})$.
Given a pair $({\gls{parameter}}, j) \in \gls{trainingset} \times \{1, \dots, N_\rbasis+1\}$, we define the associated right hand
side as
\begin{equation}
  g_{{\gls{parameter}},j}(\phi) :=
  \begin{cases}
    - a_{\gls{parameter}}(\psi_j, \phi) & \qquad \text{if } j \le N_\rbasis\\
    \dualpair{f_{\gls{parameter}}}{\phi} & \qquad \text{if } j = N_\rbasis+1\,,
  \end{cases}
\end{equation}
where $\psi_j$ denotes the $j$-th basis function of $\rbasis_{\rcoupling{\gls{lfspace}}}$.
We then construct the reduced cell space $\gls{rlfspacewb}$
as the classical reduced basis space
with respect to the following parameterized local problem:
Given a pair $({\gls{parameter}}, j) \in \gls{trainingset} \times \{1, \dots, N_\rbasis+1\}$, find $u_{{\gls{parameter}},j} \in \gls{lfspacewb}$ such that
\begin{equation}
a_{\gls{parameter}}(u_{{\gls{parameter}},j},\phi) = g_{{\gls{parameter}},j}(\phi) \qquad \forall \phi \in \gls{lfspacewb}.
\end{equation}
The corresponding reduced solutions are hence defined as:
Find $\wt u_{{\gls{parameter}},j} \in \gls{rlfspacewb}$ such that:
\begin{equation}
a_{\gls{parameter}}(\wt u_{{\gls{parameter}},j},\phi) = g_{{\gls{parameter}},j}(\phi) \qquad \forall \phi \in \gls{rlfspacewb}
\end{equation}
Both problems have unique solutions in the coercive case.
For the LocalGreedy we use the standard Reduced Basis residual error estimator, i.e.
\begin{equation}
\norm{u_{{\gls{parameter}},j} - \wt u_{{\gls{parameter}},j}}_{\gls{lfspace}} \leq \Delta_{cell}(\wt u_{{\gls{parameter}},j}) := \frac{1}{\gls{coercconst}_{LB}} \norm{ \gls{residual}_{{\gls{parameter}},j}(\wt u_{{\gls{parameter}},j}) }_{\gls{lfspacewb}'}\,,
\end{equation}
with the local residual
\begin{eqnarray}
\gls{residual}_{{\gls{parameter}}, j}: \gls{lfspacewb} &\rightarrow& \gls{lfspacewb}'\\
\varphi &\mapsto& g_{{\gls{parameter}}, j}(\cdot) - a_{\gls{parameter}}(\varphi, \cdot)
\nonumber
\end{eqnarray}
and a lower bound for the coercivity constant $\gls{coercconst}_{LB}$.
The idea of using a local greedy to generate a local space for all
possible boundary values can also be found in \cite{Iapichino2012b,Antonietti2016}.
\begin{algorithm2e}[t]
\DontPrintSemicolon%
\SetAlgoVlined%
\SetKwFunction{LocalGreedy}{LocalGreedy}%
\SetKwInOut{Input}{Input}
\SetKwInOut{Output}{Output}
\Fn{\LocalGreedy{$i, \varepsilon_\mathrm{greedy}$}}{
  \Input{space identifier $i$,\\ greedy tolerance $\varepsilon_\mathrm{greedy}$}
  \Output{reduced local subspace $\gls{rlfspacewb}$}
  $\rbasis_{\gls{rlfspacewb}} \leftarrow \emptyset$ \;
  \While{
    $\max\limits_{\substack{
        \mathllap{{\gls{parameter}}} \in \mathrlap{\gls{trainingset}} \\
        \mathllap{j} \in \mathrlap{\{1, \dots, N_\rbasis +1 \}}}
    }\ \Delta_{cell}(\wt{u}_{{\gls{parameter}},j}) > \varepsilon_\mathrm{greedy}$
  }%
  {
    $\hat {\gls{parameter}},\jhat \leftarrow \argmax\limits_{\substack{
        \mathllap{{\gls{parameter}}} \in \mathrlap{\gls{trainingset}} \\
        \mathllap{j} \in \mathrlap{\{1, \dots, N_\rbasis +1 \}}}}%
    \ \Delta_{cell}(\wt{u}_{{\gls{parameter}},j})$\;
    \Find{$u_{\hat {\gls{parameter}}, \jhat} \in \gls{lfspacewb}$ \st}{
      $
      \qquad a_{\hat {\gls{parameter}}}(u_{\hat {\gls{parameter}}, \jhat},\phi) = g_{\hat {\gls{parameter}}, \jhat}(\phi) \qquad \forall \phi \in \gls{lfspacewb}
      $}
    $u_{\hat {\gls{parameter}}, \jhat} \leftarrow u_{\hat {\gls{parameter}}, \jhat} - \sum\limits_{\mathclap{\phi\in\rbasis_{\gls{rlfspacewb}}}} {( \phi, u_{\hat {\gls{parameter}}, \jhat})_{\gls{fspace}}} \, \phi$\;
    $\rbasis_{\gls{rlfspacewb}} \leftarrow \rbasis_{\gls{rlfspacewb}} \cup \left\{\norm{u_{\hat {\gls{parameter}}, \jhat}}_{\gls{fspace}}^{-1}u_{\hat {\gls{parameter}}, \jhat}\right\}$\;
  }
  \Return $\spanset(\rbasis_{\gls{rlfspace}})$\;
}
\caption[LocalGreedy to construct local cell spaces.]
        {LocalGreedy to construct local cell spaces $\gls{rlfspacewb}, \ i \in \gls{upsilon}_0$.}
\label{algo:local_greedy}
\end{algorithm2e}

\section{Enrichment Procedure}
\label{sec:enrichment}
The first \gls{arbilomod} solution is obtained using the initial reduced local subspaces generated
using the local training and greedy procedures described in \cref{sec:training_and_greedy}.
The error of this solution is bounded by an a posteriori error estimator
\gls{locerrest_g}. The definition of this a posteriori error estimator
is postponed to \cref{chap:a posteriori}.
If this solution is not good enough according
to the a posteriori error estimator, the solution is improved by
enriching the reduced local subspaces and then solving the global reduced
problem again.
The full procedure is given in \cref{algo:online_enrichment_arbilomod}
and described in the following.

For the enrichment, we use the overlapping local subspaces \gls{lfspacepou} introduced
in \eqref{sec:poudecomposition}, which are also used for the
a posteriori error estimator.
Online enrichment on the wirebasket space decomposition
itself led to stagnation in numerical experiments.
Local problems are solved in the overlapping spaces.
The original bilinear form is used, but as a right hand side
the residual of the last reduced solution is employed.
The local spaces and the parameter values for which the enrichment
is performed are selected in a D\"orfler-like \cite{Doerfler1996}
algorithm.
The thus obtained local solutions $u_l$ do not fit into our
space decomposition, as they lie in one of the overlapping spaces \gls{lfspacepou}, not
in one of the local subspaces used for the basis construction \gls{lfspacewb}. Therefore,
the $u_l$ are decomposed using the projection operators $\gls{lfspacemapwb}$
defined in \cref{def:local_projection_operators}.
In the setting of our numerical example (\cref{sec:numerical_example}), this decomposition yields at most 9 parts
(one codim-2 part, four codim-1 parts and four codim-0 parts).
Of these parts, the one worst approximated by the existing reduced
local subspace is selected for enrichment. ``Worst approximated'' is here defined
as having the largest part orthogonal to the existing reduced local subspace.
We denote the orthogonal projection on \gls{rlfspacewb} by 
$\gls{orthogonalp}_{\gls{rlfspacewb}}$,
so $(1 - \gls{orthogonalp}_{\gls{rlfspacewb}}) \gls{lfspacemapwb}(u_l)$
is the part of $\gls{lfspacemapwb}(u_l)$ orthogonal to $\gls{rlfspacewb}$.

To avoid communication, cell spaces $\gls{rlfspacewb}, i \in \gls{upsilon}_0$
are not enriched at this point. Such an enrichment would require the
communication of the added basis vector, which might be large. Instead,
only the other spaces are enriched, and the cell spaces associated with $\gls{upsilon}_0$
are regenerated using the greedy procedure from \cref{sec:codim_0_greedy}.
For the other spaces, a strong compression of the basis vectors is possible
(cf.~\cref{sec:runtime_and_communication}).

This selection of the local spaces can lead to one reduced local space
being enriched several times in one iteration. Numerical experiments have
shown that this leads to poorly conditioned systems, as
the enrichment might introduce the same feature into a local basis twice.
To prevent this, the enrichment algorithm enriches each reduced
local subspace at most once per
iteration.
\begin{algorithm2e}
\DontPrintSemicolon%
\SetAlgoVlined%
\SetKwFunction{OnlineEnrichment}{OnlineEnrichment}%
\SetKwInOut{Input}{Input}
\SetKwInOut{Output}{Output}
\Fn{\OnlineEnrichment{$\gls{enrichmentfraction}, \mathrm{tol}$}}{
  \Input{enrichment fraction $\gls{enrichmentfraction}$,\\ target error $\mathrm{tol}$}
  \While{
    $\max\limits_{{\gls{parameter}} \in \gls{trainingset}}
    \ \Delta(\wt{u}_{\gls{parameter}}) > \mathrm{tol}$
  }%
  {
    $E \leftarrow \emptyset$\;
    \While{
      $\left(\sum\limits_{({\gls{parameter}},i) \in E} \norm{R_{\gls{parameter}}(\gls{rsol}_{\gls{parameter}})}_{(\gls{lfspacepou})'} \right)
      \Big/ \left(\sum\limits_{({\gls{parameter}},i) \in (\gls{trainingset} \times \{1, \dots, \gls{numdomainsol}\})} \norm{R_{\gls{parameter}}(\gls{rsol}_{\gls{parameter}})}_{(\gls{lfspacepou})'} \right)< \gls{enrichmentfraction}$
    }%
    {
    $\hat {\gls{parameter}},\hat i \leftarrow
    \argmax\limits_
    {({\gls{parameter}}, i) \in \left(\gls{trainingset} \times \{1, \dots, \gls{numdomainsol}\}\right) \setminus E}
    \norm{R_{\gls{parameter}}(\gls{rsol}_{\gls{parameter}})}_{(\gls{lfspacepou})'}$\;
    $E \leftarrow E \cup (\hat {\gls{parameter}}, \hat i)$\;
    }
    \tcc{$S$ is used for double enrichment protection.}
    $S \leftarrow \emptyset$\;
    \For{$({\gls{parameter}},i) \in E$}
    {
    \Find{$u_l \in \gls{lfspacepou}$ \st}{
      $a_{{\gls{parameter}}}(u_l, \varphi) =
      \dualpair{\gls{residual}_{\gls{parameter}}(\gls{rsol}_{{\gls{parameter}}})}{\varphi} \qquad \forall
      \varphi \in \gls{lfspacepou}$}
    $\check i \leftarrow \argmax\limits_{i \in \{1, \dots, \gls{numcoarseentities}\} \setminus \gls{upsilon}_0}\norm{ (1 - \gls{orthogonalp}_{\gls{rlfspacewb}}) \gls{lfspacemapwb}(u_l) } _{\gls{lfspace}}$\;
    \If{$\check i \notin S$}{
      $\widetilde V^\mathrm{wb}_{\check i} \leftarrow \widetilde V^\mathrm{wb}_{\check i} \oplus \spanset \left((1 - \gls{orthogonalp}_{\gls{rlfspacewb}}) \gls{lfspacemapwb}(u_l)\right)$\;
      $S \leftarrow S \cup \check i$\;
    }
    }
    run \texttt{LocalGreedys}\;
    recalculate reduced solutions\;
  }
}
\caption[Online Enrichment in ArbiLoMod.]
{Online Enrichment in \gls{arbilomod}.}
\label{algo:online_enrichment_arbilomod}
\end{algorithm2e}

\section{Runtime and Communication}
\label{sec:runtime_and_communication}
A major design goal of \gls{arbilomod} is communication avoidance and
scalability in parallel environments. In the following, we
highlight the possibilities offered by \gls{arbilomod} to
reduce communication in a parallel setup.

\tikzexternaldisable
\begin{figure}
\footnotesize
\centering
\begin{subfigure}[b]{0.3\textwidth}
\centering
\begin{minipage}{0.8\linewidth}
  \begin{tikzpicture}[scale=0.9]
    \node at (0.5+0.125,2.5) [color=Xred] {\huge 1};
    \node at (1.5+0.125,2.5) [color=Xred] {\huge 2};
    \node at (2.5+0.125,2.5) [color=Xred] {\huge 3};
    \node at (0.5+0.125,1.5) [color=Xred] {\huge 4};
    \node at (1.5+0.125,1.5) [color=Xred] {\huge 5};
    \node at (2.5+0.125,1.5) [color=Xred] {\huge 6};
    \node at (0.5+0.125,0.5) [color=Xred] {\huge 7};
    \node at (1.5+0.125,0.5) [color=Xred] {\huge 8};
    \node at (2.5+0.125,0.5) [color=Xred] {\huge 9};
    \draw [step=2.5mm,color=gray,very thin] (-0.2,-0.2) grid (3.2,3.2);
    \draw [step=1cm,black,line width=0.7mm] (-0.2,-0.2) grid (3.2,3.2);
  \end{tikzpicture}
\end{minipage}
\caption{Example domain numbering.}
\end{subfigure}\hfill
\begin{subfigure}[b]{0.65\textwidth}
\centering
\def\myscale{0.6}
\begin{tabular}{ccc}
\multicolumn{3}{l}{Create basis functions for vertex, face and cell spaces:}\\
\begin{tikzpicture}[scale=\myscale]
\draw [step=2.5mm,color=gray,very thin] (-0.2,-0.2) grid (3.2,3.2);
\draw [step=1cm,black,line width=0.7mm] (-0.2,-0.2) grid (3.2,3.2);
\draw [color=Xred,line width=1.1mm] (0,1) rectangle (2,3);
\draw [color=blue,line width=1.4mm] (1,2) circle (0.2);
\end{tikzpicture}&
\begin{tikzpicture}[scale=\myscale]
\draw [step=2.5mm,color=gray,very thin] (-0.2,-0.2) grid (3.2,3.2);
\draw [step=1cm,black,line width=0.7mm] (-0.2,-0.2) grid (3.2,3.2);
\draw [color=Xred,line width=1.1mm] (0,0) rectangle (2,3);
\draw [color=blue,line width=1.4mm] (0.8,0.8) rectangle (1.2,2.2);
\end{tikzpicture}&
\begin{tikzpicture}[scale=\myscale]
\draw [step=2.5mm,color=gray,very thin] (-0.2,-0.2) grid (3.2,3.2);
\draw [step=1cm,black,line width=0.7mm] (-0.2,-0.2) grid (3.2,3.2);
\draw [color=Xred,line width=1.1mm] (1,1) rectangle (2,2);
\draw [color=blue,line width=1.4mm] (1.2,1.2) rectangle (1.8,1.8);
\end{tikzpicture}\\[.5ex]
$\gls{lfspacewb}, i \in \gls{upsilon}_2$&
$\gls{lfspacewb}, i \in \gls{upsilon}_1$&
$\gls{lfspacewb}, i \in \gls{upsilon}_0$\\
\end{tabular}
\caption{To generate reduced space associated with blue marked entity, only geometry information in red marked area is needed.}
\end{subfigure}
\vspace*{-2ex}
\caption[Data dependencies in ArbiLoMod]
{
Data dependencies in \gls{arbilomod}:
Before online enrichment, it is possible to compute all reduced basis
function having support on a domain ($\omega_5^\mathrm{nol}$ here) using only
local information about the domain and its surrounding domains.
}
\label{fig:no_communication}
\end{figure}
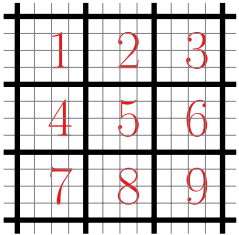
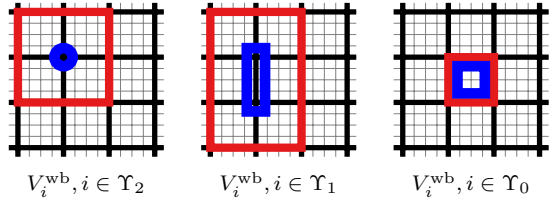
\tikzexternalenable

Similar to overlapping Domain Decomposition methods we require that
not only the local domain, but also an overlap region is available
locally. For a subdomain $\gls{subdomainnol}$ the overlap region is the domain
itself and all adjacent domains, as depicted in \cref{fig:no_communication},
i.e.\ all subdomains in the neighborhood $\gls{neigh}_i$.
As the overlap region includes the
support of all training spaces, one can compute all initial reduced local subspaces
with support in $\gls{subdomainnol}$ without further communication.
This work can be distributed on many nodes.
Afterwards, only reduced representations of the operator have to be communicated.
Using the operator decomposition
$a^b(u,v) = \sum_{i=1}^{\gls{numdomainsnol}} a_{\gls{subdomainnol}}^b(u,v)$, a global, reduced
operator is collected using an all-to-one communication of reduced matrices.
The global reduced problem is then solved on a single node.
It is assumed that the global, reduced system is sufficiently small.

If the accuracy is not sufficient, online enrichment is
performed. This step requires additional communication; first for the
evaluation of the error estimator and second to communicate new basis
vectors of reduced face spaces $\gls{rlfspacewb}, \ i \in
\gls{upsilon}_1$.
Note that it is sufficient to communicate the local projection
$\gls{lfspacemapbasic}(\psi)$ and reconstruct the actual basis
function as its extension, so that we save communication costs
proportional to the volume to surface ratio.

The evaluation of the localized error estimator requires the dual
norms of the residual in the localized spaces
$\norm{\gls{residual}_{\gls{parameter}}(\gls{rsol}_{\gls{parameter}})}_{\gls{lfspacepou}'}$. Using an offline/online
splitting, the evaluation of the error estimator can be
evaluated for the full system using only reduced quantities. The
computation of the reduced operators is performed in parallel, similar
to the basis construction in the first step. The actual evaluation of
the error estimator can be performed on a single node and is
independent of the number of degrees of freedom of the high fidelity
model.

\label{sec:H_explanation}
An important parameter for \gls{arbilomod}'s runtime is the domain size $\gls{H}$.
The domain size affects the size of the local problems, 
the amount of parallelism in the algorithm,
and the size of the reduced global problem.
An $\gls{H}$ too large leads to large local problems, while an $\gls{H}$ too small
leads to a large reduced global problem (see also the numerical
example in \cref{sec:num_results} and especially
the results in \cref{tab:H_study}).
$\gls{H}$ has to be chosen
to balance these two effects.

\section{Numerical Experiments}
\subsection{Numerical Example 1: Thermal Channels}
\label{sec:numerical_example}
\label{sec:numerical_experiments}
We first demonstrate the \gls{arbilomod} on \exb{}
as defined in \cref{sec:exb}.
The numerical experiments were performed using \gls{pymor} \cite{Milk2016}.
If not specified otherwise, we use
a mesh size of $1/\gls{h} = 200$,
a training tolerance of $\varepsilon_\mathrm{train} = 10^{-4}$,
a number of random samplings of $M = 60$,
a greedy tolerance of $\varepsilon_\mathrm{greedy} = 10^{-3}$,
a convergence criterion of $\norm{\gls{residual}_{\gls{parameter}}(\gls{rsol}_{\gls{parameter}})}_{\gls{fspaceh}'} < 10^{-2}$,
an enrichment fraction of $\gls{enrichmentfraction} = 0.5$,
a training set of size $|\gls{trainingset}| = 6$,
and the parameter for extension calculation is $\overline{\gls{parameter}} = 10^5$.
\label{sec:num_results}
\begin{figure}
\centering
\includegraphics[width=0.6\textwidth]{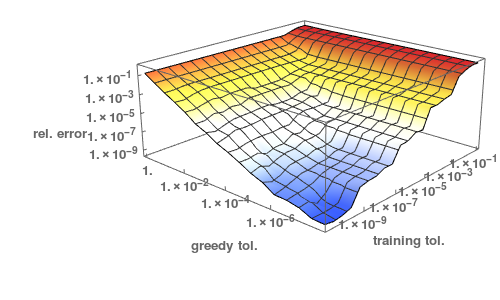}
\caption[Maximum relative $H^1$-error on training set $\gls{trainingset}$ for \exb.]
{Maximum relative $H^1$-error on training set $\gls{trainingset}$ in dependence of tolerances in codim 1 training
and codim 0 greedy for \exb. Online enrichment disabled. 
(reproduction: \cref{repro:fig:tolerances})}
\label{fig:tolerances}
\end{figure}
The initial reduced space is created using the local trainings and
greedy algorithms. In both the trainings and the greedy algorithms
a tolerance parameter steers the quality of the obtained reduced space:
In the trainings,
$\varepsilon_\mathrm{train}$
is the stopping criterion for the SnapshotGreedy
(\cref{algo:snapshot_greedy}) and the local greedys
stop when the local error estimator stays below the prescribed tolerance
$\varepsilon_\mathrm{greedy}$.
The resulting reduction errors in
dependence on the two tolerances are depicted in \cref{fig:tolerances}.
\begin{table}
\footnotesize
\centering
\begin{tabu}{c|[1pt]c|[lightgray]c|c|[lightgray]c|c|[lightgray]c|[1pt]c|[lightgray]c}
\multirow{2}{*}{geometry}&
\multicolumn{6}{c|[1pt]}{\em with training} &
\multicolumn{2}{c}{\em without training}\\
&\multicolumn{2}{c|}{trainings} &
\multicolumn{2}{c|}{greedys} &
\multicolumn{2}{c|[1pt]}{iterations} &
\multicolumn{2}{c}{iterations}
\\
\hline
&\multicolumn{2}{c|}{reuse:} &
\multicolumn{2}{c|}{reuse:} &
\multicolumn{2}{c|[1pt]}{reuse:} &
\multicolumn{2}{c}{reuse:}
\\
 & no & yes & no & yes & no & yes & no & yes\\
\tabucline[lightgray]{2-}
1 & 112 & 112 (-0 \%)  & 64 & 64 (-0 \%)  & 24 & 24 (-0 \%)
& 46 & 46 (-0 \%) \\
2 & 112 & 5 (-96 \%)  & 64 & 8 (-88 \%)  & 24 & 13 (-46 \%)
& 48 & 28 (-42 \%) \\
3 & 112 & 5 (-96 \%)  & 64 & 8 (-88 \%)  & 20 & 14 (-30 \%)
& 42 & 27 (-36 \%) \\
4 & 112 & 3 (-97 \%)  & 64 & 6 (-91 \%)  & 25 & 10 (-60 \%)
& 54 & 23 (-57 \%) \\
5 & 112 & 5 (-96 \%)  & 64 & 8 (-88 \%)  & 25 & 12 (-52 \%)
& 52 & 27 (-48 \%) \\
\end{tabu}
\caption[Number of iterations of online enrichment for \exb]
{Number of iterations of online enrichment for \exb:
(a) With and without codim 1 training.
(b) With and without reuse of basis
functions of previous simulations.
Convergence criterion: $\norm{\gls{residual}_{\gls{parameter}}(\gls{rsol}_{\gls{parameter}})}_{\gls{fspaceh}'} < 10^{-4}$,
greedy tolerance: $\varepsilon_\mathrm{greedy} = 10^{-5}$
See also \cref{fig:basisreusewithtraining,fig:basisreuse}.
(reproduction: \cref{repro:basisreuse})
}
\label{tab:basisreusecompare}
\end{table}

\begin{figure}
\centering
\input{figure_basisreuse_with_training}
\caption[Relative error with training during iteration for \exb.]
{Relative error over iteration with and without basis reuse for
after geometry change for \exb. With codim 1 training.
Convergence criterion: $\norm{\gls{residual}_{\gls{parameter}}(\gls{rsol}_{\gls{parameter}})}_{\gls{fspaceh}'} < 10^{-4}$,
greedy tolerance: $\varepsilon_\mathrm{greedy} = 10^{-5}$.
See also \cref{tab:basisreusecompare}.
(reproduction: \cref{repro:basisreuse})
}
\label{fig:basisreusewithtraining}
\end{figure}

\begin{figure}
\centering
\input{figure_basisreuse}
\caption[Relative error without training during iteration for \exb.]
{Relative $H^1$ error over iterations with and without basis reuse for 
after geometry change for \exb. Without codim 1 training.
Convergence criterion: $\norm{R_{\gls{parameter}}(\gls{rsol}_{\gls{parameter}})}_{V_h'} < 10^{-4}$,
greedy tolerance: $\varepsilon_\mathrm{greedy} = 10^{-5}$.
See also \cref{tab:basisreusecompare}.
(reproduction: \cref{repro:basisreuse})
}
\label{fig:basisreuse}
\end{figure}

\begin{figure}
\centering
\input{figure_estimatorperformance}
\caption[
Error estimators $\gls{rerrest}$, $\gls{rerrest4}$ during iteration for \exb{}.]
{
Error estimators $\gls{rerrest}$ and
$\gls{rerrest4}$ over iterations, compared to the
relative error for \exb{}. Plotted for $\gls{coercconst}_{\gls{parameter}} = \gls{decompositionboundvr} = 1$.
Simulation performed with
a convergence criterion of $\norm{R_{\gls{parameter}}(\gls{rsol}_{\gls{parameter}})}_{V_h'} < 10^{-6}$,
a training tolerance of $\varepsilon_\mathrm{train} = 10^{-5}$,
and a greedy tolerance of $\varepsilon_\mathrm{greedy} = 10^{-7}$.
(reproduction: \cref{repro:fig:estimatorperformance})
}
\label{fig:estimatorperformance}
\end{figure}

If the resulting error is too big, it can be further reduced using
iterations of online enrichment as depicted in
\cref{fig:basisreusewithtraining}. Results suggest an
very rapid decay of the error with online enrichment.
The benefits of \gls{arbilomod} can be seen in
\cref{fig:basisreusewithtraining} and 
\cref{tab:basisreusecompare}: after
the localized geometry changes, most of the
work required in the initial basis creation does not need to be
repeated and the online enrichments converge faster for
subsequent simulations, leading to less iterations.
The online enrichment presented here
converges even when started on empty bases, as depicted in
\cref{fig:basisreuse}.
It does not rely on properties
of the reduced local subspaces created by trainings and greedys.
The performance of the localized a posteriori error estimator
(defined below in \cref{sec:relative error estimators})
$\gls{rerrest4}$ can be seen in \cref{fig:estimatorperformance}.
Comparison of the localized estimator 
$\gls{rerrest4}$
with the global estimator
$\gls{rerrest}$
(also defined below in \cref{sec:relative error estimators})
shows that, for the example considered here,
the localization does not add a significant factor beyond the factor
$\gls{decompositionboundvr}$, which is
neglected in \cref{fig:estimatorperformance}.
\newcommand{\rot}[1]{\multicolumn{1}{l}{\adjustbox{angle=24,lap=\width-1em}{#1}}  }
\begin{table}
\begin{center}
\footnotesize
\begin{tabular}{r|r|r|r|r|r|r|r|r|r}
\multicolumn{1}{c}{$1/\gls{h}$} & \rot{global \gls{dof}s} & \rot{global solve time [s]} &
\rot{\# \gls{dof}s, codim 1 training space} &
\rot{avg. time per codim 1 training [s]} &
\rot{\# \gls{dof}s, codim 0 space} &
\rot{max time per codim 0 greedy [s]} & \rot{\# \gls{dof}s, reduced problem} & \rot{solve time, reduced [ms]} & \rot{max error [\permil]}
\\
\hline
200 & 80,401 & 0.656 & 7,626 & 1.02 & 1,201 & 4.9 & 1,178 & 21.8 & 1.316\\
400 & 320,801 & 4.87 & 30,251 & 5.14 & 4,901 & 7.04 & 1,151 & 22.4 & 1.433\\
600 & 721,201 & 23.6 & 67,876 & 14 & 11,101 & 10.7 & 1,116 & 19.1 & 2.035\\
800 & 1,281,601 & 41.8 & 120,501 & 29.5 & 19,801 & 17.8 & 1,101 & 17.1 & 2.735\\
1000 & 2,002,001 & 86.4 & 188,126 & 51.3 & 31,001 & 24.2 & 1,089 & 18.8 & 1.351\\
1200 & 2,882,401 & 230 & 270,751 & 81.2 & 44,701 & 36.6 & 1,082 & 18.6 & 4.462\\
1400 & 3,922,801 & 230 & 368,376 & 120 & 60,901 & 51.7 & 1,073 & 18.2 & 2.379
\end{tabular}
\end{center}
\caption[Runtimes for selected parts of ArbiLoMod for \exb{}.]
{Runtimes for selected parts of ArbiLoMod without
online enrichment for \exb{}.
``max error'' denotes $\max_{{\gls{parameter}} \in \gls{trainingset}} {\norm{u_{\gls{parameter}} -\gls{rsol}_{\gls{parameter}}}_{\gls{fspace}}}/{\norm{u_{\gls{parameter}}}_{\gls{fspace}}}$.
Runtimes measured using
a pure Python implementation, using SciPy solvers (SuperLU sequential).
Note that the global solve time is for one parameter value
while training and greedy produce spaces valid for all parameter
values in the training set $\gls{trainingset}$.
(reproduction: \cref{repro:tab:runtimes})}
\label{tab:runtimes}
\end{table}

\begin{table}
\begin{tabular}{r|r|r|r|r|r|r|r|r|r}
\multicolumn{1}{c}{$1/\gls{H}$} &
\rot{\# \gls{dof}s, codim 1 training space} &
\rot{mean training time [s]} &
\rot{max training time [s]} &
\rot{\# \gls{dof}s, codim 0 space} &
\rot{mean greedy time [s]} &
\rot{max greedy time [s]} &
\rot{\# \gls{dof}s, reduced problem} &
\rot{solve time, reduced [ms]} &
\rot{max error [\permil]}
\\
\hline
4 & 30,251 & 10.9 & 14.1 & 4,901 & 3.07 & 6.6 & 403 & 3.85 & 0.362\\
5 & 19,401 & 6.88 & 8.74 & 3,121 & 2.32 & 11.7 & 517 & 4.38 & 0.435\\
8 & 7,626 & 2.63 & 3.55 & 1,201 & 1.35 & 5.69 & 1,178 & 14.9 & 1.32\\
10 & 4,901 & 1.68 & 1.94 & 761 & 0.655 & 3.98 & 1,451 & 19.8 & 0.220\\
20 & 1,251 & 0.483 & 0.661 & 181 & 0.305 & 3.18 & 5,025 & 94.4 & 0.0804
\end{tabular}
\caption[Influence of domain size $\gls{H}$ in ArbiLoMod for \exb{}.]
{
Influence of domain size $\gls{H}$ in \gls{arbilomod} for \exb{}. Fine mesh resolution: $1/\gls{h} = 200$.
``max error'' is $\max_{{\gls{parameter}} \in \gls{trainingset}} {\norm{u_{\gls{parameter}} -\gls{rsol}_{\gls{parameter}}}_{\gls{fspace}}}/{\norm{u_{\gls{parameter}}}_{\gls{fspace}}}$.
Smaller domains lead to more parallelism and smaller local problems, but
also to more global \gls{dof}s and a worse constant in the a-posteriori error estimator.
Measured using
a pure Python implementation, using SciPy solvers (SuperLU sequential).
(reproduction: \cref{repro:tab:H_study})
}
\label{tab:H_study}
\end{table}

Even though our implementation was not tuned for performance and is not
parallel, we present some timing measurements in \cref{tab:runtimes} and \cref{tab:H_study}.
\cref{tab:runtimes} shows that, already in our unoptimized implementation, trainings and greedys
have a shorter runtime than a single global solve for problems of sufficient 
size. Taking into account that trainings and greedys create a solution
space valid in the whole parameter space, this data is a strong hint
that \gls{arbilomod} can realize its potential for acceleration 
for large 
problems, in an optimized implementation and a parallel computing environment.
Especially when the solution has to be calculated at multiple parameter values.
\cref{tab:H_study} shows the effect of choosing the domain size $\gls{H}$: 
Large domains lead to large local problems, while small domains lead to a large
reduced global problem (see also \cref{sec:H_explanation}).
\begin{figure}
\centering
\includegraphics[width=0.35\textwidth]{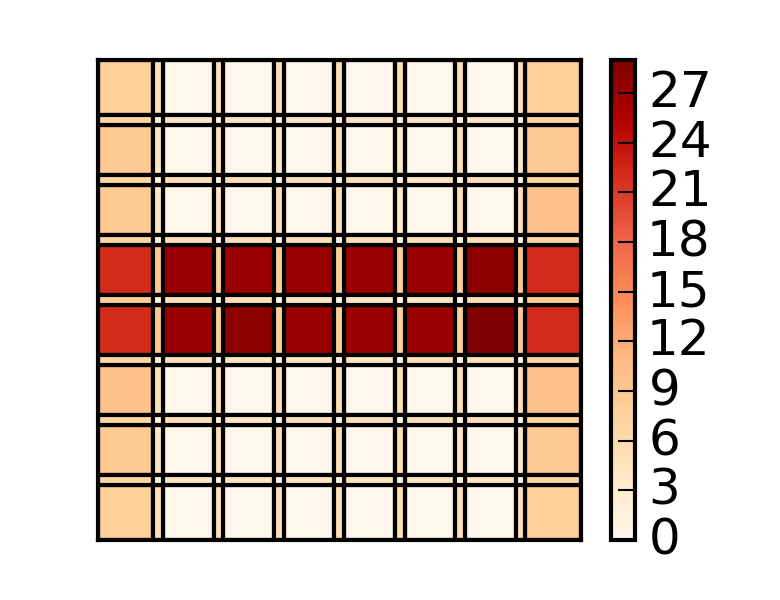}
\caption[Distribution of local basis sizes after initial training for \exb{}.]
{Distribution of local basis sizes after initial training for \exb{}.
Relative reduction error at this configuration: $1.3 \cdot 10^{-3}$.
(reproduction: \cref{repro:fig:basis_sizes})
}
\label{fig:basis_sizes}
\end{figure}

\subsection{Numerical Example 2: 2D Maxwell's}
\label{sec:arbilomod experiments for maxwell}
The application of \gls{arbilomod} to inf-sup
stable problems has several issues.
The local problems in trainings, enrichments and local
greedys do not necessarily have a unique solution
and also the global reduced problem might be unstable.
Nevertheless we apply \gls{arbilomod} without
local greedy algorithms, without online enrichment and without
a posteriori error estimator to the inf-sup stable
problem \exc{} to numerically investigate
its behavior in this case.
\subsubsection{Global Properties of Example}
Before analyzing the behavior of the localized model reduction, we discuss some
properties of the full model. For its stability, its continuity constant $\gls{contconst}$ and reduced inf-sup constants \gls{rinfsupconst} are the primary concern. They guarantee existence and uniqueness of the solution and their quotient enters the best-approximation inequality
\cref{thm:bestapproxinfsup}.
Due to the construction of the norm (see \cref{eq:maxwell energy norm}), the continuity constant cannot be larger than one, and numerics indicate that it is usually one (\cref{fig:infsupcont}). The inf-sup constant approaches zero when the frequency goes to zero. This is the well known low frequency instability of this formulation. There are remedies to this problem, but they are not considered here. The order of magnitude of the inf-sup constant is around $10^{-2}$: Due to the Robin boundaries, the problem is stable. With Dirichlet boundaries only, the inf-sup constant would drop to zero at several frequencies.
There are two drops in the inf-sup constant at ca.~770 MHz and 810 MHz. These correspond to resonances in the structure which arise when half a wavelength is the width of a channel ($\lambda / 2 \approx 1/5$).
\begin{figure}
\begin{tikzpicture}
\begin{axis}[
    xlabel=frequency / Hz,
    ylabel=inf-sup constant,
    width=0.40\textwidth,
    height=0.3\textwidth,
    no markers,
  ]
  \addplot table{infsups0.txt};
  \addplot table{infsups1.txt};
\end{axis}
\end{tikzpicture}
\begin{tikzpicture}
\begin{semilogyaxis}[
    xlabel=frequency / Hz,
    ylabel=inf-sup and continuity,
    width=0.40\textwidth,
    height=0.3\textwidth,
    no markers,
    legend pos=outer north east,
  ]
  \addplot table{infsups0.txt};
  \addplot table{infsups1.txt};
  \addplot table{continuities0.txt};
  \legend{inf-sup geo 1, inf-sup geo 2, continuity};
\end{semilogyaxis}
\end{tikzpicture}
\caption[inf-sup and continuity constant of bilinear form for \exc{}.]
{inf-sup and continuity constant of bilinear form for \exc{}. Linear and logarithmic.
(reproduction: \cref{repro:fig:infsupcont})
}
\label{fig:infsupcont}
\end{figure}
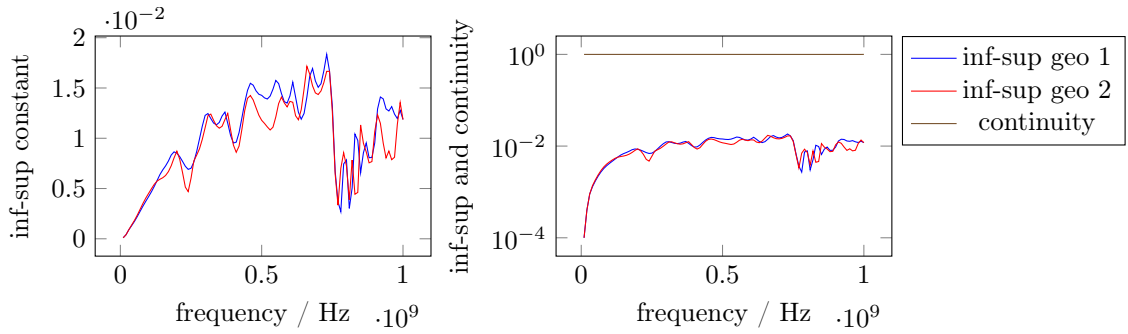
\begin{figure}
\begin{center}
\begin{tikzpicture}
\begin{semilogyaxis}[
    xlabel=basis size,
    ylabel=max rel. projection error,
    ymax=1e2,
    width=0.45\textwidth,
    height=0.3\textwidth,
    no markers,
    legend pos=outer north east,
  ]
  \addplot table{n_width0.txt};
  \addplot table{n_width1.txt};
  \legend{geometry 1, geometry 2};
\end{semilogyaxis}
\end{tikzpicture}
\end{center}
\caption[Upper bound for Kolmogorov n-width for \exc{}]
{
Error when approximating the solution set for all $\gls{parameter} \in \gls{trainingset}$ with an n-dimensional basis obtained by greedy approximation of this set. This is an upper bound for the Kolmogorov n-width. (\exc{})
(reproduction: \cref{repro:fig:global_n_width})
}
\label{fig:global_n_width}
\end{figure}

The most important question for the applicability of any reduced basis method is: is the system reducible at all,
i.e. can the solution manifold be approximated with a low dimensional solution space? The best possible answer to this
question is the Kolmogorov n-width. We measured the approximation error when approximating the solution manifold with a
basis generated by a greedy algorithm. The approximation is done by orthogonal projection onto the basis. This error is
an upper bound to the Kolmogorov n-width. Already with a basis size of 38, a relative error of $10^{-4}$ can be achieved,
see \cref{fig:global_n_width}. So this problem is well suited for reduced basis methods.

\subsubsection{Properties of Localized Spaces}
\begin{figure}
\begin{center}
\includegraphics[width=0.35\textwidth]{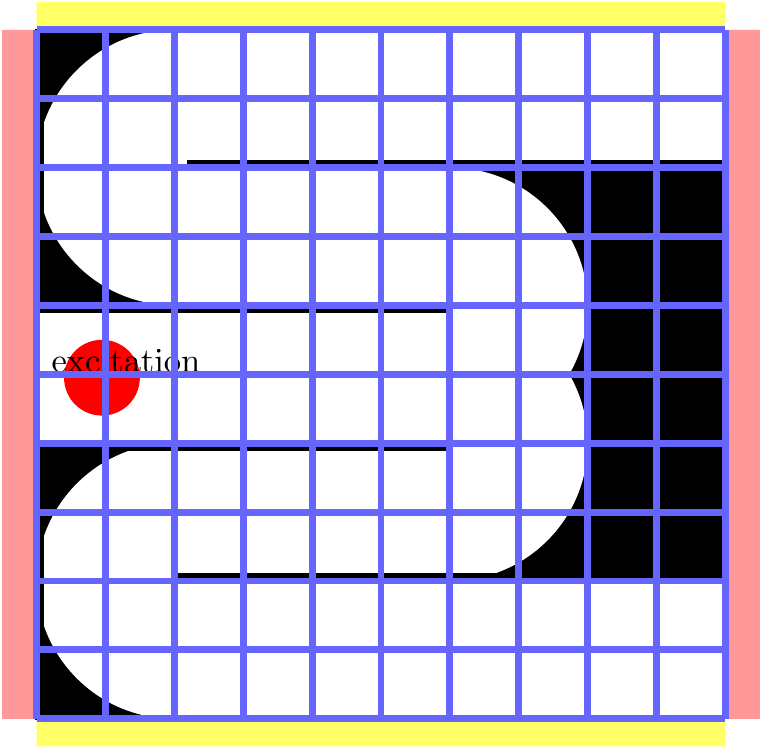}
\end{center}
\caption[Domain decomposition used for \exc{}]
{Domain decomposition used for \exc{}.}
\label{fig:exc dd}
\end{figure}

\begin{figure}
\begin{center}
\begin{tikzpicture}
\begin{semilogyaxis}[
    xlabel=basis size,
    ylabel=maximum relative error,
    ymin=1e-10,
    ymax=1e1,
    width=0.45\textwidth,
    height=0.4\textwidth,
    xmin=0,
    xmax=4000,
    no markers
  ]
  \addplot table[x index=0, y index=4]{localized0.txt};
  \addplot table[x index=0, y index=4]{localized1.txt};
  \legend{geometry 1, geometry 2};
\end{semilogyaxis}
\end{tikzpicture}
\end{center}
\caption[Maximum error when solving \exc{} with a localized basis.]
{
Maximum error when solving \exc{} with a localized basis, generated by global solves.
(reproduction: \cref{repro:fig:localized_globalsolve})
}
\label{fig:localized_globalsolve}
\end{figure}
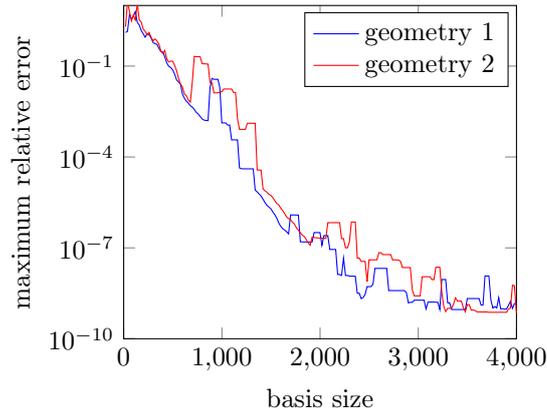
The next question is: how much do we lose by localization?
Using basis vectors with limited support, one needs a larger total
number of basis functions. To quantify this, we compare the errors with
global approximation from the previous section with the error obtained
when solving with a localized basis, using the best localized basis
we can generate. We use a 10 x 10 domain decomposition
(see \cref{fig:exc dd}) and
the space decomposition introduced in
 \cref{sec:wirebasket space decomposition}. 
To construct the best possible basis, we solve the full problem
for all parameters in the training set. For each local subspace,
we apply the corresponding projection operator
\gls{lfspacemapwb} to all global solutions and
subsequently generate a basis for these local parts of global
solutions using a snapshot greedy. The error when solving in the
resulting reduced space is depicted in
\cref{fig:localized_globalsolve}, right. Much more
basis vectors are needed, compared to the global reduced basis approach.
However the reduction in comparison to the full model (60200 \gls{dof}s)
is still significant and in contrast to standard reduced basis methods,
the reduced system matrix is not dense but block-sparse.
For a relative error of $10^{-4}$, 1080 basis vectors are necessary.

\begin{figure}
\begin{tikzpicture}
\begin{semilogyaxis}[
    xlabel=basis size,
    ylabel=maximum relative error,
    width=0.42\textwidth,
    height=0.38\textwidth,
    xmin=0,
    xmax=2500
  ]
  \addplot+[
    no markers
  ]table[x index=0, y index=4]{localized0.txt};
\end{semilogyaxis}
\end{tikzpicture}
\begin{tikzpicture}
\begin{semilogyaxis}[
    xlabel=basis size,
    ylabel=\gls{rinfsupconst},
    width=0.42\textwidth,
    height=0.38\textwidth,
    xmin=0,
    xmax=2500,
    no markers,
    legend pos=outer north east,
  ]
  \addplot table[
    x expr=\coordindex * 10 + 1,
    y index=16,
  ]{infsup_constant_0.txt};
  \addplot table[
    x expr=\coordindex * 10 + 1,
    y index=33,
  ]{infsup_constant_0.txt};
  \addplot table[
    x expr=\coordindex * 10 + 1,
    y index=69,
  ]{infsup_constant_0.txt};
  \addplot table[
    x expr=\coordindex * 10 + 1,
    y index=79,
  ]{infsup_constant_0.txt};
  \addplot table[
    x expr=\coordindex * 10 + 1,
    y index=83,
  ]{infsup_constant_0.txt};
  \addplot table[
    x expr=\coordindex * 10 + 1,
    y index=89,
  ]{infsup_constant_0.txt};
  \legend{170 MHz, 340 MHz, 700 MHz, 800 MHz, 840 MHz, 900 MHz}
\end{semilogyaxis}
\end{tikzpicture}
\caption[inf-sup constant of \exc{} vs. basis size.]
{Comparison of maximum error over all frequencies with inf-sup constant of reduced system at selected frequencies for geometry 1. Basis generated by global solves. Increased error and reduced inf-sup constant at basis size of ca.~900 (\exc{}).
(reproduction: \cref{repro:fig:infsup_drop})
}
\label{fig:infsup_drop}
\end{figure}
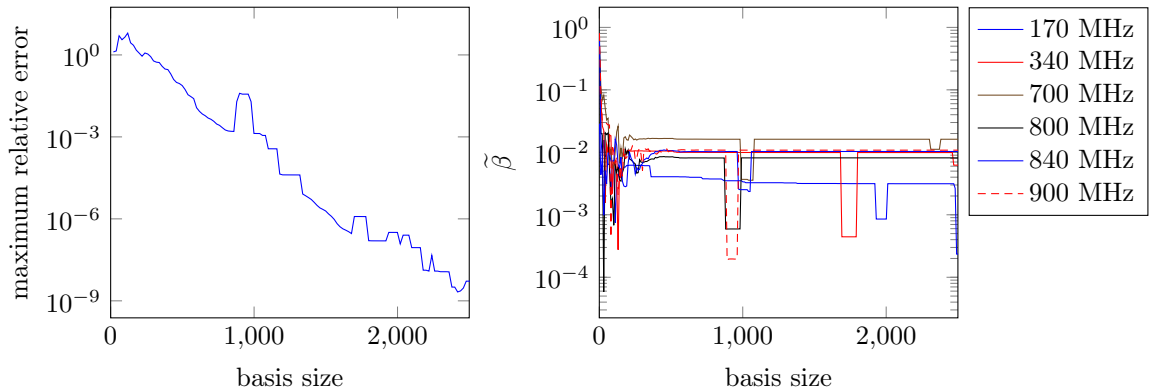

In \cref{fig:localized_globalsolve} the error is observed to jump occasionally.
This is due to the instability of a Galerkin projection of an inf-sup stable problem.
While the inf-sup constant of the reduced system is observed to be the same as the inf-sup
constant of the full system most of the time, sometimes it drops. This is depicted in \cref{fig:infsup_drop}.
For a stable reduction, a different test space is necessary. 
However, the application of the known approaches such as
\cite{carlberg2011efficient,Dahmen2014} 
to the localized setting is not straightforward.
\subsubsection{Properties of Training}
Local basis vectors should be generated using the localized training described in \cref{sec:codim_n_training}.
To judge on the quality of these basis vectors, we compare the error obtained using these basis vectors with the error obtained with local basis vectors generated by global solves. The local basis vectors generated by global solves are the reference: These are the best localized basis we can generate. The results for both geometries are depicted in \cref{fig:training_error}.
\begin{figure}
\begin{tikzpicture}
\begin{semilogyaxis}[
    xlabel=basis size,
    ylabel=error,
    width=0.48\textwidth,
    height=0.45\textwidth,
    xmin=0,
    xmax=3000,
    ymin=1e-10,
    ymax=1e3,
    no markers,
    title=geometry 1,
  ]
  \addplot table[x index=0, y index=4]{localized0.txt};
  \addplot table[x index=0, y index=4]{maxwell_training_benchmark0.txt};
  \legend{global solves, local training}
\end{semilogyaxis}
\end{tikzpicture}
\begin{tikzpicture}
\begin{semilogyaxis}[
    xlabel=basis size,
    ylabel=error,
    width=0.48\textwidth,
    height=0.45\textwidth,
    xmin=0,
    xmax=3000,
    ymin=1e-10,
    ymax=1e3,
    no markers,
    title=geometry 2,
  ]
  \addplot table[x index=0, y index=4]{localized1.txt};
  \addplot table[x index=0, y index=4]{maxwell_training_benchmark1.txt};
  \legend{global solves, local training}
\end{semilogyaxis}
\end{tikzpicture}
\caption[Maximum error over all frequencies for both geometries of \exc{}.]
{Maximum error over all frequencies for both geometries of \exc{}. Basis generated by global solves vs.\ basis generated by local training. 
(reproduction: \cref{repro:fig:training_error})
}
\label{fig:training_error}
\end{figure}
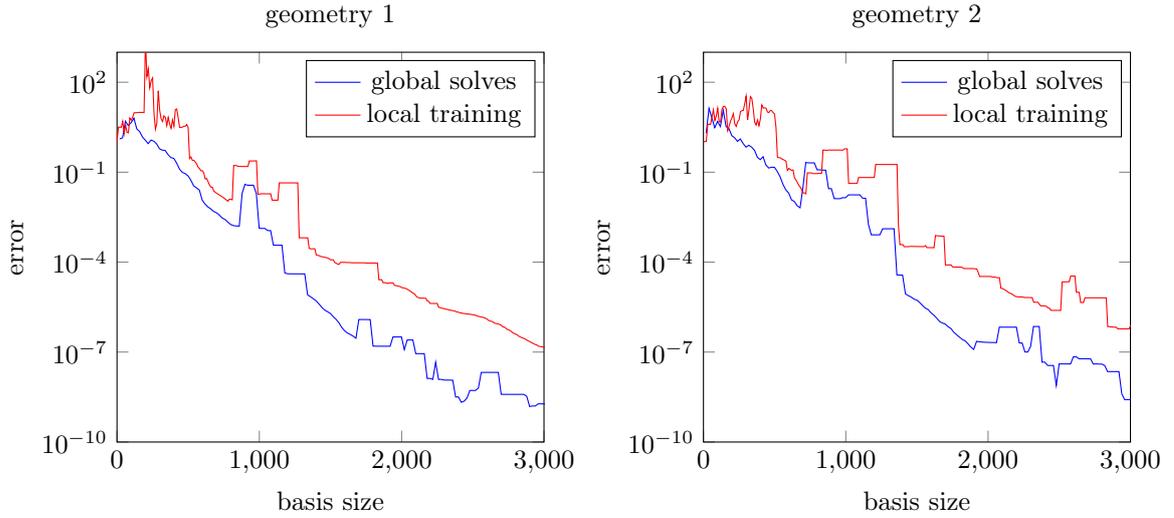
While the error decreases more slowly, we still have reasonable basis sizes with training. For a relative error of $10^{-4}$, 1280 basis vectors are necessary for geometry 1 and 1380 are necessary for geometry 2.
\subsubsection{Application to Local Geometry Change}
If we work with a relative error of 5\%, a basis of size 650 is sufficient for the first geometry and size 675 for the second.
After the geometry change, the local reduced spaces which have no change in their training domain can be reused.
Instead of solving the full system with 60200 degrees of freedom, 
the following effort is necessary per frequency point (see also \cref{fig:maxwell_example_change}).
Because the runtime is dominated by matrix factorizations, we focus on these. 
\begin{figure}
\begin{center}
\includegraphics[width=0.75\textwidth]{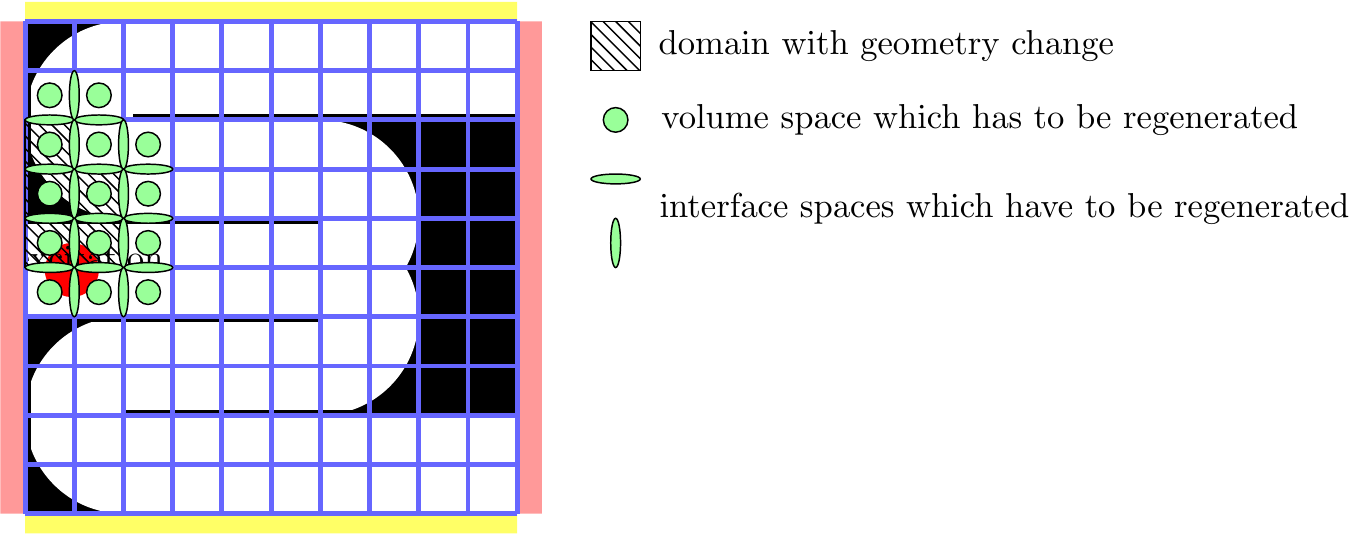}
\end{center}
\caption[Impact of geometry change in \exc{}]
{Impact of geometry change of \exc{}: 5 domains contain changes, 14 domain spaces and 20 interface spaces have to be regenerated.}
\label{fig:maxwell_example_change}
\end{figure}
\begin{itemize}
\item
14 factorizations of local problems with 5340 \gls{dof}s (volume training)
\item
20 factorizations of local problems with 3550 \gls{dof}s (interface training)
\item
1 factorization of global reduced problem with 675 \gls{dof}s (global solve)
\end{itemize}
The error between the reduced solution and the full solution in this case is 4.3\%. 
For reproduction, see \cref{repro:fig:maxwell_basis_sizes}.
The spacial distribution of basis sizes is depicted in \cref{fig:maxwell_basis_sizes}.
\begin{figure}
\begin{center}
geo 1:
\begin{minipage}[c]{0.3\textwidth}
\begin{center}
\includegraphics[width=\textwidth]{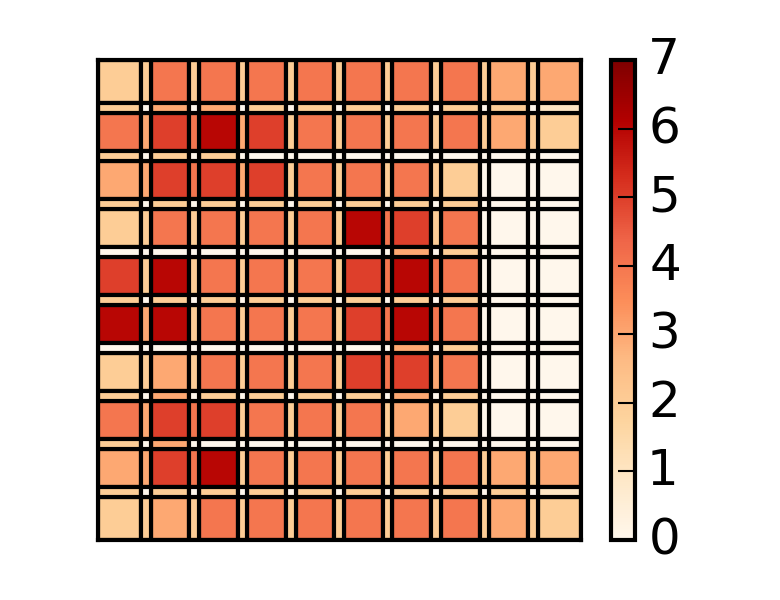}
\end{center}
\end{minipage}
\hspace{20pt}
geo 2:
\begin{minipage}[c]{0.3\textwidth}
\begin{center}
\includegraphics[width=\textwidth]{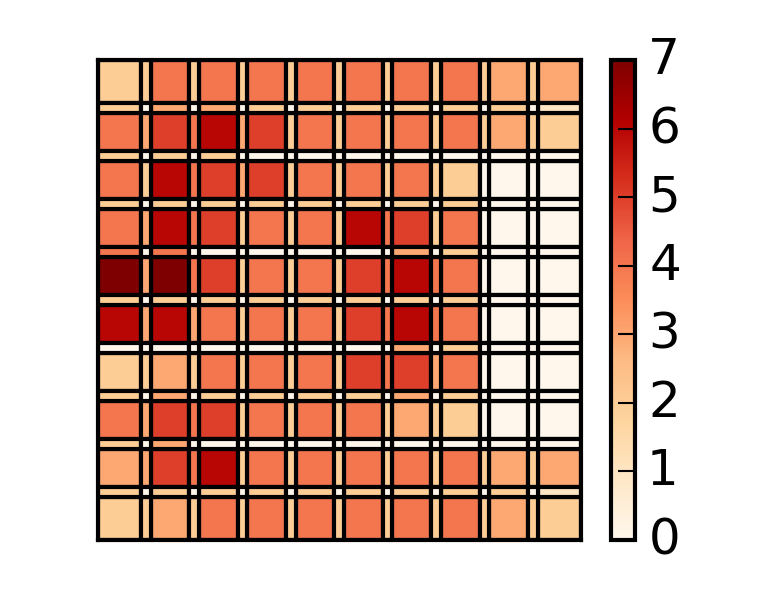}
\end{center}
\end{minipage}
\end{center}
\caption[Basis size distribution in \exc{}]
{Basis size distribution in \exc{}.
(reproduction: \cref{repro:fig:basis_sizes})
}
\label{fig:maxwell_basis_sizes}
\end{figure}

The application of \gls{arbilomod} the inf-sup stable problem of the 2D Maxwell's equation in \gls{hcurl}
shows that localized training generates bases of good quality.
A reduced model with little error for the full problem can be generated using only local solves, which can easily be parallelized.
After localized changes to the model, only in the changed region the localized bases have to be regenerated.
All other bases can be reused, which results in large
computational savings compared to a simulation from scratch.
However, the Galerkin projection leads to an unstable global reduced problem, as expected.

\clearpage
\futurebox{
\section{Ideas For Future Extension of ArbiLoMod}
\myline{}
\begin{itemize}
\item
\textbfit{Stable inf-sup reduction}\\
Using the same spaces for ansatz and test space
does not lead to a stable reduction of inf-sup stable
problems. The usual approach to 
enrich the test space with supremizers is not
feasible in the localized setting, as the
computation of supremizers is a global operation.
A construction procedure for test spaces which lead
to a stable Petrov-Galerkin projection
should be devised.
\item
\textbfit{Better Steering of Local Greedy}\\
The local greedy to construct the
codim-0 spaces in \gls{arbilomod} treats
all interface functions as being equally
important. Probably better spaces can
be generated when assigning
some form of importance factor to each
interface function during training
and considering this in the greedy.
As an alternative, also the codim-0
space could be generated by local training,
as it was done in the experiments for \exc{}.
\item
\textbfit{Systematic Comparison of Space Decompositions}\\
No systematic comparisons of different space decompositions
have been done so far, to the author's knowledge.
Benchmark examples should be defined
and solved with different space decompositions, in order
to compare their properties.
\item
\textbfit{Edge Problems in Wirebasket Space Decomposition}\\
The $\gls{extend}$ operator for the construction of
$\gls{lfspacewb}, i \in \gls{upsilon}_2$ extends
linear to zero on the edges first. Probably smaller basis
sizes can be obtained by solving a 1D problem,
is it was already suggested in \cite[Equation 2.7]{Hou1997}.
\item
\textbfit{$\gls{extend}$ Operator at Domain Boundaries}\\
Close to the domain boundaries, the
resulting basis sizes could probably be reduced
if the $\gls{extend}$ operator incorporated
the boundary conditions in the construction
of the extensions on edges.
\end{itemize}
}

\chapter{Dual Norm Localization -- A Posteriori Error Estimation}
\label{chap:a posteriori}
In this chapter, we
are concerned with localized a posteriori error estimation.
We devise two variants of an a posteriori
error estimator which can be calculated
using localized quantities only.
First we will state two propositions
in \cref{sec:abstract estimates}
in an abstract setting,
which bound
the norm of any continuous linear functional
.
We apply these to the residual to obtain
localized a posteriori error estimates.
Then we derive bounds for the constants
involved in \cref{sec:choosing spaces}.
Finally, we state a scheme for
an offline/online decomposition of 
the calculation of the dual norms involved in \cref{sec:improved splitting},
which overcomes numerical issues in traditional
offline/online splitting schemes.

\section{Requirements on Localized a Posteriori Error Estimator}
\label{sec:a_posteriori}
The model reduction error in the \gls{arbilomod} has to be controlled.
To this end, an a posteriori error estimator is used
which should have the following properties: 
\begin{enumerate}
\item It is robust and efficient.
\item It is offline/online decomposable.
\item It is parallelizable with little amount of communication.
\item After a localized geometry change, the offline computed
data in unaffected regions
can be reused.
\item
It can be used to steer adaptive enrichment of the reduced local subspaces.
\end{enumerate}
These requirements are fulfilled by the estimator presented in
the following. We develop localized bounds for the standard \gls{rb} error estimator,
\begin{equation}
\label{eq:standard rb errest}
\gls{errest} := \frac{1}{\gls{coercconst}_{\gls{parameter}}} \norm{\gls{residual}_{\gls{parameter}}(\gls{rsol}_{\gls{parameter}})}_{\gls{fspace}'}
\end{equation}
where $\gls{residual}_{\gls{parameter}}(\gls{rsol}_{\gls{parameter}}) \in \gls{fspace}^\prime$ is the global residual given as
$
	\dualpair{\gls{residual}_{\gls{parameter}}(\gls{rsol}_{\gls{parameter}})}{\varphi} = \dualpair{f_{\gls{parameter}}}{\varphi } - a_{\gls{parameter}}(\gls{rsol}_{\gls{parameter}}, \varphi) \quad \forall \varphi \in \gls{fspace}
$.
This error estimator is known to be robust and efficient (e.g.~\cite[Proposition 4.4]{Hesthaven2016}):
\begin{equation}\label{eq:global_estimator}
\norm{u_{\gls{parameter}} - \gls{rsol}_{\gls{parameter}}}_{\gls{fspace}} \leq \gls{errest} 
\leq \frac{\gls{contconst}_{\gls{parameter}}}{\gls{coercconst}_{\gls{parameter}}} \norm{u_{\gls{parameter}} - \gls{rsol}_{\gls{parameter}}}_{\gls{fspace}}.
\end{equation}
\section{Abstract Estimates for Dual Norm Localization}
\label{sec:abstract estimates}
In this section, we show two different estimates, holding 
for any continuous linear functional $F \in \gls{fspace}'$.
In later sections, we will apply these estimates to the residual
to obtain localized a posteriori error estimates.

\subsubsection{Setting}
Throughout this section, let \gls{fspace}
be any ansatz space (continuous or discrete) and
let $\{\gls{lfspace} \}_{i=1}^{\gls{numspaces}}$,
$\{ \gls{lfspacemap} \}_{i=1}^{\gls{numspaces}}$
be local spaces and mappings of a localizing space decomposition as
defined in \cref{sec:localizing space decomposition}.
No further assumptions about the space decomposition
are required here.
Let furthermore
\begin{equation}
\dot{\bigcup}_{k=1}^{\gls{orthogonal_classes}} \Upsilon_{C,k} = \left\{1, \dots, \gls{numspaces}\right\}
\end{equation}
be a partition of $\left\{1, \dots, \gls{numspaces}\right\}$
such that
\begin{equation}
       \forall 1\leq k \leq \gls{orthogonal_classes}\ : \ \forall i,j \in \Upsilon_{C,k}, \ i\ne j \ :\ V_i \perp V_j
\end{equation}
i.e.~we can partition the local subspaces into \gls{orthogonal_classes} sets
so that in each set, all spaces are orthogonal.

In this setting, we show the estimates.
To do so, we need following lemma.
\begin{lemma}[Efficiency of dual element localization]
\label{thm:efficiency}
With the assumptions from above, for any continuous linear functional $F$ on \gls{fspace} it holds that
\begin{equation}
	\label{eq:new_abstract_efficiency_estimate}
	\Big( \sum_{i=1}^{\gls{numspaces}} \norm{F}^2_{\gls{lfspace}'} \Big)^{\frac{1}{2}}
	\leq \sqrt{\gls{orthogonal_classes}}  \norm{F}_{\gls{fspace}'}.
\end{equation}
\end{lemma}
\begin{proof}
Let $V_i \perp V_j$ be some subspaces of \gls{fspace}, and 
let $F$ be a continuous linear functional on 
$V_i \oplus V_j$.
If $v_{F,1} \in V_i$ and $v_{F,2} \in V_j$ are the 
Riesz representatives of $F$ in 
$V_i$ and $V_j$,
then due to the orthogonality of $V_i$ and $V_j$,
$v_{F,1} + v_{F,2}$ is the Riesz representative
of $F$ on $V_i \oplus V_j$. Thus,
\begin{align}
\norm{F}_{(\gls{lfspace} \oplus V_j)^\prime}^2 &= \norm{v_{F,1} + v_{F,2}}_{\gls{fspace}}^2 \nonumber \\
                                  &= \norm{v_{F,1}}_{\gls{fspace}}^2 + \norm{v_{F,2}}_{\gls{fspace}}^2 = \norm{F}_{\gls{lfspace}^\prime}^2 + \norm{F}_{V_j^\prime}^2,
\end{align}
where we have used the orthogonality of the spaces again. The same is 
true for a larger orthogonal sum of spaces.
We therefore obtain:
\begin{align}
\sum_{i=1}^{\gls{numspaces}}
\norm{F}_{\gls{lfspace}^\prime}^2 &= \sum_{k=1}^{\gls{orthogonal_classes}} \sum_{i\in \Upsilon_{C,k}} \norm{F}_{\gls{lfspace}^\prime}^2 \nonumber \\
    &= \sum_{k=1}^{\gls{orthogonal_classes}} \norm{F}_{(\bigoplus_{i \in \Upsilon_{C,k}}\gls{lfspace})^\prime}^2 \nonumber \\
    &\leq \gls{orthogonal_classes} \norm{F}_{\gls{fspace}^\prime}^2.
\end{align}
\end{proof}
Using this lemma, we can state the following two localizing estimates for continuous linear functionals.
In the following, we assume there is a subspace
$\gls{rfspace} \subset \gls{fspaceh}$ such
that $F(\varphi) = 0$ for all $\varphi$ in all $\gls{rfspace}$.

\begin{proposition}[Localization of dual norm with local stability constants]
\label{thm:localization_l}
With the local stability constants defined as
\begin{equation}
\gls{mapboundvr} :=
\sup_{\varphi \in \gls{fspace} \nnull}
\frac{ \norm{\left( 1 - \gls{orthogonalp}_{\gls{rfspace} \cap \gls{lfspace}} \right)\gls{lfspacemap}(\varphi)}}{\norm{\varphi}},
\end{equation}
where $\gls{orthogonalp}_{\gls{rfspace} \cap \gls{lfspace}}$ is the orthogonal projection
on $\gls{rfspace} \cap \gls{lfspace}$,
the dual norm of any continuous linear functional $F$ on \gls{fspace}
for which it holds
\begin{equation}
F(\varphi) = 0 \qquad \forall \varphi \in \gls{rfspace}
\end{equation}
can be bounded as
\begin{equation}
\norm{F}_{\gls{fspace}'} \leq \sum_{i=1}^{\gls{numspaces}} \gls{mapboundvr} \norm{F}_{\gls{lfspace}'}
\leq
\left( \sum_{i=1}^{\gls{numspaces}} \gls{mapboundvr}^2 \right)^{\frac 1 2} \norm{F}_{\gls{fspace}'}
.
\end{equation}
\end{proposition}
\begin{proof}
By definition, it holds that
\begin{equation}
\norm{F}_{\gls{fspace}'} = \sup_{\varphi \in \gls{fspace} \nnull} \frac{F(\varphi)}{\norm{\varphi}}.
\end{equation}
Applying the space decomposition $\gls{lfspacemap}$ on $\varphi$ and using the linearity
of the functional, it follows that
\begin{equation}
\norm{F}_{\gls{fspace}'} = \sup_{\varphi \in \gls{fspace} \nnull}
\sum_{i=1}^{\gls{numspaces}} \frac{F(\gls{lfspacemap}(\varphi))}{\norm{\varphi}}
.
\end{equation}
Exchanging the supremum and the sum can only increase the value of the expression.
We have
\begin{equation}
\norm{\varphi} \geq \frac{1}{\gls{mapboundvr}} \norm{\left( 1 - \gls{orthogonalp}_{\gls{rfspace} \cap \gls{lfspace}} \right)\gls{lfspacemap}(\varphi)}
\end{equation}
and it follows that
\begin{equation}
\norm{F}_{\gls{fspace}'} \leq 
\sum_{i=1}^{\gls{numspaces}}
\sup_{\varphi \in \gls{fspace} \setminus \gls{kernel}\left\{\left( 1 - \gls{orthogonalp}_{\gls{rfspace} \cap \gls{lfspace}} \right)\gls{lfspacemap}\right\}}
\gls{mapboundvr} \frac{F(\gls{lfspacemap}(\varphi))}{\norm{\left( 1 - \gls{orthogonalp}_{\gls{rfspace} \cap \gls{lfspace}} \right)\gls{lfspacemap}(\varphi)}}
.
\end{equation}
Since the image of $\gls{orthogonalp}_{\gls{rfspace} \cap \gls{lfspace}}$ is in the kernel of $F$,
it holds that
\begin{align}
&\sup_{\varphi \in \gls{fspace}  \setminus \gls{kernel}\left\{\left( 1 - \gls{orthogonalp}_{\gls{rfspace} \cap \gls{lfspace}} \right)\gls{lfspacemap}\right\}}
\frac{F(\gls{lfspacemap}(\varphi))}{\norm{\left( 1 - \gls{orthogonalp}_{\gls{rfspace} \cap \gls{lfspace}} \right)\gls{lfspacemap}(\varphi)}}
\nonumber
\\
&\qquad =
\sup_{\varphi \in \gls{fspace}  \setminus \gls{kernel}\left\{\left( 1 - \gls{orthogonalp}_{\gls{rfspace} \cap \gls{lfspace}} \right)\gls{lfspacemap}\right\}}
\frac{F\left(\gls{lfspacemap}(\varphi) - \gls{orthogonalp}_{\gls{rfspace} \cap \gls{lfspace}}\gls{lfspacemap}(\varphi)\right)}
{\norm{\left( 1 - \gls{orthogonalp}_{\gls{rfspace} \cap \gls{lfspace}} \right)\gls{lfspacemap}(\varphi)}}
\end{align}
and further
\begin{align}
&
\hspace{-40pt}
\sup_{\varphi \in \gls{fspace}  \setminus \gls{kernel}\left\{\left( 1 - \gls{orthogonalp}_{\gls{rfspace} \cap \gls{lfspace}} \right)\gls{lfspacemap}\right\}}
\frac{F(\gls{lfspacemap}(\varphi))}{\norm{\left( 1 - \gls{orthogonalp}_{\gls{rfspace} \cap \gls{lfspace}} \right)\gls{lfspacemap}(\varphi)}}
\nonumber
\\
&=
\sup_{\varphi \in \gls{fspace}  \setminus \gls{kernel}\left\{\left( 1 - \gls{orthogonalp}_{\gls{rfspace} \cap \gls{lfspace}} \right)\gls{lfspacemap}\right\}}
\frac{F\left(\left( 1 - \gls{orthogonalp}_{\gls{rfspace} \cap \gls{lfspace}} \right)\gls{lfspacemap}(\varphi)\right)}
{\norm{\left( 1 - \gls{orthogonalp}_{\gls{rfspace} \cap \gls{lfspace}} \right)\gls{lfspacemap}(\varphi)}} \nonumber \\
&=
\sup_{\varphi \in \gls{image}\left\{\left( 1 - \gls{orthogonalp}_{\gls{rfspace} \cap \gls{lfspace}} \right)\gls{lfspacemap}\right\} \nnull}
\frac{F\left(\varphi \right)}
{\norm{\varphi}} \nonumber \\
&\leq 
\sup_{\varphi \in \gls{lfspace} \nnull}
\frac{F(\varphi)}{\norm{\varphi}}
=
\norm{F}_{\gls{lfspace}'}
\end{align}
the first inequality follows.

For the second inequality we apply Cauchy-Schwarz
\begin{equation}
\sum_{i=1}^{\gls{numspaces}} \gls{mapboundvr} \norm{F}_{\gls{lfspace}'}
\leq
\left(\sum_{i=1}^{\gls{numspaces}} \gls{mapboundvr} \right)^{\frac 1 2}
\left(\sum_{i=1}^{\gls{numspaces}} \norm{F}_{\gls{lfspace}'} \right)^{\frac 1 2}
\end{equation}
and then use \cref{thm:efficiency}
to obtain the claim.
\end{proof}

\begin{proposition}[Localization of dual norm with global stability constant]
\label{thm:localization_g}
With the global stability constant defined as
\begin{equation}
\gls{decompositionboundvr} := 
\sup_{\varphi \in \gls{fspace} \nnull}
\frac{\left(\sum_{i=1}^{\gls{numspaces}} \norm{\left( 1 - \gls{orthogonalp}_{\gls{rfspace} \cap \gls{lfspace}} \right)\gls{lfspacemap}(\varphi)}^2\right)^{\frac 1 2}}
 {\norm{\varphi}},
\end{equation}
the dual norm of any continuous linear functional $F$ on \gls{fspace}
for which it holds
\begin{equation}
F(\varphi) = 0 \qquad \forall \varphi \in \gls{rfspace}
\end{equation}
can be bounded as
\begin{equation}
\norm{F}_{\gls{fspace}'}
\leq
\gls{decompositionboundvr} 
\left(
\sum_{i=1}^{\gls{numspaces}} \norm{F}_{\gls{lfspace}'}^2
\right) ^{\frac 1 2}
\leq
\gls{decompositionboundvr} \sqrt{\gls{orthogonal_classes}}
\norm{F}_{\gls{fspace}'}
.
\end{equation}
\end{proposition}
\begin{proof}
By definition, it holds that
\begin{equation}
\norm{F}_{\gls{fspace}'} = \sup_{\varphi \in \gls{fspace} \nnull} \frac{F(\varphi)}{\norm{\varphi}}.
\end{equation}
Applying the space decomposition $\gls{lfspacemap}$ on $\varphi$ and using the linearity
of the functional, it follows that
\begin{equation}
\norm{F}_{\gls{fspace}'} = \sup_{\varphi \in \gls{fspace} \nnull}
\frac{\sum_{i=1}^{\gls{numspaces}} F(\gls{lfspacemap}(\varphi))}{\norm{\varphi}}
.
\end{equation}
Since the image of $\gls{orthogonalp}_{\gls{rfspace} \cap \gls{lfspace}}$ is in the kernel of $F$,
it holds that
\begin{equation}
\norm{F}_{\gls{fspace}'} = \sup_{\varphi \in \gls{fspace} \nnull}
\frac{\sum_{i=1}^{\gls{numspaces}} F\left(\left( 1 - \gls{orthogonalp}_{\gls{rfspace} \cap \gls{lfspace}} \right)\gls{lfspacemap}(\varphi)\right)}{\norm{\varphi}}
.
\end{equation}
Using the dual norm of $F$ in
\begin{equation}
F\left(\left( 1 - \gls{orthogonalp}_{\gls{rfspace} \cap \gls{lfspace}} \right)\gls{lfspacemap}(\varphi)\right)
\leq \norm{F}_{\gls{lfspace}'} \norm{\left( 1 - \gls{orthogonalp}_{\gls{rfspace} \cap \gls{lfspace}} \right)\gls{lfspacemap}(\varphi)},
\end{equation}
applying Cauchy-Schwarz, and rearranging yields
\begin{align}
\norm{F}_{\gls{fspace}'} &= \sup_{\varphi \in \gls{fspace} \nnull}
\frac{\sum_{i=1}^{\gls{numspaces}} F\left(\left( 1 - \gls{orthogonalp}_{\gls{rfspace} \cap \gls{lfspace}} \right)\gls{lfspacemap}(\varphi)\right)}{\norm{\varphi}}
\\
&\leq 
\sup_{\varphi \in \gls{fspace} \nnull}
\frac{\sum_{i=1}^{\gls{numspaces}} \norm{F}_{\gls{lfspace}'} \norm{\left( 1 - \gls{orthogonalp}_{\gls{rfspace} \cap \gls{lfspace}} \right)\gls{lfspacemap}(\varphi)}}{\norm{\varphi}}
\\
&\leq
\sup_{\varphi \in \gls{fspace} \nnull}
\frac{
\left(\sum_{i=1}^{\gls{numspaces}} \norm{\left( 1 - \gls{orthogonalp}_{\gls{rfspace} \cap \gls{lfspace}} \right)\gls{lfspacemap}(\varphi)}^2 \right) ^{\frac 1 2}
\left(\sum_{i=1}^{\gls{numspaces}} \norm{F}_{\gls{lfspace}'}^2 \right) ^{\frac 1 2}
}{\norm{\varphi}}
\\
&=
\gls{decompositionboundvr} 
\left(
\sum_{i=1}^{\gls{numspaces}} \norm{F}_{\gls{lfspace}'}^2
\right) ^{\frac 1 2}
\end{align}
which is the first inequality.

The second inequality again follows directly from \cref{thm:efficiency}.
\end{proof}

There are some obvious relations between these stability constants,
namely
\begin{align}
\gls{decompositionboundvr} &\leq \left( \sum_{i=1}^{\gls{numspaces}} \gls{mapboundvr}^2 \right) ^{\frac 1 2}\\
\gls{mapboundvr} &\leq \gls{decompositionboundvr} \qquad \forall i \in \{1, \dots, \gls{numspaces} \}
.
\end{align}

A stability constant very similar to $\gls{decompositionboundvr}$ appears in the
analysis of overlapping domain decomposition methods
(e.g. \cite[Assumption 2.2]{Toselli2005}, \cite{Spillane2013})
and in localization of error estimators on stars (e.g. \cite{Cohen2012}).
Applying these two estimates to the residual, we obtain two efficient, localized error estimators:
\begin{corollary}[Localization of a posteriori error estimator]
\label{thm:abstract_error_estimate}
The error estimators
\gls{locerrest_l} and
\gls{locerrest_g},
defined by
\begin{align}
\gls{locerrest_l} :=&
\frac{1}{\gls{coercconst}_{\gls{parameter}}}
\sum_{i=1}^{\gls{numspaces}} \gls{mapboundvr} \norm{\gls{residual}_{\gls{parameter}}(\gls{rsol}_{\gls{parameter}})}_{\gls{lfspace}'}
\\
\gls{locerrest_g} :=&
\frac{1}{\gls{coercconst}_{\gls{parameter}}}
\gls{decompositionboundvr} 
\left(
\sum_{i=1}^{\gls{numspaces}} \norm{\gls{residual}_{\gls{parameter}}(\gls{rsol}_{\gls{parameter}})}_{\gls{lfspace}'}^2
\right) ^{\frac 1 2}
\end{align}
are robust and efficient. It holds
\begin{align}
\label{eq:efficiency_estimate 3}
\norm{ \gls{sol}_{\gls{parameter}} - \gls{rsol}_{\gls{parameter}} }_{\gls{fspace}} \leq &
    \gls{locerrest_l}
\leq \frac{\gls{contconst}_{\gls{parameter}}}{\gls{coercconst}_{\gls{parameter}}}
        \left( \sum_{i=1}^{\gls{numspaces}} \gls{mapboundvr}^2 \right)^{\frac 1 2} \sqrt{\gls{orthogonal_classes}}
        \norm{ u_{\gls{parameter}} - \gls{rsol}_{\gls{parameter}} }_{\gls{fspace}}
\\
\label{eq:efficiency_estimate 4}
\norm{ \gls{sol}_{\gls{parameter}} - \gls{rsol}_{\gls{parameter}} }_{\gls{fspace}} \leq &
    \gls{locerrest_g}
\leq \frac{\gls{contconst}_{\gls{parameter}}}{\gls{coercconst}_{\gls{parameter}}}
        \gls{decompositionboundvr} \sqrt{\gls{orthogonal_classes}}
        \norm{ u_{\gls{parameter}} - \gls{rsol}_{\gls{parameter}} }_{\gls{fspace}}
        .
\end{align}
\end{corollary}
\begin{proof}
Applying
\cref{thm:localization_l} or \cref{thm:localization_g} to
the error estimator
\begin{equation}
\gls{errest} = \frac{1}{\gls{coercconst}_{\gls{parameter}}} \norm{\gls{residual}_{\gls{parameter}}(\gls{rsol}_{\gls{parameter}})}_{\gls{fspace}'}
\end{equation}
yields,
together with \cref{eq:global_estimator}, the proposition.
\end{proof}

\subsubsection{Discussion of the two variants}
Depending on the distribution of the
residual in space, the variant with
the localized constants
\gls{locerrest_l} might deliver better
or worse
results than the variant with
globally computed constants 
\gls{locerrest_g}.
Whenever both the residual
and the constant
\gls{mapboundvr}
are distributed heterogeneously
over the domain, we expect that
the localized variant
\gls{locerrest_l}
might be superior.
When the residual and the constants
are homogeneously distributed
over the domain, we expect
the variant with globally 
computed constants 
\gls{locerrest_g}
to be superior.
As most of the computational
effort is in the computation of
the localized residuals,
it might be the best strategy 
to always compute both
\gls{locerrest_l} and \gls{locerrest_g}
and use the better one.
In our experiments in \cref{sec:numerical_example}
we used \gls{locerrest_g}.

Offline/online decomposition of this error estimator can be done by 
applying the usual strategy for offline/online decomposition
used with the standard \gls{rb} error estimator (see e.g. \cite[Sec.  4.2.5]{Hesthaven2016})
or the numerically more stable approach shown below in \cref{sec:improved splitting}
to every dual norm in the localized error estimators.
\section{Choosing Spaces and Bounding Constants}
\label{sec:choosing spaces}
The error estimators defined in \cref{thm:abstract_error_estimate}
work for any space decomposition fulfilling the assumptions in
\cref{sec:localizing space decomposition}.
It is important to realize that there is no requirement
to use the same space decomposition
for the decomposition of the ansatz space
and for the a posteriori error estimator.
In order to obtain good constants, the space decomposition needs to be chosen carefully.
For the space decomposition used in the a posteriori error estimator,
there are different requirements than for the space decomposition
used in constructing the ansatz space.
While the main goal in the selection of the localization
of the ansatz space for the construction of the reduced
space was to achieve a small dimension of the reduced space
and especially a small dimension of the reduced spaces
coupling the domains, for the space decomposition used
for the a posteriori error estimator, the following two
requirements are dominant.
First, the subspaces should be spanned by FE ansatz functions, allowing the
residual to be easily evaluated on these spaces. 
If we used a wirebasket space decomposition for the a posteriori
error estimator, as it was done in \cite{Smetana2015},
it would be necessary to construct a basis for every
local space, in turn needing the construction
of extensions of every fine grid \gls{dof} on the interfaces.
This leads to very high computational cost in the offline phase,
which is not a problem in the \gls{scrbe} with its complete
offline/online decomposition, but is a problem in the
\gls{arbilomod} setting.
Second, the inner product
matrices on the subspaces should be sparse, as the inner product matrix
has to be solved in the computation of the dual norms.

In order to achieve these two goals,
we use the partition of unity decomposition 
given in \cref{sec:introduction space decompositions}.
We define the overlapping domains \gls{subdomainol}
in terms of the non overlapping domains
\gls{subdomainnol}.
\begin{definition}[Overlapping space decomposition]
Let $\mathcal{V}_i$, $i \in \{1, \dots, N_{\mathcal{V}_i}\}$
be the inner vertices of the coarse mesh.
We define the overlapping subdomains as
\begin{equation}
\gls{subdomainol} := \overset{\circ}{\overline{\bigcup_{j \in \gls{domains}_{\mathcal{V}_i}} \omega_j^\mathrm{nol}}}
\qquad i \in \{1, \dots, N_{\mathcal{V}_i}\}
\end{equation}
and we define linear interpolation operator \gls{feinterpolation}
to be the nodal Lagrange interpolation.
It holds $\gls{numdomainsol} =  N_{\mathcal{V}_i}$.
\end{definition}

Recall the definition of $\gls{domains}_{\mathcal{V}_i}$
from \cref{eq:domains in mesh}:
$\gls{domains}_{\mathcal{V}_i}$ is
the set of indices of subdomains that contain the coarse mesh vertex $\mathcal{V}_i$, i.e.
\begin{equation}
\gls{domains}_{\mathcal{V}_j} :=
\Big\{i \in \{1, \dots, \gls{numdomainsnol}\} \ \Big| \ \mathcal{V}_j \subseteq \overline{\gls{subdomainnol}} \Big\}.
\end{equation}

With these definitions, the local space \gls{lfspacepou} is
the span of all \gls{fe} basis functions having support only on the corresponding overlapping domain.
With $\gls{domains}_\psi$ defined as in \cref{eq:febasisdomains}, we have
\begin{equation}
\gls{lfspacepou} = \spanset\Big\{\psi \in \gls{febasis} \ \Big| \ \gls{domains}_\psi \subseteq \gls{domains}_{\mathcal{V}_i} \Big\}.
\label{eq:overlapping_space_def}
\end{equation}
Note that we have $\gls{lfspacepou} \subseteq H^1_0(\gls{subdomainol})$
for the inner domains with $\partial \gls{subdomainol} \cap \partial \gls{domain} = \emptyset$.
Contrary to $\gls{lfspacewb}$ or $\gls{lfspacebasic}$, these spaces do not form a direct sum decomposition of $\gls{fspaceh}$.
We next state a first estimate on the partition of unity constant \gls{decompositionboundvr} of \cref{thm:abstract_error_estimate}.
The resulting estimate thus 
depends on $\gls{H}^{-1}$, where $\gls{H}$ is the minimum diameter of the subdomains of the macro partition.
Typically, the size of the macro partition is moderate such that $\gls{H}^{-1}$ is small.
However, in the following \cref{thm:pu_bound} we will show that the constant $\gls{decompositionboundvr}$ can be actually bounded 
independent of $\gls{H}$, when we choose a partition of unity that is contained in the reduced space $\gls{rfspace}$.
\begin{proposition}[Estimate of localization constant]
\label{thm:pu_bound_0}
Let \gls{fspace} be a subspace of \gls{h1} equipped with the usual $H^1$-norm.
Let $\gls{overlappingspaces} := \max_{x \in \gls{domain}} \#\{i \in \{1, \dots, \gls{numdomainsol}\} | \ x \in \gls{subdomainol}\}$ be the maximum 
number of overlapping domains \gls{subdomainol} overlapping in any point $x$ of $\gls{domain}$
and let $H_i:= \diam (\gls{subdomainol})$, $i \in  \{1, \dots, \gls{numdomainsol}\}$ and $\gls{H} := \min_{ \{1, \dots, \gls{numdomainsol}\}} H_i$.
Furthermore, assume that there exist partition of unity functions
$\gls{pouf} \in H^{1,\infty}(\gls{domain})$, $i \in  \{1, \dots, \gls{numdomainsol}\}$
and a linear interpolation operator $\gls{feinterpolation}: \gls{fspace} \longrightarrow \gls{fspaceh}$ such that
\begin{enumerate}[label=(\roman*)]
	\item $\sum_{i=1}^{\gls{numdomainsol}} \gls{pouf}(x) = 1$ for all $x \in \gls{domain}$,
	\item $\gls{pouf}(x) \geq 0$ for all $x$ in \gls{domain} and all $i$ in $\{1, \dots, \gls{numdomainsol}\}$.
	\item $\norm{\nabla \gls{pouf}}_\infty \leq \gls{pufuncconstant}  H_i^{-1}$ for all $i \in \{1, \dots, \gls{numdomainsol}\}$,
	\item $\gls{feinterpolation}(\varphi) = \varphi$ for all $\varphi \in \gls{fspaceh}$,
	\item $\norm{\gls{feinterpolation}(\gls{pouf} v_h) - \gls{pouf} v_h}_{\gls{fspace}} \leq \gls{interpolationbound} \norm{v_h}_{H^1(\gls{subdomainol})}$ for all $i \in \{1, \dots, \gls{numdomainsol}\}$ and all $v_h \in \gls{fspaceh}$. 
\end{enumerate}
Then we have:
\begin{equation}
	\gls{decompositionboundvr} \leq \sqrt{4 + 2\gls{interpolationbound}^2 + 4\left(\gls{pufuncconstant}  \gls{H}^{-1}\right)^2}\cdot \sqrt{\gls{overlappingspaces}}.
\end{equation}
\end{proposition}
\begin{proof}
We compute the bound for $\gls{decompositionboundvr}$ using the partition of unity and the interpolation operator.
Recall the definition
\begin{equation}
	\gls{lfspacemappou}(\varphi) := \gls{feinterpolation}(\gls{pouf}  \varphi), \qquad i \in \{1, \dots, \gls{numdomainsol}\}.	
\end{equation}
Using this definition and \textit{(v)} we have for any $\varphi \in \gls{fspaceh}$
\begin{align}
   \sum_{i=1}^{\gls{numdomainsol}} \norm{\gls{lfspacemappou}(\varphi)}_{\gls{fspace}}^2 
   &= \sum_{i=1}^{\gls{numdomainsol}} \norm{\gls{feinterpolation}(\gls{pouf}  \varphi)}_{\gls{fspace}}^2 
   \nonumber
   \\
   &= \sum_{i=1}^{\gls{numdomainsol}} \norm{\gls{feinterpolation}(\gls{pouf}  \varphi) - \gls{pouf}\varphi + \gls{pouf}\varphi}_{\gls{fspace}}^2 
   \nonumber
   \\
   &\leq 2 \sum_{i=1}^{\gls{numdomainsol}} 
     \norm{\gls{feinterpolation}(\gls{pouf}  \varphi) - \gls{pouf}\varphi}_{\gls{fspace}}^2  + \norm{\gls{pouf}\varphi}_{\gls{fspace}}^2 
     \nonumber
   \\
   &\leq 2 \sum_{i=1}^{\gls{numdomainsol}}  \gls{interpolationbound}^2 \norm{\varphi}_{H^1(\gls{subdomainol})}^2 + \norm{\gls{pouf}  \varphi}_{\gls{fspace}}^2 .
   \label{eq:constant est 11}
\end{align}
Furthermore, using the properties of the partition of unity, we have
\begin{align}
\norm{\gls{pouf}  \varphi}_{\gls{fspace}}^2
&= \int_{\gls{subdomainol}} |\nabla (\gls{pouf} \varphi) |^2 + |\gls{pouf} \varphi |^2 \dx
\nonumber
\\
&= \int_{\gls{subdomainol}} |\nabla \gls{pouf} \varphi + \gls{pouf} \nabla \varphi |^2 + |\gls{pouf} \varphi |^2 \dx
\nonumber
\\
&\leq \int_{\gls{subdomainol}} 2|\nabla \gls{pouf} \varphi|^2 + 2|\gls{pouf} \nabla \varphi |^2 + |\gls{pouf} \varphi |^2 \dx
\nonumber
\\
&\leq \int_{\gls{subdomainol}} 2\left(\gls{pufuncconstant}  H_i^{-1}\right)^2|\varphi|^2 + 2| \nabla \varphi |^2 + |\varphi |^2 \dx
\nonumber
\\
&= \left( 1 + 2\left(\gls{pufuncconstant}  H_i^{-1}\right)^2 \right) \norm{\varphi}_{L^2(\gls{subdomainol})}^2
    + 2 \norm{\nabla \varphi}_{L^2(\gls{subdomainol})}^2
\nonumber
\\
&\leq \left( 2 + 2\left(\gls{pufuncconstant}  H_i^{-1}\right)^2 \right) \norm{\varphi}_{H^1(\gls{subdomainol})}^2
.
   \label{eq:constant est 12}
\end{align}
Inserting \cref{eq:constant est 12} into \cref{eq:constant est 11} yields
\begin{align}
\sum_{i=1}^{\gls{numdomainsol}} \norm{\gls{lfspacemappou}(\varphi)}_{\gls{fspace}}^2 
&\leq 2 \sum_{i=1}^{\gls{numdomainsol}}  \gls{interpolationbound}^2 \norm{\varphi}_{H^1(\gls{subdomainol})}^2 + \left( 2 + 2\left(\gls{pufuncconstant}  H_i^{-1}\right)^2 \right) \norm{\varphi}_{H^1(\gls{subdomainol})}^2
\nonumber
\\
&\leq \left( 4 + 2 \gls{interpolationbound}^2 + 4\left(\gls{pufuncconstant}  H_i^{-1}\right)^2 \right)
\sum_{i=1}^{\gls{numdomainsol}}  \norm{\varphi}_{H^1(\gls{subdomainol})}^2
\nonumber
\\
&\leq \left( 4 + 2 \gls{interpolationbound}^2 + 4\left(\gls{pufuncconstant}  H_i^{-1}\right)^2 \right)
\gls{overlappingspaces} \norm{\varphi}_{\gls{fspace}}^2
.
\end{align}
This gives us the estimate.
\end{proof}
A similar estimate for the energy norm is given in \cref{thm:upper bound for cpu in energy norm}.

For the rectangular domain decomposition used in 
the numerical example below, the constant $\gls{overlappingspaces}$ is $\gls{overlappingspaces} = 2^{\gls{dimension}} = 4$:
it is possible to divide the overlapping domains into four classes, so that
within each class, no domain overlaps with any other
(cf.\ \cite[Sec. 5]{Chung2015a}).

Furthermore, 
the 
coercivity constant $\gls{coercconst}_{\gls{parameter}}$ and the stability constant $\gls{contconst}_{\gls{parameter}}$,
or estimates, are required.
For the numerical example \exb{}
used in \cref{sec:numerical_example}, those
can be calculated analytically.
In general this is not possible.
The numerical computation of a lower bound for the coercivity constant
was subject of extensive research in the \gls{rb} community
(see e.g. \cite{Huynh2007,Chen2008})%
, but these methods require the calculation of the
coercivity constant at some parameter values and thus
require the solution of a global, fine scale problem.
To the authors' knowledge, there are no publications
on localization of these methods so far.

The upper bound on the constant $\gls{decompositionboundvr}$ in \cref{thm:pu_bound_0} depends on the domain size $\gls{H}$ approximately
like $\gls{H}^{-1}$. As the domain size is considered a constant in the
context of \gls{arbilomod}, the error estimator is already considered efficient with this bound. In the next proposition, we however 
show that the constant \gls{decompositionboundvr} can indeed be bounded independent of $\gls{H}$, if we exploit that the residual vanishes on the 
reduced space $\gls{rfspace}$. 
\begin{proposition}[$\gls{H}$ independent localization constant estimate]
\label{thm:pu_bound}
Let $\gls{pouf}$, $i \in \{1, \dots, \gls{numdomainsol}\}$ be a partition of unity and $\gls{feinterpolation}$
an interpolation operator satisfying the prerequisites of \cref{thm:pu_bound_0}.
Furthermore, assume $V = H^1_0(\Omega)$ and that $\gls{pouf} \in \gls{rfspace} \cap \gls{lfspacepou}$ for $i \in
D^{\rm int} := \{i \in \{1, \dots, \gls{numdomainsol}\} | \ \overline{\gls{subdomainol}} \cap \partial \gls{domain} = \emptyset\}$,
e.g. $\gls{pouf}$ is included in the reduced space (see \cref{sec:decomposition,sec:codim_2_spaces}
above).
Then the following estimate holds:
\begin{equation}
\gls{decompositionboundvr} \leq \sqrt{4 + 2\gls{interpolationbound}^2 + 4(\gls{pufuncconstant} \gls{scaledpoincareconstant})^2} \cdot \sqrt{\gls{overlappingspaces}},
\end{equation}
with a Poincaré-inequality constant $\gls{scaledpoincareconstant}$ (see proof below) that does not depend 
on the fine or coarse mesh sizes ($\gls{h}, \gls{H}$). 
\end{proposition}
\begin{proof}
For arbitrary $\varphi \in \gls{fspaceh}$ let $\bar \varphi_i := \frac{1}{|\gls{subdomainol}|}\int_{\gls{subdomainol}} \varphi$.
We then have with $D^{\rm ext}:=
\{1, \dots, \gls{numdomainsol}\} \setminus D^{\rm int}$
\begin{align}
\gls{decompositionboundvr} &:= 
\sup_{\varphi \in \gls{fspaceh} \nnull}
\frac{\left(\sum_{i=1}^{\gls{numdomainsol}} \norm{\left( 1 - \gls{orthogonalp}_{\gls{rfspace} \cap \gls{lfspacepou}} \right)\gls{lfspacemappou}(\varphi)}^2\right)^{\frac 1 2}}
 {\norm{\varphi}}
 \nonumber
 \\
 &= 
\sup_{\varphi \in \gls{fspaceh} \nnull}
\frac{\left(
\sum_{i\in D^\mathrm{int}} \norm{\left( 1 - \gls{orthogonalp}_{\gls{rfspace} \cap \gls{lfspacepou}} \right)\gls{lfspacemappou}(\varphi)}^2
+
\sum_{i\in D^\mathrm{ext}} \norm{\left( 1 - \gls{orthogonalp}_{\gls{rfspace} \cap \gls{lfspacepou}} \right)\gls{lfspacemappou}(\varphi)}^2
\right)^{\frac 1 2}}
 {\norm{\varphi}}
 \nonumber
\\
 &\leq
\sup_{\varphi \in \gls{fspaceh} \nnull}
\frac{\left(
\sum_{i\in D^\mathrm{int}} \norm{\gls{lfspacemappou}(\varphi) - \bar \varphi_i \gls{pouf}}^2
+
\sum_{i\in D^\mathrm{ext}} \norm{\gls{lfspacemappou}(\varphi)}^2
\right)^{\frac 1 2}}
 {\norm{\varphi}},
\end{align}
where we have used that by construction $\bar \varphi_i \gls{pouf} \in { \gls{rfspace} \cap \gls{lfspacepou}}$ for all $i \in
D^\mathrm{int}$.
For any $\varphi \in \gls{fspaceh}$ and $i \in D^\mathrm{int}$ we then have 
$\norm{\gls{lfspacemappou}(\varphi) - \bar \varphi_i \gls{pouf}}_{\gls{fspace}}^2 \leq 2\gls{interpolationbound}^2 \norm{\nabla \varphi}_{L^2(\gls{subdomainol})}^2 + 2
    \norm{(\varphi - \bar \varphi_i) \gls{pouf}}_{\gls{fspace}}^2 $, where
\begin{align}
    \norm{(\varphi - \bar \varphi_i) \gls{pouf}}_{\gls{fspace}}^2 
    &\leq \int_{\Omega_\xi} 
      2|\nabla (\varphi - \bar \varphi_i) \gls{pouf} |^2(x) + 2| (\varphi - \bar \varphi_i) \nabla \gls{pouf} |^2(x) \dx \nonumber \\
      &\hspace*{5em} + \norm{(\varphi - \bar \varphi_i) \gls{pouf}}^2_{L^2(\gls{subdomainol})}.
 \end{align}
With a rescaled Poincaré-type inequality 
$$
	\norm{\varphi - \bar \varphi_i}_{L^2(\gls{subdomainol})} \leq \gls{scaledpoincareconstant} H_i \norm{\nabla \varphi}_{L^2(\gls{subdomainol})},
$$
and
$
	\norm{\varphi - \bar \varphi_i}_{L^2(\gls{subdomainol})} \leq \norm{\varphi}_{L^2(\gls{subdomainol})},
$
we get
\begin{align}
   \int_{\Omega_\xi} 
   2|\nabla (\varphi - \bar{\varphi}_\xi) \gls{pouf} |^2(x) &+ 2| (\varphi - \bar \varphi_i) \nabla \gls{pouf} |^2(x) \dx + \norm{(\varphi - \bar \varphi_i) \gls{pouf}}^2_{L^2(\gls{subdomainol})} \nonumber \\
       & \qquad\qquad\leq (2 + 2(\gls{pufuncconstant} \gls{scaledpoincareconstant})^2 )  \norm{\nabla \varphi}^2_{L^2(\gls{subdomainol})} +
   \norm{\varphi}^2_{L^2(\gls{subdomainol})}.
 \end{align}
In analogy we obtain for the boundary terms, i.e. $i \in D^\mathrm{ext}$, the estimates
\begin{equation}
\norm{\gls{lfspacemappou}(\varphi)}_{\gls{fspace}}^2 \leq 2\gls{interpolationbound}^2 \norm{\nabla \varphi}_{L^2(\gls{subdomainol})} + 2
    \norm{\varphi \gls{pouf}}_{\gls{fspace}}^2 ,
\end{equation}
and
\begin{align}
    \norm{ \varphi \gls{pouf} }_{\gls{fspace}}^2 
    &\leq \int_{\Omega_\xi} 
    2|\nabla \varphi \gls{pouf}|^2(x) + 2| \varphi \nabla \gls{pouf} |^2(x) \dx + \norm{ \varphi  \gls{pouf} }^2_{L^2(\gls{subdomainol})} \nonumber \\
    & \leq (2 + 2(\gls{pufuncconstant} \gls{scaledpoincareconstant})^2) \norm{\nabla \varphi}^2_{L^2(\gls{subdomainol})} +
      \norm{\varphi}^2_{L^2(\gls{subdomainol})}
 \end{align}
using a rescaled Poincaré-type inequality which holds for $i \in D^{\rm ext}$ as 
$\varphi \in \gls{fspaceh}$ has zero boundary values, i.e.
$$
	\norm{\varphi}_{L^2(\gls{subdomainol})} \leq \gls{scaledpoincareconstant} H_i \norm{\nabla \varphi}_{L^2(\gls{subdomainol})}.
$$
Summing up all contributions we then have

\begin{align}
  \sum_{i \in D^\mathrm{int}} & \norm{\gls{lfspacemappou}(\varphi) - \bar \varphi_i \gls{pouf}}_{\gls{fspace}}^2
    + \sum_{i \in D^\mathrm{ext}}  \norm{\gls{lfspacemappou}(\varphi)}_{\gls{fspace}}^2 \nonumber \\
	      &\leq \sum_{i \in \{1, \dots, \gls{numdomainsol}\}} 2\gls{interpolationbound}^2 \norm{\nabla \varphi}_{L^2(\gls{subdomainol})}^2 + 2\left[ (2 + 2(\gls{pufuncconstant} \gls{scaledpoincareconstant})^2) \norm{\nabla
      \varphi}^2_{L^2(\gls{subdomainol})} + \norm{\varphi}^2_{L^2(\gls{subdomainol})} \right] \nonumber\\
	    & \leq (4 + 2\gls{interpolationbound}^2 + 4(\gls{pufuncconstant} \gls{scaledpoincareconstant})^2) \gls{overlappingspaces}  \norm{\varphi}^2_{V}.
\end{align}	      
This gives us the estimate.
\end{proof}

\cref{thm:pu_bound} gives a bound on $\gls{decompositionboundvr}$ that depends on the contrast of the underlying 
diffusion coefficient if $\gls{pouf} \in \gls{rfspace}$, $i \in D^\mathrm{int}$ is chosen as the \gls{msfem} type hat functions 
as suggested in \cref{sec:decomposition} above. However, it is independent on the mesh sizes $\gls{h}, \gls{H}$. 
A crucial ingredient to obtain this bound is the fact that we 
included this macroscopic partition of unity in our reduced approximation space $\gls{rfspace}$. 
If alternatively we would chose $\gls{pouf} \in \gls{rfspace}$ to be the traditional Lagrange hat functions, 
the bound on $\gls{decompositionboundvr}$ in  \cref{thm:pu_bound} would be independent of the contrast.
In fact, we might expect that $\gls{decompositionboundvr}$ behaves much better then the upper bound due to 
the approximation properties of the reduced space. It would actually be possible to compute $\gls{decompositionboundvr}$
for given $\gls{fspace}, \gls{rfspace}$ which would however be computationally expensive and thus not of any use 
in practical applications. \cref{thm:pu_bound}, however shows that the localized a posteriori 
error estimator in \cref{thm:abstract_error_estimate} in the context of \gls{arbilomod} is indeed robust and efficient, 
even with respect to $\gls{H} \to 0$.

Comparing with other localized \gls{rb} and multiscale methods,
one observes a difference in the scaling of the efficiency constants.
While in our case, $c_{pu}$ is independent of both $\gls{h}$ and $\gls{H}$,
the a posteriori error estimator published for \gls{lrbms} has a
$\gls{H}/\gls{h}$ dependency
\cite[Theorem 4.6]{Ohlberger2015}
and in the certification framework for \gls{scrbe}, a $\gls{h}^{-1/2}$ scaling appears
\cite[Proposition 4.5]{Smetana2015}, which, however,
can be removed by using $H^{1/2}$-orthogonal port modes
\cite[Corollary A.1]{Smetana2015}.
The (a priori) error estimators published for \gls{gmsfem} in \cite{Chung2014b} also 
have no dependency on $\gls{H}$ or $\gls{h}$. However, they also rely on specific
properties of the basis generation. 
The error estimator presented
here is independent of the method the basis is generated with, which
is advantageous as basis generation algorithms and the error
estimator can be developed further independently.
Also in the analysis of the \gls{dgrbe}
Pacciarini et.al.~have
a factor of $\gls{h}^{-1/2}$ in the
a priori analysis \cite{Antonietti2016}
and in the a posteriori error estimator \cite{Pacciarini2016}.

So far we did not use properties of the bilinear form other than
coercivity and continuity. Assuming locality of the bilinear form 
as in \eqref{eq:weak heat equation}, we get a local efficiency
estimate and an improved global efficiency estimate.
\begin{proposition}[Localized efficiency estimate]
\label{thm:local_efficiency}
Let the bilinear form $\gls{a}_{\gls{parameter}}$ be given by \eqref{eq:weak heat equation}.
Then we have the localized efficiency estimate
\begin{equation}
\norm{\gls{residual}_{\gls{parameter}}(\gls{rsol}_{\gls{parameter}})}_{\gls{lfspacepou}'} \leq
\gls{contconst}_{\gls{parameter}} \norm{\nabla \left(u_{\gls{parameter}} - \gls{rsol}_{\gls{parameter}} \right)}_{L^2(\gls{subdomainol})}.
\end{equation}
\end{proposition}
\begin{proof}
Using the error identity 
\begin{equation}
a_{\gls{parameter}}(u_{\gls{parameter}} - \gls{rsol}_{\gls{parameter}}, \varphi) = \dualpair{\gls{residual}_{\gls{parameter}}(\gls{rsol}_{\gls{parameter}}) }{ \varphi },
\nonumber
\end{equation}
we obtain for any $\varphi \in \gls{lfspacepou}$
\begin{align}
	\dualpair{\gls{residual}_{\gls{parameter}}(\gls{rsol}_{\gls{parameter}})}{\varphi} &= \int_{\gls{domain}} \sigma_{\gls{parameter}}(x)\nabla(u_{\gls{parameter}} - \gls{rsol}_{\gls{parameter}})(x) \nabla \varphi(x) \dx \nonumber \\
	               &= \int_{\gls{subdomainol}} \gls{heat_conductivity}_{\gls{parameter}}(x)\nabla(u_{\gls{parameter}} - \gls{rsol}_{\gls{parameter}})(x) \nabla \varphi(x) \dx \nonumber \\
		       &\leq \gls{contconst}_{\gls{parameter}} \norm{\nabla \left(u_{\gls{parameter}} - \gls{rsol}_{\gls{parameter}} \right)}_{L^2(\gls{subdomainol})} \norm{\varphi}_{\gls{fspace}},
\end{align}
from which the statement follows.
\end{proof}
\begin{remark}[Improved efficiency estimates]
Under the assumptions of \cref{thm:local_efficiency},
we have the improved efficiency estimate
\begin{equation}
\gls{locerrest_g}
	\leq \frac{\gls{contconst}_{\gls{parameter}}  \sqrt{\gls{overlappingspaces}}  \gls{decompositionboundvr}}{\gls{coercconst}_{\gls{parameter}}}   \norm{ u_{\gls{parameter}} - \gls{rsol}_{\gls{parameter}} }_{\gls{fspace}}.
	\nonumber
\end{equation}
In many cases, a better constant can be found. Finite Element ansatz functions
are usually not orthogonal if they share support. So if 
$\gls{overlappingspaces}$ spaces have support in one point in space, they have
to be placed in different groups when designing a partition
for \cref{thm:efficiency}, so $\gls{overlappingspaces} \leq \gls{orthogonal_classes}$
 (cf.\ \cite[p. 67]{Toselli2005}).
\end{remark}
Reviewing the five desired properties of an a posteriori error estimator
at the beginning of this section, we see that the presented
error estimator is robust and efficient (1) and is offline/online
decomposable (2). Parallelization can be done over the spaces
$\gls{lfspacepou}$. Only online data has to be transferred, so there is 
little communication (3). The offline/online decomposition
only has to be repeated for a space $\gls{lfspacepou}$, if a new basis
function with support in $\gls{subdomainol}$ was added. So reuse in unchanged
regions is possible (4). How the adaptive enrichment is steered (5)
was already shown in \cref{sec:enrichment}.

\section{Relative Error Estimators}
\label{sec:relative error estimators}
From the error estimators for the absolute error, we can construct 
error estimators for the relative error. Estimates for the relative error
are given for example in \cite[Proposition 4.4]{Hesthaven2016},
but the estimates used here are slightly sharper.
The estimates presented in the following were already published in \cite[Lemma 3.2.2]{hessphdthesis},
but efficiency was not shown there.
\begin{proposition}[Relative error estimators]
Assuming $\norm{\gls{rsol}_{\gls{parameter}}}_{\gls{fspace}} > \gls{errest}$
and  $\norm{\gls{rsol}_{\gls{parameter}}}_{\gls{fspace}} > \gls{locerrest_g}$,
the error estimators defined by
\begin{align}
\gls{rerrest} 
    &:= \frac{\gls{errest}}{\norm{\gls{rsol}_{\gls{parameter}}}_{\gls{fspace}} - \gls{errest}}
\nonumber \\
\gls{rerrest4} 
    &:= \frac{\gls{locerrest_g}}{\norm{\gls{rsol}_{\gls{parameter}}}_{\gls{fspace}} - \gls{locerrest_g}}
\end{align}
are robust and efficient:
\begin{eqnarray}
\frac{\norm{\gls{sol}_{\gls{parameter}} - \gls{rsol}_{\gls{parameter}}}_{\gls{fspace}}}{\norm{\gls{sol}_{\gls{parameter}}}_{\gls{fspace}}}
&\leq \gls{rerrest}
&\leq \left( 1 + 2 \gls{rerrest} \right) \frac{\gls{contconst}_{\gls{parameter}}}{\gls{coercconst}_{\gls{parameter}}}
\frac{\norm{\gls{sol}_{\gls{parameter}} - \gls{rsol}_{\gls{parameter}}}_{\gls{fspace}}}{\norm{\gls{sol}_{\gls{parameter}}}_{\gls{fspace}}}
\nonumber \\
\frac{\norm{\gls{sol}_{\gls{parameter}} - \gls{rsol}_{\gls{parameter}}}_{\gls{fspace}}}{\norm{\gls{sol}_{\gls{parameter}}}_{\gls{fspace}}}
&\leq \gls{rerrest4}
&\leq \left( 1 + 2 \gls{rerrest4} \right) 
\frac{\gls{contconst}_{\gls{parameter}}}{\gls{coercconst}_{\gls{parameter}}}
        \gls{decompositionboundvr} \sqrt{\gls{orthogonal_classes}}
\frac{\norm{\gls{sol}_{\gls{parameter}} - \gls{rsol}_{\gls{parameter}}}_{\gls{fspace}}}{\norm{\gls{sol}_{\gls{parameter}}}_{\gls{fspace}}}
\end{eqnarray}
\end{proposition}See \cref{chap:relative error estimators} for
an extensive analysis of these relative error estimators.

\section{Improved Offline/Online Splitting}
\label{sec:improved splitting}
One crucial property of an a posteriori error estimator
for model order reduction is the
possibility of an offline/online splitting
of its evaluation.
An offline/online splitting is a scheme to evaluate
the a posteriori error estimator where all
computations in the online phase are independent of the dimension 
of the unreduced model.
For the standard residual based \gls{rb}
error estimator
\cref{eq:standard rb errest}
a splitting scheme based on the gram matrix
of all contributions to the
residual is widely used (see e.g. \cite[Sec.  4.2.5]{Hesthaven2016}).
We will call this splitting
the \oldsplit{}
in the following.
An offline/online splitting for the localized
a posteriori error estimators defined in 
\cref{thm:abstract_error_estimate}
can be obtained by applying the 
\oldsplit{}
to all localized dual space norms
$\norm{\gls{residual}_{\gls{parameter}}(\gls{sol}_{\gls{parameter}})}_{\gls{lfspace}'}$.
However, as observed by several authors
\cite[pp.\ 148--149]{Patera2007}\cite{Yano2014b}\cite{Benner2012}, 
the implementation of this offline/online splitting shows poor numerical
accuracy due to round-off errors which can render the estimator unusable
when the given problem is badly conditioned and high accuracy
is required.
Observations suggest that the estimator typically 
stagnates at a relative error of order $\sqrt{\varepsilon}$,
where $\varepsilon$ is the machine accuracy of the
floating point hardware used.

In the following, we propose a new algorithm to evaluate
the norm of the residual which does not
suffer the severe numerical problems of the traditional approach,
is free of approximations, has only small computational overhead
and is easy to implement.
It is based on representing
the residual with respect to a dedicated orthonormal basis.
The new approach is denoted by
\newsplit{} in the following.
It was published in \cite{Buhr2014}.

To our knowledge, at that time
there was only one other contribution in which a numerically stable algorithm
for evaluation of the estimator is presented
\cite{Casenave2012,Casenave2014}.
This approach however comes at the price of
a computationally more expensive ``online phase'' (in \cite{Casenave2012})
or increased complexity of
offline computations (in \cite{Casenave2014}) by application
of the empirical interpolation method, which in turn requires additional stabilization.
Later, two more approaches were published.
In \cite{Sommer2015} a splitting is proposed which is equivalent
to the \newsplit{} presented below. Only
a Gram-Schmidt algorithm is replaced by a singular value decomposition.
Later, in \cite{Chen2017,Chen2018a}, a splitting
is proposed which is also equivalent, 
but it was proposed to use a rank revealing QR decomposition
instead of the Gram-Schmidt algorithm.
In the following, we present
both the \oldsplit{}
as well as the \newsplit{}
for the computation of the dual norm of the residual
$\norm{\gls{residual}_{\gls{parameter}}(\gls{sol}_{\gls{parameter}})}_{\gls{lfspace}'}$.
We assume that $\gls{sol}_{\gls{parameter}}$
is the solution of a parametric variation problem
as defined in \cref{def:variational problem,sec:parameterized problem}
and that the bilinear form $\gls{a}_{\gls{parameter}}$
and the linear form $\gls{f}_{\gls{parameter}}$
have an affine decomposition.
As the splitting scheme is the same for the global
a posteriori error estimator and the localized setting,
we only consider the simpler global case here.

To calculate the dual norm of the residual $\gls{residual}_{\gls{parameter}}(\gls{rsol}_{\gls{parameter}})$ we make use of the
fact that the norm of an element of $\gls{rfspace}^\prime$ is equal to the norm
of its Riesz representative. Denoting by $\gls{riesz} : \gls{fspaceh}' \rightarrow \gls{fspaceh}$
the Riesz isomorphism and assuming the existence of a computable lower bound $\gls{coercconst}_{{\gls{parameter}},LB} \leq \gls{coercconst}_{\gls{parameter}}$ for the coercivity constant,
we obtain a bound for the error containing only computable quantities:
\begin{equation}
\label{eq:computable_bound}
\norm{\gls{sol}_{\gls{parameter}} - \gls{rsol}_{\gls{parameter}}}_{\gls{fspaceh}} \leq \frac{1}{\gls{coercconst}_{{\gls{parameter}},LB}}
\norm{\gls{riesz}(\gls{residual}_{\gls{parameter}}(\gls{rsol}_{\gls{parameter}}))}_{\gls{fspaceh}}
\end{equation}
\subsubsection{Traditional Offline/Online Splitting}
In order to avoid high-dimensional calculations
during the online phase, the residual $\gls{residual}_{\gls{parameter}}(\gls{rsol}_{\gls{parameter}})$ can be
rewritten using the affine decompositions \cref{eq:affineaf} and
a basis representation of $\gls{rsol}_{\gls{parameter}}$. Let
$\{\widetilde{\psi}_1, \dots, \widetilde{\psi}_\rbdim\}$
be a basis of $\gls{rfspace}$ and let  
$
\gls{rsol}_{\gls{parameter}} = \sum_{i=1}^{\rbdim} \coef{\gls{rsol}_{\gls{parameter}}}_i \widetilde{\psi}_i
$,
then the Riesz representative of the residual is given as
\begin{equation}
	\gls{riesz}(\gls{residual}_{\gls{parameter}}(\gls{rsol}_{\gls{parameter}})) = 
\sum_{q=1}^{Q_f} \theta_f^q({\gls{parameter}}) \gls{riesz}(f^q) - \sum_{q=1}^{Q_a} \sum_{i=1}^{\rbdim}
\theta_a^q({\gls{parameter}}) \coef{\gls{rsol}_{\gls{parameter}}}_i \gls{riesz}(a^q(\widetilde{\psi}_i,\, \cdot\,))\,.
\label{eq:residual}
\end{equation}
To simplify notation, we rename the $\etadim := Q_f$ + $Q_a  \rbdim$ linear 
coefficients $\theta_f^q({\gls{parameter}})$ and $\theta_a^q({\gls{parameter}}) \coef{\gls{rsol}_{\gls{parameter}}}_i$ to $\alpha_k$ 
and the vectors $\gls{riesz}(f^q)$ and $\gls{riesz}(a^q(\widetilde{\psi}_l,\cdot))$ to
$\eta_k$, i.e.~$\gls{riesz}(\gls{residual}_{\gls{parameter}}(\gls{rsol}_{\gls{parameter}})) = \sum_{k=1}^{\etadim} \alpha_k \eta_k$.
The space $\spn\{\eta_1, \dots, \eta_\etadim\}$ is denoted by $\etaspace$.
For the norm of the residual we obtain
\begin{equation}
\label{eq:traditional_offline_online}
\norm{\gls{riesz}(\gls{residual}_{\gls{parameter}}(\gls{rsol}_{\gls{parameter}}))}_{\gls{fspaceh}} = \left( \sum_{k=1}^{\etadim}
	\sum_{l=1}^{\etadim} \alpha_k \alpha_l \left( \eta_k , \eta_l
	\right)_{\gls{fspaceh}} \right)^\frac{1}{2}.
\end{equation}
Using this representation, an offline/online decomposition of the
error bound is possible by pre-computing the inner products $(\eta_k, \eta_l)_{\gls{fspaceh}}$
during the offline stage. In the
online stage, only the sum in \cref{eq:traditional_offline_online} has to be evaluated. As the number of summands is independent of the dimension of $\gls{fspaceh}$,
an online run-time independent of the dimension of $\gls{fspaceh}$ is achieved.

While this approach leads to an efficient computation of the residual
norm, it shows poor numerical stability: in the sum
\cref{eq:traditional_offline_online}, terms with a relative error of order of machine
accuracy $\varepsilon$ are added. Therefore, the sum shows an absolute error of at least $\varepsilon$ times the largest value of $|\alpha_k \alpha_l (\eta_k, \eta_l)_{\gls{fspaceh}}|$,
and the error in the norm of the residual 
is thus at least of order $\sqrt{\varepsilon}\cdot\sqrt{\max_{k,l}(|\alpha_k \alpha_l (\eta_k,
	\eta_l)_{\gls{fspaceh}}|)}$. This is in agreement with the observation that this
algorithm stops converging at relative errors of order
$\sqrt{\varepsilon}$ (see \cref{sec:thermal block numericalexample}).
\subsubsection{Improved Offline/Online Splitting}
While the floating point evaluation of \cref{eq:traditional_offline_online} shows poor
numerical accuracy, note that the evaluation of 
\begin{equation}
\label{eq:hdevaluation}
\norm{\gls{riesz}(\gls{residual}_{\gls{parameter}}(\gls{rsol}_{\gls{parameter}}))}_{\gls{fspaceh}} = 
\left(\sum_{k=1}^{\etadim} \alpha_k \eta_k,\sum_{k=1}^{\etadim} \alpha_k \eta_k\right)_{\gls{fspaceh}}^\frac{1}{2}
\end{equation}
is numerically stable.
Based on this observation, we propose a new algorithm to evaluate
$\norm{\gls{riesz}(\gls{residual}_{\gls{parameter}}(\gls{rsol}_{\gls{parameter}}))}_{\gls{fspaceh}}$ which is offline/online decomposable while maintaining the algorithmic
structure of \cref{eq:hdevaluation} to ensure stability.

The algorithm we propose evaluates \cref{eq:hdevaluation} in the subspace
$\etaspace$ 
using an orthonormal basis for this space. It comprises three steps:
1. The construction of an orthonormal basis $\Psi^\eta = \{ \etabase_1, \dots,
\etabase_\etadim \}$ of $\etaspace$,
2. the evaluation of the basis coefficients of $\eta_k$ w.r.t.~the basis $\Psi^\eta$ and
3. the evaluation of \cref{eq:hdevaluation} using this basis representation.
Note that this approach is offline/online decomposable: Steps 1 and 2 can be 
done offline, without knowing the parameter, while step 3 can be performed online. 
The size of the basis $\Psi^\eta$ does not depend on the dimension of $\gls{fspaceh}$.

In principle, any orthonormalization algorithm 
applied to $\{\eta_k\}_{k=1}^{\etadim}$ can be used for the computation of
the basis $\Psi^\eta$. Note, however, that the algorithm has to compute the basis
with very high numerical accuracy. As an example, the standard modified
Gram-Schmidt algorithm usually fails to deliver the required accuracy. For the
numerical example in \cref{sec:thermal block numericalexample}, we have chosen an
improved variant of the modified Gram-Schmidt algorithm, where vectors are
re-orthonormalized until a sufficient accuracy is achieved (\cref{alg:gram-schmidt adaptive}).

After the basis $\Psi^\eta$ has been constructed using an appropriate orthonormalization algorithm, 
we can compute for each $\eta_k$ ($1 \leq k \leq \etadim)$ basis representations
$\eta_k = \sum_{i=1}^{\etadim} \overline{\eta}_{k,i} \etabase_i$, where
$
\overline{\eta}_{k,i} = \left( \etabase_i , \eta_k \right) _{\gls{fspaceh}}
$ due to the orthonormality of $\Psi^\eta$.
The right-hand side of \cref{eq:hdevaluation} can then be evaluated as
\begin{equation}
\label{eq:newestimator}
\norm{\gls{riesz}(\gls{residual}_{\gls{parameter}}(\gls{rsol}_{\gls{parameter}}))}_{\gls{fspaceh}}
= \left( \sum_{i=1}^\etadim \left( \sum_{k=1}^\etadim \alpha_k \overline \eta_{k,i} \right)^2 \right)^\frac{1}{2}
\end{equation}
which executes in time independent of the dimension of $\gls{fspaceh}$ and is observed to be numerically stable.
\subsubsection{Run-Time Complexities}
During the offline phase, both the traditional and the new algorithm have to calculate
all Riesz representatives appearing in \cref{eq:residual}. This requires the application of the inverse 
of the inner product matrix for $\gls{fspaceh}$, which 
can be computed in complexity $\mathcal{O}(\hdim \log(\hdim))$ 
with appropriate preconditioners.
As there are $\etadim$ Riesz representatives to be calculated, the overall run-time
of this step is of order $\mathcal{O}(\etadim \hdim \log(\hdim))$.
The traditional algorithm proceeds with calculating all
inner products $(\eta_k,\eta_l)_{\gls{fspace}}$ in \cref{eq:traditional_offline_online}, having a complexity of $\mathcal{O}(\etadim^2 \hdim)$.
Thus the overall complexity of the offline phase for the traditional algorithm is
$\mathcal{O}(\etadim^2 \hdim + \etadim \hdim \log(\hdim))$.

After computing the Riesz representatives in \cref{eq:residual},
the improved algorithm generates the orthonormal basis $\Psi^\eta$.
In practice it was observed that at most four re-iterations per vector are required during
orthonormalization with \cref{alg:gram-schmidt adaptive}. Thus, 
choosing this algorithm for the generation of $\Psi^\eta$ leads
to a run-time complexity of $\mathcal{O}(\etadim^2 \hdim)$ for this step.
The calculation of the $\etadim^2$ basis coefficients
$\overline \eta_{k,i} = \left( \eta_k , \etabase_i \right) _{\gls{fspaceh}}$
has again complexity $\mathcal{O}(\etadim^2 \hdim)$,
resulting in a total complexity of the offline phase for the new algorithm of
$\mathcal{O}(\etadim^2 \hdim + \etadim \hdim \log(\hdim))$, as for the traditional algorithm.

During the online phase, the right-hand sides of \cref{eq:traditional_offline_online}, resp.~\cref{eq:newestimator},
are evaluated using the pre-computed quantities $(\eta_k,\eta_l)_{\gls{fspace}}$, resp.~$\overline \eta_{k,i}$.
In both cases, a run-time of $ \mathcal{O}(\etadim^2)$ is required.

All in all, both algorithms for evaluating \cref{eq:computable_bound} show the same
run-time complexity, in the online phase as well as during the offline phase (\cref{tab:complexities}).
Note that 
$\mathcal{O}(\etadim^2) = \mathcal{O}(Q_f^2 + Q_a^2 \rbdim^2) = \mathcal{O}(\rbdim^2)$
for increasing reduced space dimensions.
\subsection*{Numerical Example}
\label{sec:thermal block numericalexample}
To demonstrate the performance of the
\newsplit{},
we apply it to the \gls{rb} approximation of \exa{}.
Since the splitting is the same in the
standard \gls{rb} and in the localized setting,
we apply it to the simpler case of standard \gls{rb}.
For basis generation, we weak greedy,
steered by the a posteriori error estimator
under consideration is used:
The reduced space $\gls{rfspace}$ is constructed from the linear span 
of solutions
to \cref{def:variational problem} for parameters selected by the following greedy search procedure:
Starting with $\gls{rfspace}_0 := \{ 0 \} \subset \gls{fspaceh}$,
in each iteration step the reduced problem \cref{def:reduced variational problem}
is solved and an error estimator is evaluated at all parameters
${\gls{parameter}}$ of a given discrete training set $\gls{trainingset} \subset \gls{parameterspace}$.
If the maximum estimated error is below a prescribed tolerance $tol$,
the algorithm stops. Otherwise, the high-dimensional problem \cref{def:variational problem}
is solved for the parameter ${\gls{parameter}}^*_n$ maximizing the estimated error
and the reduced space is extended by the obtained solution snapshot: 
$\gls{rfspace}_{n+1} := \gls{rfspace}_{n} \oplus \spn\{ u_{{\gls{parameter}}^*_n}\}$.

A reduced space of dimension 40 was generated with the weak greedy algorithm
using our new algorithm for the evaluation of the error estimator.
An equidistant training set of $6^4$ parameters was used.
Finally, for each $n$-dimensional reduced subspace $\tilde{V}_n$ ($0\leq n \leq
40)$ produced by the greedy algorithm we computed the maximum reduction error and the maximum estimated reduction errors
using both the traditional and our improved algorithm on 30 randomly selected new parameters in
$\gls{parameterspace}$ (\cref{fig:errors}).
Our results clearly indicate the breakdown
of the traditional algorithm for more than 25 basis vectors at a relative error of about $10^{-7} \approx
\sqrt{\varepsilon}$ whereas our new algorithm remains efficient for all tested basis sizes.

To underline the need for accurate error estimation in order to obtain reduced spaces of high 
approximation quality, we repeated the same experiment using the traditional algorithm for error
estimation during basis generation (\cref{fig:traderrors}). While the maximum model reduction error
still improves from $10^{-7}$ to $10^{-8}$ after the breakdown of the error estimator,
the final reduced space approximates the solution manifold 3 orders of magnitude worse than
the space obtained with our improved algorithm.

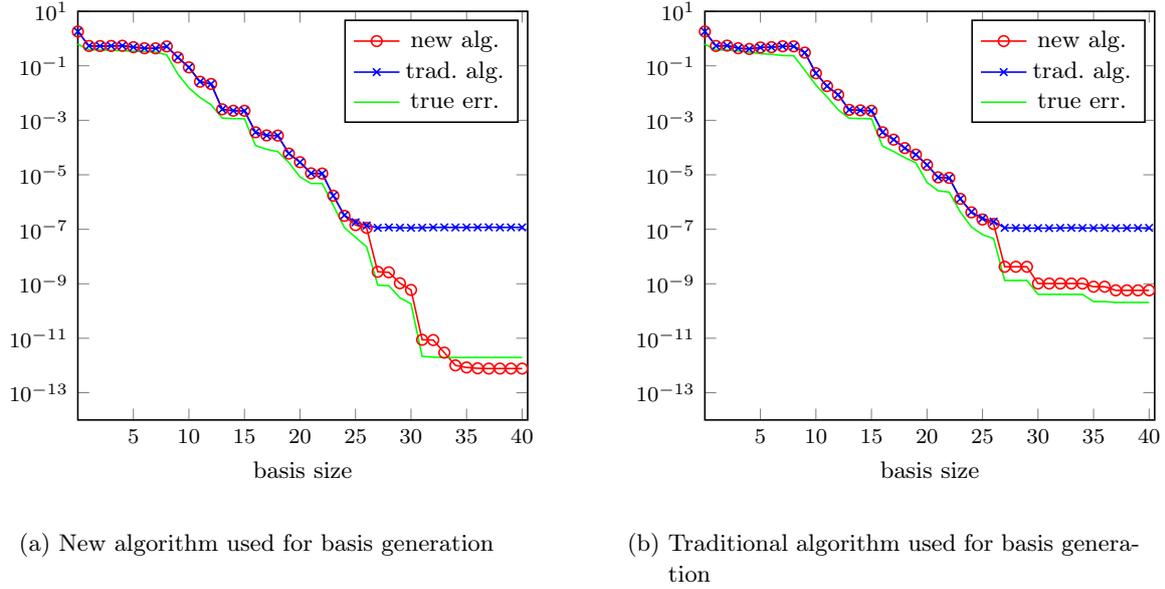
\begin{figure}
\begin{subfigure}[t]{0.45\textwidth}
\begin{center}
	\begin{tikzpicture}
		\begin{semilogyaxis}[xmin=0, xmax=40.5, ymin=1e-14, ymax=10, xtick={5, 10, 15, 20, 25, 30, 35, 40},
			semithick, small, width=7.5cm, height=7.0cm, %
			minor tick style={draw=none},
                                     xlabel={basis size}, legend style={font=\small}]
			\addplot [mark=o, red] table[x=N, y=estnew] {code/thermalblock/residual_basis.txt};
			\addplot [mark=x, blue] table[x=N, y=esttrad] {code/thermalblock/residual_basis.txt};
			\addplot [mark=none, green] table[x=N, y=err] {code/thermalblock/residual_basis.txt};
                        \legend{new alg., trad.~alg., true err.}
		\end{semilogyaxis}
	\end{tikzpicture}
\end{center}
\caption{New algorithm used for basis generation}
\label{fig:errors}
\end{subfigure}
\hfill
\begin{subfigure}[t]{0.45\textwidth}
\begin{center}
	\begin{tikzpicture}
		\begin{semilogyaxis}[xmin=0, xmax=40.5, ymin=1e-14, ymax=10, xtick={5, 10, 15, 20, 25, 30, 35, 40},
			semithick, small, width=7.5cm, height=7.0cm, %
			minor tick style={draw=none},
                                     xlabel={basis size}, legend style={font=\small}]
			\addplot [mark=o, red] table[x=N, y=estnew] {code/thermalblock/traditional.txt};
			\addplot [mark=x, blue] table[x=N, y=esttrad] {code/thermalblock/traditional.txt};
			\addplot [mark=none, green] table[x=N, y=err] {code/thermalblock/traditional.txt};
                        \legend{new alg., trad.~alg., true err.}
		\end{semilogyaxis}
	\end{tikzpicture}
\end{center}
\caption{Traditional algorithm used for basis generation}
\label{fig:traderrors}
\end{subfigure}
\caption[Maximum relative reduction errors and estimated errors in \exa{}.]
{Evaluation of proposed offline/online splitting: maximum relative reduction errors and estimated reduction
		errors ($H^1$-norm) for numerical example \exa{}.
                (reproduction: \cref{repro:fig:staberrest})
}
\label{fig:staberrest}
\end{figure}

\begin{table}
\caption[Run-time complexities of traditional and new algorithm for error estimator.]
{Run-time complexities of traditional and new algorithm for evaluation of the error estimator.}
\label{tab:complexities}
\begin{center}
\setlength{\tabcolsep}{4.5mm}
\begin{tabular}{rcc}
\toprule
stage & offline & online \\
\midrule
traditional  & $\mathcal{O}(\etadim^2 \hdim) + \mathcal{O}(\etadim \hdim \log(\hdim))$ & $\mathcal{O}(\etadim^2)$ \\
new  & $\mathcal{O}(\etadim^2 \hdim) + \mathcal{O}(\etadim^2 \hdim) + \mathcal{O}(\etadim \hdim \log(\hdim))$ & $\mathcal{O}(\etadim^2)$ \\
\bottomrule
\end{tabular}
\end{center}
\end{table}

\futurebox{
\section{Future Research on Localized a Posteriori Error Estimation}
\myline{}
\begin{itemize}
\item
\textbfit{Localization of randomized a posteriori error estimators}\\
Recently, a residual based, randomized a posteriori error
estimator has been proposed in \cite{Smetana2019},
having favorable properties.
Especially, it circumvents the explicit computation of the
inf-sup constant and it shows a very good efficiency.
It should be investigated if and how this approach
could be adapted to the localized setting.
\item
\textbfit{Online computation of estimator constants}\\
While we present a priori estimates for some of the
constants involved in the localized a posteriori
error estimators, it should be possible
to compute the localized constants
\gls{mapboundvr} on the fly using
our randomized norm estimator 
shown in \cref{sec:random a posteriori}.
It should also be possible to update
the estimates during basis adaption and
online enrichment, when the reduced
ansatz space grows and thus the constants
\gls{mapboundvr} shrink.
\end{itemize}
}

\chapter{Improved Local Training}
\label{chap:training}
The training algorithm introduced in
\cref{sec:codim_n_training},
\cref{algo:training},
is a heuristic approach.
It has the drawbacks of
first having the parameter $M$, the number of random samples,
which has to be chosen by the user,
and second it lacks provable convergence results.
In this chapter, we introduce an improved training algorithm
which chooses the number of random samples
adaptively, based on a prescribed accuracy
and for which probabilistic a priori convergence bounds can be shown.
The probabilistic a priori bounds hold with high (selectable) probability.
First we define an abstract training configuration,
corresponding to an abstract localizing space decomposition.
Then we show an a priori error estimate in this abstract setting.
We proceed introducing the randomized range recovery algorithm,
which generates local approximation spaces and controls
one of the terms in the abstract estimate.
We verify our results on two simple test examples,
\exrangea{} and \exrangeb{},
and evaluate the behavior of the range recovery algorithm on the \exd{}.
\section{Abstract Training Configuration}
Approximating the image of a suitably defined transfer operator
to create approximation spaces in localized model order
reduction has been done in \cite{bablip11} for the \gls{gfem}
and in \cite{Smetana2016} for the \gls{prscrbe}.
This formulation has been combined with a randomized approach
by the authors in \cite{Buhr2018} for the \gls{gfem}.
All of these approaches were specific to one space decomposition.
In the following, we introduce a ``localized training
configuration'' which provides an abstract framework
for the generation of localized approximation spaces,
based on an abstract localizing space decomposition.
\begin{definition}[Localized training configuration]
\label{def:localized training configuration}
Let 
$\left\{
\gls{domain}, \gls{fspace},
\{\gls{subdomain}\}_{i=1}^{\gls{numspaces}},
\{\gls{lfspace}\}_{i=1}^{\gls{numspaces}},
\{\gls{lfspacemap}\}_{i=1}^{\gls{numspaces}}
\right\}$
be a localizing space decomposition as defined in
\cref{def:localizing space decomposition}
and let \gls{sol} be the solution of a non parametric variational problem 
as defined in \cref{def:variational problem}.
We define a ``localized training configuration''
to be a set of
\begin{enumerate}
\item
\gls{numspaces} training domains $\gls{traindomain} \subset \gls{domain}$
with the maximum 
number overlapping in any point $x$ of $\gls{domain}$
denoted by
\begin{equation}
\gls{overlappingspacestrain} := \max_{x \in \gls{domain}} \#\{i \in \{1, \dots, \gls{numspaces}\} | \ x \in \gls{traindomain}\}
,
\end{equation}
\item
\gls{numspaces} Hilbert spaces
$\gls{lfspacetrain}$
equipped with a norm satisfying
\begin{equation}
\sum_{i=1}^{\gls{numspaces}} 
\norm{\gls{sol}|_{\gls{traindomain}}}^2
\leq
\gls{overlappingspacestrain}
\norm{\gls{sol}}^2,
\end{equation}
\item
\gls{numspaces} Hilbert spaces \gls{tracespace},
\item
\gls{numspaces} bounded linear operators $\gls{tracespacemap} : \gls{lfspacetrain} \rightarrow \gls{tracespace}$,
\item
\gls{numspaces} affine linear and compact operators $\gls{transferop} : \gls{tracespace} \rightarrow \gls{lfspace}$
\end{enumerate}
for which it holds 
\begin{equation}
\label{eq:abstractapriori1}
\gls{sol} = \sum_{i=1}^{\gls{numspaces}} \gls{transferop}\gls{tracespacemap}\gls{sol}|_{\gls{traindomain}} .
\end{equation}
In the case of a parameterized problem as introduced in 
\cref{sec:parameterized problem}, the operators \gls{transferop}
are replaced by parameterized operators \gls{transferopmu} for which it holds
\begin{equation}
\gls{sol}_{\gls{parameter}}
= \sum_{i=1}^{\gls{numspaces}} \gls{transferopmu}\gls{tracespacemap}\gls{sol}_{\gls{parameter}}|_{\gls{traindomain}}
\qquad \forall \gls{parameter} \in \gls{parameterspace}
.
\end{equation}
\end{definition}
In the \gls{arbilomod} setting,
the training domains \gls{traindomain} are
the support of the spaces $\training{\gls{lfspacewb}}$
defined in \cref{eq:arbilomod training}.
These domains are usually
oversampling domains of the domains \gls{subdomain},
but this is not required.
The spaces \gls{lfspacetrain} can usually be
defined as 
$\gls{lfspacetrain} := \left\{ \varphi|_{\gls{traindomain}} \ \Big| \ \varphi \in \gls{fspace} \right\}$
with the norm and inner product of \gls{fspace} where the integrals are restricted to \gls{traindomain}.
Some norms might however not be a norm when restricted to a subdomain, 
as for example the $H^1$ semi-norm on $\gls{h1}$ or the energy norm of a Laplace problem.
In these cases, the definition has to be modified.
In the \gls{arbilomod} setting, the simple restriction can be used.
In the choice of the spaces \gls{tracespace}, there is a lot of freedom
when designing a localized training configuration.
Their inner product can be chosen to obtain good constants in the
estimates presented below or to make constants easily computable.
The spaces \gls{tracespace} usually represent some form of boundary
values, e.g.~Dirichlet or Robin boundary values on
$\partial \gls{traindomain}$.
In the \gls{arbilomod} setting, \gls{tracespace}
are the spaces 
$\coupling{\training{\gls{lfspacewb}}}$
defined in \cref{eq:arbilomod training coupling},
which are spanned by all \gls{fe} functions
having support on $\partial \gls{traindomain}$.
The operators $\gls{tracespacemap}$
are usually the corresponding trace operators, 
or the extraction of the \gls{fe}-coefficients
of $\coupling{\training{\gls{lfspacewb}}}$
in the \gls{arbilomod}.
The operators \gls{transferop}, which we denote
by the name ``transfer operators'',
are usually defined by solving a local 
version of the underlying problem with the boundary
values in \gls{tracespace}.
In \gls{arbilomod}, this corresponds
to the local solve in \cref{eq:trainingproblem}.

In the following, we show an a priori error estimate
in this abstract setting.
We denote the linear part of \gls{transferop} by \gls{transferopl} and the
affine part by $\gls{transferopa} \in \gls{lfspace}$, it holds
\begin{equation}
\gls{transferop}\varphi = \gls{transferopl}\varphi + \gls{transferopa}
\qquad \forall \varphi \in \gls{tracespace}
.
\end{equation}
As in \cref{sec:abstract estimates}
let
\begin{equation}
\label{eq:classes in training}
\dot{\bigcup}_{k=1}^{\gls{orthogonal_classes}} \Upsilon_{C,k} = \left\{1, \dots, \gls{numspaces}\right\}
\end{equation}
be a partition of $\left\{1, \dots, \gls{numspaces}\right\}$
such that
\begin{equation}
       \forall 1\leq k \leq \gls{orthogonal_classes}\ : \ \forall i,j \in \Upsilon_{C,k}, \ i\ne j \ :\ V_i \perp V_j
\end{equation}
i.e.~we can partition the local subspaces into \gls{orthogonal_classes} sets
so that in each set, all spaces are orthogonal.
Using these classes, we can formulate the following lemma.

\begin{lemma}[Norm of sum of localized functions]
\label{thm:better than triangle}
For functions $\{\varphi_i\}_{i=1}^{\gls{numspaces}}$
with
\begin{equation}
\varphi_i \in \gls{lfspace} \qquad \forall i \in \{1, \dots, \gls{numspaces} \}
\end{equation}
it holds
\begin{equation}
\norm{\sum_{i=1}^{\gls{numspaces}} \varphi_i}^2 \leq \gls{orthogonal_classes} \sum_{i=1}^{\gls{numspaces}} \norm{\varphi_i}^2
.
\end{equation}
\end{lemma}
\begin{proof}
Using the classes described in \cref{eq:classes in training}, we have the identity
\begin{equation}
\norm{\sum_{i=1}^{\gls{numspaces}} \varphi_i}
=
\norm{\sum_{k=1}^{\gls{orthogonal_classes}} \sum_{i\in \Upsilon_{C,k}} \varphi_i}
.
\end{equation}
Using the triangle inequality and Cauchy-Schwarz, it follows
\begin{align}
\norm{\sum_{i=1}^{\gls{numspaces}} \varphi_i}
&\leq
\sum_{k=1}^{\gls{orthogonal_classes}} \norm{\sum_{i\in \Upsilon_{C,k}} \varphi_i}
\\
&\leq
\sqrt{\gls{orthogonal_classes}}
\left(\sum_{k=1}^{\gls{orthogonal_classes}} \norm{\sum_{i\in \Upsilon_{C,k}} \varphi_i}^2\right)^{\frac 1 2}
.
\end{align}
Due to the orthogonality of the spaces in each class $\Upsilon_{C,k}$
it holds 
\begin{equation}
\norm{\sum_{i\in \Upsilon_{C,k}} \varphi_i}^2
=
\sum_{i\in \Upsilon_{C,k}} \norm{\varphi_i}^2
\end{equation}
and it follows
\begin{align}
\norm{\sum_{i=1}^{\gls{numspaces}} \varphi_i}
&\leq
\sqrt{\gls{orthogonal_classes}}
\left(\sum_{k=1}^{\gls{orthogonal_classes}} \sum_{i\in \Upsilon_{C,k}} \norm{\varphi_i}^2\right)^{\frac 1 2}
\\
&=
\sqrt{\gls{orthogonal_classes}}
\left(\sum_{i=1}^{\gls{numspaces}} \norm{\varphi_i}^2\right)^{\frac 1 2}
.
\end{align}
\end{proof}
The following proposition is inspired by \cite[Theorem 4.6]{EickhornThesis} but
improves on this, as the error bound derived \cite[Theorem 4.6]{EickhornThesis}
scales in $\sqrt{\gls{numspaces}}$ while the error bound
presented in the following is independent of the number of local spaces
and local domains.
\begin{proposition}[Training a priori estimate]
\label{thm:training a priori}
Let 
$\left\{
\gls{domain}, \gls{fspace},
\{\gls{subdomain}\}_{i=1}^{\gls{numspaces}},
\{\gls{lfspace}\}_{i=1}^{\gls{numspaces}},
\{\gls{lfspacemap}\}_{i=1}^{\gls{numspaces}}
\right\}$
be a localizing space decomposition as defined in
\cref{def:localizing space decomposition}.
Let further\\
$\left\{
\{\gls{traindomain}\}_{i=1}^{\gls{numspaces}},
\{\gls{lfspacetrain}\}_{i=1}^{\gls{numspaces}},
\{\gls{tracespace}\}_{i=1}^{\gls{numspaces}},
\{\gls{tracespacemap}\}_{i=1}^{\gls{numspaces}},
\{\gls{transferop}\}_{i=1}^{\gls{numspaces}}
\right\}$
be a localized training configuration
as defined in \cref{def:localized training configuration}.
Let \gls{sol} be the solution of a non parametric, symmetric, coercive variational problem 
as defined in \cref{def:variational problem}
and let \gls{rsol} be the reduced solution of a 
localized Galerkin projection onto $\gls{rfspace} = \sum_{i=1}^{\gls{numspaces}} \gls{rlfspace}$ as introduced in \cref{sec:lmor}.
Further assume
\begin{equation}
\gls{transferopa} \in \gls{rlfspace} \qquad \forall i \in \{1, \dots, \gls{numspaces} \}.
\end{equation}
Then it holds
\begin{equation}
\frac{\norm{\gls{sol} - \gls{rsol}}}{\norm{\gls{sol}}}
\leq
\sqrt{\frac{\gls{contconst}}{\gls{coercconst}}}
\sqrt{\gls{orthogonal_classes} \gls{overlappingspacestrain}}
\max_{i \in \{1, \dots, \gls{numspaces}\}}
\norm{\left(1 - \gls{orthogonalp}_{\gls{rlfspace}} \right) \gls{transferopl}}
\norm{\gls{tracespacemap}}
.
\end{equation}
\end{proposition}

\begin{proof}
Due to the coercivity and symmetry of the problem, the best approximation theorem
(\cref{thm:coercive bestapproximation}) holds. We have
\begin{equation}
\norm{\gls{sol} - \gls{rsol}} \leq \sqrt{\frac{\gls{contconst}}{\gls{coercconst}}} \inf_{\widetilde v \in \gls{rfspace}} \norm{\gls{sol} - \widetilde v}
\end{equation}
and thereby
\begin{equation}
\label{eq:abstractapriorisum}
\norm{\gls{sol} - \gls{rsol}} \leq \sqrt{\frac{\gls{contconst}}{\gls{coercconst}}} \norm{\gls{sol} - u_c}
\end{equation}
with
\begin{equation}
\label{eq:abstractapriori2}
u_c := 
\sum_{i=1}^{\gls{numspaces}} \gls{orthogonalp}_{\gls{rlfspace}}\gls{transferop}\gls{tracespacemap}\gls{sol}|_{\gls{traindomain}}
.
\end{equation}
Plugging \cref{eq:abstractapriori1} and \cref{eq:abstractapriori2} into \cref{eq:abstractapriorisum},
we obtain
\begin{align}
\norm{\gls{sol} - \gls{rsol}} &\leq \sqrt{\frac{\gls{contconst}}{\gls{coercconst}}} \norm{
\sum_{i=1}^{\gls{numspaces}}                                   \gls{transferop}\gls{tracespacemap}\gls{sol}|_{\gls{traindomain}} - 
\sum_{i=1}^{\gls{numspaces}} \gls{orthogonalp}_{\gls{rlfspace}}\gls{transferop}\gls{tracespacemap}\gls{sol}|_{\gls{traindomain}}}\nonumber \\
&=
\sqrt{\frac{\gls{contconst}}{\gls{coercconst}}} \norm{
\sum_{i=1}^{\gls{numspaces}} \left(1 - \gls{orthogonalp}_{\gls{rlfspace}} \right)                             \gls{transferop}\gls{tracespacemap}\gls{sol}|_{\gls{traindomain}}}
.
\end{align}
Splitting the operator \gls{transferop} into its linear and its affine part, it follows that
\begin{align}
\norm{\gls{sol} - \gls{rsol}}
&\leq
\sqrt{\frac{\gls{contconst}}{\gls{coercconst}}}
\norm{
\sum_{i=1}^{\gls{numspaces}} \left(1 - \gls{orthogonalp}_{\gls{rlfspace}} \right) \gls{transferopl}\gls{tracespacemap}\gls{sol}|_{\gls{traindomain}}
+
\left(1 - \gls{orthogonalp}_{\gls{rlfspace}} \right) \gls{transferopa}
}
\nonumber \\
&=
\sqrt{\frac{\gls{contconst}}{\gls{coercconst}}}
\norm{
\sum_{i=1}^{\gls{numspaces}} \left(1 - \gls{orthogonalp}_{\gls{rlfspace}} \right) \gls{transferopl}\gls{tracespacemap}\gls{sol}|_{\gls{traindomain}}
}
.
\end{align}
In the last step, we used the assumption that 
$\gls{transferopa} \in \gls{rlfspace}$.
Using \cref{thm:better than triangle}
we see that
\begin{align}
\norm{\gls{sol} - \gls{rsol}}^2
\leq
\frac{\gls{contconst}}{\gls{coercconst}}
\gls{orthogonal_classes}
\sum_{i=1}^{\gls{numspaces}} \norm{\left(1 - \gls{orthogonalp}_{\gls{rlfspace}} \right) \gls{transferopl}\gls{tracespacemap}\gls{sol}|_{\gls{traindomain}}
}^2
.
\end{align}
Using the operator norms
\begin{equation}
\norm{\left(1 - \gls{orthogonalp}_{\gls{rlfspace}} \right) \gls{transferopl}\gls{tracespacemap}\gls{sol}|_{\gls{traindomain}}}
\leq
\norm{\left(1 - \gls{orthogonalp}_{\gls{rlfspace}} \right) \gls{transferopl}}
\norm{\gls{tracespacemap}}
\norm{\gls{sol}|_{\gls{traindomain}}}
\end{equation}
it follows that
\begin{align}
\norm{\gls{sol} - \gls{rsol}}^2
&\leq
\frac{\gls{contconst}}{\gls{coercconst}}
\gls{orthogonal_classes}
\sum_{i=1}^{\gls{numspaces}}
\norm{\left(1 - \gls{orthogonalp}_{\gls{rlfspace}} \right) \gls{transferopl}}^2
\norm{\gls{tracespacemap}}^2
\norm{\gls{sol}|_{\gls{traindomain}}}^2
\\
&\leq
\frac{\gls{contconst}}{\gls{coercconst}}
\gls{orthogonal_classes}
\max_{i \in \{1, \dots, \gls{numspaces}\}}
\norm{\left(1 - \gls{orthogonalp}_{\gls{rlfspace}} \right) \gls{transferopl}}^2
\norm{\gls{tracespacemap}}^2
\quad
\sum_{i=1}^{\gls{numspaces}} 
\norm{\gls{sol}|_{\gls{traindomain}}}^2
\\
&\leq
\frac{\gls{contconst}}{\gls{coercconst}}
\gls{orthogonal_classes}
\max_{i \in \{1, \dots, \gls{numspaces}\}}
\norm{\left(1 - \gls{orthogonalp}_{\gls{rlfspace}} \right) \gls{transferopl}}^2
\norm{\gls{tracespacemap}}^2
\quad
\gls{overlappingspacestrain}
\norm{\gls{sol}}^2
\end{align}
and thereby the claim.
\end{proof}
\cref{thm:training a priori} gives guidance when designing localized
training configurations. While the term 
$\norm{\left(1 - \gls{orthogonalp}_{\gls{rlfspace}} \right) \gls{transferopl}}$
will be controlled by the algorithm given below,
the operator norms $\norm{\gls{tracespacemap}}$ have to be
controlled otherwise. 
\section{Randomized Range Finder Algorithm}
\label{sec:randomized_range_finder}
In this section we present the randomized range finder algorithm, which, given a compact linear
operator between two finite dimensional Hilbert spaces, generates a space to approximate
the image of the given operator.
It is intended to be used to generate spaces \gls{rlfspace} 
and control the term 
$\norm{\left(1 - \gls{orthogonalp}_{\gls{rlfspace}} \right) \gls{transferopl}}$
in \cref{thm:training a priori}.
While this algorithm is used to create local approximation spaces in this thesis,
it is presented in an abstract setting in this section in order to 
facilitate reuse in other contexts.
We drop the index $i$ and denote the space the transfer operator
maps to by $\gls{transfer_range}$ instead of $\gls{lfspace}$.
The reduced approximation space we denote by $\gls{transfer_rrange}$
instead of $\gls{rlfspace}$ where $n$ is the dimension
of $\gls{transfer_rrange}$.
We only consider the non parametric, real valued case in this chapter.
The parametric case and the complex valued case are subject to future
research, see also
\cref{sec:future of randomized range finder}.

We restrict ourselves to affine linear operators between finite dimensional Hilbert spaces.
The restriction to finite dimensional spaces originates in the way we draw random vectors.
An extension to compact operators between Hilbert spaces of infinite dimension
would be possible but would require a new way of drawing random vectors.
An extension to nonlinear operators would be a larger project.
The restriction to operators between finite dimensional spaces is hardly a 
restriction in practice, as one usually uses finite dimensional approximations
of infinite dimensional spaces anyway in numerical codes.

The algorithm is an adaption of the range finder algorithm
from randomized numerical linear algebra (RandNLA) to the setting
of linear compact operators between Hilbert spaces. A version of this algorithm for matrices
is given in \cite[Algorithm 4.2]{Halko2011}.
The results of \cite{Halko2011} are recovered if we
discretize the linear operator in an orthonormal
basis.

\subsection{Randomized Numerical Linear Algebra}
Randomized algorithms have gained considerable attention
in the field of numerical linear algebra during the last decade.
In many cases, randomized algorithms have better performance
and better parallelization behavior
than classical algorithms while delivering comparable accuracy.
The interest of the artificial intelligence (AI) and 
machine learning (ML) communities in approximate matrix
factorizations of large matrices has accelerated this
development.

Some of the early publications by Drineas, 
Kannan, and Mahoney
in 2006 used the term 
``Fast Monte Carlo Algorithms for Matrices''
\cite{randmcla1,randmcla2,DrKaMa06b}.
They published algorithms for randomized
matrix multiplication, randomized low-rank approximation
of matrices and randomized compressed matrix decompositions.
Per-Gunnar Martinsson and coworkers
published a series of articles
\cite{CGMR05,liberty2007randomized,MaRoTy11}
and the highly cited SIAM review article \cite{Halko2011}.
Recently, also open-source software was released \cite{VorMar15}.

The content of this section has been developed in collaboration with
Kathrin Smetana and has been published in \cite{Buhr2018}.
See references therein for an extensive literature review.

\subsection{Setting}
Let $\gls{transferoperator}$ be an affine linear operator between the finite dimensional 
Hilbert spaces $\gls{transfer_source}$ and $\gls{transfer_range}$
\begin{equation}
\gls{transferoperator}: \gls{transfer_source} \rightarrow \gls{transfer_range}
\end{equation}
where $\gls{transfer_source}$ is of dimension $\gls{dimension_transfer_source}$ and $\gls{transfer_range}$ is of dimension $\gls{dimension_transfer_range}$.
We denote the linear part of \gls{transferoperator} by \gls{transferoperatorl} and the
affine part by $\gls{transferoperatora} \in \gls{transfer_range}$, it holds
\begin{equation}
\gls{transferoperator}\varphi = \gls{transferoperatorl}\varphi + \gls{transferoperatora}
\qquad \forall \varphi \in \gls{transfer_source}
.
\end{equation}
Let further $\psi_1^S, \dots, \psi_{\gls{dimension_transfer_source}}^S$ be a basis of $\gls{transfer_source}$
and
$\psi_1^R, \dots, \psi_{\gls{dimension_transfer_range}}^R$ be a basis of $\gls{transfer_range}$.

\gls{transferoperator}he target quantity in the approximation of the image of $\gls{transferoperator}$
is the operator norm
\begin{equation}
\norm{\left(1 - \gls{orthogonalp}_{\gls{transfer_rrange}} \right) \gls{transferoperatorl}} = \sup_{\varphi \in \gls{transfer_source} \setminus \{0\}}
\frac{\norm{\left(1 - \gls{orthogonalp}_{\gls{transfer_rrange}} \right) \gls{transferoperatorl}\varphi}_R}{\norm{\varphi}_S}
\end{equation}
where $\gls{orthogonalp}_{\gls{transfer_rrange}}$ is the orthogonal projection on
the approximation space $\gls{transfer_rrange} \subset \gls{transfer_range}$, which is
to be generated.

To distinguish clearly between vectors in Hilbert spaces and their
coordinate representation, 
we mark all coordinate vectors and matrices with an underline and
introduce Ritz isomorphisms 
$\gls{ritzds}: \gls{transfer_source} \rightarrow \gls{R}^{\gls{dimension_transfer_source}}$ and $D_R: \gls{transfer_range} \rightarrow \gls{R}^{\gls{dimension_transfer_range}}$ which map elements 
from $\gls{transfer_source}$ or $\gls{transfer_range}$ to a vector containing their basis coefficients in $\gls{R}^{\gls{dimension_transfer_source}}$ or $\gls{R}^{\gls{dimension_transfer_range}}$, 
respectively. For instance, $D_{S}$ maps a function $\xi = \sum_{i=1}^{\gls{dimension_transfer_source}}\underline{\xi}_{i} \psi_i^S \in \gls{transfer_source}$ 
to $\underline{\xi} \in \gls{R}^{\gls{dimension_transfer_source}}$. As a result we have the matrix of the operator $\gls{transferoperator}$ 
as $\underline{\gls{transferoperator}} = D_{R} \gls{transferoperator} D_{S}^{-1}$.

Finally, we denote by $\gls{source_inner}$ the inner product matrix of $\gls{transfer_source}$ and by $\gls{range_inner}$ the inner product matrix of $\gls{transfer_range}$.
\subsection{Optimal Spaces}
The best possible space of given dimension $n$ to minimize
$\norm{\left(1 - \gls{orthogonalp}_{\gls{transfer_rrange}} \right) \gls{transferoperatorl}}$
is the space spanned by the first $n$ left singular vectors of  \gls{transferoperatorl}.
The \gls{svd} exists due to the compactness of \gls{transferoperatorl}.
Denoting the \gls{svd} of \gls{transferoperatorl} by
\begin{equation}
\gls{transferoperatorl} \varphi = \sum_{i=1}^{N_T} \gls{lsv}_i \gls{singular_value}_i (\gls{rsv}_i, \varphi)_{\gls{transfer_source}} \qquad \forall \varphi \in \gls{transfer_source},
\end{equation}
where $N_T$ is the rank of \gls{transferoperatorl} and it holds
$(\gls{lsv}_i, \gls{lsv}_j)_{\gls{transfer_range}} = \delta_{ij}$,
$(\gls{rsv}_i, \gls{rsv}_j)_{\gls{transfer_source}} = \delta_{ij}$
and
$\gls{singular_value}_i \in \gls{R}^+$.
The optimal spaces defined as
\begin{equation}
\widetilde R^n_{\mathrm{opt}} := \mathrm{span} \{\gls{lsv}_1, \dots, \gls{lsv}_n\}
\end{equation}
minimize
$\norm{\left(1 - \gls{orthogonalp}_{\gls{transfer_rrange}} \right) \gls{transferoperatorl}}$
among all spaces of dimension $n$. It holds
\begin{align}
\norm{\left(1 - \gls{orthogonalp}_{\widetilde R^n_{\mathrm{opt}}} \right) \gls{transferoperatorl}}
&=
\sup_{\varphi \in \gls{transfer_source} \nnull}
\frac{
\norm{\left(1 - \gls{orthogonalp}_{\widetilde R^n_{\mathrm{opt}}} \right) \gls{transferoperatorl}\varphi}
}{
\norm{\varphi}
}
\nonumber
\\
&=
\sup_{\varphi \in \gls{transfer_source} \nnull}
\frac{
\norm{\left(1 - \gls{orthogonalp}_{\widetilde R^n_{\mathrm{opt}}} \right)
\sum_{i=1}^{N_T} \gls{lsv}_i \gls{singular_value}_i (\gls{rsv}_i, \varphi)_{\gls{transfer_source}}
}
}{
\norm{\varphi}
}
\nonumber
\\
&=
\sup_{\varphi \in \gls{transfer_source} \nnull}
\frac{
\norm{
\sum_{i=n+1}^{N_T} \gls{lsv}_i \gls{singular_value}_i (\gls{rsv}_i, \varphi)_{\gls{transfer_source}}
}
}{
\norm{\varphi}
}
\nonumber
\\
&=\gls{singular_value}_{n+1}
.
\end{align}
These optimal spaces were used in 
\cite{bablip11} and \cite{Smetana2016}.
The randomized algorithm presented in the following
approximates these spaces.
In the numerical experiments below, we use these
spaces as a benchmark.
\subsection{Generation of Random Vectors}
\label{sec:generation of vectors}
One of two central parts in the randomized range finder algorithm is the
generation of random vectors in the space $\gls{transfer_source}$.
As we will see later in the analysis, the only
important property of the generation of random vectors is
the distribution of the inner product with an arbitrary normed vector.
For the analysis which follows later, it is
necessary that the distribution of the inner product
with a normed vector is Gauss distributed with zero
mean and a standard deviation $s$ for which limits 
$s_\mathrm{min} \leq s \leq s_\mathrm{max}$ are known.

The easiest way of drawing random vectors is to use
independent Gauss-normal distributed coefficients to the
basis vectors $\psi_i^S$, i.e.~use $\gls{ritzdsinv}\underline r$
with a vector of independent, Gauss-normal distributed entries $\underline r$.
This drawing strategy
has the required properties, as we show in the following lemma.

\begin{lemma}[Distribution of inner product]
The inner product of a normed vector $v$ in $\gls{transfer_source}$ with a random normal vector
$(v, \gls{ritzdsinv}\underline r)_S$ 
is a Gaussian distributed random variable with mean zero and variance $s^2$, where
\begin{equation}
\lambda^{\gls{source_inner}}_{min} \leq s^2 \leq \lambda^{\gls{source_inner}}_{max}.
\end{equation}
Where $\lambda^{\gls{source_inner}}_{min}$ and $\lambda^{\gls{source_inner}}_{max}$
are the smallest and largest eigenvalue of the matrix of the inner product in $\gls{transfer_source}$.
\end{lemma}

\begin{proof}
We use the spectral decomposition of the inner product matrix \\
$\gls{source_inner} = \sum_{i=1}^{\gls{dimension_transfer_source}} \gls{source_inner_eigenvector} \lambda^{\gls{source_inner}}_i \gls{source_inner_eigenvector}^T$
with eigenvalues $\lambda^{\gls{source_inner}}_i$ and eigenvectors $\gls{source_inner_eigenvector}$.
It holds
\begin{align}
(v, \gls{ritzdsinv} \underline r)_S
&= \sum_{i=1}^{\gls{dimension_transfer_source}} (\gls{ritzds} v)^T \gls{source_inner_eigenvector} \lambda^{\gls{source_inner}}_i \gls{source_inner_eigenvector}^T \underline r . 
\end{align}
As $\gls{source_inner_eigenvector}$ is normed with respect to the euclidean inner product, the
term $\gls{source_inner_eigenvector}^T \underline r$ is a normal distributed random variable.
Using the rules for addition and scalar multiplication of Gaussian random variables,
one sees that the inner product $(v, \gls{ritzdsinv}\underline r)_S$ is a Gaussian random variable with
variance
\begin{equation}
s^2 = \sum_{i=1}^{\gls{dimension_transfer_source}} \left( (\gls{ritzds} v)^T \gls{source_inner_eigenvector} \lambda^{\gls{source_inner}}_i  \right)^2 .
\end{equation}
The variance $s^2$ can easily be bounded as follows:
\begin{eqnarray}
s^2 
&= \sum_{i=1}^{\gls{dimension_transfer_source}} \left( (\gls{ritzds} v) ^T \gls{source_inner_eigenvector} \lambda^{\gls{source_inner}}_i  \right) ^2 
\leq& \sum_{i=1}^{\gls{dimension_transfer_source}} \left( (\gls{ritzds} v) ^T \gls{source_inner_eigenvector} \right) ^2 \lambda^{\gls{source_inner}}_i \max_i (\lambda^{\gls{source_inner}}_i)
= \lambda^{\gls{source_inner}}_{max}\\
s^2 
&= \sum_{i=1}^{\gls{dimension_transfer_source}} \left( (\gls{ritzds} v) ^T \gls{source_inner_eigenvector} \lambda^{\gls{source_inner}}_i  \right) ^2
\geq& \sum_{i=1}^{\gls{dimension_transfer_source}} \left( (\gls{ritzds} v) ^T \gls{source_inner_eigenvector} \right) ^2 \lambda^{\gls{source_inner}}_i \min_i (\lambda^{\gls{source_inner}}_i)
= \lambda^{\gls{source_inner}}_{min}
.
\end{eqnarray}
It holds
\begin{equation}
\sum_{i=1}^{\gls{dimension_transfer_source}} \left( (\gls{ritzds} v) ^T \gls{source_inner_eigenvector} \right) ^2 \lambda^{\gls{source_inner}}_i = 1
\end{equation}
because of the spectral decomposition of \gls{source_inner} and
the assumption that $v$ is normed.
It holds
\begin{align}
\sum_{i=1}^{\gls{dimension_transfer_source}} \left( (\gls{ritzds} v) ^T \gls{source_inner_eigenvector} \right) ^2 \lambda^{\gls{source_inner}}_i
=&~\sum_{i=1}^{\gls{dimension_transfer_source}} \left( (\gls{ritzds} v) ^T \gls{source_inner_eigenvector} \right) \lambda^{\gls{source_inner}}_i
\left( \gls{source_inner_eigenvector}^T (\gls{ritzds} v) \right)
\nonumber \\
=&~
(\gls{ritzds} v) ^T
\left(
\sum_{i=1}^{\gls{dimension_transfer_source}}
\gls{source_inner_eigenvector}
\lambda^{\gls{source_inner}}_i
\gls{source_inner_eigenvector}^T
\right)
(\gls{ritzds} v)\nonumber \\
=&~ (\gls{ritzds} v) ^T \gls{source_inner} (\gls{ritzds} v)
\nonumber \\
=&~ \norm{v}_{\gls{transfer_source}}^2 = 1
.
\end{align}
\end{proof}
\subsection{A Probabilistic a Posteriori Norm Estimator}
\label{sec:random a posteriori}
The second central part to build the
adaptive algorithm is a convergence criterion.
To this end, we use a probabilistic norm estimator, 
giving an upper bound for the operator
norm of the given operator and apply this norm estimator 
on $\left(1 - \gls{orthogonalp}_{\gls{transfer_rrange}} \right) \gls{transferoperatorl}$.
The norm estimator is simple to compute. 
Given the number of test vectors $\gls{number_testvectors}$,
it applies the
operator $\gls{number_testvectors}$ times to a random vector and uses the maximum
of the $\gls{number_testvectors}$ norms of the results.
The number $\gls{number_testvectors}$ is a user supplied parameter.
\begin{definition}[Probabilistic a posteriori norm estimator]\label{def:norm estimator}
To estimate the operator norm of a linear operator $\gls{some_operator}: \gls{transfer_source} \rightarrow \gls{transfer_range}$ of rank smaller or equal to $N_O$, we define the a posteriori norm estimator $\gls{randomized_a_posteriori}(\gls{some_operator}, \gls{number_testvectors}, \gls{epstestfail}) $ for $\gls{number_testvectors}$ test vectors as
\begin{equation}\label{eq:a posteriori error estimator}
\gls{randomized_a_posteriori}(\gls{some_operator}, \gls{number_testvectors}, \gls{epstestfail}) 
:= \cest \max_{i \in 1, \dots, \gls{number_testvectors}} \norm{ \gls{some_operator} \ \gls{ritzdsinv} \ \underline r_i }_R.
\end{equation}
Here, $\cest$ is defined as
\begin{equation}
\cest := 
\left( \sqrt{2 \lambda^{\gls{source_inner}}_{min}} \ \mathrm{erf}^{-1} ( \sqrt[\gls{number_testvectors}]{\gls{epstestfail} } ) \right)^{-1}.
,
\end{equation}
$\underline r_i$ are random normal vectors,
and $\lambda^{\gls{source_inner}}_{min}$ is the smallest eigenvalue of the matrix of the inner product in $\gls{transfer_source}$.
\label{definition:test}
\end{definition}
This norm estimator gives an upper bound to the operator norm with high probability,
as we show in the following proposition.
\begin{proposition}[Norm estimator failure probability]
\label{thm:errest_reliable}
The norm estimator 
$
\gls{randomized_a_posteriori}(\gls{some_operator}, \gls{number_testvectors}, \gls{epstestfail})
$
is an upper bound of the operator norm $\norm{\gls{some_operator}}$ with probability greater or equal than $(1 - \gls{epstestfail})$:
\begin{equation}
\gls{probof}\Big( \norm{\gls{some_operator}} \leq \gls{randomized_a_posteriori}(\gls{some_operator}, \gls{number_testvectors}, \gls{epstestfail}) \Big) \geq 1 - \gls{epstestfail}
.
\end{equation}
\end{proposition}

\begin{proof}
We analyze the probability for the event that the norm estimator
$
\gls{randomized_a_posteriori}(\gls{some_operator}, \gls{number_testvectors}, \gls{epstestfail})
$
is smaller than the operator norm $\norm{\gls{some_operator}}$:
\begin{eqnarray}
\gls{probof}\Big( \gls{randomized_a_posteriori}(\gls{some_operator}, \gls{number_testvectors}, \gls{epstestfail}) < \norm{\gls{some_operator}} \Big)
&=&
\gls{probof}\Big( 
\cest
\max_{i \in 1, \dots, \gls{number_testvectors}} \norm{ \gls{some_operator} \ \gls{ritzdsinv} \ \underline r_i }_R 
< \norm{\gls{some_operator}}
\Big)
.
\end{eqnarray}
The probability that all test vector norms are smaller than a 
certain value is the 
the product of 
the probabilities 
that each test vector is smaller than that value.
So with a new random normal vector $\underline r$ it holds
\begin{eqnarray}
\gls{probof}\Big( \gls{randomized_a_posteriori}(\gls{some_operator}, \gls{number_testvectors}, \gls{epstestfail}) < \norm{\gls{some_operator}} \Big)&=&
\gls{probof}\Big( 
\cest
\norm{ \gls{some_operator} \ \gls{ritzdsinv} \ \underline r }_R 
< \norm{\gls{some_operator}}
\Big)^{\gls{number_testvectors}}.
\end{eqnarray}
Using the singular value decomposition of the operator $\gls{some_operator}$: $\gls{some_operator} \varphi = \sum_i u_i \gls{singular_value}_i (v_i, \varphi)_{\gls{source_inner}}$ we obtain
\begin{eqnarray}
\gls{probof}\Big( \gls{randomized_a_posteriori}(\gls{some_operator}, \gls{number_testvectors}, \gls{epstestfail}) < \norm{\gls{some_operator}} \Big)&\leq&
\gls{probof}\Big( 
\cest
\norm{ u_1 \gls{singular_value}_1 \ \left(v_1, \gls{ritzdsinv} \ \underline r\right)_S }_R 
< \norm{\gls{some_operator}}
\Big)^{\gls{number_testvectors}} \nonumber \\
&=&
\gls{probof}\Big( 
\cest
\gls{singular_value}_1 \ \left|\left(v_1, \gls{ritzdsinv} \ \underline r\right)_S\right|
< \norm{\gls{some_operator}}
\Big)^{\gls{number_testvectors}} \nonumber \\
&=&
\gls{probof}\Big( 
\cest
\left|\left(v_1, \gls{ritzdsinv} \ \underline r\right)_S\right|
< 1
\Big)^{\gls{number_testvectors}}.
\end{eqnarray}
The inner product $\left|\left(v_1, \gls{ritzdsinv} \ \underline r\right)_S\right|$ is a Gaussian distributed 
random variable with variance greater $\lambda^{\gls{source_inner}}_{min}$, so
with a new normal distributed random variable $r'$ it holds
\begin{eqnarray}
\gls{probof}\Big( \gls{randomized_a_posteriori}(\gls{some_operator}, \gls{number_testvectors}, \gls{epstestfail}) < \norm{\gls{some_operator}} \Big)&\leq&
\gls{probof}\Big( 
\sqrt{\lambda^{\gls{source_inner}}_{min}} |r'|
<
\sqrt{2 \lambda^{\gls{source_inner}}_{min}} \cdot \mathrm{erf}^{-1}\left(\sqrt[\gls{number_testvectors}]{\gls{epstestfail}}\right)
\Big)^{\gls{number_testvectors}} \nonumber \\
&=&
\mathrm{erf} \left( \frac{\sqrt{2} \mathrm{erf}^{-1}\left(\sqrt[\gls{number_testvectors}]{\gls{epstestfail}}\right) }{\sqrt{2}} \right)^{\gls{number_testvectors}}
 = \gls{epstestfail}.
\end{eqnarray}
\end{proof}

In case a good upper bound for the (numerical) rank of the operator $\gls{some_operator}$ is known,
this estimator is also effective with high probability, as the following proposition shows.

\begin{proposition}[Norm estimator effectivity]
\label{thm:errest_efficient}
Let the effectivity $\eta$ of the norm estimator $
\gls{randomized_a_posteriori}(\gls{some_operator}, \gls{number_testvectors}, \gls{epstestfail})
$ be defined as
\begin{equation}
\label{eq:effectivity}
\eta(\gls{some_operator}, \gls{number_testvectors}, \gls{epstestfail}) := 
\frac{
\gls{randomized_a_posteriori}(\gls{some_operator}, \gls{number_testvectors}, \gls{epstestfail}) 
}{
\norm{\gls{some_operator}}
}.
\end{equation}
Then, it holds
\begin{equation}
\gls{probof}\Big( \eta \leq \ceff \Big) \geq 1 - \gls{epstestfail}
,
\end{equation}
where the constant $\ceff$ is defined as
\begin{equation}
\ceff := 
\left[
Q^{-1}\left(\frac{N_O}{2}, \frac{\gls{epstestfail}}{\gls{number_testvectors}}\right)
\frac{\lambda^{\gls{source_inner}}_{max}}{\lambda^{\gls{source_inner}}_{min}}
\left( \mathrm{erf}^{-1} \left( \sqrt[\gls{number_testvectors}]{\gls{epstestfail}} \right) \right)^{-2}
\right]^{1/2}
\end{equation}
and $Q^{-1}$ is the inverse of the upper normalized incomplete gamma function;
i.e.
\begin{equation}
Q^{-1}(a,y) = x\quad \text{ when }\quad Q(a,x) = y.
\end{equation}\footnote{
Recall that the definition of the upper normalized incomplete gamma function is
$$
Q(a,x) = 
\dfrac{\int_x^\infty t^{a-1} e^{-t} \mathrm{d} t
}{
\int_0^\infty t^{a-1} e^{-t} \mathrm{d} t
}.
$$
}
\end{proposition}

\begin{proof}
The constant $\cest$ is defined as in the proof of \cref{thm:errest_reliable}.
To shorten notation, we write $\cestshort$ for $\cest$ and $\ceffshort$ for $\ceff$ within this proof. 
Invoking the definition of $\gls{randomized_a_posteriori}(\gls{some_operator}, \gls{number_testvectors}, \gls{epstestfail})$ yields
\begin{eqnarray}
\gls{probof}\Big( \gls{randomized_a_posteriori}(\gls{some_operator}, \gls{number_testvectors}, \gls{epstestfail}) > \ceffshort \norm{\gls{some_operator}} \Big)
&=&
\gls{probof}\Big( 
\cestshort
\max_{i \in 1, \dots, \gls{number_testvectors}} \norm{ \gls{some_operator} \ \gls{ritzdsinv} \ \underline r_i }_R 
> \ceffshort \norm{\gls{some_operator}}
\Big)
\end{eqnarray}
and by employing a new random normal vector $\underline r$ we obtain
\begin{eqnarray}
\gls{probof}\Big( \gls{randomized_a_posteriori}(\gls{some_operator}, \gls{number_testvectors}, \gls{epstestfail}) > \ceffshort \norm{\gls{some_operator}} \Big)&\leq&
\gls{number_testvectors}
\gls{probof}\Big( 
\cestshort
\norm{ \gls{some_operator} \ \gls{ritzdsinv} \ \underline r }_R 
> \ceffshort \norm{\gls{some_operator}}
\Big).
\end{eqnarray}
Using the singular value decomposition of the operator $\gls{some_operator}$: $\gls{some_operator} \varphi = \sum_i u_i \gls{singular_value}_i (v_i, \varphi)_S$ results in
\begin{eqnarray}
\gls{probof}\Big( \gls{randomized_a_posteriori}(\gls{some_operator}, \gls{number_testvectors}, \gls{epstestfail}) > \ceffshort \norm{\gls{some_operator}} \Big)&\leq&
\gls{number_testvectors}
\gls{probof}\Big( 
\cestshort
\norm{\sum_i u_i \gls{singular_value}_1 (v_i, \gls{ritzdsinv} \underline r)_S}_R
> \ceffshort \norm{\gls{some_operator}}
\Big).
\end{eqnarray}
For a new random normal variables $r_i$ we have
\begin{eqnarray}
\gls{probof}\Big( \gls{randomized_a_posteriori}(\gls{some_operator}, \gls{number_testvectors}, \gls{epstestfail}) > \ceffshort \norm{\gls{some_operator}} \Big)&\leq&
\gls{number_testvectors}
\gls{probof}\Big( 
\cestshort
\gls{singular_value}_1 \sqrt{\lambda^{\gls{source_inner}}_{max}} \sqrt{\sum_i r_i^2}
> \ceffshort \norm{\gls{some_operator}}
\Big) \nonumber \\
&=&
\gls{number_testvectors}
\gls{probof}\Big(
\sqrt{\sum_i r_i^2}
> \frac{\ceffshort}{\cestshort} \gls{singular_value}_1^{-1}
\sqrt{\lambda^{\gls{source_inner}}_{max}}^{-1}
\norm{\gls{some_operator}}
\Big) \nonumber \\
&=&
\gls{number_testvectors}
\gls{probof}\Big(
\sqrt{\sum_i r_i^2}
> \frac{\ceffshort}{\cestshort}
\sqrt{\lambda^{\gls{source_inner}}_{max}}^{-1}
\Big) \nonumber \\
&=&
\gls{number_testvectors}
\gls{probof}\Big(
\sum_i r_i^2
> \frac{\ceffshort^2}{\cestshort^2}
\sqrt{\lambda^{\gls{source_inner}}_{max}}^{-2}
\Big).
\end{eqnarray}
The sum of squared random normal variables is a random variable with chi-squared distribution.
Its cumulative distribution function is the incomplete, normed gamma function.
As we have a ``greater than'' relation, the upper incomplete normed gamma function is used,
which we denote by $Q(\frac k 2, \frac x 2)$ here.
Therefore, we conclude
\begin{eqnarray}
\gls{probof}\Big( \gls{randomized_a_posteriori}(\gls{some_operator}, \gls{number_testvectors}, \gls{epstestfail}) > \ceff \norm{\gls{some_operator}} \Big)&\leq&
\gls{number_testvectors} Q\left(\frac{N_O}{2}, \frac{\ceff^2}{\cest^2} \frac{1}{2 \lambda^{\gls{source_inner}}_{max}}\right) \nonumber \\
&=& \gls{epstestfail}.
\end{eqnarray}
\end{proof}

\subsection{Randomized Range Finder Algorithm}
Using the random vector generation in \cref{sec:generation of vectors}
and the randomized norm estimator in \cref{sec:random a posteriori}
applied to $\gls{some_operator} = \left(1 - \gls{orthogonalp}_{\gls{transfer_rrange}} \right) \gls{transferoperatorl}$,
we can formulate the randomized range finder algorithm.
\begin{algorithm2e}[tp]
\DontPrintSemicolon
\SetAlgoVlined
\SetKwFunction{AdaptiveRandomizedRangeApproximation}{AdaptiveRandomizedRangeApproximation}
\SetKwInOut{Input}{Input}
\SetKwInOut{Output}{Output}
\Fn{\AdaptiveRandomizedRangeApproximation{$\gls{transferoperator}, \algotol, \gls{number_testvectors}, \gls{epsalgofail}$}}{
  \Input{Operator $\gls{transferoperator}$,\\target accuracy $\algotol$,\\number of test vectors $\gls{number_testvectors}$,\\maximum failure probability $\gls{epsalgofail}$}
  \Output{space $\gls{transfer_rrange}$ with property $\gls{probof}\left(\norm{\left(1 - \gls{orthogonalp}_{\gls{transfer_rrange}} \right) \gls{transferoperatorl}} \leq \algotol \right) > \left(1 - \gls{epsalgofail}\right)$}
  \tcc{initialize basis} 
  $B \leftarrow \left\{ \gls{transferoperatora} \right\}$ \label{algo:line:basis_init}\;
  \tcc{initialize test vectors} 
  $M \leftarrow \left\{\gls{transferoperator} \gls{ritzdsinv} \underline r_1, \ \dots, \  \gls{transferoperator} \gls{ritzdsinv} \underline r_{\gls{number_testvectors}} \right\} $ \label{algo:line:test_vec_init}\;
  \tcc{orthogonalize test vectors to $\mathrm{span}(B)$}
  $M \leftarrow \left\{\left(1 - P_{\mathrm{span}(B)}\right) t \ \Big| \ t \in M  \right\}$ \label{algo:line:test_vector_update1}\;
  \tcc{determine error estimator factor}
  $\gls{epstestfail} \leftarrow \gls{epsalgofail} / N_T$ \label{algo:line:varepsilon_init} \;
  $\gls{cest} \leftarrow \left[ \sqrt{2 \lambda^{\gls{source_inner}}_{min}} \ \mathrm{erf}^{-1} \left( \sqrt[\gls{number_testvectors}]{\gls{epstestfail} } \right)
\right]^{-1}$ \label{algo:line:constant_init} \;
  \tcc{basis generation loop}
  \While{$\left( \max_{t \in M} \norm{t}_R \right) \cdot \gls{cest} > \algotol$ \label{algo:line:convergence}}{
    $B \leftarrow B \cup (\gls{transferoperator} \gls{ritzdsinv} \underline r) $ \label{algo:line:basis_extension}\;
    $B \leftarrow \mathrm{orthonormalize}(B)$ \label{algo:line:basis_extension2}\;
    \tcc{orthogonalize test vectors to $\mathrm{span}(B)$}
    $M \leftarrow \left\{\left(1 - \gls{orthogonalp}_{\mathrm{span}(B)}\right) t \ \Big| \ t \in M  \right\}$ \label{algo:line:test_vector_update}\;
  }
  \Return $\gls{transfer_rrange} = \mathrm{span}(B)$\;
}
\caption{Adaptive Randomized Range Approximation.}
\label{algo:adaptive_range_approximation}
\end{algorithm2e}
We propose an adaptive randomized range approximation algorithm that constructs an approximation space $\gls{transfer_rrange}$
by iteratively extending its basis until a convergence criterion is satisfied.
In each iteration, the basis is extended by the operator $\gls{transferoperator}$ applied to
a random function.

The full algorithm is given in \cref{algo:adaptive_range_approximation} and has four input parameters, starting with the operator $\gls{transferoperator}$, whose range should be approximated. This could be represented by a matrix, 
but in the intended context it is usually an implicitly defined operator
which is computationally expensive to evaluate.
Only the evaluation of the operator on a vector is required.
The second input parameter is the target accuracy $\algotol$ such that $\norm{\left(1 - \gls{orthogonalp}_{\gls{transfer_rrange}} \right) \gls{transferoperatorl}} \leq \algotol$.
The third input parameter is the number of test vectors $\gls{number_testvectors}$
to be used in the a posteriori error estimator. 
More test vectors
lead to spaces of lower dimension but induce a higher computational cost. 
A typical $\gls{number_testvectors}$ could be 5, 10, or 20.
The forth input parameter is the maximum failure probability $\gls{epsalgofail}$.
The algorithm might fail to deliver a space with the required
approximation properties, but only with a probability which is smaller
than this parameter. Choosing it as e.g.~$10^{-15}$ leads to an algorithm
which virtually never fails.
The algorithm returns a space which has the required approximation
properties with a probability greater than $1- \gls{epsalgofail}$
i.e.~
$P\left(\norm{\left(1 - \gls{orthogonalp}_{\gls{transfer_rrange}} \right) \gls{transferoperatorl}} \leq \algotol \right) > \left(1 - \gls{epsalgofail}\right)$.

The basis $B$ of $\gls{transfer_rrange}$ is initialized with the affine part \gls{transferoperatora} in line \ref{algo:line:basis_init},
test vectors are initialized as the operator applied to random normal vectors in 
line \ref{algo:line:test_vec_init}.
The test vectors are orthogonalized to the existing basis in line \ref{algo:line:test_vector_update1}.
Recall that $\gls{transferoperator} \gls{ritzdsinv} \underline r$ is the operator $\gls{transferoperator}$ applied to a random normal vector.
We use the term ``random normal vector'' to denote a vector whose entries are independent and identically distributed random variables
with normal distribution. The main loop of the algorithm is terminated when the a posteriori norm estimator
\cref{def:norm estimator}
applied to $\left(1 - \gls{orthogonalp}_{\gls{transfer_rrange}} \right) \gls{transferoperatorl}$ is smaller than $\algotol$.
The constant $\cest$, which appears in the error estimator,
is calculated in line \ref{algo:line:varepsilon_init} and \ref{algo:line:constant_init} using
$N_T$ --- the rank of operator $\gls{transferoperator}$. In practice $N_{T}$ is unknown
and an upper bound for $N_T$ such as $\min(\gls{dimension_transfer_source}, \gls{dimension_transfer_range})$ can be used instead.
In line \ref{algo:line:convergence} the algorithm assesses if the convergence criterion is already satisfied. Note that the term
$
\left( \max_{t \in M} \norm{t}_R \right) \cdot \cest
$
is the norm estimator \cref{eq:a posteriori error estimator} applied to $\left(1 - \gls{orthogonalp}_{\gls{transfer_rrange}} \right) \gls{transferoperatorl}$.
The test vectors are reused for all iterations.
The main loop of the algorithm consists of two parts.
First, the basis is extended in line \ref{algo:line:basis_extension}
and \ref{algo:line:basis_extension2} by applying the operator $\gls{transferoperator}$ to a random normal vector and adding the result to the basis $B$. Then the basis $B$ is orthonormalized. We emphasize that the orthonormalization is numerically challenging, 
as the basis functions are nearly linear dependent when
$\gls{transfer_rrange}$ is already a good approximation of the range of $\gls{transferoperator}$.
In the numerical experiments we use the numerically stable Gram-Schmidt with re-iteration
\cref{alg:gram-schmidt adaptive}
shown in \cref{chap:stable gram schmidt}.
Second, the test vectors are updated in 
line \ref{algo:line:test_vector_update}.

In \cref{algo:adaptive_range_approximation},
the smallest eigenvalue of the matrix of the inner product in $\gls{transfer_source}$, 
$\lambda^{\gls{source_inner}}_{min}$, or at least a lower bound for it,
is required.
The orthonormalization of $B$ in line \ref{algo:line:basis_extension2}
and the update of test vectors in 
line \ref{algo:line:test_vector_update} use the inner product in
$\gls{transfer_range}$.

The presented algorithm has good performance properties for operators $\gls{transferoperator}$ which
are expensive to evaluate. To produce the space $\gls{transfer_rrange}$ of dimension $n$, 
it evaluates the operator $n$ times to generate the basis and $\gls{number_testvectors}$
times to generate the test vectors, so in total $n + \gls{number_testvectors}$ times.
Exploiting the low rank structure of $\gls{transferoperator}$, one
could calculate the eigenvectors of $\gls{transferoperator}^*\gls{transferoperator}$ using
a Lanczos type algorithm as implemented in ARPACK \cite{Lehoucq1998}, but this would require $\mathcal{O}(n)$ evaluations of
$\gls{transferoperator}$ and $\gls{transferoperator}^*$ in every iteration, 
potentially summing up to much more than $n + \gls{number_testvectors}$ evaluations,
where the number of iterations is often not foreseeable.

\subsection{A probabilistic a priori error bound}
\label{sec:convergence_rate}
In this subsection we analyze the convergence behavior
of \cref{algo:adaptive_range_approximation}.
By applying results from randomized numerical linear algebra, 
it is possible to devise a priori bounds for the projection
error $\norm{\left(1 - \gls{orthogonalp}_{\gls{transfer_rrange}} \right) \gls{transferoperatorl}}$ and its expected value.
To do so, we first
bound 
$\norm{\left(1 - \gls{orthogonalp}_{\gls{transfer_rrange}} \right) \gls{transferoperatorl}}$
in terms of 
$\norm{ \left(1 - \underline{P}_{\gls{transfer_rrange},2}\right)\underline{\gls{transferoperatorl}} }_2$
where 
$\underline P_{\gls{transfer_rrange}, 2}$ is the matrix of the orthogonal projection on $\gls{transfer_rrange}$ in the 
euclidean inner product in $\gls{R}^{\gls{dimension_transfer_range}}$ and $\norm{\cdot}_2$ is the spectral matrix norm.
This is done in \cref{thm:tomatrix}.
In a second step, we use a result from \cite{Halko2011}
reproduced as \cref{thm:halko106108} here to obtain an a priori bound for
$\norm{ \left(1 - \underline{P}_{\gls{transfer_rrange},2}\right)\underline{\gls{transferoperatorl}} }_2$
in terms of the singular values $\underline \sigma_{i}$ of the matrix of the transfer operator.
In a third step, we devise an upper bound for the singular values of the matrix of the transfer operator
$\underline \sigma_{i}$ in terms of the singular values of the transfer operator $\sigma_{i}$
in \cref{thm:svals}.
Combining these three results, we obtain our main a priori result,
\cref{thm:convergence_rate_main}.
\begin{lemma}[Projection error of operators and their matrices]
\label{thm:tomatrix}
For some given reduced space $\gls{transfer_rrange}$ it holds
\begin{equation}
\norm{\left(1 - \gls{orthogonalp}_{\gls{transfer_rrange}} \right) \gls{transferoperatorl}} = 
\sup_{\xi \in \gls{transfer_source} \nnull} \inf_{\zeta \in \gls{transfer_rrange}} \frac{\norm{ \gls{transferoperatorl}\xi - \zeta }_{R}}{\norm{\xi }_{\gls{transfer_source}}} 
\leq \sqrt{\frac{\lambda^{\gls{range_inner}}_{max}}{\lambda^{\gls{source_inner}}_{min}}} \norm{ \left(1 - \underline{P}_{\gls{transfer_rrange},2}\right)\underline{\gls{transferoperatorl}} }_2
.
\end{equation}
\end{lemma} 
\begin{proof}
\begin{align}
\sup_{\xi \in \gls{transfer_source} \nnull} \inf_{\zeta \in \gls{transfer_rrange}} \frac{\norm{ \gls{transferoperatorl}\xi - \zeta }_{R}}{\norm{\xi }_{\gls{transfer_source}}} &= \sup_{\xi \in \gls{transfer_source} \nnull} \frac{\norm{ \gls{transferoperatorl}\xi - \gls{orthogonalp}_{\gls{transfer_rrange}}\gls{transferoperatorl}\xi }_{R}}{\norm{\xi }_{\gls{transfer_source}}} \nonumber \\
& = \sup_{\underline{\xi} \in \gls{R}^{\gls{dimension_transfer_source}} \nnull} \frac{\left((\underline{\gls{transferoperatorl}}\underline{\xi} - \underline{P}_{\gls{transfer_rrange}}\underline{\gls{transferoperatorl}}\underline{\xi})^{T}\gls{range_inner}(\underline{\gls{transferoperatorl}}\underline{\xi} - \underline{P}_{\gls{transfer_rrange}}\underline{\gls{transferoperatorl}}\underline{\xi})\right)^{1/2}}{\sqrt{\underline \xi^{T}\gls{source_inner}\underline{\xi}}} \nonumber \\
& \leq \sup_{\underline{\xi} \in \gls{R}^{\gls{dimension_transfer_source}} \nnull} \frac{\left((\underline{\gls{transferoperatorl}}\underline{\xi} - \underline{P}_{\gls{transfer_rrange},2}\underline{\gls{transferoperatorl}}\underline{\xi})^{T}\gls{range_inner}(\underline{\gls{transferoperatorl}}\underline{\xi} - \underline{P}_{\gls{transfer_rrange},2}\underline{\gls{transferoperatorl}}\underline{\xi})\right)^{1/2}}{\sqrt{\underline \xi^{T}\gls{source_inner}\underline{\xi}}} \nonumber \\
& \leq \sqrt{\frac{\lambda^{\gls{range_inner}}_{max}}{\lambda^{\gls{source_inner}}_{min}}} \sup_{\underline{\xi} \in \gls{R}^{\gls{dimension_transfer_source}} \nnull} \frac{\norm{ \underline{\gls{transferoperatorl}}\underline{\xi} - \underline{P}_{\gls{transfer_rrange},2}\underline{\gls{transferoperatorl}}\underline{\xi} }_{2} }{\norm{\underline{\xi}}_{2}}
.
\end{align}
\end{proof}

\begin{theorem}[A priori convergence of randomized range finder for matrices]\cite[Theorem 10.6]{Halko2011}
\label{thm:halko106108}
Let $\underline{\gls{transferoperator}}\in \gls{R}^{\gls{dimension_transfer_range}\times \gls{dimension_transfer_source}}$.
Then for $n\geq 4$ it holds
\begin{equation}
\gls{expectedvalue}
\left( \norm{\underline{\gls{transferoperator}} - \underline{P}_{\gls{transfer_rrange},2}\underline{\gls{transferoperator}}}_2 \right)
\leq
\min_{\overset{k+p=n}{k\geq 2, p\geq 2}}
\left[
\left( 1 + \sqrt{  \frac{k}{p-1}  } \right) \underline \sigma_{k+1} +
\frac{ e \sqrt{n}}{p} \left( \sum_{j > k} \underline{\gls{singular_value}}^2_j \right) ^{\frac{1}{2}}
\right].
\end{equation}
\end{theorem}
\begin{proof}
\cite[Theorem 10.6]{Halko2011}
\end{proof}

\begin{lemma}[Singular values of operators and their matrices]
\label{thm:svals}
Let the singular values $\underline{\gls{singular_value}}_{j}$ of the matrix $\underline{\gls{transferoperator}}$ be sorted in non-increasing order, i.e. $\underline{\gls{singular_value}}_{1} \geq \hdots \geq \underline{\gls{singular_value}}_{\gls{dimension_transfer_range}}$ and $\sigma_{j}$ be the singular values of the operator $\gls{transferoperator}$, also sorted non-increasing. Then it holds 
\begin{equation}
\underline{\gls{singular_value}}_{j} \leq
(\lambda^{\gls{source_inner}}_{max}/\lambda^{\gls{range_inner}}_{min})^{1/2}
\sigma_{j} \qquad \forall j \in \{1,\hdots,N_T\}.
\end{equation}
\end{lemma}
\begin{proof}
For notational convenience we denote within this proof the $j$-th eigenvalue of a matrix 
$\underline A$ by $\lambda_{j}(\underline A)$. All singular values for $j \in \{ 1, \dots, N_T \}$ are different from zero. Therefore, it holds
\begin{equation}
\underline \sigma_j^2 = \lambda_j(\underline{\gls{transferoperator}}^t \underline{\gls{transferoperator}})
\qquad \text{and} \qquad
\sigma_j^2 = \lambda_j(\gls{source_inner}^{-1} \underline{\gls{transferoperator}}^t \gls{range_inner} \underline{\gls{transferoperator}})
.
\end{equation}
Recall that $\underline{\gls{transferoperator}}$ is the matrix representation of $\gls{transferoperator}$ and note that 
$\gls{source_inner}^{-1} \underline{\gls{transferoperator}}^t \gls{range_inner}$
is the matrix representation of the adjoint operator $\gls{transferoperator}^*$. The non-zero eigenvalues of a product of matrices $\underline A \underline B$ 
are identical to the non-zero eigenvalues of the product $\underline B \underline A$
(see e.g.~\cite[Theorem 1.3.22]{Horn2012}), hence
\begin{equation}
\lambda_j(\gls{source_inner}^{-1} \underline{\gls{transferoperator}}^t \gls{range_inner} \underline{\gls{transferoperator}})
=
\lambda_j(
\underline{\gls{transferoperator}}^t \gls{range_inner} \underline{\gls{transferoperator}} \gls{source_inner}^{-1} )
.
\end{equation}
We may then apply the Courant minimax principle to infer 
\begin{equation}%
\lambda_{j}(\underline{\gls{transferoperator}}^{t}\gls{range_inner}\underline{\gls{transferoperator}}) (\lambda^{\gls{source_inner}}_{max})^{-1} \leq \lambda_{j}(\gls{source_inner}^{-1}\underline{\gls{transferoperator}}^{t}\gls{range_inner}\underline{\gls{transferoperator}}).
\end{equation}
Employing once again cyclic permutation and the Courant minimax principle yields
\begin{equation}\label{est2:proof eigenvalue est}
\lambda_{j}(\underline{\gls{transferoperator}}^{t}\underline{\gls{transferoperator}}) 
=\lambda_{j}(\underline{\gls{transferoperator}} \underline{\gls{transferoperator}}^t) 
\leq \lambda_{j}(\underline{\gls{transferoperator}} \underline{\gls{transferoperator}}^t \gls{range_inner}) \frac{1}{\lambda^{\gls{range_inner}}_{min}} 
\leq \lambda_{j}(\underline{\gls{transferoperator}}^{t}\gls{range_inner}\underline{\gls{transferoperator}}) \frac{1}{\lambda^{\gls{range_inner}}_{min}} \leq \lambda_{j}(\gls{source_inner}^{-1}\underline{\gls{transferoperator}}^{t}\gls{range_inner}\underline{\gls{transferoperator}})\frac{\lambda^{\gls{source_inner}}_{max}}{\lambda^{\gls{range_inner}}_{min}}
\end{equation}
and thus the claim.
\end{proof}

\begin{proposition}[A priori convergence of randomized range finder]
\label{thm:convergence_rate_main}
Let $\lambda^{\gls{source_inner}}_{max}$,
$\lambda^{\gls{source_inner}}_{min}$,
$\lambda^{\gls{range_inner}}_{max}$,
and
$\lambda^{\gls{range_inner}}_{min}$ denote the largest and smallest eigenvalues
of the inner product matrices
$\gls{source_inner}$ and
$\gls{range_inner}$, respectively and let $\gls{transfer_rrange}$ be the outcome of \cref{algo:adaptive_range_approximation}. Then, for $n\geq 4$ it holds
\begin{equation}\label{eq:a priori mean}
\gls{expectedvalue}
\norm{\left(1 - \gls{orthogonalp}_{\gls{transfer_rrange}} \right) \gls{transferoperatorl}}
\leq
\sqrt{\frac{\lambda^{\gls{range_inner}}_{max}}{\lambda^{\gls{range_inner}}_{min}}
\frac{\lambda^{\gls{source_inner}}_{max}}{\lambda^{\gls{source_inner}}_{min}}}
\min_{\overset{k+p=n}{k\geq 2, p\geq 2}}
\left[
\left( 1 + \sqrt{  \frac{k}{p-1}  } \right) \sigma_{k+1} +
\frac{ e \sqrt{n}}{p} \left( \sum_{j > k} \gls{singular_value}^2_j \right) ^{\frac{1}{2}}
\right]
.
\end{equation}
\end{proposition}
\begin{proof}
Combining \cref{thm:tomatrix}, \cref{thm:halko106108} and \cref{thm:svals}, the statement follows.
\end{proof}

The \cref{thm:convergence_rate_main} extends the results in Theorem 10.6
in \cite{Halko2011} to the case of finite dimensional linear operators. The terms consisting of the square root of the conditions of the inner product matrices $\gls{source_inner}$ and $\gls{range_inner}$ in \cref{eq:a priori mean}
are due to our generalization from the spectral matrix norm as considered in \cite{Halko2011} to inner products associated with finite dimensional Hilbert spaces.

\begin{remark}[Improved constants for a priori convergence estimate]
The result of \cref{algo:adaptive_range_approximation},
when interpreted as functions and not as coefficient vectors,
is independent of the choice of the basis in $\gls{transfer_range}$.
Disregarding numerical errors,
the result would be the same if the algorithm was executed in an orthonormal basis in $\gls{transfer_range}$.
Thus, we would expect \cref{thm:convergence_rate_main} to hold
also without the factor
$(\lambda^{\gls{range_inner}}_{max}/\lambda^{\gls{range_inner}}_{min})^{1/2}$.
\end{remark}

\section{Numerical Experiments for Isolated Range Finder}
\label{sec:rangefinder artificial test}
\label{sect:numerical_experiments}
Here we provide numerical examples which only show
the performance of the range finder algorithm, without
any model reduction.
We introduce two artificial examples, designed to
test the range finder algorithm, namely
\exrangea{} and \exrangeb{}.

In this section we demonstrate first that the reduced local spaces generated by \cref{algo:adaptive_range_approximation}
yield an approximation that converges at a nearly optimal rate. Moreover, we validate the a priori error bound in
\cref{eq:a priori mean} and the a posteriori error estimator \cref{eq:a posteriori error estimator}.
To this end, we consider two test cases
for which the singular values of the transfer operator are known (see \cite[Supplementary Materials, Section SM2]{Buhr2018} for details).

The main focus of this subsection is a thorough validation of the theoretical findings in \cref{sec:randomized_range_finder}, including a comprehensive testing on how the results depend on various parameters such as the basis size $n$, the number of test vectors $n_{t}$, and the mesh size.
In addition, CPU time measurements are given.
The second numerical example in \cref{subsect:helmholtz}
examines the behavior of the proposed algorithm in the more challenging case of the Helmholtz equation.
For the implementation of the test cases,
only NumPy and SciPy, but no FEM software library was used.

\subsubsection{Analytic interface problem}\label{subsect:analytic}
\begin{figure}
\begin{center}
\ifx\JPicScale\undefined\def\JPicScale{1}\fi
\unitlength \JPicScale mm
\begin{tikzpicture}[x=\unitlength,y=\unitlength,inner sep=0pt]
\draw [line width=0.7mm](20,20) -- (60,20);
\draw [line width=0.7mm](20,40) -- (60,40);
\definecolor{userLineColour}{rgb}{0.4,0.4,1}
\draw [line width=0.7mm,color=userLineColour](20,20) -- (20,40);
\definecolor{userLineColour}{rgb}{1,0.2,0.2}
\draw [line width=0.7mm,color=userLineColour](40,20) -- (40,40);
\definecolor{userLineColour}{rgb}{0.4,0.4,1}
\draw [line width=0.7mm,color=userLineColour](60,20) -- (60,40);
\draw (15,30) node {$\gls{boundary}_{out}$};
\draw (65,30) node {$\gls{boundary}_{out}$};
\draw (45,30) node {$\gls{boundary}_{in}$};
\draw (30,45) node {$\gls{boundary}_{N}$};
\draw (50,15) node {$\gls{boundary}_{N}$};
\draw [line width=0.4mm,<->,>=|](40,50) -- (60,50);
\draw (50,53) node {\gls{L}};
\draw [line width=0.4mm,<->,>=|](75,40) -- (75,20);
\draw (78,30) node {\gls{W}};
\end{tikzpicture}
\end{center}
\caption[Geometry of \exrangea]{Geometry of \exrangea{} and \exrangeb}
\label{fig:illustration geometry}
\end{figure}
To analyze the behavior of the proposed algorithm, we first
apply it to an analytic problem where the singular values of the transfer operator are known. 
We refer to this numerical example as \exrangea{}. We consider the problem $\mathcal{A}=-\Delta$, 
$f =0$, and assume that $\gls{domain} = (-\gls{L},\gls{L})\times (0,\gls{W})$, 
$\gls{boundary}_{out}=\{-\gls{L},\gls{L}\}\times (0,\gls{W})$, and $\gls{boundary}_{in} = \{0\}\times (0,\gls{W})$. 
Moreover, we prescribe homogeneous Neumann boundary conditions on $\partial \gls{domain} \setminus \gls{boundary}_{out}$ 
and arbitrary Dirichlet boundary conditions on $\gls{boundary}_{out}$,
see also \cref{fig:illustration geometry} (left).
The analytic solution 
is known to be
\begin{equation}\label{eq:separation of variables}
u(x_{1},x_{2}) = a_{0} + b_{0}x_{1} + \sum_{n=1}^{\infty} \cos (n\pi\frac{x_{2}}{\gls{W}}) \left[a_{n}\cosh(n\pi \frac{x_{1}}{\gls{W}}) + b_{n}\sinh(n\pi \frac{x_{1}}{\gls{W}})\right].
\end{equation}
For further discussion see \cite{Buhr2018}. This example was introduced in 
\cite[Remark 3.3]{Smetana2016}.
We define \gls{transfer_source} to be the \gls{fe} approximation of the space $L^2(\gls{boundary}_{out})$
and \gls{transfer_range} to be the \gls{fe} approximation of the $L^2(\gls{boundary}_{in})$
with the usual $L^2$-inner 
product on the respective interfaces. 
The transfer operator maps the Dirichlet data to the inner interface,
i.e.~with
$H$ as the space of all discrete solutions, we define
\begin{equation}\label{eq:mod transfer operator}
\gls{transferoperator}(v|_{\gls{boundary}_{out}}) := v|_{\gls{boundary}_{in}} \qquad \forall v \in H.
\end{equation}
As this transfer operator is linear, it holds $\gls{transferoperator} = \gls{transferoperatorl}$.
The singular values of the transfer operator are
$
\sigma_{i}=1/ \left(\sqrt{2}\cosh((i-1)\pi \gls{L} / \gls{W})\right).
$ 

For the experiments, we use $\gls{L} = \gls{W} = 1$, unless stated otherwise. We discretize the problem by meshing it with a regular mesh of squares of size $\gls{h} \cdot \gls{h}$, 
where $1/\gls{h}$ ranges from 20 to 320 in the experiments. On each square, we use bilinear 
Q1 ansatz functions,
which results in e.g.~51,681 \gls{dof}s, $\gls{dimension_transfer_source} = 322$ and $\gls{dimension_transfer_range}= 161$ for $1/\gls{h}=160$.

In \cref{fig:random modes} the first five basis vectors
as generated by \cref{algo:adaptive_range_approximation}
in one particular run
are shown side by side with the first five basis vectors
of the optimal space, i.e.~the optimal modes in
\cref{fig:optimal modes}.
While not identical, the basis functions
generated using the randomized approach are
smooth and have strong similarity with the optimal ones.
Unless stated otherwise, we
present statistics over 100,000 evaluations,
use a maximum failure probability of $\gls{epsalgofail} = 10^{-15}$,
and use $\min(\gls{dimension_transfer_source}, \gls{dimension_transfer_range})$ as an upper bound for 
the rank of $\gls{transferoperatorl}$, $N_T$.

\begin{figure}
\centering
\begin{subfigure}[t]{0.4\textwidth}
\begin{tikzpicture}
\begin{axis}[
    width=5.5cm,
    height=4.5cm,
    xlabel=$x_2$,
    ylabel=$\phi_i^{sp}(x_2)$,
    legend pos=outer north east,
    ymin=-2.2,
    ymax=2.2,
    grid=both,
    grid style={line width=.1pt, draw=gray!20},
    major grid style={line width=.2pt,draw=gray!50},
  ]
  \addplot+[blue,  mark=x,        mark repeat=40, thick] table[x expr=\coordindex / 160, y index=0] {omodes.dat};
  \addplot+[oran,  mark=triangle, mark repeat=40, thick] table[x expr=\coordindex / 160, y index=1] {omodes.dat};
  \addplot+[lila,  mark=diamond,  mark repeat=40, thick] table[x expr=\coordindex / 160, y index=2] {omodes.dat};
  \addplot+[gruen,  mark=o,        mark repeat=40, thick] table[x expr=\coordindex / 160, y index=3] {omodes.dat};
  \addplot+[grau,  mark=square,   mark repeat=40, thick] table[x expr=\coordindex / 160, y index=4] {omodes.dat};
\end{axis}
\end{tikzpicture}
\subcaption{Optimal basis of $\widetilde R^5$}
\label{fig:optimal modes}
\end{subfigure}\quad
\begin{subfigure}[t]{0.45\textwidth}
\centering
\begin{tikzpicture}
\begin{axis}[
    width=5.5cm,
    height=4.5cm,
    xlabel=$x_2$,
    ylabel=$\phi_i^{rnd}(x_2)$,
    ymin=-2.2,
    ymax=2.2,
    legend pos=outer north east,
    grid=both,
    grid style={line width=.1pt, draw=gray!20},
    major grid style={line width=.2pt,draw=gray!50},
  ]
  \addplot+[blue,  mark=x,        mark repeat=40, thick] table[x expr=\coordindex / 160, y index=0] {modes.dat};
  \addplot+[oran,  mark=triangle, mark repeat=40, thick] table[x expr=\coordindex / 160, y index=1] {modes.dat};
  \addplot+[lila,  mark=diamond,  mark repeat=40, thick] table[x expr=\coordindex / 160, y index=2] {modes.dat};
  \addplot+[gruen,  mark=o,        mark repeat=40, thick] table[x expr=\coordindex / 160, y index=3] {modes.dat};
  \addplot+[grau,  mark=square,   mark repeat=40, thick] table[x expr=\coordindex / 160, y index=4] {modes.dat};
  \legend{1,2,3,4,5}
\end{axis}
\end{tikzpicture}
\subcaption{Example basis of $\widetilde R^5$ generated by \cref{algo:adaptive_range_approximation}}
\label{fig:random modes}
\end{subfigure}
\caption[Comparison of optimal with randomized basis functions for \exrangea{}.]
{
Comparison of optimal basis functions with the basis functions generated by 
\cref{algo:adaptive_range_approximation} for \exrangea{}.
Basis functions are normalized to an $L^2(\gls{boundary}_{in})$ norm of one.
(reproduction: \cref{repro:img:modes})
}
\label{img:modes}
\end{figure}

\begin{figure}
\centering
\begin{subfigure}[t]{0.55\textwidth}
\centering
\begin{tikzpicture}
\begin{semilogyaxis}[
    width=5cm,
    height=5cm,
    xmin=-1,
    xmax=13,
    ymin=1e-17,
    ymax=1e2,
    xlabel=basis size $n$,
    ylabel=$\norm{\left(1 - \gls{orthogonalp}_{\gls{transfer_rrange}} \right) \gls{transferoperatorl}}$,
    grid=both,
    grid style={line width=.1pt, draw=gray!20},
    major grid style={line width=.2pt,draw=gray!50},
    minor xtick={1,2,3,4,5,6,7,8,9,10,11,12,13,14,15},
    minor ytick={1e-15, 1e-14, 1e-13, 1e-12, 1e-11, 1e-10, 1e-9, 1e-8, 1e-7, 1e-6, 1e-5, 1e-4, 1e-3, 1e-2, 1e-1, 1e0, 1e1, 1e2, 1e3},
    ytick={1e0, 1e-5, 1e-10, 1e-15},
    max space between ticks=25,
    legend style={at={(1.4,1.35)},anchor=north east},
  ]
  \addplot+[densely dotted, blue,  mark=o, mark options={solid}] table[x expr=\coordindex,y index=4] {twosquares_percentiles.dat};
  \addplot+[densely dashed, blue,  mark=x, mark options={solid}] table[x expr=\coordindex,y index=3] {twosquares_percentiles.dat};
  \addplot+[solid,          black, thick, mark=x] table[x expr=\coordindex,y index=2] {twosquares_percentiles.dat};
  \addplot+[densely dashed, gruen, mark=x, mark options={solid}] table[x expr=\coordindex,y index=1] {twosquares_percentiles.dat};
  \addplot+[densely dotted, gruen, mark=o, mark options={solid}] table[x expr=\coordindex,y index=0] {twosquares_percentiles.dat};
  \addplot+[solid,          red,   mark=0, thick] table[x index=0, y index=1] {experiment_twosquares_svddecay_8.dat};
  \legend{max, 75 percentile, 50 percentile, 25 percentile, min, $\sigma_{i+1}$};
\end{semilogyaxis}
\end{tikzpicture}
\subcaption{Percentiles, worst case, and best case}
\label{fig:deviation percentiles}
\end{subfigure}
\begin{subfigure}[t]{0.4\textwidth}
\centering
\begin{tikzpicture}
\begin{semilogyaxis}[
    width=5cm,
    height=5cm,
    xmin=-1,
    xmax=13,
    ymin=1e-17,
    ymax=1e2,
    xlabel=basis size $n$,
    yticklabels={,,},
    grid=both,
    grid style={line width=.1pt, draw=gray!20},
    major grid style={line width=.2pt,draw=gray!50},
    minor xtick={1,2,3,4,5,6,7,8,9,10,11,12,13,14,15},
    minor ytick={1e-16, 1e-15, 1e-14, 1e-13, 1e-12, 1e-11, 1e-10, 1e-9, 1e-8, 1e-7, 1e-6, 1e-5, 1e-4, 1e-3, 1e-2, 1e-1, 1e0, 1e1, 1e2, 1e3},
    ytick={1e0, 1e-5, 1e-10, 1e-15},
    max space between ticks=25,
    legend style={at={(1.1,1.35)},anchor=north east},
  ]
  \addplot+[] table[x index=0, y index=1]{twosquares_expectation.dat};
  \addplot+[solid,          black, thick, mark=x] table[x expr=\coordindex,y index=0] {twosquares_means.dat};
  \legend{a priori limit for mean, mean};
\end{semilogyaxis}
\end{tikzpicture}
\subcaption{Mean of deviation and a priori limit}
\label{fig:deviation mean}
\end{subfigure}
\caption[Projection error over basis size for \exrangea{}.]
{Projection 
error $\sup_{\xi \in \gls{transfer_source} \nnull} \inf_{\zeta \in \gls{transfer_rrange}} \frac{\norm{ \gls{transferoperator}\xi - \zeta }_{\gls{transfer_range}}}{\norm{\xi }_{\gls{transfer_source}}} = \norm{\left(1 - \gls{orthogonalp}_{\gls{transfer_rrange}} \right) \gls{transferoperatorl}}$
over basis size $n$ for \exrangea{}
with mesh size $\gls{h}=1/160$.
(reproduction: \cref{repro:fig:interface_a_priori})
}
\label{fig:interface_a_priori}
\end{figure}

We first quantify the approximation quality of the spaces
$\gls{transfer_rrange}$ in dependence of the basis size $n$, disregarding
the adaptive nature of \cref{algo:adaptive_range_approximation}.
In \cref{fig:deviation percentiles}, statistics
over the achieved projection error $\norm{\left(1 - \gls{orthogonalp}_{\gls{transfer_rrange}} \right) \gls{transferoperatorl}}$
are shown along with the singular values 
$\gls{singular_value}_{n+1}$ of the transfer operator $\gls{transferoperator}$.
$\gls{singular_value}_{n+1}$ is a lower bound for the
projection error and it is the projection error
that is achieved using an optimal basis.
It shows that while the algorithm most of the time
produces a basis nearly as good as the optimal basis, sometimes it
needs two or three basis vectors more. 
This is in line with the predictions by theory, see the discussion
after \cref{thm:convergence_rate_main}.
The mean value of the projection error 
converges with the same rate as the a priori error bound
given in
\cref{thm:convergence_rate_main}
with increasing basis size. The a priori error bound is consistently
around three orders of magnitude larger than the actual error, 
until the actual error hits the numerical noise between
$10^{-14}$ and $10^{-15}$, see \cref{fig:deviation mean}. This is mainly due to the fact that the singular 
values decay very fast for the present example and an index shift in the singular values by $p\geq 2$ as 
required by the a priori error bound \cref{eq:a priori mean} therefore results in a much smaller error 
than predicted by the a priori error bound.  
Note that we have 
$(\lambda_{max}^{\gls{range_inner}}/\lambda_{min}^{\gls{range_inner}})^{1/2}\approx (\lambda_{max}^{\gls{source_inner}}/\lambda_{min}^{\gls{source_inner}})^{1/2}\approx 2$. 

The adaptive behavior of \cref{algo:adaptive_range_approximation}
is analyzed in \cref{fig:interface_adaptive_quartiles}.
\cref{fig:adaptive performance} shows that
for $\gls{number_testvectors} = 10$
the algorithm succeeded to generate a space with the
requested approximation quality every single time in 
the 100,000 test runs and most of the time,
the approximation quality is about one or two
orders of magnitude better than required.
\cref{fig:number of testvecs} shows
the influence of the number of test vectors $\gls{number_testvectors}$:
With a low number of test vectors like 3 or 5, 
the algorithm produces spaces with an approximation
quality much better than requested, which 
is unfavorable as the basis sizes are larger than 
necessary. 10 or 20 test vectors seem to be a 
good compromise, as enlarging $\gls{number_testvectors}$
to 40 or 80 results in only little improvements
while increasing computational cost. This different behavior of \cref{algo:adaptive_range_approximation} for various numbers of test vectors $n_{t}$ is due to the scaling of the effectivity of the a posteriori error estimator $\eta\left(\left(1 - \gls{orthogonalp}_{\gls{transfer_rrange}} \right) \gls{transferoperatorl}, \gls{number_testvectors}, \gls{epstestfail}\right)$ as defined in \cref{eq:effectivity} in the number of test vectors $n_{t}$: The median effectivity $\eta\left(\left(1 - \gls{orthogonalp}_{\gls{transfer_rrange}} \right) \gls{transferoperatorl}, \gls{number_testvectors}, \gls{epstestfail}\right)$
is 29.2 for $\gls{number_testvectors} = 10$, 10.4 for $\gls{number_testvectors} = 20$, and 6.1 for $\gls{number_testvectors} = 40$. We may thus also  conclude that the a posteriori error estimator \cref{eq:a posteriori error estimator} is a sharp bound for the present test case.

The quality of the produced spaces $\gls{transfer_rrange}$
should be independent of the mesh size $\gls{h}$.
\cref{fig:h dependency deviation}
confirms this.
After a preasymptotic regime, the deviation $\norm{\left(1 - \gls{orthogonalp}_{\gls{transfer_rrange}} \right) \gls{transferoperatorl}}$
is independent of the mesh size.
In the preasymptotic regime, the finite element
space is not capable of approximating the
corresponding modes.
But while the deviation $\norm{\left(1 - \gls{orthogonalp}_{\gls{transfer_rrange}} \right) \gls{transferoperatorl}}$
is independent of the mesh size, the norm of the test
vectors used in the a posteriori error estimator
in \cref{algo:adaptive_range_approximation}
is not (see \cref{fig:h dependency testvecnorm}).
The maximum norm of test vectors scales with the deviation and with $\sqrt{\gls{h}}$.
In the adaptive algorithm, the scaling with $\sqrt{\gls{h}}$ is compensated by the factor 
$({\lambda^{\gls{source_inner}}_{min}})^{-1/2}$ in $\cest$.
To analyze the behavior in $\gls{h}$, the geometry parameters were chosen as $\gls{L}=0.5$ and $\gls{W}=1$
to have a slower decay of the singular values of the transfer operator.

To examine CPU times we use \exrangea{}
in a larger configuration
with $\gls{L}=1$, $\gls{W}=8$ and $1/\gls{h} = 200$.
This results in 638.799 unknowns,
$\gls{dimension_transfer_source} = 3202$, and $\gls{dimension_transfer_range} = 1601$.
The
measured CPU times
for a simple, single threaded implementation
are given in 
\cref{tab:cpu_times}.
The transfer operator is implemented
implicitly. Its matrix is not
assembled. Instead, the corresponding problem
is solved using the sparse direct solver 
SuperLU \cite{superlu_ug99,superlu99}
each
time the operator is applied.
For \cref{algo:adaptive_range_approximation},
a target accuracy $\algotol$ of $10^{-4}$,
the number of testvectors $\gls{number_testvectors} = 20$,
and a maximum failure probability $\gls{epsalgofail} = 10^{-15}$
is used. In one test run, it resulted in an approximation space $\gls{transfer_rrange}$
of dimension $39$. It only evaluated the operator $n + \gls{number_testvectors} = 59$ times.
Each operator evaluation was measured to take $0.301$ seconds,
so a runtime of approximately $(n + \gls{number_testvectors}) * 0.301\mathrm{s} \approx 17.8$s is expected.
The measured runtime of 20.4 seconds is slightly higher, due to the
orthonormalization of the basis vectors and the projection of the test vectors.

CPU times for the calculation of the optimal space of same size
are given for comparison. 
The ``eigs'' function in ``scipy.sparse.linalg'', which is based on ARPACK, is used
to find the eigensystem of
$\gls{transferoperator} \gls{transferoperator}^*$.
However, the calculation using ARPACK is not adaptive. 
To employ ARPACK, the required number of vectors
has to be known in advance,
which is why we expect that in general,
the comparison would be even more in favor of the adaptive randomized
algorithm.

\begin{table}
\begin{center}
\begin{tabular}{r|c}
\multicolumn{2}{c}{\emph{Properties of transfer operator}}\\
\hline
unknowns of corresponding problem& 638,799\\
LU factorization time & 14.1 s\\
operator evaluation time & 0.301 s \\
adjoint operator evaluation time & 0.301 s
\end{tabular}\\
\vspace{10pt}
\begin{tabular}{r|c|c}
\multicolumn{3}{c}{\emph{Properties of basis generation}}\\
\cline{2-3}
& \cref{algo:adaptive_range_approximation} & Scipy/ARPACK \\
\hline
(resulting) basis size $n$ & 39 & 39 \\
operator evaluations & 59 & 79 \\
adjoint operator evaluations & 0 & 79 \\
\hline
execution time (w/o factorization) & 20.4 s & 47.9 s
\end{tabular}
\end{center}
\caption[CPU times for \exrangea{}.]
{
CPU times for \exrangea{}
with $\gls{L}=1$, $\gls{W}=8$ and $1/\gls{h} = 200$
in single threaded implementation.
(reproduction: \cref{repro:tab:cpu_times})
}
\label{tab:cpu_times}
\end{table}

\begin{figure}
\centering
\begin{subfigure}[t]{0.54\textwidth}
\centering
\begin{tikzpicture}
\begin{loglogaxis}[
    width=6cm,
    height=4cm,
    xmin=1e-13,
    xmax=1e4,
    x dir=reverse,
    ymin=1e-15,
    ymax=1e1,
    xlabel=target error $\algotol$,
    ylabel=$\norm{\left(1 - \gls{orthogonalp}_{\gls{transfer_rrange}} \right) \gls{transferoperatorl}}$,
    grid=both,
    grid style={line width=.1pt, draw=gray!20},
    major grid style={line width=.2pt,draw=gray!50},
    minor xtick={1e-15, 1e-14, 1e-13, 1e-12, 1e-11, 1e-10, 1e-9, 1e-8, 1e-7, 1e-6, 1e-5, 1e-4, 1e-3, 1e-2, 1e-1, 1e0, 1e1, 1e2, 1e3},
    minor ytick={1e-15, 1e-14, 1e-13, 1e-12, 1e-11, 1e-10, 1e-9, 1e-8, 1e-7, 1e-6, 1e-5, 1e-4, 1e-3, 1e-2, 1e-1, 1e0, 1e1, 1e2, 1e3},
    xtick={1e3, 1e0, 1e-3, 1e-6, 1e-9, 1e-12},
    ytick={1e3, 1e0, 1e-3, 1e-6, 1e-9, 1e-12},
    max space between ticks=25,
    legend style={at={(1.2,1.75)},anchor=north east},
  ]
  \addplot+[densely dotted, blue,  thick, mark=none] table[x index=0,y index=5] {twosquares_adaptive_convergence.dat};
  \addplot+[densely dashed, blue,  thick, mark=none] table[x index=0,y index=4] {twosquares_adaptive_convergence.dat};
  \addplot+[solid,          black, thick, mark=none] table[x index=0,y index=3] {twosquares_adaptive_convergence.dat};
  \addplot+[densely dashed, red,   thick, mark=none] table[x index=0,y index=2] {twosquares_adaptive_convergence.dat};
  \addplot+[densely dotted, red,   thick, mark=none] table[x index=0,y index=1] {twosquares_adaptive_convergence.dat};
    \addplot+[solid, grau, thick, mark=o,mark repeat=10] table[x index=0,y index=0] {twosquares_adaptive_convergence.dat};
  \legend{max, 75 percentile, 50 percentile, 25 percentile, min,$y=x$};
\end{loglogaxis}
\end{tikzpicture}
\subcaption{Quartiles for 10 test vectors.}
\label{fig:adaptive performance}
\end{subfigure}
\begin{subfigure}[t]{0.44\textwidth}
\centering
\begin{tikzpicture}
\begin{loglogaxis}[
    width=6cm,
    height=4cm,
    xmin=1e-13,
    xmax=1e4,
    x dir=reverse,
    ymin=1e-15,
    ymax=1e1,
    xlabel=target error $\algotol$,
    yticklabels={,,},
    grid=both,
    grid style={line width=.1pt, draw=gray!20},
    major grid style={line width=.2pt,draw=gray!50},
    minor xtick={1e-15, 1e-14, 1e-13, 1e-12, 1e-11, 1e-10, 1e-9, 1e-8, 1e-7, 1e-6, 1e-5, 1e-4, 1e-3, 1e-2, 1e-1, 1e0, 1e1, 1e2, 1e3},
    minor ytick={1e-15, 1e-14, 1e-13, 1e-12, 1e-11, 1e-10, 1e-9, 1e-8, 1e-7, 1e-6, 1e-5, 1e-4, 1e-3, 1e-2, 1e-1, 1e0, 1e1, 1e2, 1e3},
    xtick={1e3, 1e0, 1e-3, 1e-6, 1e-9, 1e-12},
    ytick={1e3, 1e0, 1e-3, 1e-6, 1e-9, 1e-12},
    legend style={at={(1.2,1.75)},anchor=north east},
  ]
  \addplot+[mark=square,      mark repeat=10, thick] table[x index=0,y index=1] {twosquares_adaptive_num_testvecs.dat};
  \addplot+[mark=o     , mark repeat=12, thick,solid] table[x index=0,y index=2] {twosquares_adaptive_num_testvecs.dat};
  \addplot+[mark=none,   mark repeat=14, thick,solid] table[x index=0,y index=3] {twosquares_adaptive_num_testvecs.dat};
  \addplot+[mark=triangle,  mark repeat=16, thick,solid] table[x index=0,y index=4] {twosquares_adaptive_num_testvecs.dat};
  \addplot+[mark=diamond, mark repeat=18, thick,solid] table[x index=0,y index=5] {twosquares_adaptive_num_testvecs.dat};
  \addplot+[mark=x,      mark repeat=20, thick,solid] table[x index=0,y index=6] {twosquares_adaptive_num_testvecs.dat};
    \addplot+[densely dashed,grau,thick,mark=none] table[x index=0,y index=0] {twosquares_adaptive_num_testvecs.dat};
  \legend{
    $\gls{number_testvectors} = 3$,
    $\gls{number_testvectors} = 5$,
    $\gls{number_testvectors} = 10$,
    $\gls{number_testvectors} = 20$,
    $\gls{number_testvectors} = 40$,
    $\gls{number_testvectors} = 80$,
    $y=x$,
  };
\end{loglogaxis}
\end{tikzpicture}
\subcaption{Maximum error for given number of test vectors.}
\label{fig:number of testvecs}
\end{subfigure}
\caption[Projection error over target projection error for \exrangea{}.]
{Projection
error $\sup_{\xi \in \gls{transfer_source} \nnull} \inf_{\zeta \in \gls{transfer_rrange}} \frac{\norm{ \gls{transferoperator}\xi - \zeta }_{\gls{transfer_range}}}{\norm{\xi }_{\gls{transfer_source}}} = \norm{\left(1 - \gls{orthogonalp}_{\gls{transfer_rrange}} \right) \gls{transferoperatorl}}$
over target projection error for \exrangea{}
with mesh size $\gls{h}=1/160$.
(reproduction: \cref{repro:fig:interface_adaptive_quartiles})
}
\label{fig:interface_adaptive_quartiles}
\end{figure}

\begin{figure}
\centering
\begin{subfigure}[t]{0.4\textwidth}
\centering
\begin{tikzpicture}
\begin{loglogaxis}[
    width=4.5cm,
    height=4.5cm,
    xtick={20,40,80,160,320},
    xticklabels={20,40,80,160,320},
    ytick={1e-10,1e-8, 1e-6, 1e-4, 1e-2, 1e0},
    xlabel=1/\gls{h},
    ylabel=median$\left(\norm{\left(1 - \gls{orthogonalp}_{\gls{transfer_rrange}} \right) \gls{transferoperatorl}}\right)$,
    legend pos=outer north east,
    grid=both,
    major grid style={line width=.2pt,draw=gray!70},
  ]
  \addplot+[] table[x index=0, y index=1] {experiment_h_deviation.dat};
  \addplot+[] table[x index=0, y index=2] {experiment_h_deviation.dat};
  \addplot+[] table[x index=0, y index=3] {experiment_h_deviation.dat};
  \addplot+[] table[x index=0, y index=4] {experiment_h_deviation.dat};
  \addplot+[] table[x index=0, y index=5] {experiment_h_deviation.dat};
  \addplot+[] table[x index=0, y index=6] {experiment_h_deviation.dat};
\end{loglogaxis}
\end{tikzpicture}
\subcaption{\gls{h} dependency of deviation}
\label{fig:h dependency deviation}
\end{subfigure}
\begin{subfigure}[t]{0.55\textwidth}
\centering
\begin{tikzpicture}
\begin{loglogaxis}[
    width=4.5cm,
    height=4.5cm,
    xtick={20,40,80,160,320},
    xticklabels={20,40,80,160,320},
    ytick={0.1,0.2,0.3,0.4},
    yticklabels={0.1,0.2,0.3,0.4},
    xlabel=1/\gls{h},
    ylabel=median$\left(\frac{\max_i \norm{\left(1 - \gls{orthogonalp}_{\gls{transfer_rrange}} \right) \gls{transferoperatorl}\gls{ritzdsinv} \underline r_i}}{\norm{\left(1 - \gls{orthogonalp}_{\gls{transfer_rrange}} \right) \gls{transferoperatorl}}}\right)$,
    legend pos=outer north east,
    grid=both,
    major grid style={line width=.2pt,draw=gray!70},
  ]
  \addplot+[] table[x index=0, y index=1] {experiment_h_testvecs.dat};
  \addplot+[] table[x index=0, y index=2] {experiment_h_testvecs.dat};
  \addplot+[] table[x index=0, y index=3] {experiment_h_testvecs.dat};
  \addplot+[] table[x index=0, y index=4] {experiment_h_testvecs.dat};
  \addplot+[] table[x index=0, y index=5] {experiment_h_testvecs.dat};
  \addplot+[] table[x index=0, y index=6] {experiment_h_testvecs.dat};
  \addplot+[mark=none, domain=20:320] {2*x^(-0.5)};
  \legend{n=2, n=4, n=6, n=8, n=10, n=12, $c \sqrt{\gls{h}}$};
\end{loglogaxis}
\end{tikzpicture}
\subcaption{\gls{h} dependency of testvector norm}
\label{fig:h dependency testvecnorm}
\end{subfigure}
\caption[\gls{h} dependency in \exrangea{}]
{\gls{h} dependency in \exrangea{};
Statistics over 10,000 samples.
(reproduction: \cref{repro:fig:h_dependency})
}
\label{fig:h_dependency}
\end{figure}

\subsubsection{Helmholtz Equation}\label{subsect:helmholtz}
In this subsection we analyze the behavior of the proposed algorithm
in a numerical test case approximating the solution of
the Helmholtz equation. The domain $\gls{domain}$, the boundaries $\gls{boundary}_{in}$
and $\gls{boundary}_{out}$ and the boundary conditions are the same as in \cref{subsect:analytic},
only the operator $\mathcal{A}$ differs and is defined as
$\mathcal{A} := -\Delta - \gls{helmholtz_wavenumber}^2$ in this subsection where \gls{helmholtz_wavenumber}
is the wave number.
As for \exrangea{}, it has 51,681 \gls{dof}s, $\gls{dimension_transfer_source} = 322$ and $\gls{dimension_transfer_range}= 161$ for $1/\gls{h}=160$.
We refer to this numerical example as \exrangeb{}.
We assume the problem to be inf-sup stable and thus uniquely solvable,
which is the case as long as it is not in a resonant configuration.

\input{figure_helmholtz_1.tex}
For $\gls{helmholtz_wavenumber}=0$ we obtain \exrangea{}.
We observe that the singular values of the transfer operator
first have a plateau and then decay exponentially,
see \cref{fig:helmholtz_svd}.
The longer the plateau, the faster is the exponential decay.
The length of the plateau is observed to be very close
to the length of the inner interface divided by 
a half wavelength, i.e.~ $1/(\lambda / 2) = \gls{helmholtz_wavenumber} / \pi$.
Comparing this with the
analysis of Finite Element methods for
the Helmholtz equation
(cf.~\cite{ihlenburg2006finite}),
one finds this similar to the
``minimal resolution condition''
$1/\gls{h} \geq \sqrt{12}/\gls{helmholtz_wavenumber}$.
However, we do not observe an equivalent of
the ``pollution effect'':
When doubling the frequency, 
doubling the number of degrees of freedom
is more than sufficient to maintain 
the same approximation quality.
For example, at $\gls{helmholtz_wavenumber}=10$ the
optimal space of size $5$ reaches an
approximation quality of about $10^{-6.7}$.
For $\gls{helmholtz_wavenumber}=20$, the optimal space of size $10$
reaches an approximation quality of
about $10^{-10.5}$ (\cref{fig:helmholtz_svd}).
Of course, both the discretization used
to compute the transfer operator as well
as the approximation scheme used to
solve the global problem using the 
locally computed approximation spaces
might still suffer from the pollution effect.

\cref{algo:adaptive_range_approximation}
succeeds to generate reduced spaces $\gls{transfer_rrange}$
which achieve a projection error {\color{white}{aaa}}
$\norm{\left(1 - \gls{orthogonalp}_{\gls{transfer_rrange}} \right) \gls{transferoperatorl}}$
which is close to the optimal projection error
given by the singular values of the transfer operator.
We show results for $\gls{helmholtz_wavenumber}=30$ in \cref{fig:helmholtz_a_priori}.
Also in the adaptive case, we observe the expected behavior, see \cref{fig:helmholtz adaptive performance}
and \cref{fig:helmholtz number of testvecs}.
The plateaus which can be observed in \cref{fig:helmholtz adaptive performance}
are due to the very fast decay of the singular values.
For example the first plateau is at an error of about $10^{-3}$, which
is the error usually achieved at a basis size of 10 (cf.~\cref{fig:helmholtz deviation percentiles}).
The next plateau at an error of about $10^{-7}$ corresponds to a basis size of 11.

\section{Olimex A64 Example}
\label{sec:olimex decay}
\begin{figure}
\begin{center}
\includegraphics[width=0.45\textwidth]{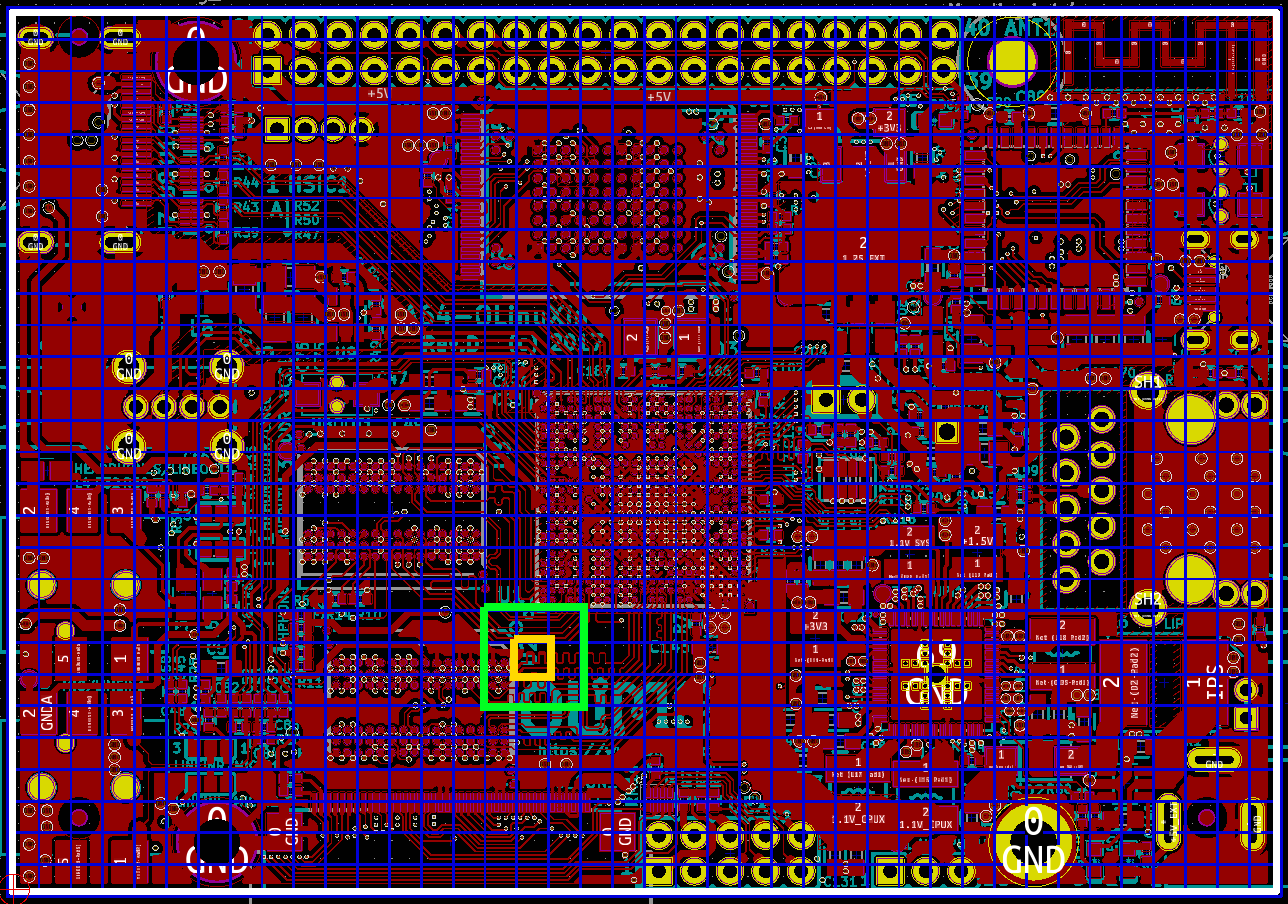}
\hspace{10pt}
\includegraphics[width=0.4\textwidth]{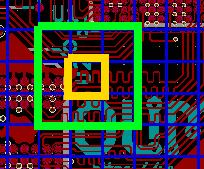}
\end{center}
\caption[Domain decomposition used for \exd{}.]{
Domain decomposition used for \exd{}
with 40 $\times$ 28 domains.
For each domain $\gls{subdomain}$ (e.g. $\omega_{816}$, {\color{yellow} yellow})
we introduce an oversampling domain $\omega^*_i$ (e.g. $\omega_{816}^*$, {\color{green} green}).
}
\label{fig:olimex dd}
\end{figure}

\cref{thm:training a priori},
our results in \cref{chap:arbilomod}
as well as the results in \cite{Buhr2018}
suggest that good reduction errors can be achieved
if the local approximation spaces capture
the most important parts of the image of
the transfer operators and thus 
$\norm{\left(1 - \gls{orthogonalp}_{\gls{rlfspace}} \right) \gls{transferopl}}$
is small.
Both the theoretical results about
the randomized range finder in \cref{sec:convergence_rate}
as well as the numerical results in \cref{sec:rangefinder artificial test}
show that good approximation spaces can be generated
in reasonable time if the singular value decay
of the transfer operators is strong enough.
In this section, we compute the singular value decay
of transfer operators defined in \exd{}
to judge on the applicability of
our methodology to this example.

We analyze \exd{} described in \cref{sec:olimex description}
for a fixed frequency of $\gls{angular_frequency} = 2 \pi \cdot 10^5 \mathrm{Hz}$.
We use a two dimensional non overlapping domain decomposition
$\left\{ \gls{subdomainnol} \right\}_{i=1}^{\gls{numdomainsnol}}$
with $\gls{numdomainsnol} = 1120$
depicted in \cref{fig:olimex dd}.
The size of each domain is approximately $(2mm)^3$.
Due to its simplicity, we use the
basic space decomposition defined in \cref{def:basic_space_def},
where \gls{fe} ansatz functions are grouped by the domains they
have support in.
We assume that one could obtain better results using the
wirebasket space decomposition.
We focus on domain $\omega_{816}$,
which is shown along with its oversampling
domain $\omega_{816}^*$ in
\cref{fig:olimex dd,fig:olimex example solution}.
It contains a part of the traces connecting
the CPU to the RAM and has
exceptional high geometric complexity.

We define the transfer operator as in
\cref{sec:codim_n_training},
except that we project to \gls{lfspacebasic}
instead of \gls{lfspacewb} at the end.
So let $\gls{coarsemeshentity}_{816}$ be $\omega_{816}$,
then we define
\begin{align}
T_{816}^l :\qquad& \coupling{\training{V_{816}^\mathrm{basic}}} \rightarrow V_{816}^\mathrm{basic} \nonumber \\
& \varphi \mapsto \mathcal{P}_{V_{816}^\mathrm{basic}} ( u_l )
\end{align}
where $u_l \in \training{V_{816}^\mathrm{basic}}$ solves
\begin{align}
\gls{a}(u_l + \varphi, \psi) = 0 \qquad \forall \psi \in \training{V_{816}^\mathrm{basic}}
.
\end{align}
For simplicity, we use the euclidean norm
of the coefficient vectors in $V_{816}^\mathrm{basic}$ and compare
with the singular values of the matrix of $T_{816}^l$.
In \cref{fig:olimexdecay}, the singular values
of the matrix of $T_{816}^l$ are shown.
Already at a basis size of 50,
a relative reduction of almost three orders of magnitude can be seen.
At a basis size of 100, a relative reduction of about 4.5
orders of magnitude can be seen. 
For comparison, the value of the a posteriori
norm estimator used in the randomized range finder
is plotted.
The randomized norm estimator behaves
approximately like
$\norm{\left(1 - \gls{orthogonalp}_{\gls{transfer_rrange}} \right) \gls{transferoperatorl}}$
but is much faster to compute.
As expected, it behaves like
$
\underline{\gls{singular_value}}_{n+1} \sqrt{n}
$.
Runtimes for the randomized range finder applied to 40 different domains
can be seen in \cref{fig:olimex rangefinder runtimes}.
For most domains, the runtime was between 5 and 10 minutes,
in an unoptimized, sequential implementation.

While these preliminary results are not conclusive,
they are encouraging.
The singular value decay suggests that good approximation
spaces of dimension less than 100 per domain exist and
can be generated in reasonable time. In a domain
decomposition with about 1000 domains, this would
result in a global reduced model of around
100 000 unknowns, which could probably be solved quickly,
given that the matrices of a localized reduction
scheme are block-sparse.
The generation of the reduced spaces
can easily be parallelized and executed in a cloud environment.

\begin{figure}
\begin{center}
\includegraphics[width=0.45\textwidth]{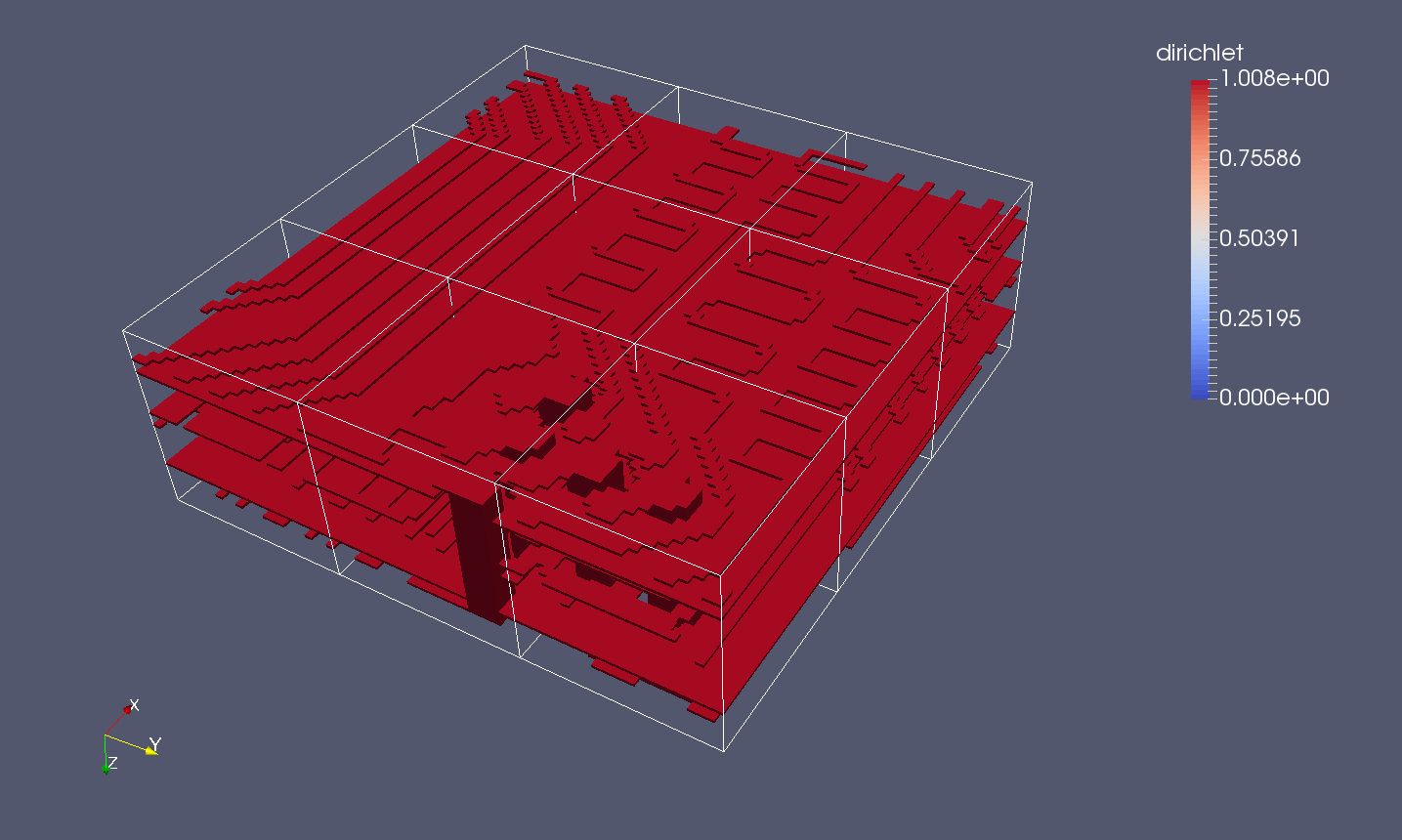}
\hspace{10pt}
\includegraphics[width=0.45\textwidth]{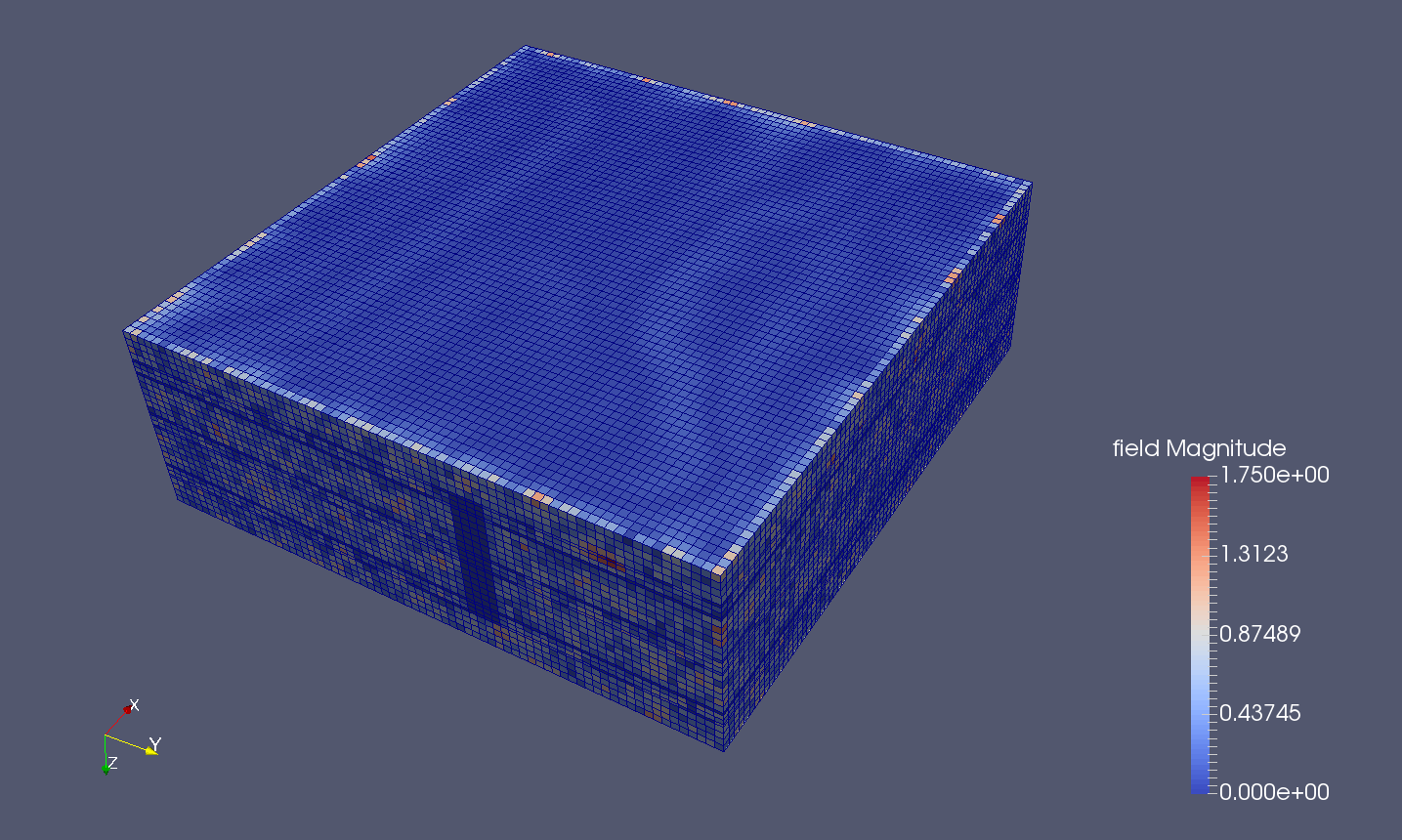}
\end{center}
\caption[Example solution in oversampling domain in \exd{}.]
{Example randomized solution in oversampling domain of $\omega_{816}$ of \exd{}
(approx.~350.000 unknowns).}
\label{fig:olimex example solution}
\end{figure}
\clearpage
\begin{figure}
\begin{center}
\begin{tikzpicture}
\begin{semilogyaxis}[
    width=8cm,
    xlabel=$n$,
]
\addplot [mark=none, thick, red] table [x expr=\coordindex, y index=0, header=false] {svals_816.dat};
\addplot [mark=none, thick, red, densely dashed] table [x expr=\coordindex, y index=1, header=false] {svals_816.dat};
\addplot [mark=none, thick, blue] table [x expr=\coordindex, y index=0, header=false] {maxnorms_2norm_816.dat};
\legend{$\underline{\gls{singular_value}}_{n+1}$, $\underline{\gls{singular_value}}_{n+1} \sqrt{n}$, $\Delta(\widetilde R_n)$}
\end{semilogyaxis}
\end{tikzpicture}
\end{center}
\caption[Singular value decay of $\gls{transferoperator}_{816}$ for \exd{}.]
{Singular value decay of $\gls{transferoperator}_{816}$
and decay of randomized norm estimator
for \exd{}.
(reproduction: \cref{repro:fig:olimexdecay})
}
\label{fig:olimexdecay}
\end{figure}

\begin{figure}
\begin{center}
\begin{tikzpicture}
\begin{axis}[
    width=14cm,
    height=6.5cm,
    ybar stacked,
    bar width=2mm,
    ymin=0,
    legend style={at={(1.1,1.1)},anchor=north east},
    xlabel=domain,
    ylabel=time/$s$,
]

\addplot+[ybar] table [x=domainnum, y=assemblytime] {timings.dat} ;
\addplot+[ybar] table [x=domainnum, y=factorizationtime] {timings.dat} ;
\addplot+[ybar] table [x=domainnum, y=testvector_generation_time] {timings.dat} ;
\addplot+[ybar] table [x=domainnum, y=basis_generation_time] {timings.dat} ;
\legend{assembly, factorization, testvector generation, basis generation}
\end{axis}
\end{tikzpicture}
\end{center}
\caption[Runtimes for randomized range finder on \exd{}.]
{Runtimes for randomized range finder on \exd{}.\\
(reproduction: \cref{repro:fig:olimexdecay})
}
\label{fig:olimex rangefinder runtimes}
\end{figure}

\clearpage
\futurebox{
\section{Future Research on Randomized Range Finder}
\label{sec:future of randomized range finder}
\myline{}
While a fast and efficient algorithm was devised
for approximating the image of a linear operator,
there is lots of room for improvement and future
research. In the following, we state
what we consider the most interesting or important
improvements which could be tackled in future
research projects.
\begin{itemize}
\item
\textbfit{Parameterized problems}\\
The presented algorithm approximates the image
of a linear operator. In the context of model
reduction, one often deals with parameterized problems, which results in 
parameterized transfer operators \gls{transferopmu}.
The only published ways the authors are aware of to approximate the image
of a parameterized operator are the ``spectral greedy''
in \cite[Algorithm 4.1]{Smetana2016}
and the \gls{pod} strategy in 
\cite{Taddei2018}.
Both approaches require the execution of the randomized range finder
at every training parameter and are thus very expensive
for large training sets \gls{trainingset}.
It would be desirable to have an algorithm to approximate the image of a
parameterized operator without evaluating it at each training parameter.
\item
\textbfit{Complex numbers}\\
The results presented here are only valid for the real case.
An extension to complex numbers should be done. This, however, would
require reevaluation of lot of results from RandNLA.
Michael Mahoney and his coworkers only deal with the real case 
in their publications. Per-Gunnar Martinsson and his coworkers
mention the complex case in their publications. The
relevant proofs are, however, only done for the real case.
Most results in RandNLA use estimates for eigenvalues
of random matrices like the condition number
estimates in \cite{Chen2005}, which are published
also for complex matrices.
Recreating the results from RandNLA for complex matrices,
starting at the results from \cite{Chen2005}
to finally devise a complex randomized range finder algorithm
would be an interesting project.
\item
\textbfit{Better random vectors}\\
The drawing of random vectors presented here is simple
and the quality of the generated spaces deteriorates when
the condition of the matrix of the inner product in $\gls{transfer_source}$
becomes high. A better way of drawing random vectors should
be developed, taking into account the inner product in $\gls{transfer_source}$,
removing the dependency on the condition of the matrix
of the inner product in $\gls{transfer_source}$.
\item
\textbfit{Better a posteriori error estimator}\\
The presented a posteriori error estimator
only uses the maximum of all norms of all test vectors
and thus looses information.
Probably, a better a posteriori error estimator can be devised
using information from all test vector norms.
\item
\textbfit{Better bases through postprocessing}\\
In the presented version of the randomized range finder,
all basis vectors produced are used.
It would probably lead to better basis vectors if one generated
some more vectors than necessary and reduce the number
using an \gls{svd} afterwards.
\item
\textbfit{Infinite dimensional case}\\
It would be interesting to formulate a randomized range finder
which works on compact operators in Hilbert spaces.
This would require a new way of drawing random vectors. 
\end{itemize}
}
\clearpage
\futurebox{
\section{Future Research for Transferoperator Based Training}
While the previous items were concerned with 
improving the randomized range finder algorithm itself,
the training procedure can also be further improved.
\myline{}
\begin{itemize}
\item
\textbfit{Better Transfer Operators}\\
While we can approximate the image of the presented
transfer operators in an optimal way, there might
be better transfer operators with faster decay
of singular values.
Recently, slightly better results have been demonstrated
by using Robin- instead of Dirichlet boundary values
to formulate the transfer operator
\cite{EickhornThesis}.
\item
\textbfit{Application to \gls{arbilomod}}\\
While the application of \cref{algo:adaptive_range_approximation}
in the \gls{arbilomod} context is straightforward,
application of \cref{thm:training a priori}
needs estimates for $\norm{\gls{tracespacemap}}$,
which are not available so far.
\end{itemize}
}

\chapter{Online Enrichment}
\label{chap:enrichment}
In this chapter, we introduce two
algorithms for localized online enrichment
in an abstract localized model order reduction setting.
We show a priori convergence rates and verify the results
in a numerical example.
\section{Introduction}
By the term ``online enrichment'', we denote
the extension of our localized approximation spaces $\gls{rlfspace}$,
based on a previous solution of the global, reduced problem.
In contrast to the trainings, where only local information is used,
online enrichment uses global information, by using the previous global, reduced solution.
Online enrichment serves two purposes.
First, it can be used in a situation where 
local approximation spaces have been generated using a training procedure,
but the reduced solution is not good enough according to the
a posteriori error estimator.
Second, it provides a means to generate local approximation 
spaces in settings where a training procedure fails to generate
good approximation spaces.
This could be for instance in problems with many channels,
of which only very few are used.
Enrichment is done by solving the original problem
in a local space, using the residual of the last
solution as right hand side.

Online enrichment is a widely used technique in localized model order reduction methods.
\gls{arbilomod}, introduced in \cite{Buhr2017}, employs online enrichment.
For the \gls{lrbms} which was introduced in \cite{Albrecht2012},
online enrichment was presented in \cite{Albrecht2013} and in \cite{Ohlberger2015}.
For the \gls{gmsfem} which was introduced in \cite{Efendiev2013},
online enrichment was discussed in \cite{Chung2015a}.
For the \gls{cemgmsfem} which was introduced in \cite{Chung2017a}, 
online enrichment was discussed in \cite{Chung2017c}.
However, no a priori convergence analysis is available for the \gls{arbilomod} and the \gls{lrbms}.
For the \gls{gmsfem} and \gls{cemgmsfem}, exponential convergence was shown in \cite{Chung2015a} and \cite{Chung2017c}.

\section{Enrichment Algorithms}
\IncMargin{1em}
\begin{algorithm2e}
\DontPrintSemicolon%
\SetAlgoVlined%
$n \leftarrow 0$\;
$\gls{rfspace}_n \leftarrow \spanset{\{0\}}$\;
\While{
  not converged
}%
{
  \tcc{solve reduced system}
  \Find{$\gls{rsol}_n \in \gls{rfspace}_n \st$}{
    $\gls{a}(\gls{rsol}_n, \varphi) = \gls{f}(\varphi) \qquad \forall \varphi \in \gls{rfspace}_n$}
  \tcc{find maximum local residual}
  $k \leftarrow \argmax\limits_{i=1, \dots, \gls{numspaces}} \norm{\gls{residual}(\gls{rsol}_n)}_{\gls{lfspace}'}$\;
  \tcc{solve local enrichment problem}
  \Find{$\localenrichment \in V_k \st$}{
    $\gls{a}(\localenrichment, \varphi) = \gls{residual}(\gls{rsol}_n)(\varphi) \qquad \forall \varphi \in V_k$}
  \tcc{form enriched space}
  $\gls{rfspace}_{n+1} \leftarrow \gls{rfspace}_n \oplus \spanset{\{\localenrichment\}}$\;
  $n \leftarrow n+1$\;
}
\caption{Residual Based Online Enrichment.}
\label{algo:online_enrichment}
\end{algorithm2e}
\DecMargin{1em}

\IncMargin{1em}
\begin{algorithm2e}
\DontPrintSemicolon%
\SetAlgoVlined%
$n \leftarrow 0$\;
$\gls{rfspace}_n \leftarrow \spanset{\{0\}}$\;
\While{
  not converged
}%
{
  \tcc{solve reduced system}
  \Find{$\gls{rsol}_n \in \gls{rfspace}_n \ \st$}{
    $\gls{a}(\gls{rsol}_n, \varphi) = \gls{f}(\varphi) \qquad \forall \varphi \in \gls{rfspace}_n$}
  \tcc{solve local enriched problems}
  \For{$i = 1, \dots, \gls{numspaces}$}{
    \Find{$\localenrichedsolution{i} \in \gls{rfspace}_n \oplus \gls{lfspace} \ \st$}{
      $\gls{a}(\localenrichedsolution{i}, \varphi) = \gls{f}(\varphi) \qquad \forall \varphi \in \gls{rfspace}_n \oplus \gls{lfspace}$}
  }
  \tcc{find maximum solution shift}
  $k \leftarrow \argmax\limits_{i=1, \dots, \gls{numospaces}} \norm{\gls{rsol}_n - \localenrichedsolution{i}}_a$\;
  \tcc{form enriched space}
  $\gls{rfspace}_{n+1} \leftarrow \gls{rfspace}_n \oplus \spanset{\{\localenrichedsolution{k}\}}$\;
  $n \leftarrow n+1$\;
}
\caption{Globally Coupled Online Enrichment.}
\label{algo:globally_coupled_online_enrichment}
\end{algorithm2e}
\DecMargin{1em}

In \cref{sec:enrichment}, an enrichment
procedure was introduced. This is, however, only
a heuristic approach for which no a priori convergence bounds can be shown.
In this section, we introduce two enrichment procedures in
an abstract setting as introduced in
\cref{sec:localized model order reduction},
for which we can show a priori convergence rates in the following section.

\subsection*{Setting}
Let 
$\left\{
\gls{domain}, \gls{fspace},
\{\gls{subdomain}\}_{i=1}^{\gls{numspaces}},
\{\gls{lfspace}\}_{i=1}^{\gls{numspaces}},
\{\gls{lfspacemap}\}_{i=1}^{\gls{numspaces}}
\right\}$
be a localizing space decomposition as defined in
\cref{def:localizing space decomposition}
and let \gls{sol} be the solution of a non parametric variational problem 
as defined in \cref{def:variational problem}.
Let \gls{a} be symmetric and coercive
and let the full space \gls{fspace} be equipped with
the energy norm induced by \gls{a}.
Starting with the nullspace $\gls{rfspace}_0$, we construct
a sequence of subspaces of $V$ which we denote by
$\gls{rfspace}_n$.
The full problem is reduced by Galerkin projection on these
reduced spaces.
We denote the solutions of the reduced problems by $\gls{rsol}_n$.
Each reduced space $\gls{rfspace}_n$ is constructed by enriching the previous
reduced space with an additional basis function $\localenrichment$, 
which lies in one of the localized spaces $\gls{lfspace}$.

\subsection*{Residual Based Enrichment}
The residual based enrichment algorithm 
(given as \cref{algo:online_enrichment})
first selects the local enrichment space from which the 
enrichment function $\localenrichment$ is taken.
The local space $V_k$ which maximizes the dual norm
of the residual $\norm{\gls{residual}(\gls{rsol}_n)}_{\gls{lfspace}'} = \sup_{\varphi \in \gls{lfspace} \nnull} \frac{\gls{residual}(\gls{rsol}_n)(\varphi)}{\norm{\varphi}_a}$ is chosen,
i.e.
\begin{equation}
k:= \argmax\limits_{i\in \{1, \dots, \gls{numspaces}\}} \norm{\gls{residual}(\gls{rsol}_n)}_{\gls{lfspace}'}.
\end{equation}
The residual $\gls{residual}(\gls{rsol}_n) \in \gls{fspace}'$ is defined as
$\gls{residual}(\gls{rsol}_n)(\cdot) := \gls{f}(\cdot) - \gls{a}(\gls{rsol}_n, \cdot)$.
Then a localized problem is formed by 
a Galerkin projection
of the original problem onto this local space
$V_k$,
and replacing the right hand side $\gls{f}$ by the last residual $\gls{residual}(\gls{rsol}_n)$.
The solution of the localized problem is the enrichment function $\localenrichment$.
\subsection*{Globally Coupled Local Enrichment}
The globally coupled local enrichment algorithm
(given as \cref{algo:globally_coupled_online_enrichment})
couples the global reduced space with the full local space.
First it iterates over all local spaces $\gls{lfspace}$
and solves the coupled problem: It solves
the original problem projected on the space
$\gls{rfspace}_n \oplus \gls{lfspace}$, the solution of
this coupled problem is called $\localenrichedsolution{i}$.
Then the local space $V_k$ is selected which maximizes
the change in the solution $\norm{\gls{rsol}_n - \localenrichedsolution{k}}_a$, i.e.~
\begin{equation}
k:= \argmax\limits_{i\in \{1, \dots, \gls{numspaces}\}}
\norm{\gls{rsol}_n - \localenrichedsolution{i}}_a
.
\end{equation}
The function $\localenrichedsolution{k}$ is used to enrich the space $\gls{rfspace}_n$.
Note that this is an enrichment in $V_k$, even though $\localenrichedsolution{k}$ has global support.

\FloatBarrier
\section{A Priori Convergence Estimates}
\label{sec:enrichment convergence}
First we prove exponential convergence for the residual based enrichment.
\begin{theorem}[Exponential convergence of residual based online enrichment]
\label{maintheorem}
With the assumptions from above,
for the reduced solutions $\gls{rsol}_{n+1}$ in \cref{algo:online_enrichment} it holds
\begin{equation}
\norm{\gls{sol} - \gls{rsol}_{n+1}}_a
\leq c_\mathrm{rbe} \cdot 
\norm{\gls{sol} - \gls{rsol}_{n}}_a
\end{equation}
with
\begin{equation}
\label{defc}
c_\mathrm{rbe} := \sqrt{1 - \frac{1}{\gls{numspaces}} \frac{1}{\gls{decompositionboundvr}^2}}
\end{equation}
with the constant $\gls{decompositionboundvr}$ as defined in \cref{thm:localization_g}.
\end{theorem}

\begin{proof}
As we use the energy norm, the solution is the best approximation
\begin{equation}
\norm{\gls{rsol}_{n+1} - \gls{sol}}_a \leq \norm{\varphi - \gls{sol}}_a \qquad \forall \varphi \in \gls{rfspace}_{n+1} .
\end{equation}
This holds for $\varphi = \gls{rsol}_n + \alpha \localenrichment$ for all $\alpha$ in $\gls{R}$.
Because of the symmetry of $\gls{a}$ it holds
\begin{equation}
\nonumber
\norm{\gls{rsol}_n + \alpha \localenrichment - \gls{sol}}_a^2 = \norm{\gls{rsol}_n - \gls{sol}}_a^2 - 2 \alpha \gls{residual}(\gls{rsol}_n)(\localenrichment) + \alpha^2 \norm{\localenrichment}_a^2 .
\end{equation}
This term is minimized by choosing $\alpha = \gls{residual}(\gls{rsol}_n)(\localenrichment) / \norm{\localenrichment}_a^2$. We use this $\alpha$ and 
realize that
$\gls{residual}(\gls{rsol}_n)(\localenrichment) = \norm{\localenrichment}_a^2 = \norm{\gls{residual}(\gls{rsol}_n)}_{V_k'}^2$,
because $\localenrichment$ is the Riesz representative of $\gls{residual}(\gls{rsol}_n)$ in $V_k$ in the energy norm.
It follows that
\begin{equation}
\label{eq:chungend}
\norm{\gls{rsol}_{n+1} - \gls{sol}}_a^2 \leq \norm{\gls{rsol}_n - \gls{sol}}_a^2 - \norm{\gls{residual}(\gls{rsol}_n)}_{V_k'}^2 .
\end{equation}
Till this point, the proof followed the structure given in \cite[Section 4]{Chung2015a}. 
We defined $k$ to select the largest local residual, so it holds
\begin{equation}
\label{eq:r1}
\norm{\gls{residual}(\gls{rsol}_n)}_{V_k'}^2 \geq \frac{1}{\gls{numspaces}} \sum_{i=1}^{\gls{numspaces}} \norm{\gls{residual}(\gls{rsol}_n)}_{\gls{lfspace}'}^2 .
\end{equation}
Furthermore, from \cref{thm:localization_g} we know
\begin{equation}
\label{eq:r2}
\norm{\gls{residual}(\gls{rsol}_n)}_{V'} \leq \gls{decompositionboundvr} \left(\sum_{i=1}^{\gls{numspaces}} \norm{\gls{residual}(\gls{rsol}_n)}_{\gls{lfspace}'}^2 \right) ^{\frac 1 2}
\end{equation}
As we use the energy norm, we have 
\begin{equation}
\label{eq:r3}
\norm{\gls{residual}(\gls{rsol}_n)}_{V'}^2 = \norm{\gls{rsol}_n - \gls{sol}}_a^2 .
\end{equation}
Combining \eqref{eq:r1}, \eqref{eq:r2}, and \eqref{eq:r3} we obtain
\begin{equation}
\label{eq:rall}
\norm{\gls{residual}(\gls{rsol}_n)}_{V_k'}^2 \geq \frac{1}{\gls{numspaces}} \frac{1}{\gls{decompositionboundvr}^2} \norm{\gls{rsol}_n - \gls{sol}}_a^2 .
\end{equation}
Combining \eqref{eq:rall} with \eqref{eq:chungend} yields
\begin{equation}
\norm{\gls{rsol}_{n+1} - \gls{sol}}_a^2 \leq \left( 1 - \frac{1}{\gls{numspaces}} \frac{1}{\gls{decompositionboundvr}^2} \right) \norm{\gls{rsol}_n - \gls{sol}}_a^2
\end{equation}
and thus the claim.
\end{proof}

\begin{corollary}[Exponential convergence from iteration 0]
\label{cor1}
For the reduced solutions $\gls{rsol}_n$ in \cref{algo:online_enrichment} it holds
\begin{equation}
\norm{\gls{rsol}_{n} - \gls{sol}}_a
\leq c_\mathrm{rbe}^n \cdot 
\norm{\gls{sol}}_a
\end{equation}
with $c_\mathrm{rbe}$ as defined in \eqref{defc}.
\end{corollary}
\begin{proof}
This follows from \cref{maintheorem}, because 
$\gls{rsol}_0 = 0$.
\end{proof}

\subsection*{Globally coupled enrichment}
The globally coupled enrichment given in \cref{algo:globally_coupled_online_enrichment}
is the optimal enrichment: Among all enrichment functions
from all local spaces, it selects the one which minimizes the
resulting error in the energy norm.
\begin{theorem}[Optimality of globally coupled online enrichment (\cref{algo:globally_coupled_online_enrichment})]
For the reduced solutions $\gls{rsol}_{n+1}$ in \cref{algo:globally_coupled_online_enrichment} it holds
\begin{equation}
\begin{multlined}
\norm{\gls{sol} - \gls{rsol}_{n+1}}_a = %
\min_{i = 1, \dots, \gls{numspaces}}
\min_{\psi_e \in \gls{lfspace}}
\Big\{
\norm{\gls{sol} - \gls{rsol}_e}_a
\ \Big| \ 
\gls{rsol}_e \in \gls{rfspace}_n \oplus \spanset \{ \psi_e \}
\ \mathrm{solves} \qquad \qquad \qquad  \hfill \\ 
\hspace{20pt} a(\gls{rsol}_e, \varphi) = f(\varphi) \qquad \forall \varphi \in \gls{rfspace}_n  \oplus \spanset \{ \psi_e \}
\Big\} .
\end{multlined}
\end{equation}
\end{theorem}
\begin{proof}
First we realize that the solution $\gls{rsol}_{n+1}$ is identical to 
$\localenrichedsolution{k}$, because $\localenrichedsolution{k}$ solves
$a(\localenrichedsolution{k}, \varphi) = f(\varphi)$ in
$\gls{rfspace}_{n+1}$, which is a subspace of $\gls{rfspace}_n \oplus V_k$ and the solution is unique:
\begin{equation}
\gls{rsol}_{n+1} = \localenrichedsolution{k} .
\end{equation}
Second, since $\gls{sol} - \localenrichedsolution{i}$ is $a$-orthogonal to
$\localenrichedsolution{i} - \gls{rsol}_n$, it holds
\begin{equation}
\norm{\gls{sol} - \gls{rsol}_n}_a^2 = \norm{\gls{sol} - \localenrichedsolution{i}}_a^2 + \norm{\localenrichedsolution{i} - \gls{rsol}_n}_a^2
\end{equation}
and thus 
\begin{equation}
k = \argmax\limits_{i\in \{1, \dots, \gls{numspaces}\}} \norm{\gls{rsol}_n 
- \localenrichedsolution{i}}_a
\end{equation}
implies
\begin{equation}
k = \argmin\limits_{i\in \{1, \dots, \gls{numspaces}\}} \norm{\gls{sol} - \localenrichedsolution{i}}_a.
\end{equation}
So $\localenrichedsolution{k}$ is closest to \gls{sol} among all $\localenrichedsolution{i}$.

Third, $\localenrichedsolution{i}$ is the best approximation in $\gls{rfspace}_n \oplus \gls{lfspace}$.
It is not possible to get closer to \gls{sol} with any other enrichment in $\gls{lfspace}$.
\end{proof}

\begin{corollary}[Exponential convergence of globally coupled online enrichment]
The results in \cref{maintheorem} and \cref{cor1} also hold for \cref{algo:globally_coupled_online_enrichment}.
\end{corollary}
\begin{proof}
The enrichment in \cref{algo:globally_coupled_online_enrichment} is optimal
in every step,
so it is not worse than the enrichment of \cref{algo:online_enrichment}.
Any bound for \cref{algo:online_enrichment} holds also for 
\cref{algo:globally_coupled_online_enrichment}.
\end{proof}
\section{Application to Partition of Unity Decomposition}
While the two enrichment algorithms given above can be formulated
in an abstract setting, there are some constraints in implementing them.
The residual based enrichment requires the evaluation of dual norms
of the residuals in \gls{lfspace}, which usually require a basis
of these spaces. This would be prohibitively expensive
for e.g.~the wirebasket space decomposition, as the construction
of a basis for e.g.~an edge space would require one local solve
operation for every \gls{dof} on the edge. In addition,
the storage requirements of these vectors would be high.
Furthermore, even if the basis was available, the
computation of the Riesz representative for the evaluation of the
dual norm could be computationally expensive, as the matrix
for the inner product could become dense.
A similar argument holds for the globally coupled local enrichment.
Here the solution in $\gls{rfspace}_n \oplus \gls{lfspace}$
would require a basis of \gls{lfspace}.
From the different space decompositions introduced in this thesis,
only the partition of unity decomposition introduced in
\cref{sec:poudecomposition} allows for a straightforward implementation,
as its local spaces are spanned by \gls{fe} ansatz functions.
So a basis is already available and also the computation of
Riesz representatives for the evaluation of the dual norm
is quickly possible, as the inner product matrices are sparse.
We use the partition of unity decomposition in the following numerical example.

In \cref{sec:choosing spaces}, we derived an upper bound for \gls{decompositionboundvr}
for the partition of unity decomposition and the $H^1$-norm.
As the a priori estimates derived here require the energy norm, we give
an upper bound for \gls{decompositionboundvr} in the energy norm
in the following.
\begin{proposition}[Upper bound of $\gls{decompositionboundvr}$]
\label{thm:upper bound for cpu in energy norm}
With the assumptions about the energy norm from above,
for the partition of unity decomposition as defined in \cref{sec:poudecomposition}
based on the partition of unity functions $\{\gls{pouf}\}_{i=1}^{\gls{numdomainsol}}$.
Let $\gls{overlappingspaces}$ be the maximum number of functions $\gls{pouf}$ having support
in any given point $x$ in $\gls{domain}$, let $\frac{\gls{heat_conductivity}_{max}}{\gls{heat_conductivity}_{min}}$
being the contrast of the problem, let $\gls{friedrichs}$ the constant of the
Friedrich's inequality on $\gls{domain}$ and $\norm{\cdot}_\infty$
the infinity norm.
Neglecting the interpolation operator mapping back to the \gls{fe} space, it holds that
\begin{equation}
\gls{decompositionboundvr}^2 \leq
2 \gls{overlappingspaces} \left(\gls{friedrichs}
\frac{\gls{heat_conductivity}_{max}}{\gls{heat_conductivity}_{min}}
\max_{i\in \{1, \dots, \gls{numspaces}\}} \norm{\nabla \gls{pouf}}_\infty^2 + \max_{i\in \{1, \dots, \gls{numspaces}\}} \norm{\gls{pouf}}_\infty^2
\right).
\end{equation}
\end{proposition}
\begin{proof}
Starting from 
\begin{equation}
\gls{decompositionboundvr} \leq
\sup_{\varphi \in \gls{fspace} \nnull}
\frac{
  \sum_{i=1}^{\gls{numdomainsol}} \norm{\gls{pouf} \varphi}_a^2
}{
  \norm{\varphi}_a^2
}
\end{equation}
we estimate 
$\sum_{i=1}^{\gls{numspaces}} \norm{\gls{pouf} \varphi}_a^2$. It holds that
\begin{eqnarray}
\sum_{i=1}^{\gls{numspaces}} \norm{\gls{pouf} \varphi}_a^2
&=&
\sum_{i=1}^{\gls{numspaces}} \int_{\gls{subdomainol}} \gls{heat_conductivity} |\nabla(\gls{pouf} \varphi)|^2 \dx
\\
\label{eqlast}
&\leq&
\sum_{i=1}^{\gls{numspaces}} \int_{\gls{subdomainol}} \gls{heat_conductivity} \left(2|(\nabla \gls{pouf}) \varphi|^2 + 2|\gls{pouf} (\nabla \varphi)|^2 \right) \dx .
\end{eqnarray}
For the second term in \eqref{eqlast} it holds that
\begin{equation}
\label{part1}
\sum_{i=1}^{\gls{numspaces}} \int_{\gls{subdomainol}} 2 \gls{heat_conductivity} |\gls{pouf} (\nabla \varphi)|^2 \dx
\leq 2 \gls{overlappingspaces} \max_{i\in \{1, \dots, \gls{numspaces}\}} \norm{\gls{pouf}}_\infty^2 \norm{\varphi}_a^2
\end{equation}
and for the first term we have
\begin{equation}
\sum_{i=1}^{\gls{numspaces}} \int_{\gls{subdomainol}} 2 \gls{heat_conductivity} |(\nabla \gls{pouf}) \varphi|^2 \dx
\leq
2 \gls{overlappingspaces} \gls{heat_conductivity}_{max} \max_{i\in \{1, \dots, \gls{numspaces}\}} \norm{\nabla \gls{pouf}}_\infty^2 \int_{\gls{domain}} |\varphi|^2 \dx
\nonumber
\end{equation}
\begin{equation}
\label{part2}
\leq
2 \gls{overlappingspaces} \gls{friedrichs} \frac{\gls{heat_conductivity}_{max}}{\gls{heat_conductivity}_{min}} \max_{i\in \{1, \dots, \gls{numspaces}\}} \norm{\nabla \gls{pouf}}_\infty^2 \norm{\varphi}_a^2
.
\end{equation}
Combining these yields the claim.
\end{proof}
\section{Numerical Example}
\label{sec:enrichment numerical example}
\begin{figure}
\centering
{%
\setlength{\fboxsep}{0pt}%
\setlength{\fboxrule}{1pt}%
\fbox{\includegraphics[width=0.2\textwidth]{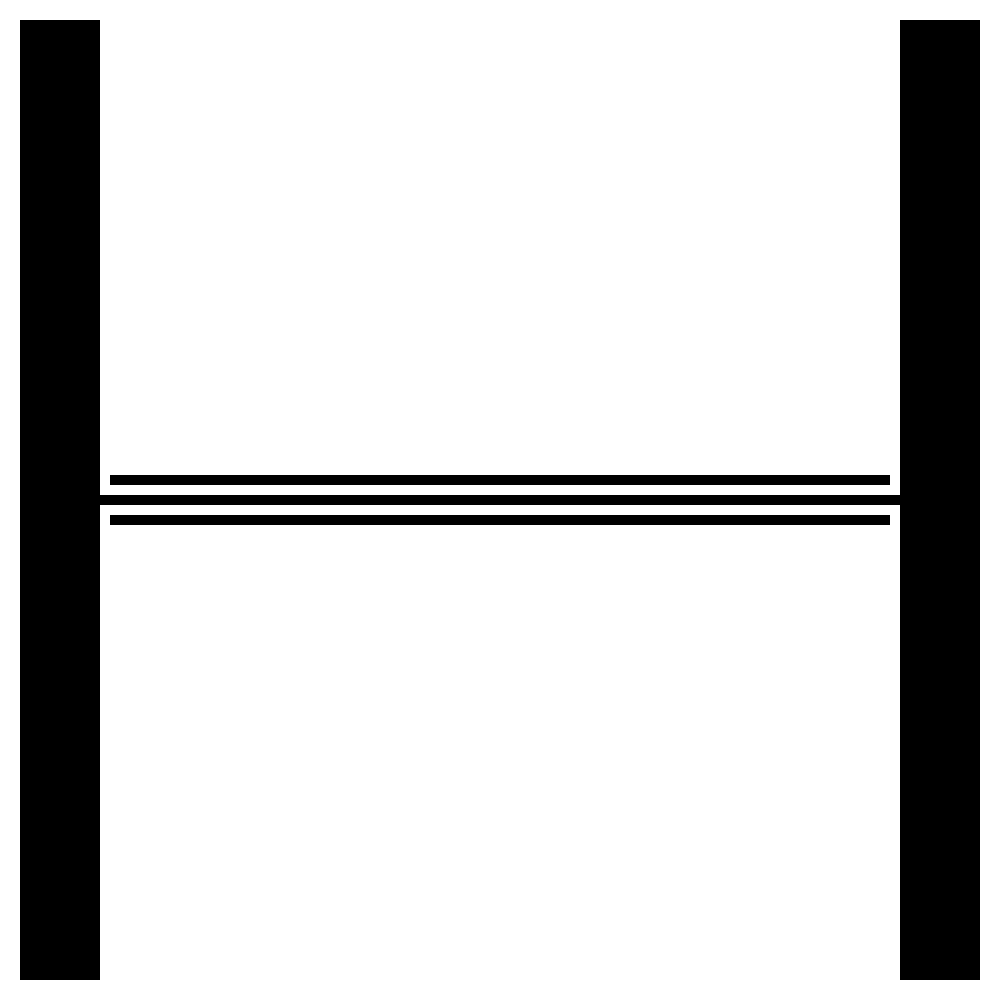}}
}%
\ \ \ 
\includegraphics[width=0.2\textwidth]{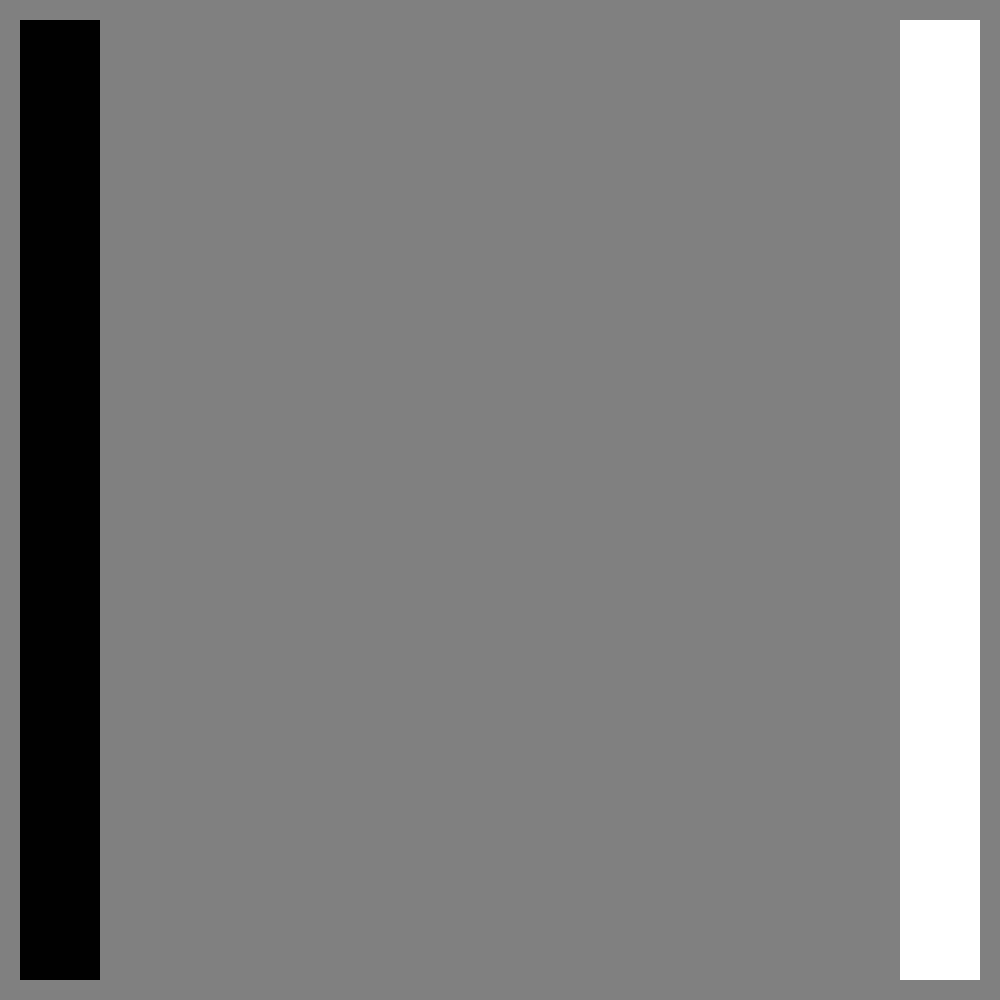}
\includegraphics[width=0.34\textwidth]{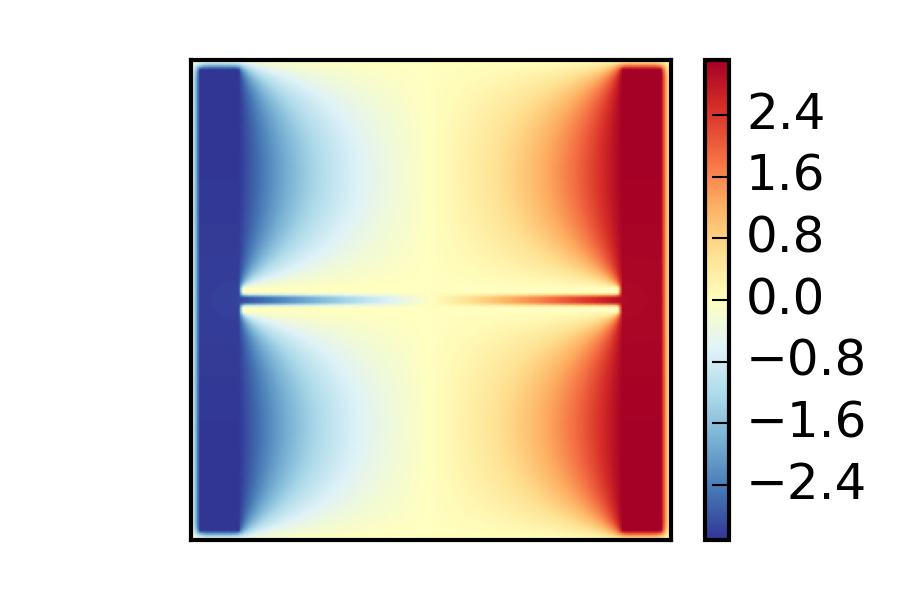}
\caption[Coefficient field $\gls{heat_conductivity}$ and right hand side $f$ of \exbv{}.]
        {Geometry of \exbv{}. Left: Coefficient field $\gls{heat_conductivity}$. White is 1, black is $10^5$. Middle: right hand side $\gls{heatsource}$. Black is $-10^5$, gray is $0$, white is $10^5$. Right: reference solution $\gls{sol}$.
(reproduction: \cref{repro:fig:enrichment geometry})
}
\label{fig:enrichment geometry}
\end{figure}

\input{fig_dd}
While the method and the proofs work both in two and three dimensions,
experiments were conducted for the two dimensional case only.
The example used here is the \exbv{}, a variant of
the \exb{} but driven by a right hand side instead
of inhomogeneous boundary conditions.
We define $\gls{domain}$ to be the unit square $(0,1)^2$.
We discretize the problem using P1 finite elements on 
a regular grid of $200\times 200$ squares, each divided
into four triangles, resulting in 80,401 degrees of freedom.
We use a coefficient field $\gls{heat_conductivity}$ with high contrast
($\gls{heat_conductivity}_{max}/\gls{heat_conductivity}_{min} = 10^5$) and high
conductivity channels to get interesting behavior
(see \cref{fig:enrichment geometry}).
As domain decomposition $\gls{subdomainol}$ we use domains of size 
$0.2\times 0.2$ with overlap $0.1$, resulting in $81$
subdomains (\cref{fig:dd}).
We use a partition of unity space decomposition
as defined in \cref{sec:poudecomposition}
with a conforming partition of unity \gls{pouf}
with
$\max_{i\in \{1, \dots, \gls{numspaces}\}} \norm{\nabla \gls{pouf}}_\infty^2 = 2 \gls{H}^{-2} = 200$ and
$\max_{i\in \{1, \dots, \gls{numspaces}\}} \norm{\gls{pouf}}_\infty^2 = 1$.
The resulting error decay is shown in \cref{fig:error decay zoom,fig:error decay,fig:convergence rate}.
To compare with the theory,
we calculate an upper bound for $\gls{decompositionboundvr}$ using
$\gls{overlappingspaces}=4$,
$\gls{friedrichs} = 1/(\sqrt{2}\pi)$,
$\gls{heat_conductivity}_{max} / \gls{heat_conductivity}_{min} = 10^5$,
and obtain $\gls{decompositionboundvr}^2 \leq 3.6013 \cdot 10 ^ 7$.
This results in an estimate of $1-c \geq 1.714 \cdot 10^{-10}$.
The rate of convergence observed in the experiment is several
orders of magnitude better than the rate guaranteed by the
a priori theory
and is close to the optimal convergence rate
(\cref{fig:convergence rate}).

To investigate the reason for this,
we plot the quotient of the larger part and the smaller part
of the estimates \cref{eq:chungend,eq:r1,eq:r2} in \cref{fig:ineq}.
It can be observed that the estimate \cref{eq:chungend}
is rather sharp, except when the error drops after
a plateau.
In estimate \cref{eq:r1}, around one order of magnitude is lost.
This could be improved by not enriching only one space but
using a marking strategy instead.
However, the main reason for the a priori theory to be so 
pessimistic seems to be the estimate in \cref{eq:r2}.
\tikzexternaldisable
\begin{figure}
\centering
\begin{tikzpicture}
\begin{semilogyaxis}[
    width=7cm,
    height=6cm,
    xmin=-5,
    xmax=105,
    ymin=3e-3,
    ymax=3e0,
    xlabel=Iteration $n$,
    ylabel=relative energy error,
    legend pos=north east,
    grid=both,
    grid style={line width=.1pt, draw=gray!20},
    major grid style={line width=.2pt,draw=gray!50},
    ytick={1e-6, 1e-5, 1e-4, 1e-3, 1e-2, 1e-1, 1e0},
    minor ytick={1e-15, 1e-14, 1e-13, 1e-12, 1e-11, 1e-10, 1e-9, 1e-8, 1e-7, 1e-6, 1e-5, 1e-4, 1e-3, 1e-2, 1e-1, 1e0, 1e1, 1e2, 1e3},
    minor xtick={10,30,50,70,90},
    legend style={at={(1.1,1.1)},anchor=north east},
  ]
\addplot [solid, thick] table [x expr=\coordindex,y index=0] {errors.dat};
\addplot [densely dashed, thick, red] table [x expr=\coordindex,y index=0] {errors_g_c.dat};
\legend{\cref{algo:online_enrichment}, \cref{algo:globally_coupled_online_enrichment}}
\end{semilogyaxis}
\end{tikzpicture}
\caption[Error decay during online enrichment for \exbv{} (zoom)]
{Decay of relative energy error during iteration for \exbv{} (zoom).
(reproduction: \cref{repro:fig:exconv})
}
\label{fig:error decay zoom}
\end{figure}
\begin{figure}
\centering
\begin{tikzpicture}
\begin{semilogyaxis}[
    width=7cm,
    height=6cm,
    xmax=610,
    xlabel=Iteration $n$,
    ylabel=relative energy error,
    legend pos=north east,
    grid=both,
    grid style={line width=.1pt, draw=gray!20},
    major grid style={line width=.2pt,draw=gray!50},
    ytick={1e-12, 1e-9, 1e-6, 1e-3, 1},
    minor ytick={1e-15, 1e-14, 1e-13, 1e-12, 1e-11, 1e-10, 1e-9, 1e-8, 1e-7, 1e-6, 1e-5, 1e-4, 1e-3, 1e-2, 1e-1, 1e0, 1e1, 1e2, 1e3},
    xtick={0,100,200,300,400,500,600},
    minor xtick={50,150,250,350,450,550,650},
    legend style={at={(1.1,1.1)},anchor=north east},
  ]
\addplot [solid, thick] table [x expr=\coordindex,y index=0] {errors.dat};
\addplot [densely dashed, thick, red] table [x expr=\coordindex,y index=0] {errors_g_c.dat};
\draw [red, thick,rounded corners] (axis cs:-5,3e-3) rectangle (axis cs:105,3e0);
\node at (axis cs:120,1e-1) [ anchor=north west] {\cref{fig:error decay zoom}};
\fill [red, opacity=0.2] (axis cs:400,1e-12) rectangle (1000, 10);
\node at (axis cs:430, 1e-9) [red, anchor=north west, rotate=90] {numerical noise};
\legend{\cref{algo:online_enrichment}, \cref{algo:globally_coupled_online_enrichment}}
\end{semilogyaxis}
\end{tikzpicture}
\caption[Decay of relative energy error during iteration in \exbv{}.]
{Decay of relative energy error during iteration.
(reproduction: \cref{repro:fig:exconv})
}
\label{fig:error decay}
\end{figure}
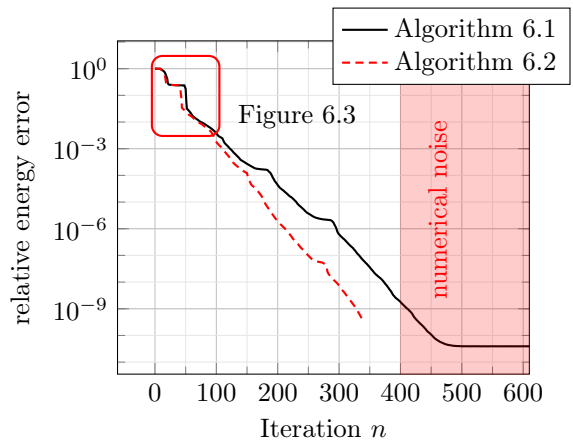
\clearpage
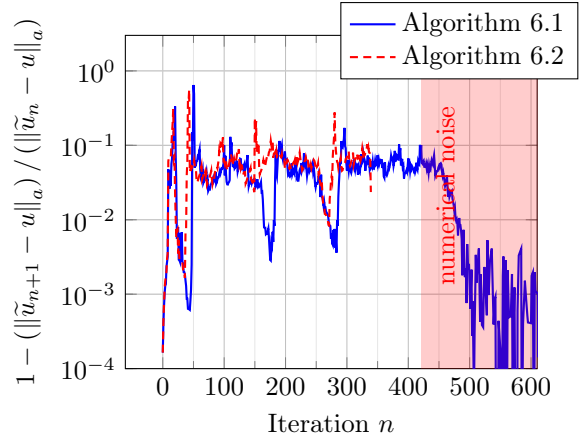
\begin{figure}
\centering
\begin{tikzpicture}
\begin{semilogyaxis}[
    width=7cm,
    height=6cm,
    xmax=610,
    ymin = 1e-4,
    ymax=3e0,
    xlabel=Iteration $n$,
    ylabel=$1 - \left(\norm{\widetilde u_{n+1} - u}_a\right)/\left(\norm{\widetilde u_n - u}_a\right)$,,
    legend pos=north east,
    grid=both,
    grid style={line width=.1pt, draw=gray!20},
    major grid style={line width=.2pt,draw=gray!50},
    ytick={1e-15, 1e-14, 1e-13, 1e-12, 1e-11, 1e-10, 1e-9, 1e-8, 1e-7, 1e-6, 1e-5, 1e-4, 1e-3, 1e-2, 1e-1, 1e0, 1e1, 1e2, 1e3},
    minor ytick={1e-15, 1e-14, 1e-13, 1e-12, 1e-11, 1e-10, 1e-9, 1e-8, 1e-7, 1e-6, 1e-5, 1e-4, 1e-3, 1e-2, 1e-1, 1e0, 1e1, 1e2, 1e3},
    xtick={0,100,200,300,400,500,600},
    minor xtick={50,150,250,350,450,550,650},
    legend style={at={(1.1,1.1)},anchor=north east},
  ]
\addplot [solid, thick, blue] table [x expr=\coordindex,y index=0] {convergence.dat};
\addplot [densely dashed, thick, red] table [x expr=\coordindex,y index=0] {convergence_g_c.dat};
\draw [red, thick] (axis cs:-100,1.714e-10) -- (axis cs:610,1.714e-10);
\node at (axis cs:-40,1e-8) [red, anchor=north west] {guaranteed rate $(1-c)$};
\fill [red, opacity=0.2] (axis cs:420,1e-20) rectangle (1000, 1000);
\node at (axis cs:430, 1e-3) [red, anchor=north west, rotate=90] {numerical noise};
\legend{\cref{algo:online_enrichment}, \cref{algo:globally_coupled_online_enrichment}}
\end{semilogyaxis}
\end{tikzpicture}
\caption[Convergence of online enrichment in \exbv{}.]
{Convergence of online enrichment in \exbv{}. 1 would be convergence within one iteration, 0 would be stagnation.
(reproduction: \cref{repro:fig:exconv})
}
\label{fig:convergence rate}
\end{figure}
\begin{figure}
\centering
\begin{tikzpicture}
\begin{semilogyaxis}[
    width=7cm,
    height=6cm,
    xmax=610,
    xlabel=Iteration $n$,
    legend pos=north east,
    grid=both,
    grid style={line width=.1pt, draw=gray!20},
    major grid style={line width=.2pt,draw=gray!50},
    xtick={0,100,200,300,400,500,600},
    minor xtick={50,150,250,350,450,550,650},
    legend style={at={(1.1,0.95)},anchor=south east},
  ]
\addplot [solid, thick, blue] table [x expr=\coordindex,y index=0] {ineq.dat};
\addplot [densely dashed, thick, red,] table [x expr=\coordindex,y index=1] {ineq.dat};
\addplot [densely dotted, thick, green,] table [x expr=\coordindex,y index=2] {ineq.dat};
\fill [red, opacity=0.2] (axis cs:420,1e-20) rectangle (1000, 100);
\node at (axis cs:430, 1e1) [red, anchor=north west, rotate=90] {numerical noise};
\legend{
  equation \eqref{eq:chungend},
  equation \eqref{eq:r1},
  equation \eqref{eq:r2},
}
\end{semilogyaxis}
\end{tikzpicture}
\caption[Sharpness of inequalities in \cref{algo:online_enrichment}]
{Sharpness of inequalities in \cref{algo:online_enrichment}:
For   equation \eqref{eq:chungend},
  $\left(\norm{\gls{rsol}_n - \gls{sol}}_a^2 - \norm{\gls{residual}(\gls{rsol}_n)}_{O_k'}^2\right)/\left(
  \norm{\gls{rsol}_{n+1} - \gls{sol}}_a^2\right)$ is plotted.
For   equation \eqref{eq:r1},
$\left(\norm{\gls{residual}(\gls{rsol}_n)}_{V_k'}^2\right) / \left(
\frac{1}{\gls{numspaces}} \sum_{i=1}^{\gls{numspaces}} \norm{\gls{residual}(\gls{rsol}_n)}_{V_k'}^2\right) $ is plotted.
For   equation \eqref{eq:r2},
$\left(\gls{decompositionboundvr}^2 \sum_{i=1}^{\gls{numspaces}} \norm{\gls{residual}(\gls{rsol}_n)}_{\gls{lfspace}'}^2\right)/\left(
\norm{\gls{residual}(\gls{rsol}_n)}_{V'}^2\right)$ is plotted.
(reproduction: \cref{repro:fig:exconv})
}
\label{fig:ineq}
\end{figure}

\clearpage
\section{Discussion}
\tikzexternalenable
In this chapter, exponential convergence
of two enrichment algorithms has been shown.
The second enrichment
algorithm, the
``globally coupled local enrichment''
for which a form of optimality can be shown,
is a new idea which was not published
before, to the author's knowledge.
The results shown here are a first
step.
To be integrated with the \gls{arbilomod},
several improvements are necessary.
First, the algorithms presented here only
use one space decomposition.
In the context of the \gls{arbilomod},
in the enrichment introduced in \cref{sec:enrichment},
we used two different space decompositions
for their favorable properties
in approximation and error estimation, respectively.
Furthermore, an extension to parameterized problems
is required, ideally
combining the convergence rate with the
Kolomogorov n-width of the
localized solution manifolds.
\clearpage
\noindent
\futurebox{
\section{Future Research in the Context of Localized Online Enrichment}
\myline
\begin{itemize}
\item
\textbfit{Numerical Experiment Exposing Training Breakdown}\\
We suspect that online enrichment is crucial
for examples with many channels of which only few
are used.
It would be interesting to see this confirmed in a 
numerical experiment.
A numerical example should be devised
where training as described in \cref{chap:training}
produces very large local bases and
enrichment produces only few basis vectors.
\item
\textbfit{Introduce Basis Cleanup Step}\\
Online enrichment increases the basis size,
possibly after a local geometry change.
It could also happen that local basis vectors
become superfluous. We think that localized
reduced basis methods should comprise a
``cleanup step'', where unnecessary basis functions
are removed. Implementing and analyzing
this would be an interesting project.
\item
\textbfit{Different Boundary Conditions}\\
The local problems in online enrichment
are always formulated with Dirichlet boundary
conditions in this thesis.
Probably better results can be obtained
using Robin-type boundary conditions
or \gls{dtn} operators.
\item
\textbfit{Parameterized Case}\\
The results in this chapter are only for the
non parametric case. The online
enrichment presented before in
\cref{sec:enrichment}
was designed for the parametric case.
It would be desirable if
convergence results in the parametric
case could be derived.
\end{itemize}
}

\chapter{Conclusion}
\label{chap:conclusion}
Localized model order reduction
is a reliable simulation tool for problems
with fine scale features but without
scale separation.
The favorable parallelization
opportunities match current
trends to cloud based computing.
Possibilities for data reuse after
localized geometry changes
promise fast interactive simulations.

We have shown that local bases of reasonable size can
be generated only local information
at acceptable computational cost
for the stationary heat conduction.
For the wave equation, the
same holds true as long as the
wavelength is large in comparison
to the local domain size.
Results from randomized numerical linear algebra
helped us building a solid theoretical
foundation for randomized training
algorithms.
Both localized a posteriori error estimators
as well as localized online enrichment
algorithms have been derived and analyzed.
However, for localized a posteriori error estimation
for inf-sup stable problems, the question
of how to quickly derive a lower bound
for the inf-sup constant is still open.

Localized model order reduction
is a field of active research, which seems
to start maturing now.
The main questions
(as outlined in 
\cref{sec:research questions})
and their answers are similar,
regardless of whether it is
\gls{gfem}, \gls{prscrbe}, \gls{gmsfem}, \gls{lrbms}, or \gls{arbilomod}.
We hope that the community will find a common formulation
and join forces to proceed faster in the future.
Maybe the abstract localizing space decomposition
and the abstract localized training configuration in this
thesis serve as an inspiration.

The relation between domain decomposition methods
and localized model order reduction still seems
to be not completely understood.
The training step in localized model order reduction
seems to be related to the generation of a coarse space
in domain decomposition methods and the
online enrichment seems related to the iteration
step in domain decomposition.
Knowledge transfer could probably benefit both communities.

The largest future challenge
on the way to a quick and reliable solver
for electromagnetic fields inside
of printed circuit boards, which was our initial
motivating example, seems to be
a stable localized reduction scheme for
inf-sup stable problems,
which is not available to date.
\clearpage
\section*{Contributions}
New contributions of the thesis are the following.
\begin{itemize}
\item
Application: The use of localized model order reduction for handling
of arbitrary changes is proposed.
\item
Application: The singular value decay of a transfer operator defined
in a \gls{pcb} board was computed, suggesting good applicability
of training based localized model order reduction.
\item
Formulation: An abstract definition of a localizing space decomposition
is proposed, providing a common formulation for different
localization strategies.
\item
Formulation: Ab abstract definition of a localized training configuration
is proposed, providing a common framework for different training
strategies. An a priori error bound is provided, formulated
completely in the abstract setting.
\item
Space Decomposition: A definition of the wirebasket space decomposition is given
which is valid in 2 and 3 dimensions for both Lagrange and Nédélec ansatz functions.
\item
Training:
A randomized training algorithm with provable convergence rates is proposed.
\item
A posteriori error estimator:
A localized, offline/online decomposable a posteriori error estimator with
computable constants was proposed.
\item
A posteriori error estimator: 
A new strategy for offline/online decomposition of residual based
error estimators in \gls{rb} contexts was proposed, drastically
reducing numerical noise.
\item
Online Enrichment: For the simplest case of online enrichment, exponential convergence was proven.
\item
Online Enrichment: The globally coupled online enrichment was proposed,
which leads to optimal convergence of online enrichment for coercive problems.
\end{itemize}

\appendix
\chapter{Reproduction of Results}
For all results presented in this thesis,
source code for reproduction is provided,
which is published in \cite{diss_zenodo}.
This chapter contains instructions
to reproduce each figure in this thesis.
The provided source code generates
data files, which are used to 
generate the plots in this document.
The plots are generated using pgfplots
\cite{feuersanger2011manual}.
The sources to generate the plots from
the data files can be found in the 
\LaTeX ~source of this thesis,
which is provided with this thesis.
The format of the data files is unspecified,
the meaning of the columns can be
inferred from the generating scripts.

Most examples require Python 2.7
and \gls{pymor} 0.4.
Some visualizations are generated using \gls{mathematica}.
The code requires some libraries to be installed.
It is advised to disable parallelization by setting
the environment variable \texttt{OMP\_NUM\_THREADS} to ``1'',
as the automatic parallelization inside numpy
and scipy often leads to increased execution
time in the small localized problems.

The following sequence of commands leads
to a working installation on Ubuntu 18.04 LTS:

\begin{verbatim}
$> sudo apt-get install python-pip python-virtualenv python-numpy \
   python-scipy python-pyside cython python-matplotlib python-dev \
   python git python-pil
$> export PATH_TO_VIRTUALENV=~/virtualenv
$> virtualenv --system-site-packages $PATH_TO_VIRTUALENV
$> source $PATH_TO_VIRTUALENV/bin/activate
$> pip install --upgrade numpy scipy progress pyopengl ipython
$> pip install --upgrade git+https://github.com/pymor/pymor.git@0.4.x
$> pip install --upgrade \
   git+https://github.com/sdrave/simdb.git@143d6b9f9f5b7bf55ab2bf\
   56d3cb19166cc1abcb
$> export SIMDB_PATH=~/simdb
\end{verbatim}

\section{Reproduction of \cref{fig:nedelec_ansatz}}
\label{repro:fig:nedelec_ansatz}
In the directory \path{code/nedelec_ansatz_function} execute the
\gls{mathematica} notebook
\path{simplex_2d.nb}. The visualizations appear in the notebook.

\section{Reproduction of \cref{fig:thermal block solution}}
\label{repro:fig:thermal block solution}
In the directory \path{code/thermalblock} run the file
\path{full_solution.py}. 
A window will open showing the solution. 
\section{Reproduction of \cref{fig:structures}}
\label{repro:fig:structures}
In the directory \path{code/arbilomod} run the file \path{create_full_solutions.py}.
The solution plots are created in ``sol1.png'' to ``sol5.png''.

\section{Reproduction of \cref{fig:maxwell solutions}}
\label{repro:fig:maxwell solutions}
In the directory \path{code/arbilomod_maxwell} run the file \path{maxwell_create_solutions.py}.
The solution plots are created in 
\path{msol_0_1.png} to \path{msol_0_3.png} and \path{msol_1_1.png} to \path{msol_1_3.png}.

\section{Reproduction of \cref{fig:vertexextension}}
\label{repro:fig:vertexextension}
In the directory
\path{code/vertex_extension_plot}
run the \gls{mathematica} notebook
\path{generate_vertex_extension_plots.nb}.
The images will be exported as ``png''-images
in the directory of the notebook.

\section{Reproduction of \cref{fig:tolerances}}
\label{repro:fig:tolerances}
In the directory \path{code/arbilomod}
run the file \path{experiment_tolerances.py}
and afterwards\\
\path{postprocessing_tolerances.py}.
The results are written to
\path{errors_tolerances.txt}.
The 
\gls{mathematica} file 
\path{create_tolerances_plot.nb}
generates the image from the txt file.

\section{Reproduction of
\cref{tab:basisreusecompare},
\cref{fig:basisreusewithtraining}, and
\cref{fig:basisreuse}
}
\label{repro:basisreuse}
In the directory \path{code/arbilomod}
run the file \path{experiment_basisreuse.py}
and the file 
\path{experiment_basisreuse_with_training.py}.
Then run \path{postprocessing_basisreuse.py}
and \path{postprocessing_basisreuse_with_training.py}.
The data for \cref{tab:basisreusecompare} is printed to stdout.
The data for 
\cref{fig:basisreusewithtraining}
is written to
\path{basisreuse_training_with_reuse_0.dat} to
\path{basisreuse_training_with_reuse_4.dat}
and
\path{basisreuse_training_without_reuse_0.dat} to
\path{basisreuse_training_without_reuse_4.dat}.
The data for 
\cref{fig:basisreuse}
is written to
\path{basisreuse_with_reuse_0.dat} to
\path{basisreuse_with_reuse_4.dat}
and
\path{basisreuse_without_reuse_0.dat} to
\path{basisreuse_without_reuse_4.dat}.

\section{Reproduction of \cref{fig:estimatorperformance}}
\label{repro:fig:estimatorperformance}
In the directory \path{code/arbilomod}
run the file \path{experiment_estimatorperformance.py}.
Then run \path{postprocessing_estimatorperformance.py}.
The data 
is written to
\path{estimatorperformance_0.dat} to
\path{estimatorperformance_4.dat}.

\section{Reproduction of \cref{tab:runtimes}}
\label{repro:tab:runtimes}
Set the environment variable OMP\_NUM\_THREADS to 1
by executing 
\texttt{export OMP\_NUM\_THREADS=1}
for single-thread performance measurement.
It is not used by the provided scripts directly,
but inside of numpy and/or scipy.
Then, in the directory \path{code/arbilomod}
run the file \path{experiment_create_timings.py}.
Then run 
\path{postprocessing_create_timings.py}.
The timings are written to stdout.

\section{Reproduction of \cref{tab:H_study}}
\label{repro:tab:H_study}
In directory \path{code/arbilomod}
run the file \path{experiment_H_study.py}
and afterwards

\path{postprocessing_H_study.py}.
Results are written to stdout.

\section{Reproduction of \cref{fig:basis_sizes}}
\label{repro:fig:basis_sizes}
In directory \path{code/arbilomod}
run the file \path{experiment_draw_basis_sizes.py}
and afterwards\\
\path{postprocessing_draw_basis_sizes.py}.
The picture will be written to \path{basis_sizes.png}.

\section{Reproduction of \cref{fig:infsupcont}}
\label{repro:fig:infsupcont}
In directory \path{code/arbilomod_maxwell}
run the file \path{maxwell_calculate_infsup.py}.
The inf-sup constants are written to 
\path{infsups0.txt} for the first geometry
and to \path{infsups1.txt} for the second geometry.
The continuity constants are written to \path{continuities0.txt}
for the first geometry and to \path{continuities1.txt}
for the second geometry.

\section{Reproduction of \cref{fig:global_n_width}}
\label{repro:fig:global_n_width}
In directory \path{code/arbilomod_maxwell}
run the file \path{maxwell_global_n_width.py}.
The data for geometry 1 will be in
\path{n_width0.txt} and the data for
geometry 2 will be in \path{n_width1.txt}.

\section{Reproduction of \cref{fig:localized_globalsolve}}
\label{repro:fig:localized_globalsolve}
In directory \path{code/arbilomod_maxwell}
run the file \path{maxwell_local_n_width.py}.
The data for geometry 1 will be in
\path{localized0.txt} and the data for
geometry 2 will be in \path{localized1.txt}.

\section{Reproduction of \cref{fig:infsup_drop}}
\label{repro:fig:infsup_drop}
In directory \path{code/arbilomod_maxwell}
run the file \path{maxwell_infsup_during_reduction.py}.
The inf-sup constants are written to
\path{infsup_constant_0.txt}
and
\path{infsup_constant_1.txt}.
The convergence graph on the left is in 
\path{localized0.txt}, generated
by \path{maxwell_local_n_width.py},
see above.

\section{Reproduction of \cref{fig:training_error}}
\label{repro:fig:training_error}
In directory \path{code/arbilomod_maxwell}
run the file \path{maxwell_training_benchmark.py}.
The errors will be in 
\path{maxwell_training_benchmark0.txt}
and
\path{maxwell_training_benchmark1.txt}.
The data for global solves
is in \path{localized0.txt}
and \path{localized1.txt}, 
 generated
by 
\path{maxwell_local_n_width.py},
see above.
\section{Reproduction of \cref{fig:maxwell_basis_sizes}}
\label{repro:fig:maxwell_basis_sizes}
In directory \path{code/arbilomod_maxwell}
run the file 
\path{experiment_maxwell_geochange.py}
and afterwards

\path{postprocessing_draw_basis_sizes_maxwell_geochange.py}.
The images are written to 

\path{basis_sizes_maxwell_0.png}
and
\path{basis_sizes_maxwell_1.png}.

\section{Reproduction of \cref{fig:staberrest}}
\label{repro:fig:staberrest}
In directory \path{code/thermalblock}
run the file \path{run_thermalblock.sh}.
The results will be in\\ \path{traditional.txt}
for bases generated with the \oldsplit{}
and in 
\path{residual_basis.txt} for bases generated
with the \newsplit{}.

\section{Reproduction of \cref{img:modes}}
\label{repro:img:modes}
In directory \path{code/rangefindertest/example1}
run \path{experiment_twosquares_images.py}.
Results will be in 
\path{modes.dat} and 
\path{omodes.dat}.
If you encounter errors related to umfpack, uninstall
scikit-umfpack (pip uninstall scikit-umfpack)
as its python bindings are buggy at the time of writing.

\section{Reproduction of \cref{fig:interface_a_priori}}
\label{repro:fig:interface_a_priori}
In directory \path{code/rangefindertest/example1},
run the four files
\path{experiment_twosquares_svddecay.py},
\path{experiment_twosquares_algo_convergence.py},
\path{postprocessing_twosquares_operatornorm_decay.py}, and
\path{postprocessing_twosquares_a_priori_comparison.py}
(in that order).

Results for Figure (a) will be in \path{twosquares_percentiles.dat}.
Results for Figure (b) will be in \path{twosquares_means.dat} and \path{twosquares_expectation.dat}.

\section{Reproduction of \cref{tab:cpu_times}}
\label{repro:tab:cpu_times}
Set the environment variable OMP\_NUM\_THREADS to 1
by executing 
\path{export OMP_NUM_THREADS=1}
for single-thread performance measurement.
In directory
 \path{code/rangefindertest/cpu_times}
run file \path{measure_cpu_times.py}.
Results will be on stdout.

\section{Reproduction of \cref{fig:interface_adaptive_quartiles}}
\label{repro:fig:interface_adaptive_quartiles}
In directory \path{code/rangefindertest/example1},
run files \path{experiment_twosquares_svddecay.py}
and
\path{experiment_twosquares_algo_convergence.py}.

To reproduce (a), run
\path{postprocessing_twosquares_adaptive_convergence.py},
results will be in \path{twosquares_adaptive_convergence.dat}.

To reproduce (b), run
\path{postprocessing_twosquares_adaptive_num_testvecs.py},
results will be in \path{twosquares_adaptive_num_testvecs.dat}.

\section{Reproduction of \cref{fig:h_dependency}}
\label{repro:fig:h_dependency}
In directory \path{code/rangefindertest/example1},
run files \path{experiment_h_svddecay.py},

\path{experiment_h_algo_convergence.py},
and \path{postprocessing_h.py}.
Results will be in 
\path{experiment_h_deviation.dat}
 and 
\path{experiment_h_testvecs.dat}.

\section{Reproduction of \cref{fig:helmholtz_svd}}
\label{repro:fig:helmholtz_svd}
In directory \path{code/rangefindertest/example2},
run \path{helmholtz.py}.
Results will be in 
\path{helmholtz_selected.dat} (for left figure)
   and \path{helmholtz_3d.dat} (for right figure).

\section{Reproduction of \cref{fig:helmholtz_a_priori}}
\label{repro:fig:helmholtz_a_priori}
In directory \path{code/rangefindertest/example2},
run \path{helmholtz.py}.
Then run 
\path{postprocessing_helmholtz_operatornorm_decay.py}.
Result for left figure will be in 
\path{helmholtz_percentiles.dat}.
Then run 
\path{postprocessing_helmholtz_a_priori_comparison.py}
and\\
\path{postprocessing_helmholtz_operatornorm_decay.py}.
Results for right figure will be in \path{helmholtz_expectation.dat} and \path{helmholtz_means.dat}.

\section{Reproduction of \cref{fig:helmholtz_interface_adaptive_quartiles}}
\label{repro:fig:helmholtz_interface_adaptive_quartiles}
In directory \path{code/rangefindertest/example2},
run \path{helmholtz.py}.
Then run 
\path{postprocessing_helmholtz_adaptive_convergence.py}.
Result for left figure will be in 
\path{helmholtz_adaptive_convergence.dat}.
Then run 
\path{postprocessing_helmholtz_adaptive_num_testvecs.py}.
Result for right figure will be in
 \path{helmholtz_adaptive_num_testvecs.dat}.

\section{Reproduction of \cref{fig:olimexdecay,fig:olimex rangefinder runtimes}}
\label{repro:fig:olimexdecay}
The Olimex A64 board design is a KiCad project file.
We provide the project file used here at
\cite{Olimex_zenodo}.
The code to generate the transfer operators 
is developed in a fork of KiCad. 
KiCad's unit test mechanism was used to set this up.
The complete KiCad source code with the changes
to compute the transfer operators can be found
in \path{code/kicad}.

To build it, you need a \Cpp development
environment. On a Ubuntu 18.04 LTS
the following command leads to a working
environment:

\begin{verbatim}
$> sudo apt-get install build-essential cmake libwxgtk3.0-gtk3-dev \
   libglew-dev libglm-dev libcairo2-dev libboost1.65-all-dev swig \
   python-wxgtk3.0 libssl-dev libsuitesparse-dev libcurl4-openssl-dev
\end{verbatim}
To build it, create a build directory.
Inside of the build directory, execute
\\
\texttt{cmake -DCMAKE\_BUILD\_TYPE=Release -DKICAD\_SPICE=OFF -DBUILD\_GITHUB\_PLUGIN=OFF\textbackslash \\  -DKICAD\_USE\_OCE=OFF path/to/code/kicad}.\\
Enter the subdirectory \texttt{qa/polygon\_generator}
and run \texttt{make}.
After successful build, 
save the file \path{A64-OlinuXino_Rev_D.kicad_pcb}
from \cite{Olimex_zenodo} 
and the file \path{kicad.ini}
from the 
\path{code/kicad/qa/polygon_generator}
in the current directory and
execute 
\path{test_polygon_generator}.

The script will compute the transfer operators
for 40 domains and run the adaptive range finder
algorithm on all of these.
The timings for \cref{fig:olimex rangefinder runtimes}
are written to \path{timings.dat}.

Additionaly, the script will output matrices in 
\path{sysmat_816.mtx} and \path{rsmat_816.mtx}
from which the transfer operator can be computed.
Install umfpack for python in a setup
as specified above by using the command
\path{pip install --upgrade scikit-umfpack}.
Copy the file \path{apply_arpack.py}
from \path{code/kicad/qa/polygon_generator}
to the local directory and execute it.
The singular values for the transfer operator
visible in \cref{fig:olimexdecay}
will be written to \path{svals_816.dat}
and the maximum testvector values
will be written to \path{maxnorms_2norm_816.dat}.

\section{Reproduction of \cref{fig:enrichment geometry}}
\label{repro:fig:enrichment geometry}
In directory \path{code/enrichment},
run \path{create_full_solution.py}.
The result is written to 
\path{solution.png}.

\section{Reproduction of \cref{fig:error decay zoom,fig:error decay,fig:convergence rate,fig:ineq}}
\label{repro:fig:exconv}
In directory \path{code/enrichment},
run the following four files:
\path{experiment_residual_based.py},
\path{experiment_globally_coupled.py},

\path{postprocessing_residual_based.py},
\path{postprocessing_globally_coupled.py}.
The data is written to data files:
\path{errors.dat} will contain relative energy-, h1- and l2-errors for residual based enrichment,
\path{convergence.dat} will contain data for \cref{fig:convergence rate} 
for residual based enrichment,
\path{ineq.dat} will contain data for \cref{fig:ineq},
\path{errors_g_c.dat} will contain relative energy errors for globally coupled enrichment,
\path{convergence_g_c.dat} will contain data for \cref{fig:convergence rate}  for globally coupled enrichment.

\section{Reproduction of \cref{fig:gs orthogonality error}, \cref{fig:gs reproduction error}, and \cref{fig:gs num iterations}}
\label{repro:fig:gs stuff}
In the directory \path{code/gram_schmidt} run the file \path{experiment.py}.
For each algorithm, a file is created:
\path{Schmidt-GS.dat},
\path{Gram-GS.dat},
\path{ReIt-GS_2.dat},
\path{ReIt-GS_3.dat},
\path{AReIt-GS.dat}.
Each contains the orthogonality error in the first column
and the reproduction error in the second column.
In addition, the file \path{iterations.dat} is created,
containing the median number of iterations
for each vector.

\section{Reproduction of \cref{tab:assemblyperformance}}
\label{repro:tab:assemblyperformance}
To regenerate the timings for the proposed
code, run the tests in 
\path{code/high_performance_assembly}
using Valgrind/Callgrind. One output file is written
for each evaluation of the functions.

To regenerate the timings for the dune code,
add the files in
\path{code/high_performance_assembly/dune}
in a dune module, compile them and run them with
Valgrind/Callgrind. One output file is written for each
evaluation. Note that the first two evaluations
show higher instruction counts because of
initializations.
The tests were performed with dune 2.6 and dune-pdelab 2.0.

\chapter{A Posteriori Error Estimation for Relative Errors}
\label{chap:relative error estimators}
\section{Introduction}
In the context of model order reduction, 
it is often the case that an a posteriori error estimator
for the absolute error is available, but 
an estimator for the relative error is required.
It is possible to derive an upper bound for the relative
error based on the error estimator for the
absolute error and the norm of the reduced solution.
Most introductory texts to
reduced basis methods state one, which we will
call the classic relative error estimator
in the following. This approach is widely used
in the model order reduction community.
In our manuscript introducing \gls{arbilomod},
we used a different approach \cite[Proposition 5.10]{Buhr2017}. 
This approach can also be found in \cite[Lemma 3.2.2]{hessphdthesis}.
In this section, we compare these two approaches
and conclude that the new approach is better than
the classic approach in almost every aspect.
This is not a new result, it is known to experts 
in the field. However, to our knowledge,
this is the first comprehensive comparison of the
two approaches.
\section{Two relative error estimators}
\subsection{Setting}
We assume a full model with solution $\gls{sol}$
and a reduced model with solution $\gls{rsol}$ are given
and furthermore, an efficient
a posteriori error estimator $\Delta(\gls{rsol})$ satisfying
\begin{equation}
\norm{\gls{sol} - \gls{rsol}} \leq \Delta(\gls{rsol}) \leq \eta \norm{\gls{sol} - \gls{rsol}}
\end{equation}
with efficiency constant $\eta$ is available.
Based on that, an efficient relative a posteriori error estimator
$\Delta_\mathrm{rel}(\gls{rsol})$ should be derived satisfying
\begin{equation}
\frac{\norm{\gls{sol} - \gls{rsol}}}{\norm{\gls{sol}}} \leq \Delta_\mathrm{rel}(\gls{rsol}) \leq \eta_\mathrm{rel} \frac{\norm{\gls{sol} - \gls{rsol}}}{\norm{\gls{sol}}}
\end{equation}
with efficiency constant $\eta_\mathrm{rel}$.
It should not require the norm of the solution $\norm{\gls{sol}}$, as
this is usually not available in the context
of a posteriori error estimation.
\section{Definition of Relative a Posteriori Error Estimators}
In this subsection, two 
relative a posteriori error estimators are defined.
The first one is the classic one, which can be found in many
introductory texts on reduced basis methods, for
example in \cite[Proposition 4D]{Patera2007}, 
in \cite[Proposition 2.27]{Haasdonk2017a} or
\cite[Proposition 4.4]{Hesthaven2016}.
\begin{definition}[Classic relative a posteriori error estimator]
Given a solution $\gls{sol}$ and a reduced solution $\gls{rsol}$
and an a posteriori error estimator $\Delta(\gls{rsol})$,
the classic relative a posteriori error estimator is defined
by
\begin{equation}
\Delta^\mathrm{c}_\mathrm{rel}(\gls{rsol}) := 2 \frac{\Delta(\gls{rsol})}{\norm{\gls{rsol}}} .
\end{equation}
\end{definition}
The second approach is the ``new'' relative a posteriori error estimator,
which was used by the authors in \cite[Proposition 5.10]{Buhr2017} but
can also be found in other publications, for example \cite[Lemma 3.2.2]{hessphdthesis}.
\begin{definition}[New relative a posteriori error estimator]
Given a solution $\gls{sol}$ and a reduced solution $\gls{rsol}$
and an a posteriori error estimator $\Delta(\gls{rsol})$,
the new relative a posteriori error estimator is defined
by
\begin{equation}
\Delta^\mathrm{new}_\mathrm{rel}(\gls{rsol}) := \frac{\Delta(\gls{rsol})}{\norm{\gls{rsol}} - \Delta(\gls{rsol})} .
\end{equation}
\end{definition}
These two approaches will be analyzed in the following sections.
\begin{remark}[Dividing by reduced error norm]
Sometimes, the quantity $\frac{\Delta(\gls{rsol})}{\norm{\gls{rsol}}}$
is used to analyze the relative error, for example in \cite[Section 9.1]{Quarteroni2016}.
While this is a very good approximation
of the relative error when $\Delta(\gls{rsol}) \ll \norm{\gls{rsol}}$,
it is not an upper bound to the relative error $\norm{\gls{sol} - \gls{rsol}} / \norm{\gls{sol}}$.
\end{remark}
\section{Rigorousity and Effectivity of Relative a Posteriori Error Estimators}
Assuming that the underlying absolute a posteriori error estimator is
rigorous and efficient with efficiency constant $\eta$, i.e.
assuming that
\begin{equation}
\norm{\gls{sol} - \gls{rsol}} \leq \Delta(\gls{rsol}) \leq \eta \norm{\gls{sol} - \gls{rsol}},
\end{equation}
we show that the two relative a posteriori error estimators 
are in turn rigorous and efficient for the relative error.
These results can also be found in the references cited above.
The following proposition is for the classic approach.
\begin{proposition}[Classic relative a posteriori error estimator]
If $\Delta(\gls{rsol}) \leq \frac 1 2 \norm{\gls{rsol}}$,
it holds that
\begin{equation}
\frac{\norm{\gls{sol} - \gls{rsol}}}{\norm{\gls{sol}}}
\leq
\Delta^\mathrm{c}_\mathrm{rel}(\gls{rsol})
\leq 
\eta^\mathrm{c}_\mathrm{rel}
\frac{\norm{\gls{sol} - \gls{rsol}}}{\norm{\gls{sol}}}
\end{equation}
with $\eta^\mathrm{c}_\mathrm{rel} := 3 \eta$.
\end{proposition}
\begin{proof}
First, we realize that $\norm{\gls{rsol}} \leq 2 \norm{\gls{sol}}$ because
\begin{equation}
\norm{\gls{sol}}
= \norm{\gls{sol} - \gls{rsol} + \gls{rsol}}
\geq \norm{\gls{rsol}} - \norm{\gls{sol} - \gls{rsol}}
\geq \norm{\gls{rsol}} - \Delta(\gls{rsol})
\geq \norm{\gls{rsol}} - \tfrac 1 2 \norm{\gls{rsol}}
= \tfrac 1 2 \norm{\gls{rsol}} .
\end{equation}
Using this, the first inequality of the proposition follows by
\begin{equation}
\frac{\norm{\gls{sol} - \gls{rsol}}}{\norm{\gls{sol}}}
\leq \frac{\Delta(\gls{rsol})}{\norm{\gls{sol}}}
= \frac{\Delta(\gls{rsol})}{\norm{\gls{rsol}}} \frac{\norm{\gls{rsol}}}{\norm{\gls{sol}}}
\leq 2 \frac{\Delta(\gls{rsol})}{\norm{\gls{rsol}}}
= \Delta_\mathrm{rel}^\mathrm{c}(\gls{rsol}) .
\end{equation}

For the second inequality, we first realize that $\norm{\gls{sol}} \leq \frac 3 2 \norm{\gls{rsol}}$, because
\begin{equation}
\norm{\gls{sol}} = \norm{\gls{sol} - \gls{rsol} + \gls{rsol}} 
\leq \norm{\gls{sol} - \gls{rsol}} + \norm{\gls{rsol}}
\leq \Delta(\gls{rsol}) + \norm{\gls{rsol}}
\leq \tfrac 1 2 \norm{\gls{rsol}} + \norm{\gls{rsol}}
= \tfrac 3 2 \norm{\gls{rsol}}
\end{equation}
and use this to show
\begin{equation}
\Delta_\mathrm{rel}^\mathrm{c}(\gls{rsol})
= \frac{2 \Delta(\gls{rsol})}{\norm{\gls{rsol}}}
\leq \frac{2 \eta \norm{\gls{sol} - \gls{rsol}}}{\norm{\gls{rsol}}}
\leq \frac{2 \eta \norm{\gls{sol} - \gls{rsol}}}{\frac 2 3 \norm{\gls{sol}}}
= 3 \eta \frac{\norm{\gls{sol} - \gls{rsol}}}{\norm{\gls{sol}}} .
\end{equation}
\end{proof}
For the new relative a posteriori error estimator, 
a very similar result can be obtained, which is given in the following proposition.
\begin{proposition}[New relative a posteriori error estimator]
If $\Delta(\gls{rsol}) < \norm{\gls{rsol}}$,
it holds that
\begin{equation}
\frac{\norm{\gls{sol} - \gls{rsol}}}{\norm{\gls{sol}}}
\leq
\Delta^\mathrm{new}_\mathrm{rel}(\gls{rsol})
\leq 
\eta^\mathrm{new}_\mathrm{rel}
\frac{\norm{\gls{sol} - \gls{rsol}}}{\norm{\gls{sol}}}
\end{equation}
with $\eta^\mathrm{new}_\mathrm{rel} := (1 + 2 \Delta_\mathrm{rel}^\mathrm{new}(\gls{rsol}) ) \eta$.
\end{proposition}
\begin{proof}
Realizing that $\left(\norm{\gls{rsol}} - \Delta(\gls{rsol}) \right)\leq \norm{\gls{sol}}$,
it is easy to see that
\begin{equation}
\frac{\norm{\gls{sol} - \gls{rsol}}}{\norm{\gls{sol}}}
\leq
\frac{\Delta(\gls{rsol})}{\norm{\gls{sol}}}
\leq
\frac{\Delta(\gls{rsol})}{\norm{\gls{rsol}} - \Delta(\gls{rsol})},
\end{equation}
which is the first inequality. 
Using 
\begin{equation}
\frac{
\norm{\gls{rsol}} + \Delta(\gls{rsol})
}{
\norm{\gls{rsol}} - \Delta(\gls{rsol})
}
 = 1 + 2 \Delta_\mathrm{rel}^\mathrm{new}(\gls{rsol})
\end{equation}
the second inequality can be shown:
\begin{eqnarray}
\Delta_\mathrm{rel}^\mathrm{new}(\gls{rsol}) = 
\frac{\Delta(\gls{rsol})}{\norm{\gls{rsol}} - \Delta(\gls{rsol})}
& \leq &\eta \frac{\norm{\gls{sol} - \gls{rsol}}}{\norm{\gls{rsol}} - \Delta(\gls{rsol})}\\
&=&  \eta \frac{\norm{\gls{sol} - \gls{rsol}}}{\norm{\gls{rsol}} + \Delta(\gls{rsol})} \left( 1 + 2 \Delta_\mathrm{rel}^\mathrm{new}(\gls{rsol}) \right)\\
&\leq& \eta \frac{\norm{\gls{sol} - \gls{rsol}}}{\norm{\gls{sol}}} \left( 1 + 2 \Delta_\mathrm{rel}^\mathrm{new}(\gls{rsol}) \right).
\end{eqnarray}
\end{proof}
\section{Comparison of Relative A Posteriori Error Estimators}
In the following, we argue that the new relative
a posteriori error estimator is better than the classic one in almost every regard.
\subsection{Comparison of Range of Validity}
The classic relative a posteriori error estimator is valid
if $\Delta(\gls{rsol}) \leq \frac 1 2 \norm{\gls{rsol}}$.
The new approach is valid if $\Delta(\gls{rsol}) < \norm{\gls{rsol}}$.
So the range of validity for the new estimator is larger than
and completely contains the range of validity of the classic
relative a posteriori error estimator.
\subsection{Comparison of Sharpness}
\begin{proposition}[Sharpness of relative a posteriori error estimators]
Whenever both relative a posteriori error estimators are defined,
i.e.~when $\Delta(\gls{rsol}) \leq \frac 1 2 \norm{\gls{rsol}}$,
the new relative a posteriori error estimator is not worse
than the classic one, in mathematical terms:
\begin{equation}
\Delta(\gls{rsol}) \leq \frac 1 2 \norm{\gls{rsol}}
\qquad
\Rightarrow
\qquad
\Delta_\mathrm{rel}^\mathrm{new}(\gls{rsol}) \leq \Delta_\mathrm{rel}^\mathrm{c}(\gls{rsol}) .
\end{equation}
\end{proposition}
\begin{proof}
This can be shown by the following simple computation.
\begin{align}
&&\Delta(\gls{rsol}) &\leq \frac 1 2 \norm{\gls{rsol}} \\
&\Rightarrow &\frac 1 2 \norm{\gls{rsol}} &\leq \norm{\gls{rsol}} - \Delta(\gls{rsol}) \label{eq:toinvert} \\
&\Rightarrow &\frac{2}{\norm{\gls{rsol}}} &\geq \frac{1}{\norm{\gls{rsol}} - \Delta(\gls{rsol})}\\
&\Rightarrow &\frac{2 \Delta(\gls{rsol})}{\norm{\gls{rsol}}} &\geq \frac{\Delta(\gls{rsol})}{\norm{\gls{rsol}} - \Delta(\gls{rsol})}\\
&\Rightarrow &\Delta_\mathrm{rel}^\mathrm{c}(\gls{rsol}) &\geq \Delta_\mathrm{rel}^\mathrm{new}(\gls{rsol})
\end{align}
We can always invert in \cref{eq:toinvert} as $\norm{\gls{rsol}} - \Delta(\gls{rsol})$ is positive.
\end{proof}
\subsection{Comparison of Efficiency}
\begin{proposition}[Efficiency constant of relative a posteriori error estimators]
Whenever both relative a posteriori error estimators are defined,
i.e.~when $\Delta(\gls{rsol}) \leq \frac 1 2 \norm{\gls{rsol}}$,
the new relative a posteriori error estimator has a better
efficiency constant
than the classic one, in mathematical terms:
\begin{equation}
\Delta(\gls{rsol}) \leq \frac 1 2 \norm{\gls{rsol}}
\qquad
\Rightarrow
\qquad
\eta_\mathrm{rel}^\mathrm{new} \leq \eta_\mathrm{rel}^\mathrm{c} .
\end{equation}
\end{proposition}
\begin{proof}
This can be shown by the following simple computation.
\begin{align}
&&\Delta(\gls{rsol}) &\leq \frac 1 2 \norm{\gls{rsol}} \\
&\Rightarrow& \Delta(\gls{rsol}) &\leq \norm{\gls{rsol}} - \Delta(\gls{rsol}) \label{eq:refme}\\
&\Rightarrow& \frac{ \Delta(\gls{rsol}) }{ \norm{\gls{rsol}} - \Delta(\gls{rsol})} &\leq  1\\
&\Rightarrow& \Delta_\mathrm{rel}^\mathrm{new} &\leq 1 \\
&\Rightarrow& \left(1 + 2\Delta_\mathrm{rel}^\mathrm{new} \right)\eta &\leq 3\eta \\
&\Rightarrow& \eta_\mathrm{rel}^\mathrm{new} &\leq  \eta_\mathrm{rel}^\mathrm{c}
\end{align}
We can always divide by $\norm{\gls{rsol}} - \Delta(\gls{rsol})$ in \cref{eq:refme} as this
term is positive.
\end{proof}
\section{Conclusion on Relative A Posteriori Error Estimators}
The approach to construct an a posteriori error estimator
for the relative error presented above and named
the ``new'' approach is in almost every regard
better than the ``classic'' approach. It has a wider
range of validity, gives tighter error bounds and
has a better efficiency constant.
However, it might be more convenient to use
the classic approach in some cases, for example
when analyzing rates of convergence.

\chapter{Numerically Stable Gram-Schmidt Algorithm}
\label{chap:stable gram schmidt}
\section{Introduction}
The Gram-Schmidt algorithm
is a widely used orthonormalization procedure.
It exists in different variants,
four of which we present in the following.
Some of them, which are widely used 
in numerical codes today, suffer from
numerical instabilities when the input vectors
are almost linear dependent.
In our publication on numerically stable a posteriori
estimation, we also published a numerically more stable
variant of the Gram-Schmidt algorithm \cite[Algorithm 1]{Buhr2014}.
In this chapter, we compare the different variants
of the Gram-Schmidt algorithm and compare their
performance in numerical experiments.
\section{Algorithms}
\begin{algorithm2e}[t]
\KwIn{vectors $\varphi_i$, $i \in 1, \dots, N$}
\KwOut{orthonormal vectors $\psi_i$}
\For{$i \leftarrow 1, \dots, N$}{
  $\psi_i \leftarrow \varphi_i - \sum_{j=1}^{i-1} (\psi_i, \varphi_j) ~ \psi_j$ \;
  $\psi_i \leftarrow \psi_i / \norm{\psi_i}$ \;
}
\caption{Erhard Schmidt's variant of Gram-Schmidt algorithm.}
\label{alg:schmidt gram-schmidt}
\end{algorithm2e}

\begin{algorithm2e}[t]
\KwIn{vectors $v_i$, $i \in 1, \dots, N$}
\KwOut{orthonormal vectors $v_i$}
\For{$i \leftarrow 1, \dots, N$}{
  \For{$j \leftarrow 1, \dots, (i-1)$}{
    $v_i \leftarrow v_i - (v_j,v_i) ~ v_j$\;
  }
  $v_i \leftarrow v_i / \norm{v_i}$\;
}
\caption{Jørgen Pedersen Gram's variant of Gram-Schmidt algorithm.}
\label{alg:gram gram-schmidt}
\end{algorithm2e}

\begin{algorithm2e}[t]
\KwIn{vectors $v_i$, $i \in 1, \dots, N$\\$n_r$: number of iterations}
\KwOut{orthonormal vectors $v_i$}
\For{$i \leftarrow 1, \dots, N$}{
  \For{$r \leftarrow 1, \dots, n_r$}{
    \For{$j \leftarrow 1, \dots, (i-1)$}{
      $v_i \leftarrow v_i - (v_j,v_i) ~ v_j$\;
    }
  }
  $v_i \leftarrow v_i / \norm{v_i}$\;
}
\caption{Gram-Schmidt with re-iteration.}
\label{alg:gram-schmidt reiteration}
\end{algorithm2e}

\begin{algorithm2e}[t]
\KwIn{vectors $v_i$, $i \in 1, \dots, N$}
\KwOut{orthonormal vectors $v_i$}
\For{$i \leftarrow 1, \dots, N$}{
  $v_i \leftarrow v_i / \norm{v_i} $\;
  \Repeat{$\mathrm{newnorm}  > 0.1$}{
    \For{$j \leftarrow 1, \dots, (i-1)$}{
      $v_i \leftarrow v_i - (v_j,v_i) ~ v_j$\;
    }
    newnorm $\leftarrow \norm{v_i}$\;
    $v_i \leftarrow v_i / \mathrm{newnorm} $\;
  }
}
\caption{Gram-Schmidt with adaptive re-iteration.}
\label{alg:gram-schmidt adaptive}
\end{algorithm2e}

In this section, we present four variants of
the Gram-Schmidt algorithm. All of them
are mathematically equivalent and would
yield the same results if executed in exact
arithmetic.
Erhard Schmidt published his algorithm in 
1907 \cite[p.442]{PPN235181684_0063}.
It is a one pass algorithm. For each vector, 
first, all inner products with the already orthonormal 
vectors are 
computed. Then, the projections on the already orthonormal
vectors is subtracted.
We present Schmidt's variant in \cref{alg:schmidt gram-schmidt}
on page \pageref{alg:schmidt gram-schmidt}.
Jørgen Pedersen Gram published his algorithm already 1883 
\cite{PPN243919689_0094}.
Gram's algorithm is usually called 
``Modified Gram-Schmidt'' (MGS) in today's literature.
It is also a one pass algorithm.
In contrast to Schmidt's variant,
the calculation of the inner products with the
already orthonormal vectors and the subtraction
of the projections are interleaved. This can lead
to smaller numerical errors.
However, there is still the problem that
later orthogonalization steps can 
break the orthogonality achieved in earlier steps.
We present Gram's variant in \cref{alg:gram gram-schmidt}
on page \pageref{alg:gram gram-schmidt}.
These numerical instabilities are often
addressed by executing the Gram-variant multiple
times, usually two times. We give 
the (Modified) Gram-Schmidt algorithm
with re-iteration in \cref{alg:gram-schmidt reiteration}
on page \pageref{alg:gram-schmidt reiteration}.
However, we observed in numerical experiments
that two iterations are usually sufficient, but not always.
Always executing three or more iterations would be
a waste of computational resources. 
We further observed that numerical accuracy is
lost when the norm of the vector under consideration
shrinks.
This led us to devise the Gram-Schmidt algorithm
with adaptive re-iteration, which
iterates until the norm of the vector under consideration
does not shrink any more.
This new algorithm is given in \cref{alg:gram-schmidt adaptive}
on page \pageref{alg:gram-schmidt adaptive}.
We reiterate, if the norm of the vector under
consideration has decreased by more than a factor of
0.1. This factor can be tuned:
Setting it close to one (e.g.~0.99) leads to high numerical accuracy
with more iterations.
A smaller factor (e.g.~0.01) leads to less numerical accuracy
and less iterations.
\section{Numerical Comparison}
We perform a numerical comparison of the four algorithms
in $\gls{C}^{10000}$ with the inner product defined by
$(u,v) := u^H v$.
To numerically test the four algorithms, we
follow a four step procedure.
In a first step, we randomly generate an orthonormal
set of $N$ vectors which we denote by $\varphi_i,
i \in 1, \dots, N$
with $(\varphi_i, \varphi_j) = \delta_{ij}$.
In a second step, we form almost linear
dependent test vectors
using the definition
\begin{equation}
t_i := \sum_{j=1}^i \left( \frac 1 2 \right) ^j \varphi_j .
\end{equation}
These vectors are linear independent in exact arithmetic,
but numerical effects render them linear dependent.
In a third step, we hand these vectors to one of the
Gram-Schmidt algorithms and denote its result
by $\psi_i$.
In exact arithmetic, the vectors $\psi_i$ 
are identical to $\varphi_i$.
In a fourth step, we quantify the errors.
To do so, we calculate two
quantities. The first is the 
orthonormality error, judging whether
the resulting vectors are orthonormal. We
define it as
\begin{equation}
e_{1,i} := \max_{j \in 1, \dots, (i-1)} | (\psi_i, \psi_j)| .
\end{equation}
The second is the reproduction error.
It quantifies whether the original vectors are
recovered by the algorithm. We define it as
\begin{equation}
e_{2,i} := \norm{\varphi_i - \psi_i} .
\end{equation}
Because we are interested in worst-case behavior,
we run each test 200 times and report the worst
result.
\begin{figure}
\centering
\begin{tikzpicture}
\begin{semilogyaxis}[
    width=9cm,
    height=6cm,
    ymin=1e-18,
    ymax=1e3,
    xlabel=$i$,
    ylabel=$e_{1,i}$,
    legend pos=outer north east,
    grid=both,
    grid style={line width=.1pt, draw=gray!20},
    major grid style={line width=.2pt,draw=gray!50},
    ytick={1e-15,1e-12, 1e-9, 1e-6, 1e-3, 1},
    minor ytick={1e-18,1e-17,1e-16,1e-15, 1e-14, 1e-13, 1e-12, 1e-11, 1e-10, 1e-9, 1e-8, 1e-7, 1e-6, 1e-5, 1e-4, 1e-3, 1e-2, 1e-1, 1e0, 1e1, 1e2, 1e3},
  ]
  \addplot [solid, thick, mark=o, blue, mark repeat=10] table [x expr=\coordindex,y index=0] {code/gram_schmidt/Schmidt-GS.dat};
  \addplot [solid, thick, mark=none, red] table [x expr=\coordindex,y index=0] {code/gram_schmidt/Gram-GS.dat};
  \addplot +[solid, thick, mark=square, oran, mark repeat=6] table [x expr=\coordindex,y index=0] {code/gram_schmidt/ReIt-GS_2.dat};
  \addplot +[solid, thick, mark=diamond, dunkelgruen, mark repeat=11] table [x expr=\coordindex,y index=0] {code/gram_schmidt/ReIt-GS_3.dat};
  \addplot +[densely dashed, thick, mark=none] table [x expr=\coordindex,y index=0] {code/gram_schmidt/AReIt-GS.dat};
\legend{
  Schmidt GS,
  Gram GS (MGS),
  Re-Iterated GS (2 Iterations),
  Re-Iterated GS (3 Iterations),
  Adaptive Re-Iterated GS
};
\end{semilogyaxis}
\end{tikzpicture}
\caption[Orthogonality Error of Gram-Schmidt algorithms.]
{Orthogonality Error of Gram-Schmidt algorithms.
(reproduction: \cref{repro:fig:gs stuff})
}
\label{fig:gs orthogonality error}
\end{figure}
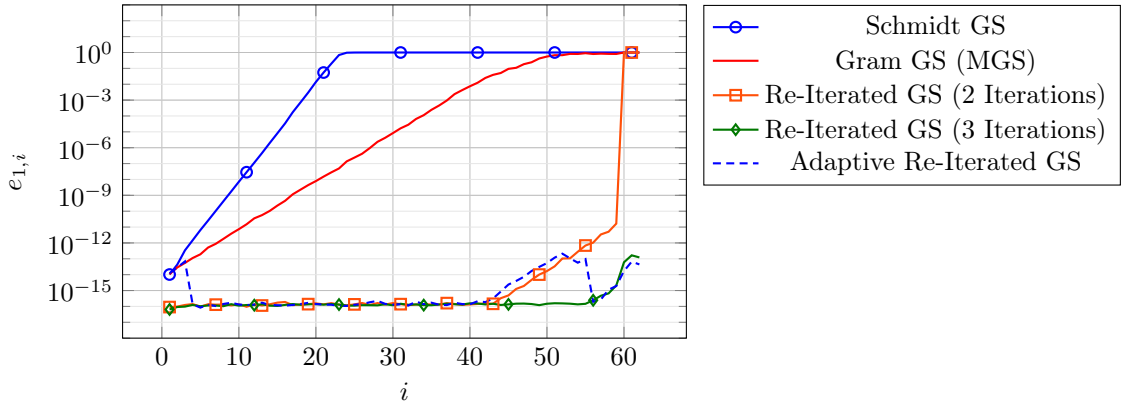
\begin{figure}
\centering
\begin{tikzpicture}
\begin{semilogyaxis}[
    width=9cm,
    height=6cm,
    ymin=1e-18,
    ymax=1e3,
    xlabel=$i$,
    ylabel=$e_{2,i}$,
    legend pos=outer north east,
    grid=both,
    grid style={line width=.1pt, draw=gray!20},
    major grid style={line width=.2pt,draw=gray!50},
    ytick={1e-15,1e-12, 1e-9, 1e-6, 1e-3, 1},
    minor ytick={1e-18,1e-17,1e-16,1e-15, 1e-14, 1e-13, 1e-12, 1e-11, 1e-10, 1e-9, 1e-8, 1e-7, 1e-6, 1e-5, 1e-4, 1e-3, 1e-2, 1e-1, 1e0, 1e1, 1e2, 1e3},
  ]
  \addplot [solid, thick, mark=o, blue, mark repeat=10] table [x expr=\coordindex,y index=1] {code/gram_schmidt/Schmidt-GS.dat};
  \addplot [solid, thick, mark=none, red] table [x expr=\coordindex,y index=1] {code/gram_schmidt/Gram-GS.dat};
  \addplot +[solid, thick, mark=square, oran, mark repeat=6] table [x expr=\coordindex,y index=1] {code/gram_schmidt/ReIt-GS_2.dat};
  \addplot +[solid, thick, mark=diamond, dunkelgruen, mark repeat=11] table [x expr=\coordindex,y index=1] {code/gram_schmidt/ReIt-GS_3.dat};
  \addplot +[densely dashed, thick, mark=none] table [x expr=\coordindex,y index=1] {code/gram_schmidt/AReIt-GS.dat};
\legend{
  Schmidt GS,
  Gram GS (MGS),
  Re-Iterated GS (2 Iterations),
  Re-Iterated GS (3 Iterations),
  Adaptive Re-Iterated GS
};
\end{semilogyaxis}
\end{tikzpicture}
\caption[Reproduction Error of Gram-Schmidt algorithms.]
{Reproduction Error of Gram-Schmidt algorithms.
(reproduction: \cref{repro:fig:gs stuff})
}
\label{fig:gs reproduction error}
\end{figure}
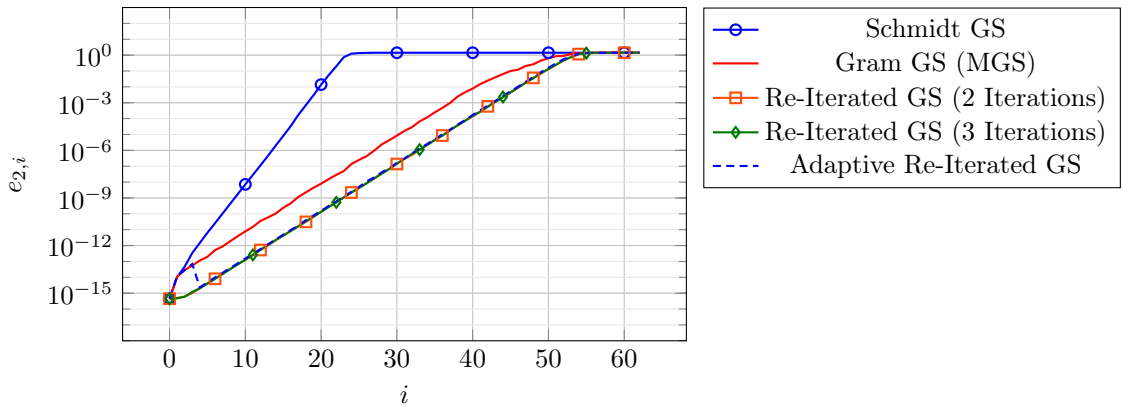
\begin{figure}
\centering
\begin{tikzpicture}
\begin{axis}[
    width=9cm,
    height=6cm,
    ymin=0,
    ymax=5,
    xlabel=$i$,
    ylabel=iterations (median),
    legend pos=outer north east,
    grid=both,
    grid style={line width=.1pt, draw=gray!20},
    major grid style={line width=.2pt,draw=gray!50},
  ]
  \addplot [densely dashed, thick, mark=none, blue, mark repeat=10] table [x expr=\coordindex,y index=0] {code/gram_schmidt/iterations.dat};
\legend{
  Adaptive Re-Iterated
};
\end{axis}
\end{tikzpicture}
\caption[Median of Iterations in Adaptive Re-Iterated Gram-Schmidt.]
{Median of Iterations in Adaptive Re-Iterated Gram-Schmidt
(reproduction: \cref{repro:fig:gs stuff})
}
\label{fig:gs num iterations}
\end{figure}
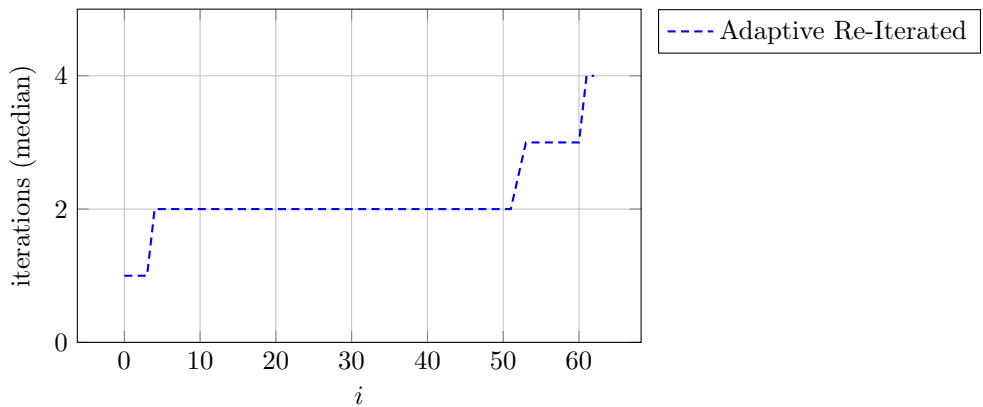
We can see in the results that
re-iteration is crucial for good numerical
performance.
The orthogonality error, plotted in
\cref{fig:gs orthogonality error}
stays at very low levels for all
re-iterated versions, while
it grows exponentially for
the two traditional variants.
Gram's version of the algorithm (the ``MGS'')
performs slightly better than
Schmidt's version.
The reproduction error grows
exponentially for all variants.
That is due to the construction of
the testvectors. The way they are
constructed looses information
about the later vectors. This
lost information can not be recovered
by the Gram-Schmidt algorithm.
All re-iterated variants have
a reproduction error which
is indistinguishable from
the optimal error, see 
\cref{fig:gs reproduction error}.
The two traditional variants
perform worse and again, 
Gram's version performs better
than Schmidt's version.
The adaptive variant
yields results as good as the re-iterated
algorithm with fixed number of iterations.
The median of the number of iterations 
performed by the adaptive variant is shown in
\cref{fig:gs num iterations}.

\section{Summary}
The Gram-Schmidt algorithm in the form
it was published by Gram and Schmidt
suffers from numerical instability
when applied to nearly linear dependent
vectors.
This problem can be mitigated by re-iteration,
effectively executing the algorithm several times.
In most cases, two iterations are sufficient.
Adaptive re-iteration has two advantages.
In comparison to a fixed number of two iterations,
it saves computational resources
when re-iteration is not necessary.
This might happen for example
when the input vectors are already nearly
orthogonal.
Second, it assures a low numerical error
on nearly linear dependent input vectors
by re-iterating as necessary.

\chapter{High Performance Nédélec Assembly}
At the core of each \gls{fe} software
is an assembly routine computing
local integrals on mesh elements.
Usually these integrals are computed
by a quadrature summation procedure.
We followed a different path
and implemented functions
to compute the local integrals based
on automatic symbolic computations
and code generation for 
lowest order Nédélec ansatz functions.
The results are presented in the following.
We list these results here for completeness.
There have been multiple publications
considering fast \gls{fe} assembly recently,
see e.g.~\cite{Melenk20014339,OlgaardWells2010b,hardware-exadune}.
\section{Problem Setting}
\label{sec:high performance assembly}
In \gls{fe} assembly, the core routine
has to compute the terms occurring in the bilinear
form on each mesh element $K$, for all combinations
of ansatz functions $\psi$ which have support on that element.
For lowest order Nédélec ansatz functions
(\cref{sec:nedelec ansatz})
and the bilinear form defined in 
\cref{eq:maxwell bilinear linear form}, it is the two
integrals
\begin{eqnarray}
I_{1,ij} &:=& \int_K \psi_i \psi_j \dx \label{eq:idid integral}\\
I_{2,ij} &:=& \int_K (\nabla \times \psi_i) \cdot (\nabla \times \psi_j) \dx \label{eq:curlcurl integral}
.
\end{eqnarray}
In order to implement a fast evaluation of these integrals,
we interpret the evaluation of these integrals as functions.
In two dimensions, a triangle is defined by the coordinates
of its corners, which can be interpreted as an element of $\gls{R}^6$.
The resulting matrices are $3 \times 3$ matrices, as three
basis functions have support on each triangle.
We define the two assembly functions 
\begin{align}
\text{ASSEMBLY2DIDID} : & \qquad \gls{R}^6 \rightarrow \gls{R}^{3 \times 3} \\
\text{ASSEMBLY2DCURLCURL} : & \qquad \gls{R}^6 \rightarrow \gls{R}^{3 \times 3}
\end{align}
mapping from the coordinates of the corners of the triangle to the
matrices containing the evaluated integrals, where
ASSEMBLY2DIDID corresponds to Equation \eqref{eq:idid integral} and\\
ASSEMBLY2DCURLCURL corresponds to Equation \eqref{eq:curlcurl integral}.
In three dimensions,
a tetrahedron is defined by the three coordinates of its four corners,
and there are six ansatz functions having support on each tetrahedron.
We define accordingly
\begin{align}
\text{ASSEMBLY3DIDID} : & \qquad \gls{R}^{12} \rightarrow \gls{R}^{6 \times 6} \\
\text{ASSEMBLY3DCURLCURL} : & \qquad \gls{R}^{12} \rightarrow \gls{R}^{6 \times 6}
.
\end{align}
The usual procedure to evaluate these functions is to use
a mapping to a reference element and 
a quadrature formula on the reference element.
The integrand is then evaluated at each quadrature point,
the evaluated values are added using quadrature weights.
We accelerated this procedure using automatic symbolic computation
and code generation.
The software package \gls{mathematica} can evaluate 
the integrals with the corner coordinates as symbolic
parameters. The resulting formulas can easily be transformed to
\Cpp code. Also a completely automatic code generation
would be possible, but yielded slightly inferior results.
\vspace{-5pt}
\section{Source Code}
The \gls{mathematica} code to generate
the assembly routines can be found in the directory\\
\texttt{code/high\_performance\_assembly/}.
In the file \texttt{simplex\_2d.nb},
the analytic expressions for triangles for both
integrals in question are computed.
In the file \texttt{simplex\_3d.nb},
the analytic expressions for tetrahedra for
both integrals are computed.
Note that only one entry of the element matrix
is computed. The entry can be chosen in the \gls{mathematica} file.
From \gls{mathematica}'s output,
the assembly routines
in the files
\texttt{simplex\_2d\_curlcurl.hh},
\texttt{simplex\_2d\_idid.hh},\\
\texttt{simplex\_3d\_curlcurl.hh}, and
\texttt{simplex\_3d\_idid.hh}
are created, containing the local assembly.
The code is templated, making it suitable
for both assembly using the \mintinline{cpp}{double}
data type or a high precision data type.
As an example, the two dimensional
curlcurl assembly is shown 
in 
\cref{listing:2dcurlcurl}.
\begin{listing}
\cppfile{code/high_performance_assembly/simplex_2d_curlcurl.hh}
\caption{C++ curlcurl assembly for a triangle.}
\label{listing:2dcurlcurl}
\end{listing}
\vspace{-5pt}
\section{Quality Control}
\begin{table}
\begin{center}
\begin{tabular}{r|c|c|c}
Operator & dune-pdelab & proposed & factor \\
\hline
    IDID 3D & 11 003 & 770 & 14.3\\
CURLCURL 3D & 11 260 & 454 & 24.8\\
    IDID 2D &  4 336 & 139 & 31.2\\
CURLCURL 2D &  1 704 &  41 & 41.6
\end{tabular}
\end{center}
\vspace{-5pt}
\caption[Number of instructions for one local integral evaluation]
{Number of instructions for one local matrix evaluation (Valgrind measurement).
(reproduction: \cref{repro:tab:assemblyperformance})
}
\label{tab:assemblyperformance}
\end{table}
To check the correctness of 
the assembly functions, a \gls{mathematica}
file was written which computes the local integrals
on a set of random elements and generates \Cpp
code with the reference results.
All numbers are written as strings in 
predefined accuracy. 
The four files\\
 \texttt{simplex\_(2|3)d\_(idid|curlcurl)\_testcases.nb}
generate the four header files\\
 \texttt{simplex\_(2|3)d\_(idid|curlcurl)\_testcases.hh},
containing the test cases.
The \gls{mathematica} notebooks require 
the notebooks \texttt{simplex\_2d.nb} or
\texttt{simplex\_3d.nb} to be executed
before.
Based on these files, the results
are compared with
reference results generated by \gls{mathematica}.
\section{Performance Measurement}
We evaluated the performance of the generated routines
by executing them in Valgrind \cite{valgrind} / Callgrind.
Valgrind simulates a CPU and outputs 
the number of instructions used
in this simulated CPU.
In comparison to timing based performance measurements,
Valgrind measurements are exactly reproducible.
The measured numbers are given in \cref{tab:assemblyperformance}.
The proposed code is between 14.3 and 41.6 times faster
than the dune-pdelab implementation.

\chapter[Handling of Complex Numbers in pyMOR]
{Handling of Complex Numbers in \gls{pymor}}
\label{chap:complex handling in pymor}
As part of this research project, the software package \gls{pymor}
\cite{Milk2016} was adapted to handle complex numbers.
A question is to decide whether the inner product
should be antilinear in its {\color{blue}first} or in its {\color{red}second} argument.
Both are possible choices, but it is beneficial to be as
consistent as possible with the rest of the world.
Because this question arises very often in discussions,
we document our findings here.
A short survey yields the following:
\begin{itemize}
\item
Software packages:
\begin{itemize}
\item Numpy \cite{numpy} does no complex conjugation in its ``numpy.inner'' function.
It does, however, conjugate the {\color{blue}first} argument in its ``numpy.vdot'' function.
\item DUNE \cite{Bastian2008} does complex conjugation on the {\color{blue}first} argument in
the ``dot'' function, see the file ``dune/common/dotproduct.hh'', line 16
in git hash\\
3091dfd1d65f9c50b9cb5e56c0cde4c9614ba816.
\item Fenics \cite{AlnaesBlechta2015a} does complex conjugation on the {\color{blue}first}
argument, see for example the file ``dolfin/la/EigenVector.cpp'', line 252, 
in\\git hash b99d156dfcbf97d958f539172bbf74addeef1b4e, 
where the arguments are forwarded in the same order to the ``dot'' function of Eigen, which does complex conjugation on the
{\color{blue}first} argument.
\item Eigen \cite{eigenweb} does complex conjugation on the {\color{blue}first} argument,
see the file ``Eigen/src/Core/Dot.h'', line 61 
in\\Mercurial hash 1b5aaf8108e8594083d45d43c08c2c113748cbda.
\item Deal.ii \cite{dealII90} 
is not consistent. In 
git hash\\
884c11cbce4595fea7337aa72e6745f928879247, one finds the following.
\begin{itemize}
\item
It does complex conjugation in the {\color{red}second} argument,
see file\\
``dealii/include/deal.II/lac/vector.h'', line 438.
\item
It does complex conjugation in the {\color{blue}first} argument,
see file \\``dealii/include/deal.II/lac/petsc\_vector\_base.h'', line 518.
\item
It does complex conjugation in the {\color{red}second} argument, 
see file \\ ``dealii/include/deal.II/lac/vector\_operations\_internal.h'', line 744.
\end{itemize}
\item Matlab does complex conjugation in the {\color{blue}first} argument, tested with version R2017a.
\item PetSc \cite{petsc-web-page}:
The function ``VecDot'' does complex conjugation of the {\color{red}second} argument.
see the file ``petsc/src/vec/vec/interface/rvector.c'' line 93 in git commit
57f8f5fd4131aa625156f794bdbd51339abe2d6c.
\item In BLAS there are the functions
``zdotu'' doing no complex conjugation and ``zdotc'' doing complex
conjugation of the {\color{blue}first} argument.
\end{itemize}
\item
Publications
\begin{itemize}
\item Barbara Verfürth
defines the sesquilinear form to conjugate the test function, see for example
\cite[equation (2.2)]{Ohlberger2018a}.
\item Jan Heasthaven defines the inner product to be antilinear in the
second argument and the sesquilinear form to be antilinear in the test function
in \cite[page 973]{Chen2010}.
\item Peter Monk defines the sesquilinear form to be antilinear in the test function
in his book ``Finite Element Methods for Maxwell's Equations'' \cite[equation (4.1)]{monk2003finite}.
\item Dolean et.al.~define the sesquilinear form to conjugate the test function in \cite{DOLEAN2008435}.
\item C{\'e}a, Jean defines the sesquilinear form to conjugate the test function in \cite{cea1964approximation}.
\end{itemize}
\end{itemize}
We observe that there is no consensus
whether the inner product should be antilinear in the first or
second argument.
Since most software have implemented it antilinear in the
first argument, \gls{pymor} follows this route and
has the inner product antilinear in the first argument.

In the formulation of the sesquilinear form it is
common to do a complex conjugation of the test function.
\gls{pymor} follows this pattern.

This implies that the functionals used as right hand side
are antilinear.

\backmatter
\pagestyle{gloss}
\printglossary[type=main,title=List of Symbols]             %
\printglossary[type=acronym]                                %
\bibliographystyle{bibgen}
\clearpage
\pagestyle{bib}
\bibliography{refs}
\cleardoublepage
\vfill
\hfill
\cleardoublepage

\end{document}